\documentclass[12pt,oneside]{amsart}

\usepackage{geometry}                
\usepackage{graphicx}
\usepackage{amssymb}
\usepackage{color}
\usepackage{bbm, dsfont}
\usepackage[hidelinks]{hyperref}

\hypersetup{
	colorlinks=false,
	pdfborder={0 0 0},
	pdfborderstyle={/S/U/W 0},
}

\numberwithin{equation}{section}

\usepackage{verbatim}

\usepackage[dvipsnames]{xcolor}

\usepackage{ulem}
\newtheorem{theorem}{Theorem}[section]
\newtheorem{lemma}[theorem]{Lemma}
\newtheorem{corollary}[theorem]{Corollary}
\newtheorem{proposition}[theorem]{Proposition}

\theoremstyle{definition}

\theoremstyle{remark}
\newtheorem{remark}[theorem]{Remark}

\newtheorem*{remark*}{Note}

\numberwithin{equation}{section}

\usepackage{amsmath}
\usepackage{mathrsfs}

\usepackage{verbatim}
\usepackage{amsmath,amssymb,amsthm}           
\usepackage{amssymb}            
\usepackage{amsfonts}            
\usepackage{mathrsfs}          
\usepackage{amsthm}
\usepackage{algorithm}  
\usepackage{algorithmicx}  
\usepackage{algpseudocode}  
\usepackage{mathtools}
\usepackage{commath}
\usepackage{physics}
\usepackage{bm}
\usepackage{graphicx}

\usepackage{float}
\usepackage{listings}
\usepackage{subfigure}
\usepackage{multirow}
\usepackage{color}
\usepackage{bbm}
\usepackage[export]{adjustbox}
\usepackage{enumerate}
\usepackage{bbm}

\usepackage{amsmath}
\usepackage{mathrsfs}

\newcommand{\RNum}[1]{\uppercase\expandafter{\romannumeral #1\relax}}

\newcommand{\specificthanks}[1]{\@fnsymbol{#1}}

\DeclareFontFamily{OML}{rsfs}{\skewchar\font'177}
\DeclareFontShape{OML}{rsfs}{m}{n}{ <5> <6> rsfs5 <7> <8> <9>
	rsfs7 <10> <10.95> <12> <14.4> <17.28> <20.74> <24.88> rsfs10 }{}
\DeclareMathAlphabet{\mathfs}{OML}{rsfs}{m}{n}

\newcounter{cnstcnt}
\newcommand{\cl}{%
	\refstepcounter{cnstcnt}%
	\ensuremath{c_{\thecnstcnt}}}
\newcommand{\cref}[1]{\ensuremath{c_{\ref*{#1}}}}

\newcounter{newcnstcnt}
\newcommand{\Cl}{%
	\refstepcounter{newcnstcnt}%
	\ensuremath{C_{\thenewcnstcnt}}}
\newcommand{\Cref}[1]{\ensuremath{C_{\ref*{#1}}}}

\DeclareFontFamily{U}{mathx}{}
\DeclareFontShape{U}{mathx}{m}{n}{<-> mathx10}{}
\DeclareSymbolFont{mathx}{U}{mathx}{m}{n}
\DeclareMathAccent{\widehat}{0}{mathx}{"70}
\DeclareMathAccent{\widecheck}{0}{mathx}{"71}

\begin{document}

	\title{Separation and cut edge in macroscopic clusters for metric graph Gaussian free fields}

		\author{Zhenhao Cai$^1$}
		\address[Zhenhao Cai]{Faculty of Mathematics and Computer Science, Weizmann Institute of Science}
		\email{zhenhao.cai@weizmann.ac.il}
		\thanks{$^1$Faculty of Mathematics and Computer Science, Weizmann Institute of Science}

		\author{Jian Ding$^2$}
		\address[Jian Ding]{New Cornerstone Science Laboratory, School of Mathematical Sciences, Peking University}
		\email{dingjian@math.pku.edu.cn}
	\thanks{$^2$New Cornerstone Science Laboratory, School of Mathematical Sciences, Peking University}
	
	
	\maketitle
	%
	%

	 	\begin{abstract}
	 	 We prove that for the Gaussian free field (GFF) on the metric graph of $\mathbb{Z}^d$ (for all $d\ge 3$ except the critical dimension $d_c=6$), with uniformly positive probability there exist two distinct sign clusters of diameter at least $cN$ within a box of size $N$ such that their graph distance is less than $N^{-[(d-2)\vee (2d-8)]}$. This phenomenon contrasts sharply with the two-dimensional case, where the distance between two macroscopic clusters is typically comparable to their diameters, following from the basic property of the scaling limit ``conformal loop ensembles'' $\mathrm{CLE}_4$ (Sheffield-Werner'2001).


 	          As a byproduct, we derive that the number of pivotal edges for the one-arm event (i.e., the sign cluster containing the origin has diameter at least $N$) is typically of order $N^{(\frac{d}{2}-1)\land 2}$. This immediately implies that for the incipient infinite cluster (IIC) of the metric graph GFF, the dimension of cut edges (i.e., edges whose removal disconnects the IIC) equals $(\frac{d}{2}-1)\land 2$. Translated in the language of critical loop soups (whose clusters, by the isomorphism theorem, have the same distribution as GFF sign clusters), this leads to the analogous estimates where the counterpart of a pivotal edge is a pivotal loop at scale $1$. This result hints at the new and possibly surprising idea that already in dimension $3$, microscopic loops (even those at scale $1$) play a crucial role in the construction of macroscopic loop clusters.



  	 	\end{abstract}

\section{Introduction}\label{section_intro}

In this paper, we study the geometric properties of Gaussian free field (GFF) sign clusters on the metric graph $\widetilde{\mathbb{Z}}^d$. For precision, we first recall the definitions of $\widetilde{\mathbb{Z}}^d$ and the GFF on this graph. Unless otherwise specified, we assume that $d\ge 3$. For the $d$-dimensional integer lattice $\mathbb{Z}^d$, we denote its edge set by $\mathbb{L}^d:=\{\{x,y\}:x,y\in \mathbb{Z}^d, \|x-y\|=1\}$, where $\|\cdot \|$ represents the Euclidean distance. For each edge $e=\{x,y\} \in \mathbb{L}^d$, consider a compact interval $I_e$ of length $d$, whose endpoints are identical to $x$ and $y$ respectively. The metric graph $\widetilde{\mathbb{Z}}^d$ is defined as the union of $I_e$ for all $e\in \mathbb{L}^d$. The GFF $\{\widetilde{\phi}_v\}_{v\in \widetilde{\mathbb{Z}}^d}$ can be generated as follows: 
\begin{enumerate}

	\item[(i)]  Sample a discrete GFF $\{\phi_x\}_{x\in \mathbb{Z}^d}$, i.e., mean-zero Gaussian random variables with covariance $\mathbb{E}\big[ \phi_x\phi_y \big] = G(x,y)$ for all $x,y\in \mathbb{Z}^d$, where $G(x,y)$ is the Green's function, defined as the expected number of visits to $y$ by a simple random walk on $\mathbb{Z}^d$ starting from $x$.

	\item[(ii)] For each $e=\{x,y\}\in \mathbb{L}^d$, independently sample a Brownian bridge on $I_e$ (associated with a Brownian motion with variance $2$ at time $1$) with boundary conditions $\phi_x$ at $x$ and $\phi_y$ at $y$ (sampled by Step (i)). For any $v\in \widetilde{\mathbb{Z}}^d$, let $\widetilde{\phi}_v$ be the value at $v$ of the Brownian bridge on the unique interval $I_e$ containing $v$.

\end{enumerate}
The level-set $\widetilde{E}^{\ge h}:=\{v\in \widetilde{\mathbb{Z}}^d: \widetilde{\phi}_v\ge h \}$ for $h\in \mathbb{R}$ have been extensively studied. A fundamental property of $\widetilde{E}^{\ge h}$ is that the critical threshold for percolation exactly equals zero. Precisely, it was proved in \cite{lupu2016loop} that $\widetilde{E}^{\ge h}$ percolates (i.e., contains an infinite connected component) if and only if $h<0$. The key ingredient for proving this critical threshold is the following two-point function estimate: 
\begin{equation}\label{two-point1}
 	\mathbb{P}\big(x\xleftrightarrow{\ge 0} y  \big)= \pi^{-1}\arcsin\Big( \tfrac{G(x,y)}{\sqrt{G(x,x)G(y,y)}} \Big)\asymp \|x-y\|^{2-d},\ \  \forall x\neq y \in \mathbb{Z}^d. 
 \end{equation} 
 Here $A_1\xleftrightarrow{\ge 0} A_2$ denotes the event that that there exists a path in $\widetilde{E}^{\ge 0}$ connecting $A_1$ and $A_2$, and $f\asymp g$ means that $f$ and $g$ satisfy $cg\le f\le Cg$ for some constants $C>c>0$ depending only on $d$. This estimate has since served as a starting point for a quantitative analysis of the growth of critical level-sets. In particular, the exact order of the (critical) one-arm probability has been established through a series of works \cite{ding2020percolation, drewitz2023critical, cai2024high, drewitz2023arm, drewitz2024critical, cai2024one}. Specifically, for the one-arm probability $\theta_d(N):=\mathbb{P}\big(\bm{0}\xleftrightarrow{\ge 0} \partial B(N)\big)$, where $\bm{0}:=(0,0,\cdots,0)$ is the origin of $\mathbb{Z}^d$, $B(N):=[-N,N]^d\cap \mathbb{Z}^d$ is the box of side length $\lfloor 2N \rfloor$ centered at $\bm{0}$, and $\partial A:=\{x\in A: \exists y\in \mathbb{Z}^d\setminus A\ \text{such that}\ \{x,y\}\in \mathbb{L}^d\}$ is the boundary of $A\subset \mathbb{Z}^d$, it is known that 
  \begin{align}
	&\text{when}\ 3\le d<6,\ \ \ \ \ \ \ \ \  \theta_d(N) \asymp N^{-\frac{d}{2}+1};\label{one_arm_low} \\
	&\text{when}\ d=6,\ \ \ \ \  N^{-2}\lesssim \theta_6(N) \lesssim N^{-2+\varsigma(N)}, \  \text{where}\ \varsigma(N):= \tfrac{\ln\ln(N)}{[\ln(N)]^{1/2}}\ll 1; \label{one_arm_6} \\
	&\text{when}\ d>6,\ \ \ \ \ \ \ \ \ \ \ \ \ \ \ \theta_d(N) \asymp N^{-2}.\label{one_arm_high}
\end{align} 
  Here $f\lesssim g$ means that the functions $f$ and $g$ satisfy $f\le Cg$ for some constant $C>0$ depending only on $d$. Notably, it has been conjectured in \cite{cai2024one} that $\theta_6(N)\asymp N^{-2}[\ln(N)]^\delta$ for some $\delta>0$. To avoid the additional complexity caused by the diverging disparity between the current upper and lower bounds on $\theta_6(N)$, we assume that $d\neq 6$ unless stated otherwise. (The techniques developed for $d\neq 6$ can usually be extended to $d=6$, leading to weaker estimates with error terms similar to that in (\ref{one_arm_6}).) As a natural extension of the one-arm probability $\theta_d(N)$, the crossing probability $\rho_d(n,N):=\mathbb{P}\big(B(n)\xleftrightarrow{\ge 0}\partial B(N)\big)$ for $N> n\ge 1$ was also computed. Precisely, it has been proved in \cite{cai2024one} that 
   \begin{align}
	&\text{when}\ 3\le d<6,\  \rho_d(n,N) \asymp \big(n/N\big)^{\frac{d}{2}-1};\label{crossing_low} \\
	&\text{when}\ d>6, \ \ \ \ \   \ \    \rho_d(n,N) \asymp  (n^{d-4}N^{-2})\land 1. \label{crossing_high}
\end{align}
Moreover, other quantities, such as volume exponent \cite{cai2024quasi, drewitz2024cluster} and cluster dimension \cite{cai2024incipient, drewitz2024cluster}, were also studied. At a more structural level, the incipient infinite cluster (IIC), a key concept introduced by \cite{kesten1986incipient} to understand the behavior of large clusters at criticality, was constructed and analyzed in \cite{cai2024quasi, cai2024incipient, werner2025switching}. Specifically, building on the quasi-multiplicativity established in \cite{cai2024quasi}, it was proved in \cite{cai2024incipient} that for any $d\ge 3$ with $d\neq 6$, the following four limiting measures all exist and are the same: 
\begin{equation}\label{iic1}
\mathbb{P}^{(1)}_{d,\mathrm{IIC}}(\cdot):=	\lim\limits_{N\to \infty}\mathbb{P} \big( \cdot   \mid  \bm{0}\xleftrightarrow{\ge 0} \partial B(N)  \big) ,	 
\end{equation}
\begin{equation}\label{iic2}
	\mathbb{P}^{(2)}_{d,\mathrm{IIC}}(\cdot):=	\lim\limits_{h \uparrow 0}\mathbb{P} \big( \cdot   \mid  \bm{0}\xleftrightarrow{\ge h } \infty  \big) ,	 
\end{equation}
 \begin{equation}\label{iic3}
	\mathbb{P}^{(3)}_{d,\mathrm{IIC}}(\cdot):=	\lim\limits_{x \to \infty}\mathbb{P} \big( \cdot   \mid  \bm{0}\xleftrightarrow{\ge 0} x   \big), 
\end{equation} 
\begin{equation}\label{iic4}
\mathbb{P}^{(4)}_{d,\mathrm{IIC}}(\cdot):=	\lim\limits_{T\to \infty}\mathbb{P} \big( \cdot   \mid  \mathrm{cap}(\mathcal{C}_{\bm{0}}^+ )\ge T  \big),  
\end{equation}
where $\mathrm{cap}(\mathcal{C}_{\bm{0}}^+ )$ represents the capacity of the cluster $\mathcal{C}_{\bm{0}}^+:=\{v\in \widetilde{\mathbb{Z}}^d: v\xleftrightarrow{\ge 0} \bm{0}\}$. The IIC refers to the random cluster containing $\bm{0}$ under any of these limiting measures. Notably, the switching identity introduced in \cite{werner2025switching} (see the precise statement in Lemma \ref{lemma_switching}) provides an elegant proof of the existence of $\mathbb{P}^{(3)}_{d,\mathrm{IIC}}(\cdot)$ for all $d\ge 3$. Although this method is only applicable to $\mathbb{P}^{(3)}_{d,\mathrm{IIC}}(\cdot)$, it remarkably yields an explicit description of the IIC---the cluster consisting of a Brownian excursion from $\bm{0}$ to infinity and a loop soup on $\widetilde{\mathbb{Z}}^d\setminus \{\bm{0}\}$ of intensity $\frac{1}{2}$ (the definitions of Brownian excursion and loop soup will be provided below). Random walks and the chemical distance on the IIC in high dimensions were studied in \cite{ganguly2024ant, ganguly2024critical}.


      The derivation of many of the aforementioned results relies crucially on the isomorphism theorem, which establishes a rigorous equivalence between the GFF and the loop soup. Specifically, the loop soup $\widetilde{\mathcal{L}}_{\alpha}$ of intensity $\alpha>0$ is a Poisson point process with intensity measure $\alpha\widetilde{\mu}$. Here $\widetilde{\mu}$ (referred to as the loop measure) is a $\sigma$-finite measure on the space of rooted loops (i.e., continuous paths on $\widetilde{\mathbb{Z}}^d$ that start and end at the same point), defined by 
      \begin{equation}\label{def_mu}
	\widetilde{\mu}(\cdot) := \int_{v\in \widetilde{\mathbb{Z}}^d} \mathrm{d}\mathrm{m}(v) \int_{0< t< \infty} t^{-1} \widetilde{q}_t(v,v)\widetilde{\mathbb{P}}^t_{v,v}(\cdot) \mathrm{d}t,
\end{equation}
where $\mathrm{m}(\cdot)$ is the Lebesgue measure on $\widetilde{\mathbb{Z}}^d$, $\widetilde{q}_t(v_1, v_2)$ is the transition density of the Brownian motion on $\widetilde{\mathbb{Z}}^d$ (see Section \ref{section_notation} for its precise definition), and $\widetilde{\mathbb{P}}^t_{v_1,v_2}(\cdot)$ denotes the law of the Brownian bridge on $\widetilde{\mathbb{Z}}^d$ from $v_1$ to $v_2$ with duration $t$ (its transition density is given by $\frac{\widetilde{q}_s(v_1,\cdot )\widetilde{q}_{t-s}(\cdot,v_2)}{\widetilde{q}_t(v_1, v_2)}$ for $0\le s\le t$). The isomorphism theorem (see \cite[Proposition 2.1]{lupu2016loop}) constructs a coupling between the GFF $\{\widetilde{\phi}_v\}_{v\in \widetilde{\mathbb{Z}}^d}$ and the critical loop soup $\widetilde{\mathcal{L}}_{1/2}$ (the criticality of the intensity $\frac{1}{2}$ for percolation was verified in \cite{lupu2016loop, chang2024percolation}) satisfying the following properties: 
\begin{enumerate}

	\item[(a)] $\widehat{\mathcal{L}}_{1/2}^v=\frac{1}{2}\widetilde{\phi}_v^2$ for all $v\in \widetilde{\mathbb{Z}}^d$, where $\widehat{\mathcal{L}}_{1/2}^v$ represents the total local time at $v$ of all loops in $\widetilde{\mathcal{L}}_{1/2}$. In particular, each sign cluster (i.e., a maximal connected subgraph where GFF values have the same sign) is exactly a (critical) loop cluster (i.e., a connected component consisting of loops in $\widetilde{\mathcal{L}}_{1/2}$).

	\item[(b)]  Given all loop clusters, the sign of the GFF values on each loop cluster is independent of the others, taking ``$+$'' (or ``$-$'') with probability $\frac{1}{2}$.

\end{enumerate}

\subsection{Graph distance between macroscopic clusters}

For most discrete stochastic models endowed with spatial structures, constructing the scaling limit usually constitutes a central research objective. Particularly, in the  two-dimensional case, the scaling limit of loop clusters (or equivalently, GFF sign clusters) has been established. Precisely, in the celebrated work \cite{sheffield2012conformal}, the conformal loop ensembles $\mathrm{CLE}_\kappa$ (where the parameter $\kappa\in (\frac{8}{3},4]$), defined as the boundaries of (continuum) loop clusters in the Brownian loop soup of intensity $\frac{(3\kappa-8)(6-\kappa)}{4\kappa}$, were shown to consist of loops in the form of Schramm-Loewner evolution ($\mathrm{SLE}_\kappa$). Afterwards, it was proved in \cite{lupu2019convergence} that for the discrete half-plane and the corresponding metric graph, the boundaries of (discrete) loop clusters converge to $\mathrm{CLE}_4$. In particular, this together with the property of $\mathrm{CLE}_4$ that every two loops do not touch each other (which follows from the non-crossing property and Property (1) of $\mathrm{CLE}_\kappa$ in \cite[Section 2.2]{sheffield2012conformal}) implies that within a box, the minimal distance between all macroscopic (discrete) loop clusters is macroscopic.

 

The main result of this paper, as a sharp contrast, shows that the aforementioned minimal distance is microscopic (and may even be $o(1)$!) when $d>2$.  Precisely, for $A\subset \widetilde{\mathbb{Z}}^d$, its diameter is defined as $\mathrm{diam}(A):=\sup_{v,w\in A}|v-w|$, where $|v-w|$ denotes the graph distance between $v$ and $w$ on $\widetilde{\mathbb{Z}}^d$. The distance between two sets $D,D'\subset \widetilde{\mathbb{Z}}^d$ is defined as $\mathrm{dist}(D,D'):=\inf_{v\in D,v'\in D'}|v-v'|$. For any $N\ge 1$ and $\delta>0$, we denote by $\mathfrak{C}_{N,\delta}$
the collection of sign clusters contained in $\widetilde{B}(N)$ with diameter at least $\delta N$, where $\widetilde{B}(N):=\cup_{e\in \mathbb{L}^d:I_e \cap (-N,N)^d\neq \emptyset}I_e$ is the box in $\widetilde{\mathbb{Z}}^d$ of radius $N$ centered at $\bm{0}$. We then define the minimal distance 
\begin{equation}\label{def_min_distance}
	\mathcal{D}_{N,\delta}:=\min\big\{\mathrm{dist}(\mathcal{C},\mathcal{C}' ) : \mathcal{C}\neq \mathcal{C}'\in \mathfrak{C}_{N,\delta} \big\}.
\end{equation}
 In addition, for any functions $f$ and $g$ that depend on $d$, we define the operation 
\begin{equation}
	f\boxdot g(d):= f(d)\cdot \mathbbm{1}_{d\le 6}+ g(d)\cdot \mathbbm{1}_{d>6}. 
\end{equation}


 \begin{theorem}\label{thm_minimal_distance}
	For any $d\ge 3$ with $d\neq 6$, there exists a constant $\cl\label{const_minimal_distance}(d)>0$ such that for any $N\ge 1$ and $\chi>0$, 
	\begin{equation}\label{ineq_minimal_distance}
		\mathbb{P}\big(\mathcal{D}_{N,\cref{const_minimal_distance}}\le \chi  \big) \asymp (\chi^{\frac{1}{2}}N^{(\frac{d}{2}-1)\boxdot (d-4)})\land 1. 
	\end{equation}
\end{theorem}


The starting point for Theorem \ref{thm_minimal_distance} is the estimates on heterochromatic two-arm probabilities established in the companion paper \cite{inpreparation_twoarm}. Specifically, for any subsets $A_i\subset \widetilde{\mathbb{Z}}^d$ for $1\le i\le 4$, consider the event $\mathsf{H}^{A_1,A_2}_{A_3,A_4}:= \big\{A_1\xleftrightarrow{\ge 0} A_2,A_3\xleftrightarrow{\le 0} A_4  \big\}$. Here $A_3\xleftrightarrow{\le 0} A_4$ denotes the event that there exists a path in the negative cluster $\widetilde{E}^{\le 0}:=\{v\in \widetilde{\mathbb{Z}}^d: \widetilde{\phi}_v\le 0 \}$ connecting $A_3$ and $A_4$. It was proved in \cite{inpreparation_twoarm} that for any $N\ge C$ and $v,v'\in \widetilde{B}(cN)$,  
 \begin{align}
	 	 &\text{when}\ \chi=:|v-v'|\le 1, \  \mathbb{P}\big(\mathsf{H}_{v',\partial B(N)}^{v,\partial B(N)}  \big)\asymp \chi^{\frac{3}{2}}N^{-[(\frac{d}{2}+1)\boxdot 4]};   \label{thm1_small_n}\\
	&\text{when}\ \chi\ge 1,\   \ \ \ \  \  \ \ \ \ \ \ \   \  \  	\mathbb{P}\big(\mathsf{H}_{v',\partial B(N)}^{v,\partial B(N)} \big)\asymp \chi^{(3-\frac{d}{2})\boxdot 0}N^{-[(\frac{d}{2}+1)\boxdot 4]}.   \label{thm1_large_n} 
	\end{align}
  As noted in \cite[Remark 1.2]{inpreparation_twoarm}, these bounds (with $\chi \asymp 1$) imply that the expected number of edges within a box of size $N$, whose endpoints are contained in distinct sign clusters with diameter at least $N$, is of order $N^{(\frac{d}{2}-1)\boxdot (d-4)}$. This observation suggests that two macroscopic sign clusters may nearly touch at many locations, leading to an extremely small graph distance between them (as the configurations near different touching points are independent). Since the first moment of the number of touching points is well understood, this naturally motivates computing its second moment. In the present setting, the resulting second-moment bound turns out to be of the same order as the square of the first moment, allowing for an application of the second moment method to verify the abundance of touching points, from which the vanishing of the minimal distance $\mathcal{D}_{N,\delta}$ follows.


  As shown in Section \ref{section_proof_thm1.1}, estimating the second moment of the number of touching points reduces to controlling the probability, associated with the joint occurrence of two touching points, that two triples of points are contained in two distinct clusters. To this end, we decompose this event into connecting events of the form $\mathsf{H}^{v,w}_{v',w'}$ through a sequence of steps, as illustrated in Figures \ref{fig1}-\ref{fig3}. In practice, due to the long-range correlations of the GFF (or equivalently, the loop soup), carrying out such a decomposition typically necessitates a multi-scale analysis, which substantially increases the technical complexity of the argument. This complexity persists even in the high-dimensional case $d\ge 7$ (see \cite[Section 7]{cai2024high}). A related multiscale analysis also appears in the construction of the IICs (see \cite[Section 3]{cai2024incipient}).

{\color{blue} 


    
    }

	 {\color{red}
	

	
	}

  \begin{remark}[Bernoulli percolation]
  For critical Bernoulli percolation on the triangular lattice, the clusters have been well understood. Specifically, the groundbreaking work \cite{SMIRNOV2001239} proved that the boundary of a cluster converges to $\mathrm{SLE}_6$. Based on this result, \cite{camia2006two, camia2007critical} (see \cite{SLECLE} for a comprehensive review) constructed the full scaling limit of all cluster boundaries, consisting of continuum loops which locally behave like $\mathrm{SLE}_6$ curves. Moreover, as stated in \cite[Theorem 2]{camia2006two}, these loops touch each other frequently, which is consistent with the behavior of GFF sign clusters established in Theorem \ref{thm_minimal_distance}.



    \end{remark}

\begin{remark}[construction of the scaling limit]
 \cite[Conjecture A]{werner2021clusters} proposed that the scaling limit of loop clusters on $\widetilde{\mathbb{Z}}^3$ is exactly the clusters of the Brownian loop soup of intensity $\frac{1}{2}$ on $\mathbb{R}^3$. As explained in \cite[Section 3]{werner2021clusters}, this conjecture corresponds to the intuition that the contribution of microscopic loops to the construction of macroscopic clusters is negligible in the scaling limit. Indeed, assuming this intuition, each macroscopic cluster depends essentially only on the macroscopic loops in $\widetilde{\mathcal{L}}_{1/2}$. As established in \cite{lawler2007random}, these loops admit a bijective approximation by continuum loops in the corresponding Brownian loop soup, (plausibly) leading to the convergence between the discrete and continuum clusters in \cite[Conjecture A]{werner2021clusters}. However, one can notice that in the approximation scheme of \cite{lawler2007random}, the Hausdorff distance between a discrete loop (of diameter $N$) and its corresponding continuum loop is typically $O(\log(N))$ (this bound is essentially sharp, referring to \cite[Theorem 2]{KMT_sharp}). Consequently, when the graph distance between two (discrete) loop clusters of diameters $N$ is substantially smaller than $O(\log(N))$ (which, by Theorem \ref{thm_minimal_distance}, occurs rather frequently), the error in the approximation of loops will, theoretically, disrupt the correspondence between discrete and continuum clusters. So, Theorem \ref{thm_minimal_distance} already indicates that there is a problem with this conjecture, as inherent fluctuations in the approximation scheme provide different outcomes for the limits of metric graph loop soups associated to a given continuum loop soup.


 In a companion paper \cite{inpreparation_gap}, we confirm these ideas in more precise terms and show that \cite[Conjecture A]{werner2021clusters} does indeed not hold. More precisely, we show that removing microscopic loops (even those at scale 1) does significantly reduce the connecting probabilities for loop soup clusters, and further show that the dimension of clusters in the Brownian loop soup in $\mathbb{R}^3$ is actually strictly smaller than the one arising from the scaling limits of metric graph clusters.


\end{remark}

\subsection{Pivotal edge and loop} 
In percolation models, a bit (vertex in site percolation, or edge in bond percolation) is called pivotal for a connecting event (i.e., an event of the form $A_1\xleftrightarrow{} A_2$) if flipping its state changes the occurrence of this event. As emphasized in many significant articles (e.g., \cite[Section 1]{garban2013pivotal}), the pivotal bit is a key concept to understand the near-critical behavior in percolation. For planar Bernoulli percolation, the establishment of the scaling limit \cite{SMIRNOV2001239} has led to a deep understanding of the distribution of pivotal points \cite{garban2013pivotal}. See also \cite{benjamini1999noise, garban2010fourier, schramm2010quantitative} for applications of pivotal points to noise sensitivity and dynamical percolation. For Bernoulli percolation on non-planar graphs and other percolation models, our understanding of pivotal bits remains rather limited. In the context of this paper, we say an edge $e\in \mathbb{L}^d$ is pivotal for the event $\big\{\bm{0}\xleftrightarrow{\ge 0} \partial B(N)\big\}$ if $\big\{\bm{0}\xleftrightarrow{\ge 0} \partial B(N)\big\}$ occurs but $\big\{\bm{0}\xleftrightarrow{\widetilde{E}^{\ge 0}\setminus I_e} \partial B(N)\big\}$ does not. Let $\mathbf{Piv}_N^+$ denote the collection of all these edges. In the next result, we show that for the one-arm event $\big\{  \bm{0}\xleftrightarrow{\ge 0} \partial B(N)\big\}$, the number of pivotal edges is typically of order $N^{(\frac{d}{2}-1)\boxdot 2}$:



\begin{theorem}\label{thm_1.5}
	For any $d\ge 3$ with $d\neq 6$, and any $\epsilon\in (0,1)$, there exist constants $\cl\label{const_GFF_pivotal1}(d,\epsilon),\Cl\label{const_GFF_pivotal3}(d,\epsilon),\Cl\label{const_GFF_pivotal4}(d,\epsilon)>0$ such that for any $N\ge \Cref{const_GFF_pivotal3}$, 
	\begin{equation}\label{new115}
		\mathbb{P}\big(  | \mathbf{Piv}_N^+  | \in \big[\cref{const_GFF_pivotal1} N^{(\frac{d}{2}-1)\boxdot 2} ,\Cref{const_GFF_pivotal4} N^{(\frac{d}{2}-1)\boxdot 2} \big] \mid \bm{0}\xleftrightarrow{\ge 0} \partial B(N) \big)\ge 1-\epsilon.
			\end{equation}
\end{theorem}

As a direct corollary of Theorem \ref{thm_1.5}, the dimension of cut edges in the IIC (as defined by any of the limiting measures in (\ref{iic1})-(\ref{iic4})) equals $(\frac{d}{2}-1)\boxdot 2$. 


 From the perspective of the loop soup, a loop is called pivotal for a connecting event if removing it from the loop soup changes the occurrence of this event. To imitate the pivotal edge in bond percolation, we restrict the discussion to the loops contained in a single interval $I_e$ and intersecting both trisection points of $I_e$ (let $\mathfrak{L}$ denote the collection of all these loops). For any $A_1,A_2\subset \widetilde{\mathbb{Z}}^d$, let $A_1\xleftrightarrow{} A_2$ denote the event that there exists a loop cluster of $\widetilde{\mathcal{L}}_{1/2}$ connecting $A_1$ and $A_2$. For any $N\ge 1$, we define the edge collection 
  \begin{equation}\label{close_def_Piv_N}
	\widetilde{\mathbf{Piv}}_N:= \big\{ e\in \mathbb{L}^d: \exists \ \text{pivotal}\ \widetilde{\ell}\in \mathfrak{L}\ \text{contained in}\ I_e\ \text{for}\  \bm{0} \xleftrightarrow{} \partial B(N) \big\}. 
\end{equation}  
 Next, we present the analogue of Theorem \ref{thm_1.5} for the loop soup $\widetilde{\mathcal{L}}_{1/2}$.


\begin{theorem}\label{thm_pivotal_loop}
For any $d\ge 3$ with $d\neq 6$, and any $\epsilon>0$, there exist constants $\cl\label{const_pivotal_1}(d,\epsilon),\Cl\label{const_pivotal_2}(d,\epsilon),\Cl\label{const_pivotal_3}(d,\epsilon)>0$ such that for any $N\ge \Cref{const_pivotal_2}$, 
	\begin{equation}\label{new116}
		\mathbb{P}\big(  \big|\widetilde{\mathbf{Piv}}_N\big|\in \big[\cref{const_pivotal_1} N^{(\frac{d}{2}-1)\boxdot 2} ,\Cref{const_pivotal_3} N^{(\frac{d}{2}-1)\boxdot 2} \big]  \mid \bm{0} \xleftrightarrow{} \partial B(N) \big)\ge 1-\epsilon.
	\end{equation}
\end{theorem}

We outline the proof of Theorem \ref{thm_pivotal_loop} as follows. On the one hand, based on the two-arm estimate in (\ref{thm1_large_n}), we show that for each edge $e=\{x,y\}$ such that $|x|\asymp \mathrm{dist}(x,\partial B(N))\asymp N$, the probability that $x$ and $y$ belong to two distinct loop clusters, one intersecting $\bm{0}$ and the other reaching $\partial B(N)$, is of order $N^{-d}$. Consequently, conditioned on $\{\bm{0} \xleftrightarrow{} \partial B(N)\}$, the expectation of $\big|\widetilde{\mathbf{Piv}}_N\big|$ is proportional to $[\theta_d(N)]^{-1}\asymp N^{(\frac{d}{2}-1)\boxdot 2}$. This together with Markov's inequality implies that $\big|\widetilde{\mathbf{Piv}}_N\big|\le CN^{(\frac{d}{2}-1)\boxdot 2}$ occurs with high probability. On the other hand, we show that the conditional probability of $\big|\widetilde{\mathbf{Piv}}_N\big|\le cN^{(\frac{d}{2}-1)\boxdot 2}$ vanishes uniformly as $c\to 0$ in the following two steps:

\begin{enumerate}

	\item[(i)] We apply the second moment method to show that for all sufficiently small $a>0$, there exists $\lambda(a)>0$ (increasing in $a$) such that conditioned on $\{\bm{0} \xleftrightarrow{} \partial B(N)\}$, with a uniformly positive probability $\delta>0$ there exist at least $\lambda(a) N^{(\frac{d}{2}-1)\boxdot 2}$ pivotal edges within the annulus $B(CaN)\setminus B(aN)$. To estimate the second moment of $\big|\widetilde{\mathbf{Piv}}_N\big|$, it suffices to control, for all edges $e_1=\{x_1,y_1\}$ and $e_2=\{x_2,y_2\}$ in $B(N)$, the probability that the events $\{\bm{0}\xleftrightarrow{} x_1\}$, $\{x_2\xleftrightarrow{} \partial B(N)\}$ and $\{y_1\xleftrightarrow{} y_2\}$ are certified by three distinct loop clusters. As in the proof of Theorem \ref{thm_minimal_distance}, we decompose events
of this type into connecting events of the form $\mathsf{H}^{v,w}_{v',w'}$, and bound their probabilities using the two-arm estimate in (\ref{thm1_large_n}).




   \item[(ii)]  We divide the box $B(N)$ into an increasing sequence of annuli $\{\widetilde{B}(C^{i+1}bN)\setminus \widetilde{B}(C^{i}bN)\}_{1\le i\le K}$, where we require $\lambda(b)>c$ and $K\gtrsim \log(1/b)$. Note that one may further require $b\to 0$ (and hence $K\to \infty$) as $c\to 0$. By adapting the decomposition of loop clusters in \cite[Section 5.1]{cai2024incipient}, we show that for the event $\{\bm{0} \xleftrightarrow{} \partial B(N)\}$, the numbers of edges containing a pivotal loop in $\mathfrak{L}$ are asymptotically independent across different annuli in $\{\widetilde{B}(C^{i+1}bN)\setminus \widetilde{B}(C^{i}bN)\}_{1\le i\le K}$. Combined with Step (i), it yields that the conditional probability of $\big|\widetilde{\mathbf{Piv}}_N\big|\le cN^{(\frac{d}{2}-1)\boxdot 2}$ is approximately upper-bounded by $(1-\delta)^{K}$, and thus vanishes as $c\to 0$ (since $K\to \infty$).


\end{enumerate}

As for Theorem \ref{thm_1.5}, note that each
$e\in \widetilde{\mathbf{Piv}}_N$ must be a pivotal edge for the event
$\{\bm{0}\xleftrightarrow{} \partial B(N)\}$; moreover, according to the
isomorphism theorem, the law of such pivotal edges coincides with that of
pivotal edges for the event $\big\{\bm{0}\xleftrightarrow{\ge 0} \partial B(N)\big\}$. In other words, $|\mathbf{Piv}_N^+|$ stochastically dominates $|\widetilde{\mathbf{Piv}}_N|$, and thus exceeds $cN^{(\frac{d}{2}-1)\boxdot 2}$ with high probability. The opposite domination fails because a pivotal edge for $\{\bm{0}\xleftrightarrow{} \partial B(N)\}$ does not necessarily contain a pivotal loop at scale $1$. Instead, such an edge may arise from the intersection of two loops which, away from this edge, lie in different loop clusters, with one cluster intersecting $\bm{0}$ while the other reaching $\partial B(N)$. To compute the contribution of this scenario to the first moment of $|\mathbf{Piv}_N^+|$, we need to control, for a given annulus, the probability that a loop cluster and a Brownian motion both cross this annulus and remain disjoint. Thanks to the explicit formula for the two-point function (see (\ref{211})), this probability can again be reduced to that of events of the form $\mathsf{H}^{v,w}_{v',w'}$, and the desired bounds then follow from the two-arm estimate (\ref{thm1_large_n}).

%

%
%
%
%
%
%

\textbf{Statements about constants.} The letters $C$ and $c$ denote positive constants, which are allowed to vary depending on context. Numerical constants (e.g., $C_1,c_2$) remain fixed throughout the paper. By convention, $C$ (with or without subscripts) denotes large constants, while $c$ represents small ones. Unless stated otherwise, a constant depends only on the dimension $d$. Dependence on other parameters must be explicitly indicated (e.g., $C(\lambda)$, $c(\delta)$).


\textbf{Organization of this paper.} In Section \ref{section_notation}, we fix the necessary notations and present some relevant results. In Section \ref{section_proof_thm1.1}, we prove Theorem \ref{thm_minimal_distance}, assuming a key estimate (see Proposition \ref{lemma_four_arms}) on the probability of two distinct sign clusters nearly touching at two given points. This estimate is then verified in Sections \ref{section_proof_two_branch}, thereby completing the proof of Theorem \ref{thm_minimal_distance}. Finally, we provide the proofs of Theorems \ref{thm_pivotal_loop} and \ref{thm_1.5} in Sections \ref{section_pivotal} and \ref{section_pivotal_dege} respectively.

  \section{Preliminaries}\label{section_notation}

In this section, we review some primary notations and useful results.

\subsection{Basic notation for graphs}

 \begin{itemize}


  	\item For any $v\in \widetilde{\mathbb{Z}}^d$, we denote by $\bar{v}$ the closest lattice point to $v$ (we break the tie in a predetermined manner).

 	\item  Recall that we have defined the boxes $B(N)$ and $\widetilde{B}(N)$ in Section \ref{section_intro}. We also denote the Euclidean ball $\mathcal{B}(N):=\{x\in \mathbb{Z}^d: \|x\|\le N \}$. For any $v\in \widetilde{\mathbb{Z}}^d$, we define $B_v(N):=\bar{v}+B(N)$, $\widetilde{B}_v(N):=\bar{v}+\widetilde{B}(N)$ and $\mathcal{B}_v(N):=\bar{v}+\mathcal{B}(N)$.

    \item For $N\ge 1$ and $\delta>0$, we denote $\mathfrak{B}^{1}_{N,\delta}:=\widetilde{B}(\delta N)$ and $\mathfrak{B}^{-1}_{N,\delta}:=[\widetilde{B}(\delta^{-1} N)]^{c}$.

 	\item For any $e\in \mathbb{L}^d$ and $v,w\in I_e$, we denote by $I_{[v,w]}$ the subinterval of $I_e$ with endpoints $v$ and $w$.

    \item We use the letter $D$ (possibly with superscripts or subscripts) to denote a subset of $\widetilde{\mathbb{Z}}^d$ consisting of finitely many compact connected components.

 	\item  The boundary of $D\subset \widetilde{\mathbb{Z}}^d$ is $\widetilde{\partial}D:=\{v\in D: \inf_{w\in \widetilde{\mathbb{Z}}^d\setminus D}|v-w|=0 \}$. Since $\widetilde{\mathbb{Z}}^d$ is locally one-dimensional, $\widetilde{\partial}D$ is a finite set.

 	\item  A path on $\widetilde{\mathbb{Z}}^d$ is a continuous function $\widetilde{\eta}:[0,T)\to \widetilde{\mathbb{Z}}^d$, where the duration $T\in \mathbb{R}^+\cup \{\infty\}$. When $T<\infty$, we denote $\widetilde{\eta}(T)=\lim_{t\uparrow T}\widetilde{\eta}(t)$. Recall that a rooted loop $\widetilde{\varrho}$ (of duration $T<\infty$) is a path satisfying $\widetilde{\varrho}(0)=\widetilde{\varrho}(T)$.

 	\item  For a path $\widetilde{\eta}$, we denote its range by $\mathrm{ran}(\widetilde{\eta}):=\{\widetilde{\eta}(t): 0\le t<T \}$.

 	\item For any set $D\subset \widetilde{\mathbb{Z}}^d$ and any path $\widetilde{\eta}$, we denote the first time when $\widetilde{\eta}$ hits $D$ by $\tau_{D}(\widetilde{\eta}):=\inf\{0\le t<T: \widetilde{\eta}(t)\in D\}$ (where we set $\tau_\emptyset:= +\infty$ for completeness). In particular, we abbreviate $\tau_{v}:=\tau_{\{v\}}$ for $v\in \widetilde{\mathbb{Z}}^d$.



 \end{itemize}

 \subsection{Brownian motion on $\widetilde{\mathbb{Z}}^d$}\label{section_brownian_motion}

The Brownian motion on $\widetilde{\mathbb{Z}}^d$, denoted by $\{\widetilde{S}_t\}_{t\ge 0}$, is a continuous-time Markov chain evolving as follows. When $\widetilde{S}_{\cdot}$ is in the interior of an interval $I_e$ ($e\in \mathbb{L}^d$), it behaves as a standard one-dimensional Brownian motion. When $\widetilde{S}_{\cdot}$ reaches a lattice point $x\in \mathbb{Z}^d$, it uniformly picks an interval $I_e$ incident to $x$, and then performs a Brownian excursion from $x$ in $I_e$; in addition, once the excursion hits an endpoint of $I_e$, the process continues evolving in the same manner from this endpoint. For each $v\in \widetilde{\mathbb{Z}}^d$, we denote by $\widetilde{\mathbb{P}}_v$ the law of $\{\widetilde{S}_t\}_{t\ge 0}$ starting from $v$. Under $\widetilde{\mathbb{P}}_v$, we let $\tau_D:=\tau_D(\widetilde{S}_{\cdot})$ denote the first time when the Brownian motion $\widetilde{S}_{\cdot}$ hits $D$. The expectation under $\widetilde{\mathbb{P}}_v$ is denoted by $\widetilde{\mathbb{E}}_v$. 

    \textbf{Green's function.} For any $D\subset \widetilde{\mathbb{Z}}^d$, the Green's function for $D$ is defined by  
    \begin{equation}
 	\widetilde{G}_D(v,w):= \int_0^{\infty} \Big\{\widetilde{q}_t(v,w)-\widetilde{\mathbb{E}}_v\big[ \widetilde{q}_{t-\tau_D}(\widetilde{S}_{\tau_D},w)\cdot \mathbbm{1}_{\tau_D<t} \big]  \Big\} dt, \ \ \forall v,w\in  \widetilde{\mathbb{Z}}^d. 
 \end{equation}
  Note that $\widetilde{G}_D$ is finite, symmetric, continuous, and is decreasing in $D$. In addition, $\widetilde{G}_D(v,w)=0$ if $\{v,w\}\cap D\neq \emptyset$. When $D=\emptyset$, we abbreviate $\widetilde{G}(\cdot,\cdot):=\widetilde{G}_\emptyset(\cdot,\cdot)$. Moreover, restricted to $\mathbb{Z}^d$, $\widetilde{G}(\cdot,\cdot)$ is identical to the Green's function $G(\cdot,\cdot)$ on $\mathbb{Z}^d$, which implies that $\widetilde{G}(v,w) \asymp (|v-w|+1)^{2-d}$ for all $v,w\in \widetilde{\mathbb{Z}}^d$. The strong Markov property of Brownian motion yields that 
 \begin{equation}\label{2.9}
 	\widetilde{G}_D(v,w)= \widetilde{\mathbb{P}}_v\big( \tau_{w}<\tau_{D} \big)\cdot \widetilde{G}_D(w,w). 
 \end{equation}
  In addition, using the potential theory, one has 
  \begin{equation}\label{order_green_function}
  	\widetilde{G}_D(w,w) \asymp \mathrm{dist}(D,w)\land 1. 
  \end{equation}

 For any $D\subset \widetilde{\mathbb{Z}}^d$, we denote by $\mathbb{P}^D$ the law of $\{\widetilde{\phi}_v\}_{v\in \widetilde{\mathbb{Z}}^d}$ conditioned on the event $\{\widetilde{\phi}_v=0,\forall v\in D\}$. Under $\mathbb{P}^D$, the GFF $\widetilde{\phi}_\cdot$ is mean-zero and has covariance 
    \begin{equation}
 	\mathbb{E}^{D}\big[\widetilde{\phi}_v\widetilde{\phi}_w \big]=\widetilde{G}_D(v,w),\ \ \forall v,w\in \widetilde{\mathbb{Z}}^d.
 \end{equation}
 Here $\mathbb{E}^D$ denotes the expectation under $\mathbb{P}^D$. Referring to \cite[Proposition 5.2]{lupu2016loop}, the formula in (\ref{two-point1}) for the two-point function is valid for $\widetilde{\phi}_\cdot\sim \mathbb{P}^D$, i.e., for any $D\subset \widetilde{\mathbb{Z}}^d$ and $v\neq w\in \widetilde{\mathbb{Z}}^d\setminus D$, 
  \begin{equation}\label{211}
  \begin{split}
  	  	\mathbb{P}^D\big( v\xleftrightarrow{\ge 0} w\big)=   &\pi^{-1}\arcsin\Big( \tfrac{\widetilde{G}_D(v,w)}{\sqrt{\widetilde{G}_D(v,v)\widetilde{G}_D(w,w)}} \Big) \\
  	  	 \asymp &  \tfrac{\widetilde{G}_D(v,w)}{\sqrt{\widetilde{G}_D(v,v)\widetilde{G}_D(w,w)}} \overset{(\ref{2.9})}{= } \widetilde{\mathbb{P}}_w\big( \tau_{v}<\tau_{D} \big)\cdot \sqrt{\tfrac{\widetilde{G}_D(v,v)}{\widetilde{G}_D(w,w)}}.
  \end{split}
  \end{equation}

\textbf{Basic properties of Brownian motion.} The first property asserts that the hitting distribution on the boundary of a Euclidean ball is comparable to the uniform distribution.


%
%

  \begin{lemma}\label{lemma_BM_uniform}
 For any $d\ge 3$, $r\ge 1$ and $x\in \partial \mathcal{B}(r)$, the following estimates hold: 
 \begin{enumerate}

 	\item  For any $v\in \mathfrak{B}^{-1}_{N,d^{-1}}$, 
 	\begin{equation}\label{revisenew_21}
 	\widetilde{\mathbb{P}}_v\big( \tau_{\partial \mathcal{B}(r)} =\tau_x <\infty \big)\asymp   r^{1-d}. 
 \end{equation}

 	\item  For any $v\in \mathfrak{B}^{1}_{N,d^{-1}}$,  	
 	 \begin{equation}\label{revisenew_22}
 	\widetilde{\mathbb{P}}_v\big( \tau_{\partial \mathcal{B}(r)} =\tau_x <\infty \big)\lesssim  r^{1-d}. 
 \end{equation}
 
 \end{enumerate}

 \end{lemma}

 \textbf{P.S.} Some points of $\partial \mathcal{B}(r)$ cannot be hit first by a Brownian motion starting from $\mathfrak{B}^{1}_{N,d^{-1}}$. For such a point $x$, the left-hand side of (\ref{revisenew_22}) is zero. Hence, the reverse inequality in (\ref{revisenew_22}) cannot hold for all $x\in \partial \mathcal{B}(r)$.


 \begin{proof}
 The inequality (\ref{revisenew_21}) follows directly from \cite[Lemma 3.7]{cai2025minimal} and \cite[Proposition 6.5.4]{lawler2010random}. We now turn to (\ref{revisenew_22}). Without loss of generality, we assume that $r$ is sufficiently large. By the last-exit decomposition (see e.g., \cite[Section 8.2]{morters2010brownian}), 
 \begin{equation}\label{revise29}
 	\begin{split}
 		\widetilde{\mathbb{P}}_v\big( \tau_{\partial \mathcal{B}(r)} =\tau_x <\infty \big) =  & \sum\nolimits_{z\in \partial \mathcal{B}(r-10)}  \widetilde{G}_{[\mathcal{B}(r-10)]^c}(v,z) \cdot \tfrac{1}{2d} \\
 		& \cdot \sum\nolimits_{z'\in [\mathcal{B}(r-10)]^c:\{z,z'\}\in \mathbb{L}^d} 	\widetilde{\mathbb{P}}_{z'}\big( \tau_{\partial \mathcal{B}(r)} =\tau_x < \tau_{\mathcal{B}(r-10)} \big).
 	\end{split}
 \end{equation}
 Meanwhile, for any $z\in \partial \mathcal{B}(r-10)$, it follows from \cite[Lemma 6.3.7]{lawler2010random} that the probability of $\widetilde{S}_{\cdot }\sim \widetilde{\mathbb{P}}_{v}$ exiting the ball $\mathcal{B}(r-10)$ at $z$ is of order $r^{1-d}$. In addition, this probability is bounded from below by 
 \begin{equation}
 	\widetilde{\mathbb{P}}_{v}\big( \tau_{z}< \tau_{[\mathcal{B}(r-10)]^c} \big)\cdot \tfrac{1}{2d}  \overset{(\ref{2.9})}{\gtrsim}  \widetilde{G}_{[\mathcal{B}(r-10)]^c}(v,z). 
 \end{equation}
 These two observations together imply that for all $z\in \partial \mathcal{B}(r-10)$, 
  \begin{equation}\label{revise_new_211}
  	\widetilde{G}_{[\mathcal{B}(r-10)]^c}(v,z) \lesssim r^{1-d}.
  \end{equation}

  Next, we estimate the hitting probability on the second line of (\ref{revise29}). For any $z'\in [\mathcal{B}(r-10)]^c$ that is adjacent to $\mathcal{B}(r-10)$, one has 
  \begin{equation}\label{revise12}
  	\widetilde{\mathbb{P}}_{z'}\big( \tau_{\partial \mathcal{B}(r)} =\tau_x < \tau_{\mathcal{B}(r-10)} \big)\le \widetilde{\mathbb{P}}_{z'}\big( \tau_x < \tau_{\mathcal{B}(r-10)\cup \partial \mathcal{B}(r+10)} \big). 
  \end{equation}
  To reach $x$ before hitting $\mathcal{B}(r-10)\cup \partial \mathcal{B}(r+10)$, the Brownian motion $\widetilde{S}_\cdot\sim \widetilde{\mathbb{P}}_{z'}$ has to cross a hyper-rectangle of length $c|z-z'|$ and width $20$, whose probability decays exponentially in $|z'-x|$. Combined with (\ref{revise12}), it yields that 
  \begin{equation}\label{new_revise_213}
  	\widetilde{\mathbb{P}}_{z'}\big( \tau_{\partial \mathcal{B}(r)} =\tau_x < \tau_{\mathcal{B}(r-10)} \big)\lesssim  e^{-c'|z'-x|}. 
  \end{equation}
  Plugging (\ref{revise_new_211}) and (\ref{new_revise_213}) into (\ref{revise29}), we obtain the desired bound (\ref{revisenew_22}). 
 \end{proof}

 As shown in the following lemma, the hitting distribution remains stable under the addition of appropriate obstacles. The proof can be found in \cite[Lemma 2.1]{cai2024quasi}.

 \begin{lemma}\label{lemma_BM_stable}
 	For any $d\ge 3$, $N\ge 1$, $x\in \partial B(N)$, $D_1 \subset \mathfrak{B}^{1}_{N,1/2}$ and $D_{-1} \subset \mathfrak{B}^{-1}_{N,1/2}$,  
 	\begin{equation}
 		\widetilde{\mathbb{P}}_{x}\big( \tau_{D_i}= \tau_{v} <\tau_{D_{-i}} \big)  \asymp \widetilde{\mathbb{P}}_{x}\big( \tau_{D_i }= \tau_{v} <\infty \big), \ \ \forall i\in \{1,-1\}\ \text{and}\ v\in \widetilde{\partial} D_i. 
 	\end{equation}
 \end{lemma}

 Next, we extend the stability in Lemma \ref{lemma_BM_stable} to more general sets.

 \begin{lemma}\label{lemma_new_ BM_stable}
 	For any $d\ge 3$, there exists $c>0$ such that for any $N\ge 1$, $i\in \{1,-1\}$, $D_i\subset \widetilde{\mathbb{Z}}^d $, $v\in   (\widetilde{\partial}D_i)\cap \mathfrak{B}^i_{N,1}$ and $D_{-i}\subset \mathfrak{B}^{-i}_{N,c}$, 
 	\begin{equation}\label{great23}
 \sum\nolimits_{x\in \partial \mathcal{B}(d^iN)} \widetilde{\mathbb{P}}_x\big( \tau_{D_i}=\tau_{v}<\tau_{D_{-i}} \big) \asymp  		\sum\nolimits_{x\in \partial \mathcal{B}(d^iN)} \widetilde{\mathbb{P}}_x\big( \tau_{D_i}=\tau_{v}< \infty \big).
 	\end{equation}
 \end{lemma}
 \textbf{P.S.} In fact, it is necessary to take the sum over $x\in \partial \mathcal{B}(d^iN)$ in (\ref{great23}). Otherwise, if one considers the hitting probability in $D_i$ for the Brownian motion starting from a fixed point $x \in \partial \mathcal{B}(d^iN)$, then imposing $D_{-i}$ as an absorbing boundary may drastically decrease this probability, as the following scenario illustrates. Let the set $D_i$ form a tunnel such that there exists a path starting from $x$ and first hitting $D_i$ within $\mathfrak{B}^i_{N,1}$, and that all such paths must intersect $\mathfrak{B}^{-i}_{N,C}$. Note that $\widetilde{\mathbb{P}}_x( \tau_{D_i}=\tau_{v}<\infty)>0$ for this $D_i$. However, for $D_{-i}=\widetilde{\partial} \mathfrak{B}^{-i}_{N,C}$, the probability $\widetilde{\mathbb{P}}_x( \tau_{D_i}=\tau_{v}<\tau_{D_{-i}})$ vanishes.  
 
 \begin{proof}
 We only provide the proof for $i=1$, since the case $i=-1$ follows similarly. For any $x\in \partial \mathcal{B}(d N)$, by the strong Markov property of Brownian motion,  
 	 \begin{equation}\label{great24}
 	 	\begin{split}
 	 		0\le & \widetilde{\mathbb{P}}_x\big( \tau_{D_1}=\tau_{v}< \infty \big) - \widetilde{\mathbb{P}}_x\big( \tau_{D_1}=\tau_{v}<\tau_{D_{-1}} \big) \\
 	 		=&  \widetilde{\mathbb{P}}_x\big( \tau_{D_{-1}} < \tau_{D_1}=\tau_{v}< \infty \big)  \\
 	 		\le & \max_{z \in \widetilde{\partial} D_{-1}}  \widetilde{\mathbb{P}}_z\big( \tau_{\partial \mathcal{B}(10dN) }<\infty \big) \cdot \max_{z'\in \partial \mathcal{B}(10dN)}   \sum\nolimits_{z''\in \partial \mathcal{B}(dN)}  \\
 	 		& \ \ \  \ \ \ \ \ \ \ \ \ \ \ \ \ \ \ \ \ \ \ \ \ \ \ \   \widetilde{\mathbb{P}}_{z'}\big( \tau_{\partial \mathcal{B}(dN)} =\tau_{z''}<\infty  \big)  
 	 		\cdot \widetilde{\mathbb{P}}_{z''}\big(\tau_{D_1}=\tau_{v}< \infty  \big). 
 	 	\end{split}
 	 \end{equation}
 	  From the potential theory of Brownian motion (see e.g., (\ref{cap1}) and (\ref{cap2}) below), one has that for any $D_{-1}\subset \mathfrak{B}^{-1}_{N,c}$ and $z \in \widetilde{\partial} D_{-1}$, 
 	  	 \begin{equation}\label{great25}
 	 	\widetilde{\mathbb{P}}_z\big( \tau_{\partial \mathcal{B}(10dN) }<\infty \big) \lesssim c^{d-2}. 
 	 \end{equation}
 	 Meanwhile, Lemma \ref{lemma_BM_uniform} implies that for any $z'\in \partial \mathcal{B}(10dN)$, 
 	 \begin{equation}\label{great26}
 	 	\widetilde{\mathbb{P}}_{z'}\big( \tau_{\partial \mathcal{B}(dN)} =\tau_{z''}<\infty  \big) \lesssim N^{1-d}. 
 	 \end{equation}
 	 Inserting (\ref{great25}) and (\ref{great26}) into (\ref{great24}), and then summing over $x\in \partial \mathcal{B}(d N)$, we get 
 	  \begin{equation}
 	  	\begin{split}
 	  		0\le &\sum\nolimits_{x\in \partial \mathcal{B}(d N)} \widetilde{\mathbb{P}}_x\big( \tau_{D_1}=\tau_{v}< \infty \big) -   \sum\nolimits_{x\in \partial \mathcal{B}(d N)} \widetilde{\mathbb{P}}_x\big( \tau_{D_1}=\tau_{v}<\tau_{D_{-1}} \big)  \\
 	  		\lesssim  & c^{d-2}N^{1-d}|\partial \mathcal{B}(d N)| \cdot  \sum\nolimits_{z''\in \partial \mathcal{B}(d N)} \widetilde{\mathbb{P}}_{z''}\big( \tau_{D_1}=\tau_{v}< \infty \big). 
 	  	\end{split}
 	  \end{equation}
 	 Combined with $|\partial \mathcal{B}(d N)|\asymp N^{d-1}$, it implies that the inequality (\ref{great23}) holds for all sufficiently small $c>0$.  
 \end{proof}
 
In what follows, we present a useful decomposition of the Green's function.
 
 \begin{lemma}\label{lemma_new_decom_green}
 	 For any $d\ge 3$, there exists $c>0$ such that for any $D\subset \widetilde{\mathbb{Z}}^d$, $R \ge 1$, and $v_1,v_2\in \widetilde{\mathbb{Z}}^d$ with $|v_1-v_2|\ge R$, 
 \begin{equation}\label{23}
 	\begin{split}
 		\widetilde{G}_{D}(v_1,v_2) \lesssim R^{2-d}\prod\nolimits_{i\in \{1,2\}}    \widetilde{\mathbb{P}}_{v_i}\big( \tau_{\partial \mathcal{B}_{v_i}(cR)}<\tau_{D} \big). 	 	
 		\end{split}
 \end{equation}
 \end{lemma}
  \begin{proof}
Using the strong Markov property of Brownian motion, we have 
 \begin{equation}\label{ineq26}
 	\begin{split}
 		\widetilde{\mathbb{P}}_{v_1}\big( \tau_{v_2}<\tau_{D} \big) \le  \widetilde{\mathbb{P}}_{v_1}\big( \tau_{\partial \mathcal{B}_{v_1}(cR)}<\tau_{D} \big) \cdot \max_{y\in \partial \mathcal{B}_{v_1}(cR) } \widetilde{\mathbb{P}}_{y}\big( \tau_{v_2}<\tau_{D} \big). 
 	\end{split}
 \end{equation}
 Therefore, it remains to show that for any $y\in \partial \mathcal{B}_{v_1}(cR)$, 
 \begin{equation}\label{revisenew_222}
 	\widetilde{\mathbb{P}}_{y}\big( \tau_{v_2}<\tau_{D} \big) \lesssim R^{2-d} [\widetilde{G}_D(v_2,v_2)]^{-1} \cdot  \widetilde{\mathbb{P}}_{v_2}\big( \tau_{\partial \mathcal{B}_{v_2}(cR)}<\tau_{D} \big).
 \end{equation}
 To achieve this, by the last-exit decomposition, one has  
 \begin{equation}\label{ineq27}
 	 \begin{split}
 	 		\widetilde{\mathbb{P}}_{y}\big( \tau_{v_2}<\tau_{D} \big) \lesssim &  \sum_{z\in \partial \mathcal{B}_{v_2}(cR)} \widetilde{G}_D(y,z) \sum_{z'\in \mathcal{B}_{v_2}(cR): \{z,z'\}\in \mathbb{L}^d} \widetilde{\mathbb{P}}_{z'}\big( \tau_{v_2}<\tau_{D\cup\partial \mathcal{B}_{v_2}(cR)} \big)\\
 	 		\lesssim & R^{2-d} \sum\nolimits_{z' \in F_R }  \widetilde{\mathbb{P}}_{z'}\big( \tau_{v_2}<\tau_{D\cup\partial \mathcal{B}_{v_2}(cR)} \big),
 	 \end{split}
 \end{equation}
where $F_R$ denotes the collection of points in $\mathcal{B}_{v_2}(cR)\setminus \partial  \mathcal{B}_{v_2}(cR)$ that are adjacent to $\partial  \mathcal{B}_{v_2}(cR)$. By (\ref{2.9}) and the symmetry of the Green's function, we have 
\begin{equation}\label{revisenew_224}
	\begin{split}
		\widetilde{\mathbb{P}}_{z'}\big( \tau_{v_2}<\tau_{D\cup\partial \mathcal{B}_{v_2}(cR)} \big) =&  \frac{\widetilde{G}_{D\cup\partial \mathcal{B}_{v_2}(cR)}(z',z')}{\widetilde{G}_{D\cup\partial \mathcal{B}_{v_2}(cR)}(v_2,v_2) }\cdot \widetilde{\mathbb{P}}_{v_2}\big( \tau_{z'}<\tau_{D\cup\partial \mathcal{B}_{v_2}(cR)} \big)\\
		\lesssim &  [\widetilde{G}_D(v_2,v_2)]^{-1} \cdot \widetilde{\mathbb{P}}_{v_2}\big( \tau_{z'}<\tau_{D\cup\partial \mathcal{B}_{v_2}(cR)} \big), 
	\end{split}
\end{equation}
 where the inequality follows from the facts that $\widetilde{G}_{D\cup\partial \mathcal{B}_{v_2}(cR)}(z',z')\lesssim 1$, and that $\widetilde{G}_{D\cup\partial \mathcal{B}_{v_2}(cR)}(v_2,v_2)\asymp \widetilde{G}_D(v_2,v_2)$ (by (\ref{order_green_function})). Inserting (\ref{revisenew_224}) into (\ref{ineq27}), one has 
 \begin{equation}\label{revise_new_225}
 	\begin{split}
 			\widetilde{\mathbb{P}}_{y}\big( \tau_{v_2}<\tau_{D} \big) \lesssim R^{2-d}  [\widetilde{G}_D(v_2,v_2)]^{-1} \cdot \sum\nolimits_{z' \in F_R } \widetilde{\mathbb{P}}_{v_2}\big( \tau_{z'}<\tau_{D\cup\partial \mathcal{B}_{v_2}(cR)} \big). 
 	\end{split}
 \end{equation}
 Note that each $z''\in \partial \mathcal{B}_{v_2}(cR)$ has a neighbor $z'\in F_R$. Moreover, in order for the Brownian motion $\widetilde{S}_\cdot\sim  \widetilde{\mathbb{P}}_{v_2}$ to first hit $\partial \mathcal{B}_{v_2}(cR)$ at $z''$, it may first reach $z'$ and then enter $z''$ through the edge $\{z',z''\}$. As a result,   
\begin{equation}
	\begin{split}
		 \widetilde{\mathbb{P}}_{v_2}\big( \tau_{\partial \mathcal{B}_{v_2}(cR)} =\tau_{z''}<\tau_{D} \big) \ge \tfrac{1}{2d} \cdot \widetilde{\mathbb{P}}_{v_2}\big( \tau_{z'}<\tau_{D\cup\partial \mathcal{B}_{v_2}(cR)} \big). 
	\end{split}
\end{equation} 
Summing over all $z''\in \partial\mathcal{B}_{v_2}(cR)$ yields 
 \begin{equation}\label{revisenew_227}
 	 \widetilde{\mathbb{P}}_{v_2}\big( \tau_{\partial \mathcal{B}_{v_2}(cR)}  <\tau_{D} \big) \gtrsim  \sum\nolimits_{z' \in F_R } \widetilde{\mathbb{P}}_{v_2}\big( \tau_{z'}<\tau_{D\cup\partial \mathcal{B}_{v_2}(cR)} \big). 
 \end{equation}
Plugging (\ref{revisenew_227}) into (\ref{revise_new_225}), we obtain (\ref{revisenew_222}), thereby completing the proof. 
  \end{proof}


 \textbf{Boundary excursion kernel.} For any $D\subset \widetilde{\mathbb{Z}}^d$, the boundary excursion kernel for $D$ is defined as follows: for any $v\neq w\in \widetilde{\partial}D$, 
  \begin{equation}\label{24}
 	\mathbb{K}_D(v,w):=\lim\limits_{\epsilon \downarrow 0} (2\epsilon)^{-1}\sum\nolimits_{v'\in \widetilde{\mathbb{Z}}^d:|v-v'|=\epsilon}\widetilde{\mathbb{P}}_{v'}\big(\tau_{D}=\tau_{w}<\infty \big).
 \end{equation} 
 By \cite[Lemma 3.2]{inpreparation_twoarm}, one has: for any $D\subset \widetilde{\mathbb{Z}}^d$ and $v,w\in \widetilde{\mathbb{Z}}^d\setminus D$ with $|v-w|\ge d$, 
 \begin{equation}
 	\mathbb{K}_{D\cup \{v,w\}}(v,w) \asymp \big[ \widetilde{G}_D(w,w) \big]^{-1}\cdot  \widetilde{\mathbb{P}}_w\big( \tau_v<\tau_D \big). 
 \end{equation}
 Combined with (\ref{211}) and $\widetilde{G}_D(w,w)\vee \widetilde{G}_D(v,v)\lesssim 1$, it implies that 
 \begin{equation}\label{for2.16}
 	\begin{split}
 		 \mathbb{P}^D\big( v\xleftrightarrow{\ge 0} w\big)  \lesssim   \mathbb{K}_{D\cup \{v,w\}}(v,w). 
 		  	\end{split}  
 \end{equation}

  \textbf{Capacity.} For any $D\subset \widetilde{\mathbb{Z}}^d$, the equilibrium measure for $D$ is defined by  
 \begin{equation}
 	\mathbb{Q}_D(v):=\lim\limits_{N\to \infty} \sum\nolimits_{w\in \partial B(N)} \mathbb{K}_{D\cup \partial B(N)}(v,w), \ \forall v\in \widetilde{\partial}D. 
 \end{equation} 
 The capacity of $D$ is defined as the total mass of the equilibrium measure, i.e., $\mathrm{cap}(D):= \sum\nolimits_{v\in \widetilde{\partial}D} \mathbb{Q}_D(v)$. The capacity is closely related to the hitting probability: for any $D\subset \widetilde{B}(N)$, $x\in \partial B(2N)$ and $D'\subset [\widetilde{B}(3N)]^c$,  
  \begin{equation}\label{cap1}
 	\widetilde{\mathbb{P}}_x\big(\tau_{D}<\tau_{D'} \big)\asymp N^{2-d}\cdot \mathrm{cap}(D). 
 \end{equation}
   In addition, the capacity of a box satisfies
   \begin{equation}\label{cap2}
   	\mathrm{cap}(\widetilde{B}(N))\asymp N^{d-2}, \ \ \forall N\ge 1. 
   \end{equation}



 \subsection{Loop soup}

   Recall the definitions of the loop measure $\widetilde{\mu}$ and the loop soup $\widetilde{\mathcal{L}}_{\alpha}$ in Section \ref{section_intro}. By forgetting the roots of all rooted loops, $\widetilde{\mu}$ can be considered as a measure on the space of equivalence classes of rooted loops, where two rooted loops belong to the same class whenever one is a cyclic time-shift of the other. Each equivalence class $\widetilde{\ell}$ is called a loop. Note that equivalent rooted loops have the same range. In light of this, we define $\mathrm{ran}(\widetilde{\ell})$ as $\mathrm{ran}(\widetilde{\varrho})$ for any $\widetilde{\varrho}\in \widetilde{\ell}$.

    \textbf{Isomorphism theorem.} For any $D\subset \widetilde{\mathbb{Z}}^d$, we denote $\widetilde{\mathcal{L}}_{\alpha}^{D}:=\widetilde{\mathcal{L}}_{\alpha}\cdot \mathbbm{1}_{\mathrm{ran}(\widetilde{\ell})\cap D=\emptyset}$. Referring to \cite[Proposition 2.1]{lupu2016loop}, the isomorphism theorem mentioned in Section \ref{section_intro} is valid for the metric graph $\widetilde{\mathbb{Z}}^d\setminus D$. I.e., there exists a coupling between $\{\widetilde{\phi}_v\}_{v\in \widetilde{\mathbb{Z}}^d}\sim \mathbb{P}^D$ and $\widetilde{\mathcal{L}}_{1/2}^{D}$ satisfying Properties (a) and (b) described in Section \ref{section_intro} (where $\widehat{\mathcal{L}}_{1/2}^{v}$ is replaced by $\widehat{\mathcal{L}}_{1/2}^{D,v}$---the total local time at $v$ of all loops in $\widetilde{\mathcal{L}}_{1/2}^{D}$). For later use of this isomorphism theorem, we record the following notation for clusters: 
    \begin{itemize}

    	\item For a point measure $\mathcal{L}$ consisting of paths, we denote by $\cup \mathcal{L}$ the union of ranges of all paths in $\mathcal{L}$. For any $D,A,A'\subset \widetilde{\mathbb{Z}}^d$, we denote by $A\xleftrightarrow{(D)}A'$ the event that there exists a path in $\cup \widetilde{\mathcal{L}}_{1/2}^D$ connecting $A$ and $A'$. We then define the cluster $\mathcal{C}_A^D:=\{v\in \widetilde{\mathbb{Z}}^d:v\xleftrightarrow{(D)}A\}$. When $D=\emptyset$, we write $\mathcal{C}_A^\emptyset$ as $\mathcal{C}_A$ for brevity.

    	\item  For any $A\subset \widetilde{\mathbb{Z}}^d$, we denote the positive (resp. negative) cluster containing $A$ by $\mathcal{C}_A^+:=\{v\in \widetilde{\mathbb{Z}}^d: v\xleftrightarrow{\ge 0} A\}$ (resp. $\mathcal{C}_A^-:=\{v\in \widetilde{\mathbb{Z}}^d: v\xleftrightarrow{\le 0} A\}$). The sign cluster containing $A$ is defined as  $\mathcal{C}_A^\pm:=\mathcal{C}_A^+\cup \mathcal{C}_A^-$.

       \item   For any loop cluster $\mathcal{C}$, let $\mathcal{F}_{\mathcal{C}}$ denote the $\sigma$-field generated by $\mathcal{C}$ and all loops included in $\mathcal{C}$. Similarly, for any sign cluster $\mathcal{C}'$, we denote by $\mathcal{F}_{\mathcal{C}'}$ the $\sigma$-field generated by $\mathcal{C}'$ and all GFF values on $\mathcal{C}'$.

    \end{itemize}

 \textbf{Restriction property.} Conditioned on $\{\mathcal{C}^{D}_{A}=F\}$ (for any proper realization $F$ of $\mathcal{C}^{D}_{A}$), the thinning property of Poisson point processes implies that the point measure $\widetilde{\mathcal{L}}_{1/2}^{D}\cdot \mathbbm{1}_{\mathrm{ran}(\widetilde{\ell})\cap \mathcal{C}^{D}_{A}=\emptyset}$ has the same distribution as $\widetilde{\mathcal{L}}_{1/2}^{D\cup F}$. From the perspective of GFF, the corresponding property also holds. I.e., for $\diamond\in \{+,-,\pm\}$, conditioned on $\{\mathcal{C}^{\diamond}_{A}=F\}$, the law of GFF values on $\widetilde{\mathbb{Z}}^d\setminus \mathcal{C}^{\diamond}_{A}$ is $\mathbb{P}^F$.

 \textbf{Connecting probability.} For any $A_1,A_2,D\subset \widetilde{\mathbb{Z}}^d$, the event $A_1\xleftrightarrow{(D)} A_2$ is measurable with respect to $\mathcal{C}_{A_1}^D$. Thus, when $\mathcal{C}_{A_1}^D$ is disjoint from $D'\subset \widetilde{\mathbb{Z}}^d$ (and hence consists only of loops in $\widetilde{\mathcal{L}}_{1/2}^{D\cup D'}$), we have $\mathcal{C}_{A_1}^D=\mathcal{C}_{A_1}^{D\cup D'}$ and $A_1\xleftrightarrow{(D \cup D')} A_2$. As a result, one has  
 \begin{equation}\label{2.14}
  \big\{ 	A_1\xleftrightarrow{(D)} A_2 \big\} \cap \big\{ A_1\xleftrightarrow{(D)} D' \big\}^c \subset  \big\{ \mathcal{C}_{A_1}^D=\mathcal{C}_{A_1}^{D\cup D'} \big\}\cap  \big\{A_1\xleftrightarrow{(D \cup D')} A_2\big\}. 
 \end{equation}
This immediately implies that the probability of the left-hand side is less equal to that of the right-hand side. In the next lemma, we show that the reverse inequality may hold (up to a constant factor) under appropriate conditions.

 \begin{lemma}\label{lemma24}
 	(1)  For any $3\le d\le 5$, there exists $c>0$ such that for any $D\subset \widetilde{\mathbb{Z}}^d$, $N\ge 1$, $i\in \{1,-1\}$, and $A_1,A_2\subset \mathfrak{B}^{i}_{N,c}$, 
 	\begin{equation}\label{to214}
 \mathbb{P}\big(A_1\xleftrightarrow{(D\cup \widetilde{\partial} \mathfrak{B}^{i}_{N,c})} A_2 \big) 	\lesssim  	\mathbb{P}\big(A_1\xleftrightarrow{(D)} A_2, \big\{A_1\xleftrightarrow{(D)} \partial B(N) \big\}^c   \big). 
 	\end{equation}
 (2) For any $3\le d\le 5$, there exists $C>0$ such that for any $D\subset \widetilde{\mathbb{Z}}^d$, $N_2> N_1\ge 1$, and $A_1,A_2\subset \widetilde{B}(N_2)\setminus \widetilde{B}(N_1)$, 
  	\begin{equation}\label{newto214}
  	\begin{split}
  		& \mathbb{P}\big(A_1\xleftrightarrow{(D\cup \partial B(N_1)\cup \partial B(N_2) )} A_2 \big) \\
  		 	\lesssim  &	\mathbb{P}\big(A_1\xleftrightarrow{(D)} A_2, \big\{A_1\xleftrightarrow{(D)} \partial B(\tfrac{N_1}{C})\cup \partial B(CN_2)) \big\}^c   \big). 
  	\end{split}
 	\end{equation}
 \end{lemma}
 \begin{proof}
(1) Since $\big\{\widetilde{\partial} \mathfrak{B}^{i}_{N,c}\xleftrightarrow{(D)} \partial B(N)  \big\}^c$ depends only on the loops intersecting $[\mathfrak{B}^{i}_{N,c}]^c$, it is independent of the event $A_1\xleftrightarrow{(D\cup \widetilde{\partial} \mathfrak{B}^{i}_{N,c})} A_2$. Therefore,  
 \begin{equation}\label{to215}
 	\begin{split}
 		 &\mathbb{P}\big(A_1\xleftrightarrow{(D\cup \widetilde{\partial} \mathfrak{B}^{i}_{N,c})} A_2, \big\{\widetilde{\partial} \mathfrak{B}^{i}_{N,c}\xleftrightarrow{(D)} \partial B(N)  \big\}^c \big)\\
 		 =&  \mathbb{P}\big(A_1\xleftrightarrow{(D\cup \widetilde{\partial} \mathfrak{B}^{i}_{N,c})} A_2  \big)  \mathbb{P}\big( \big\{\widetilde{\partial} \mathfrak{B}^{i}_{N,c}\xleftrightarrow{(D)} \partial B(N)  \big\}^c \big) 
 		\overset{(\ref{crossing_low})}{ \gtrsim}     \mathbb{P}\big(A_1\xleftrightarrow{(D\cup \widetilde{\partial} \mathfrak{B}^{i}_{N,c})} A_2  \big).
 	\end{split}
 \end{equation}  
 Meanwhile, when $\big\{ A_1\xleftrightarrow{(D\cup \widetilde{\partial} \mathfrak{B}^{i}_{N,c})} A_2\big\}\cap \big\{\widetilde{\partial} \mathfrak{B}^{i}_{N,c}\xleftrightarrow{(D)} \partial B(N)  \big\}^c$ occurs, by $\widetilde{\mathcal{L}}_{1/2}^{D\cup \widetilde{\partial} \mathfrak{B}^{i}_{N,c}}\le \widetilde{\mathcal{L}}_{1/2}^{D}$ we have $A_1\xleftrightarrow{(D)} A_2$; moreover, since $A_1\subset \mathfrak{B}^{i}_{N,c}$, one has $\big\{A_1\xleftrightarrow{(D)} \partial B(N) \big\}^c$. Consequently, we have 
 \begin{equation}\label{to217}
 	\begin{split}
 		& \mathbb{P}\big(A_1\xleftrightarrow{(D\cup \widetilde{\partial} \mathfrak{B}^{i}_{N,c})} A_2, \big\{\widetilde{\partial} \mathfrak{B}^{i}_{N,c}\xleftrightarrow{(D)} \partial B(N)  \big\}^c \big)\\
 		\le &	\mathbb{P}\big(A_1\xleftrightarrow{(D)} A_2, \big\{A_1\xleftrightarrow{(D)} \partial B(N) \big\}^c   \big).  
 	\end{split}
 \end{equation}
Combined with (\ref{to215}), it implies (\ref{to214}).

(2) For the same reason as in (\ref{to215}), we have 
\begin{equation}\label{to218}
	\begin{split}
		&\mathbb{P}\big( \big\{A_1\xleftrightarrow{(D\cup \partial B(N_1)\cup \partial B(N_2) )} A_2\big\}\cap  \big\{  B(\tfrac{N_1}{C})\xleftrightarrow{(D)} \partial B(N_1) \big\}^c  \\
		& \ \  \ \ \ \ \ \ \ \ \ \ \ \ \ \ \ \ \ \ \ \ \ \ \ \ \ \ \ \ \ \ \ \ \ \  \cap  \big\{ B( N_2)\xleftrightarrow{(D)} \partial B(CN_2)  \big\}^c \big) \\
		=&\mathbb{P}\big( A_1\xleftrightarrow{(D\cup \partial B(N_1)\cup \partial B(N_2) )} A_2 \big) \\
		&\cdot \mathbb{P}\big(  \big\{  B(\tfrac{N_1}{C})\xleftrightarrow{(D)} \partial B(N_1) \big\}^c \cap \big\{ B( N_2)\xleftrightarrow{(D)} \partial B(CN_2)  \big\}^c \big) \\
		\overset{\text{(FKG)},(\ref{crossing_low})}{\gtrsim} & \mathbb{P}\big( A_1\xleftrightarrow{(D\cup \partial B(N_1)\cup \partial B(N_2))} A_2 \big). 
	\end{split}
\end{equation}
In addition, as in (\ref{to217}), the event on the left-hand side of (\ref{to218}) implies that the cluster $\mathcal{C}_{A_1}^{D}$ is contained in the annulus $\widetilde{B}(CN_2)\setminus \widetilde{B}(\frac{N_1}{C})$. This observation together with (\ref{to218}) yields (\ref{newto214}). 
 \end{proof}

 As verified in \cite{cai2024quasi}, the connecting probability between two general sets is stable under the addition of certain zero boundary conditions. Precisely, referring to the proof of \cite[Proposition 1.8]{cai2024quasi} in \cite[Section 5]{cai2024quasi}, we have the following property:

 \begin{lemma}\label{lemma_stability_bd}
 	For any $d\ge 3$, there exist $c,c'>0$ such that for any $i\in \{1,-1\}$, $N\ge 1$, $A,D\subset \mathfrak{B}^i_{N,c}$ and $D'\subset \mathfrak{B}^{-i}_{N,c'}$, 
 	\begin{equation}
 \mathbb{P}\big(A\xleftrightarrow{(D)} \partial B(N)  \big) \asymp \mathbb{P}\big( A\xleftrightarrow{(D\cup D')} \partial B(N) \big).
 	\end{equation} 
 \end{lemma}

 In addition to Lemma \ref{lemma_stability_bd}, we also present the corresponding stability for point-to-set connecting probabilities as follows.

 \begin{lemma}\label{lemma_new_stable}
 	For any $3\le d\le 5$, there exist $C,c>0$ such that for any $N\ge C$, $x\in \partial B(N)$, $A,D\subset \mathfrak{B}^1_{N,c}$ and $D'\subset \mathfrak{B}^{-1}_{N,c}$, 
 	\begin{equation}\label{ineq_new_stable}
 		 \mathbb{P}\big(A\xleftrightarrow{(D)} x  \big) \asymp \mathbb{P}\big( A\xleftrightarrow{(D\cup D')} x \big). 
 	\end{equation}
 	For any $d\ge 7$, there exists $c'>0$ such that (\ref{ineq_new_stable}) is valid if the following additional condition holds:
 	\begin{equation}\label{condition_new_stable}
 		 \sum\nolimits_{z\in \mathbb{Z}^d: \widetilde{B}_z(1)\cap D'\neq \emptyset}|z|^{2-d}\le c'N^{6-d}. 
 		  	\end{equation}	
 \end{lemma}

 The proof of Lemma \ref{lemma_new_stable} will be provided in Appendix \ref{app0_stable}.

%

 The next lemma shows that this stability holds for the average of two-point functions on the boundary of a Euclidean ball, with more general absorbing boundaries.

 	 \begin{lemma}\label{lemma_new26}
 	 	For any $3\le d\le 5$, there exists $c>0$ such that for any $D\subset \widetilde{\mathbb{Z}}^d$, $r_{-1}>r_1\ge 1$, $i\in \{1,-1\}$, $D'\subset \mathfrak{B}^i_{r_i,c}$, and $x\in \partial \mathcal{B}(r_{-i})$, 
 	 	\begin{equation}\label{for220}
 	 		\begin{split}
 	 			 \sum\nolimits_{y\in  \partial \mathcal{B}(r_{i})}  \mathbb{P}\big( x\xleftrightarrow{(D)}  y \big) \asymp \sum\nolimits_{y\in  \partial \mathcal{B}(r_{i})}  \mathbb{P}\big( x\xleftrightarrow{(D\cup D')}  y \big). 
 	 			  	 		\end{split}
 	 	\end{equation} 
 	 \end{lemma}
 	 \textbf{P.S.} As in Lemma \ref{lemma_new_ BM_stable}, it is necessary to take the sum over $y\in  \partial \mathcal{B}(r_{i})$ in (\ref{for220}), which can be explained via an example analogous to the one proposed below (\ref{great23}). The same applies to the subsequent corollaries of Lemma \ref{lemma_new26}.

 	 \begin{proof}
 	 The proofs for $i=1$ and $i=-1$ rely on similar arguments, so we only show the details for $i=1$. Since $\widetilde{\mathcal{L}}_{1/2}^D\ge \widetilde{\mathcal{L}}_{1/2}^{D\cup D'}$, the left-hand side of (\ref{for220}) is greater equal to the right-hand side. Thus, by (\ref{211}), it suffices to show that
 	 	 \begin{equation}\label{for221}
 	 	 	 \sum_{y\in  \partial \mathcal{B}(r_{1})} \widetilde{\mathbb{P}}_{x}\big( \tau_{y}<\tau_{D}\big)  \sqrt{\tfrac{\widetilde{G}_D(y,y)}{\widetilde{G}_D(x,x)}}   \lesssim  \sum_{y\in  \partial \mathcal{B}(r_{1})} \widetilde{\mathbb{P}}_{x}\big( \tau_{y}<\tau_{D\cup D'}\big)  \sqrt{\tfrac{\widetilde{G}_{D\cup D'}(y,y)}{\widetilde{G}_{D\cup D'}(x,x)}} .  
 	 	 \end{equation}
 	 	 To see this, it follows from Lemma \ref{lemma_new_ BM_stable} that  
 	 		\begin{equation}\label{for223}
 	 		\sum\nolimits_{y\in  \partial \mathcal{B}(r_{1})}\widetilde{\mathbb{P}}_{x}\big( \tau_{y}<\tau_{D}\big) \asymp   \sum\nolimits_{y\in  \partial \mathcal{B}(r_{1})}\widetilde{\mathbb{P}}_{x}\big( \tau_{y}<\tau_{D\cup D'}\big)  . 
 	 	\end{equation}
 	 	 Meanwhile, note that $\mathrm{dist}(D',z)\ge 1$ for all $z\in \partial \mathcal{B}(r_{1})\cup  \partial \mathcal{B}(r_{-1})$. Thus, for each $w\in \{x,y\}$ (where $x\in \partial \mathcal{B}(r_{-1})$ and $y\in \partial \mathcal{B}(r_{1})$), it follows from (\ref{order_green_function}) that 
 	 	\begin{equation}\label{great236}
 	 		\widetilde{G}_D(w,w) \asymp \mathrm{dist}(D,w)\land 1 \asymp \mathrm{dist}(D\cup D',w)\land 1 \asymp \widetilde{G}_{D\cup D'}(w,w). 
 	 	\end{equation}
 	 	 Combining (\ref{for223}) and (\ref{great236}), we get (\ref{for221}) and thus complete the proof. 
 	 	  	 \end{proof}

 In what follows, we present several corollaries of Lemmas \ref{lemma24} and \ref{lemma_new26}. We say an event $\mathsf{A}$ is a two-point event for $\widetilde{\mathcal{L}}_{1/2}^D$ outside $D'$ if there exist $v,w\in \widetilde{\mathbb{Z}}^d$ such that $\mathsf{A}=\big\{v \xleftrightarrow{(D)}  w \big\}\cap \big\{v \xleftrightarrow{(D)}  D' \big\}^c$. The intersection of finitely many two-point events is called a pairwise connecting event.

 \begin{corollary}\label{coro27}
 We maintain the same conditions for $d,D,i,r_{-1}$ and $r_1$ as in Lemma \ref{lemma_new26}. Then there exist $c>c'>0$ such that the following hold:
 \begin{enumerate}

 	\item For any $x\in \partial \mathcal{B}(r_{-i})$, and any pairwise connecting event $\mathsf{A}$ for $\widetilde{\mathcal{L}}_{1/2}^D$ outside $[\mathfrak{B}^i_{r_i,c'}]^c$, 
 	\begin{equation}\label{form228}
 	\begin{split}
 		 		 \sum_{y\in \partial \mathcal{B}(r_i)} \mathbb{P}\big( \mathsf{A}\cap \big\{x\xleftrightarrow{(D)} y \big\}\cap  \big\{x\xleftrightarrow{(D)} \widetilde{\partial}\mathfrak{B}^i_{r_i,c}  \big\}^c \big)  
 		 		 \gtrsim   \mathbb{P}\big( \mathsf{A} \big) \cdot \sum_{y\in \partial \mathcal{B}(r_i)} \mathbb{P}\big( x \xleftrightarrow{(D)} y \big). 
 	\end{split}
 	\end{equation}

 	\item   For any pairwise connecting event $\mathsf{A}'$ for $\widetilde{\mathcal{L}}_{1/2}^D$ outside $\mathfrak{B}^{-1}_{r_{-1},c'}\setminus \mathfrak{B}^1_{r_1,c'} $,
 	\begin{equation}\label{form229}
 		\begin{split}
 			&	 \sum_{x\in \partial \mathcal{B}(r_1),y\in \partial \mathcal{B}(r_{-1})} \mathbb{P}\big( \mathsf{A}'\cap \big\{x\xleftrightarrow{(D)} y \big\}\cap  \big\{x\xleftrightarrow{(D)} \widetilde{\partial}\mathfrak{B}^1_{r_1,c}  \cup \widetilde{\partial}\mathfrak{B}^{-1}_{r_{-1},c}  \big\}^c \big)  \\
 				 \gtrsim & \mathbb{P}\big( \mathsf{A}' \big) \cdot  \sum_{x\in \partial \mathcal{B}(r_1),y\in \partial \mathcal{B}(r_{-1})} \mathbb{P}\big( x \xleftrightarrow{(D)} y \big).  
 		\end{split}
 	\end{equation}
 \end{enumerate}
 \end{corollary}
\begin{proof}
	(1) Let $\overline{\mathcal{C}}$ denote the union of all clusters involved in the event $\mathsf{A}$. Since $\mathsf{A}$ is outside $[\mathfrak{B}^i_{r_i,c'}]^c$, one has $\mathsf{A}\subset \{\overline{\mathcal{C}} \subset \mathfrak{B}^i_{r_i,c}\}$. Thus, using the restriction property, (\ref{to214}) and Lemma \ref{lemma_new26}, we obtain (\ref{form228}) as follows: 
	\begin{equation*}
		\begin{split}
		 &	\sum\nolimits_{y\in \partial \mathcal{B}(r_i)} \mathbb{P}\big( \mathsf{A}\cap  \big\{x\xleftrightarrow{(D)} y \big\}\cap  \big\{x\xleftrightarrow{(D)} \widetilde{\partial}\mathfrak{B}^i_{r_i,c}  \big\}^c \big) \\
			\overset{}{=}  & \mathbb{E}\big[ \mathbbm{1}_{ \mathsf{A}}  \cdot	\sum\nolimits_{y\in \partial \mathcal{B}(r_i)} \mathbb{P}\big(  x\xleftrightarrow{(D\cup \overline{\mathcal{C}})} y, \big\{x\xleftrightarrow{(D\cup \overline{\mathcal{C}})} \widetilde{\partial}\mathfrak{B}^i_{r_i,c}  \big\}^c \mid \mathcal{F}_{\overline{\mathcal{C}}} \big) \big]\\
			\overset{(\ref{to214})}{\gtrsim } & \mathbb{P}\big( \mathsf{A} \big) \cdot \sum\nolimits_{y\in \partial \mathcal{B}(r_i)} \mathbb{P}\big( x \xleftrightarrow{(D\cup \widetilde{\partial}\mathfrak{B}^i_{r_i,\sqrt{c\cdot c'}} ) } y \big) 
			\overset{(\text{Lemma}\ \ref{lemma_new26})}{\asymp  }   \mathbb{P}\big( \mathsf{A} \big) \cdot \sum\nolimits_{y\in \partial \mathcal{B}(r_i)} \mathbb{P}\big( x \xleftrightarrow{(D) } y \big).  
					\end{split}
	\end{equation*} 
	
	(2) Let $\overline{\mathcal{C}}'$ be the union of all relevant clusters in $\mathsf{A}'$, and note that $\mathsf{A}'\subset \{\overline{\mathcal{C}} \subset (\mathfrak{B}^{-1}_{r_{-1},c'}\setminus \mathfrak{B}^1_{r_1,c'} )^c \}$. Therefore, similar to the proof in Item (1), we have 
	\begin{equation*}
		\begin{split}
			&  \sum_{x\in \partial \mathcal{B}(r_1),y\in \partial \mathcal{B}(r_{-1})} \mathbb{P}\big( \mathsf{A}'\cap \big\{x\xleftrightarrow{(D)} y \big\}\cap  \big\{x\xleftrightarrow{(D)} \widetilde{\partial}\mathfrak{B}^1_{r_1,c}  \cup \widetilde{\partial}\mathfrak{B}^{-1}_{r_{-1},c}  \big\}^c \big) \\
			 \overset{}{=}  &\mathbb{E}\Big[ \mathbbm{1}_{ \mathsf{A}'}  \cdot 	\sum_{x\in \partial \mathcal{B}(r_1),y\in \partial \mathcal{B}(r_{-1})}  \mathbb{P}\big(  x\xleftrightarrow{(D\cup \overline{\mathcal{C}}')} y, \big\{x\xleftrightarrow{(D\cup \overline{\mathcal{C}}')} \widetilde{\partial}\mathfrak{B}^1_{r_1,c}  \cup \widetilde{\partial}\mathfrak{B}^{-1}_{r_{-1},c}   \big\}^c \mid \mathcal{F}_{\overline{\mathcal{C}}'} \big) \Big]\\
		\overset{(\ref{newto214})}{\gtrsim } &	\mathbb{P}\big( \mathsf{A}' \big) \cdot\sum_{x\in \partial \mathcal{B}(r_1),y\in \partial \mathcal{B}(r_{-1})} \mathbb{P}\big( x \xleftrightarrow{(D\cup \widetilde{\partial}\mathfrak{B}^1_{r_1,\sqrt{c\cdot c'}}\cup \widetilde{\partial}\mathfrak{B}^{-1}_{r_{-1},\sqrt{c\cdot c'}} ) } y \big).  
		\end{split}
	\end{equation*}	 
	Meanwhile, by applying Lemma \ref{lemma_new26} twice, one has 
	\begin{equation*}
	\begin{split}
			  \sum_{x\in \partial \mathcal{B}(r_1),y\in \partial \mathcal{B}(r_{-1})} \mathbb{P}\big( x \xleftrightarrow{(D\cup \widetilde{\partial}\mathfrak{B}^1_{r_1,\sqrt{c\cdot c'}}\cup \widetilde{\partial}\mathfrak{B}^{-1}_{r_{-1},\sqrt{c\cdot c'}} ) } y \big)  
				\asymp     \sum_{x\in \partial \mathcal{B}(r_1),y\in \partial \mathcal{B}(r_{-1})} \mathbb{P}\big( x \xleftrightarrow{(D)} y   \big).
	\end{split}
	\end{equation*}
	Plugging this into the last inequality, we obtain (\ref{form229}).
\end{proof}

 	The following result relates the escape probability to the two-point event.

\begin{corollary}\label{coronew27} 
		For any $3\le d\le 5$, there exist $C,c>0$ such that for any $R\ge 1$, $v\in \widetilde{B}(cR)$ and $D\subset \widetilde{\mathbb{Z}}^d$, 
	\begin{equation*}
		\begin{split}
			\widetilde{\mathbb{P}}_v\big(\tau_{\partial B(R)} <\tau_D \big) \lesssim  R^{-1}\sqrt{\widetilde{G}_D(v,v)}  \sum\nolimits_{z\in \partial  \mathcal{B}(\frac{R}{10})}  \mathbb{P}\big( v \xleftrightarrow{(D)} z , \big\{ v \xleftrightarrow{(D)}  \partial B(CR) \big\}^c \big).
		\end{split}
	\end{equation*}
\end{corollary}
\begin{proof} 
  	By the last-exit decomposition, we have 
 		\begin{equation}\label{for234}
 			\begin{split}
 				&\widetilde{\mathbb{P}}_v\big(\tau_{\partial B(R)} <\tau_D \big) \\
 				 \lesssim   & \sum\nolimits_{z\in \partial  \mathcal{B}(\frac{R}{10})}   \widetilde{G}_D(v,z)  
 				\cdot \sum\nolimits_{z'\in [\mathcal{B}(\frac{R}{10})]^c:\{z,z'\}\in \mathbb{L}^d} \widetilde{\mathbb{P}}_{z'}\big( \tau_{\partial B(R)} <\tau_{D\cup \mathcal{B}(\frac{R}{10})} \big). 
 			\end{split}
 		\end{equation}	
  	Using the potential theory of Brownian motion, one has 
 	\begin{equation}\label{for235}
 		\widetilde{\mathbb{P}}_{z'}\big( \tau_{\partial B(R)} <\tau_{D\cup \mathcal{B}(\frac{R}{10})} \big) \le \widetilde{\mathbb{P}}_{z'}\big( \tau_{\partial B(R)} <\tau_{ \mathcal{B}(\frac{R}{10})} \big)  \lesssim R^{-1}. 
 	\end{equation}
 		Moreover, by (\ref{211}), (\ref{to214}) and Lemma \ref{lemma_new26}, we have 
 		\begin{equation}\label{for236}
 			\begin{split}
 			\sum_{z\in \partial  \mathcal{B}(\frac{R}{10})} 	\widetilde{G}_D(v,z) \overset{(\ref{211}) }{ \lesssim} & \sqrt{\widetilde{G}_D(v,v)} \sum_{z\in \partial  \mathcal{B}(\frac{R}{10})}  \mathbb{P}\big(v\xleftrightarrow{(D)} z \big)\\
 		\overset{(\text{Lemma}\ \ref{lemma_new26})}{ \lesssim} &  \sqrt{\widetilde{G}_D(v,v)}\sum_{z\in \partial  \mathcal{B}(\frac{R}{10})}  \mathbb{P}\big(v\xleftrightarrow{(D\cup \widetilde{\partial} \mathfrak{B}^{-1}_{R,c})} z \big)\\ 
 		 		\overset{(\ref{to214})}{ \lesssim} &\sqrt{\widetilde{G}_D(v,v)}  \sum_{z\in \partial  \mathcal{B}(\frac{R}{10})} \mathbb{P}\big( v \xleftrightarrow{(D)} z , \big\{ v \xleftrightarrow{(D)} \partial B(CR)  \big\}^c \big).
 			\end{split}
 		\end{equation}
 		Plugging (\ref{for235}) and (\ref{for236}) into (\ref{for234}), we complete the proof.
\end{proof}

 	  	By Lemma \ref{lemma_new_decom_green} and Corollary \ref{coronew27}, we obtain the following decomposition.

 \begin{corollary} \label{lemma_cut_two_points}
 For any $3\le d\le 5$, there exist $C,c>0$ such that for any $D\subset \widetilde{\mathbb{Z}}^d$, $R \ge C$, and $v_1,v_2\in \widetilde{\mathbb{Z}}^d$ with $|v_1-v_2|\ge R$, 
 \begin{equation}
 	\begin{split}
 			\mathbb{P}\big( v_1\xleftrightarrow{(D)} v_2 \big) \lesssim R^{-d}\prod_{i\in \{1,2\}} \sum_{z_i\in \partial \mathcal{B}_{v_i}(cR)}  \mathbb{P}\big( v_i\xleftrightarrow{(D)} z_i, \big\{ v_i\xleftrightarrow{(D)} \partial B_{v_i}(\tfrac{R}{100}) \big\}^c \big). 
 	\end{split}
 \end{equation}
 \end{corollary}
 	 \begin{proof}
 	 	Applying (\ref{211}), Lemma \ref{lemma_new_decom_green} and Corollary \ref{coronew27}, we have 
 	 	\begin{equation}
 	 		\begin{split}
 	 			\mathbb{P}\big( v_1\xleftrightarrow{(D)} v_2 \big)  \overset{(\ref{211})}{\asymp} & \widetilde{G}_D(v_1,v_2)\prod\nolimits_{i\in \{1,2\}} \big[ \widetilde{G}_D(v_i,v_i)\big]^{-\frac{1}{2}}\\
 	 			\overset{(\text{Lemma}\ \ref{lemma_new_decom_green})}{\lesssim} &R^{2-d}\prod\nolimits_{i\in \{1,2\}} \big[ \widetilde{G}_D(v_i,v_i)\big]^{-\frac{1}{2}} \cdot    \widetilde{\mathbb{P}}_{v_i}\big( \tau_{\partial \mathcal{B}_{v_i}(10cR)}<\tau_{D} \big)\\
 	 		\overset{(\text{Corollary}\ \ref{coronew27})}{\lesssim}  & R^{-d}\prod_{i\in \{1,2\}} \sum_{z_i\in \partial \mathcal{B}_{v_i}(cR)}  \mathbb{P}\big( v_i\xleftrightarrow{(D)} z_i, \big\{ v_i\xleftrightarrow{(D)} \partial B_{v_i}(\tfrac{R}{100}) \big\}^c \big). \qedhere
 	 		\end{split}
 	 	\end{equation}

 	 \end{proof}

 	Following similar arguments as in the proof of Lemma \ref{lemma_new_ BM_stable}, we may derive the mean value property of connecting probabilities as follows.

 \begin{corollary}\label{newcoro_29}
 	For any $d\ge 3$, there exists $c>0$ such that for any $D\subset \widetilde{\mathbb{Z}}^d$, $R\ge 1$, $v\in \widetilde{B}^1_{R,c}$ and $w\in \widetilde{B}^{-1}_{R,c}$, 
 	\begin{equation}
 		\mathbb{P}\big( v\xleftrightarrow{(D)} w \big)\lesssim |w|^{2-d}R^{-1}\cdot  \sum\nolimits_{x\in \partial \mathcal{B}(R) } \mathbb{P}\big( v\xleftrightarrow{(D)} x  \big). 
 	\end{equation}
 \end{corollary}
 	\begin{proof}
 		Recall from (\ref{211}) that 
 		\begin{equation}\label{242}
 			\begin{split}
 				\mathbb{P}\big( v\xleftrightarrow{(D)} w \big)\asymp \widetilde{\mathbb{P}}_w\big( \tau_{v}<\tau_{D} \big)\cdot \sqrt{\tfrac{\widetilde{G}_D(v,v)}{\widetilde{G}_D(w,w)}}. 		
 					\end{split}
 		\end{equation}
 		By the strong Markov property of Brownian motion, one has 
 		\begin{equation}\label{243}
 			\begin{split}
 				\widetilde{\mathbb{P}}_w\big( \tau_{v}<\tau_{D} \big) \le &\widetilde{\mathbb{P}}_{w}\big( \tau_{\partial B_w(2)} <\tau_{D} \big)\cdot \max_{z\in \partial B_w(2)}  \widetilde{\mathbb{P}}_{z}\big( \tau_{B(dR)} < \infty \big) \\
 		 &	\cdot 	\max_{ z' \in \partial B(dR)}  \sum_{x\in \partial \mathcal{B}(R)  }  \widetilde{\mathbb{P}}_{z'}\big(\tau_{\partial \mathcal{B}(R) }  = \tau_{x}<\infty \big)\cdot \widetilde{\mathbb{P}}_x\big(\tau_{v}  <\tau_D \big). 
 			\end{split}
 		\end{equation}
 		 The probabilities on the right-hand side can be estimated as follows. Firstly, using the formulas \cite[(3.10)]{inpreparation_twoarm} and (\ref{order_green_function}) in turn, one has 
 		\begin{equation}\label{244}
 			\widetilde{\mathbb{P}}_{w}\big( \tau_{\partial B_w(2)} <\tau_{D} \big) \lesssim \mathrm{dist}(w,D) \land 1  \asymp \widetilde{G}_D(w,w). 
 		\end{equation}
 		 Secondly, by (\ref{cap1}) and (\ref{cap2}), we have: for any $z\in \partial B_w(2)$, 
 		 \begin{equation}
 		 	\widetilde{\mathbb{P}}_{z}\big( \tau_{B(dR)} < \infty \big)  \asymp  \big(  \tfrac{R}{|w|}\big)^{d-2}.  		 
 		 	\end{equation}
 		 Thirdly, it follows from Lemma \ref{lemma_BM_uniform} that for any $z' \in \partial B(dR)$ and $x\in \partial \mathcal{B}(R)$, 
 		 \begin{equation}
 		 	\widetilde{\mathbb{P}}_{z'}\big(\tau_{\partial \mathcal{B}(R) }  = \tau_{x}<\infty \big) \lesssim R^{1-d}. 
 		 \end{equation}
 		Moreover, the formula (\ref{211}) together with $\widetilde{G}(x,x)\lesssim 1$ implies that  
 		\begin{equation}\label{247}
 			\widetilde{\mathbb{P}}_x\big(\tau_{v}  <\tau_D \big) \lesssim \mathbb{P}\big( v\xleftrightarrow{(D)} x  \big)\cdot \big[\widetilde{G}_D(v,v)\big]^{-\frac{1}{2}}. 
 		\end{equation}
 		Plugging (\ref{244})-(\ref{247}) into (\ref{243}), we get 
 		\begin{equation}\label{248}
 			\widetilde{\mathbb{P}}_w\big( \tau_{v}<\tau_{D} \big) \lesssim |w|^{2-d}R^{-1} \cdot \tfrac{\widetilde{G}_D(w,w)}{\sqrt{\widetilde{G}_D(v,v)}}    \cdot\sum\nolimits_{x\in \partial \mathcal{B}(R)  }  \mathbb{P}\big( v\xleftrightarrow{(D)} x  \big). 
 		\end{equation}
 		Combining (\ref{242}), (\ref{248}) and $\widetilde{G}(w,w)\lesssim 1$, we complete the proof. 
 	\end{proof}

 	The following relation between the point-to-set and boundary-to-set connecting probabilities has been verified in \cite[(5.2)]{cai2024quasi}.



\begin{lemma}\label{lemma_relation}
For any $d\ge 3$ with $d\neq 6$, $D\subset \widetilde{\mathbb{Z}}^d$, $R\ge 1$, $i\in \{1,-1\}$, and $A \subset  \mathfrak{B}^{i}_{R,(10d^2)^{-1}}$,
 		\begin{equation} 
 		 			\mathbb{P}\big(A \xleftrightarrow{(D)} \partial B(R) \big) \lesssim R^{-(\frac{d}{2}\boxdot 3) } \sum\nolimits_{x\in \partial \mathcal{B}(d^{-i}R) } \mathbb{P}\big(A \xleftrightarrow{(D)} x \big). 
 		\end{equation}
\end{lemma}


Combining Lemmas \ref{lemma24}, \ref{lemma_new26} and \ref{lemma_relation}, we obtain the following estimate.

 	\begin{corollary}\label{lemma_point_to_boundary}
 		For any $3\le d\le 5$, there exist $C,c>0$ such that for any $D\subset \widetilde{\mathbb{Z}}^d$, $R\ge1 $, $v\in [\widetilde{B}(CR)]^c$, 
 		\begin{equation}
 			\mathbb{P}\big( v\xleftrightarrow{(D)} \partial B(R)  \big)\lesssim  R^{-\frac{d}{2}}\sum\nolimits_{z\in \partial \mathcal{B}(dR)}\mathbb{P}\big( v \xleftrightarrow{(D)} z , \big\{ v \xleftrightarrow{(D)} \partial B(cR) \big\}^c \big).  
 		\end{equation} 
 	\end{corollary}
 	\begin{proof}
 		 Using Lemmas \ref{lemma24}, \ref{lemma_new26} and \ref{lemma_relation}, we have 
 		\begin{equation}
 			\begin{split}
 				\mathbb{P}\big( v\xleftrightarrow{(D)} \partial B(R)  \big)\overset{(\text{Lemma}\ \ref{lemma_relation})  }{\lesssim } & R^{-\frac{d}{2}} \sum\nolimits_{z\in \partial \mathcal{B}(dR)}\mathbb{P}\big( v \xleftrightarrow{(D)} z  \big)\\
 				\overset{(\text{Lemma}\ \ref{lemma_new26})  }{\lesssim } &R^{-\frac{d}{2}} \sum\nolimits_{z\in \partial \mathcal{B}(dR)}\mathbb{P}\big( v \xleftrightarrow{(D\cup \partial B(\sqrt{c}R))} z  \big)  \\
 					\overset{(\text{Lemma}\ \ref{lemma24})  }{\lesssim } & R^{-\frac{d}{2}} \sum\nolimits_{z\in \partial \mathcal{B}(dR)}\mathbb{P}\big( v \xleftrightarrow{(D)} z , \big\{ v \xleftrightarrow{(D)} \partial B(cR) \big\}^c \big).   \qedhere
 					 			\end{split}
 		\end{equation}

 	\end{proof}

We also cite the following result on the scaling of point-to-set connecting probabilities (see \cite[Proposition 1.5]{cai2024quasi}).

 \begin{lemma}\label{lemma_point_to_set}
 	For any $d\ge 3$ with $d\neq 6$, there exists $c>0$ such that for any $R\ge 1$, $A,D\subset \widetilde{B}(cR)$ and $x_1,x_2 \in  [\widetilde{B}(R)]^c$, 
 \begin{equation}
 	\mathbb{P}\big(A\xleftrightarrow{(D)} x_1 \big) \cdot |x_1|^{d-2}\asymp \mathbb{P}\big(A\xleftrightarrow{(D)} x_2 \big) \cdot |x_2|^{d-2}. 
 \end{equation} 	
 \end{lemma}

 Combining Lemmas \ref{lemma_relation} and \ref{lemma_point_to_set}, one may obtain the following bound:

 \begin{lemma}\label{lemma_213}
 	For any $d\ge 3$ with $d\neq 6$, there exists $c>0$ such that for any $R\ge 1$, $A,D\subset \widetilde{B}(cR)$ and $x\in [\widetilde{B}(R)]^{c}$, 
 	\begin{equation}\label{ineqforlemma216}
 		\mathbb{P}\big(A \xleftrightarrow{(D)} \partial B(R) \big) \lesssim R^{-[(\frac{d}{2}-1)\boxdot 2] } |x|^{d-2}\cdot   \mathbb{P}\big(A \xleftrightarrow{(D)} x \big). 
 	\end{equation}
 \end{lemma}

In the analysis of loop clusters, decomposing connection probabilities is a fundamental technique. A key property in this context is quasi-multiplicativity (see \cite[Theorem 1.1]{cai2024quasi}), which was instrumental in the derivation of the IIC.

 \begin{lemma}[quasi-multiplicativity]\label{lemma_quasi}
 	For any $d\ge 3$ with $d\neq 6$, there exist $c,C>0$ such that for any $R\ge 1$, $A_1,D_1\subset \widetilde{B}(cR^{1\boxdot \frac{2}{d-4}})$ and $A_{-1},D_{-1}\subset \big[\widetilde{B}(CR^{1\boxdot \frac{d-4}{2}})\big]^c$,  
 	\begin{equation}\label{ineq_quasi}
 		\mathbb{P}\big( A_1\xleftrightarrow{(D_1\cup D_{-1})} A_{-1}\big) \asymp R^{0\boxdot (6-d)}\prod\nolimits_{i\in \{1,-1\}}\mathbb{P}\big( A_i \xleftrightarrow{(D_i)} \partial B(R)\big).
 	\end{equation}
 \end{lemma}

%
%

     As a supplement, we present the following decomposition on point-to-set connecting probabilities. To maintain the flow, we defer its proof to Appendix \ref{app_new_decompose_point_to_set}.

       \begin{lemma}\label{lemma_new_decompose_point_to_set}
  	For any $d\ge 3$ with $d\neq 6$, there exist $c,C>0$ such that for any $R\ge 1$, $x\in \partial B(R)$, $y\in \mathfrak{B}^{-1}_{R,c}$, and any sets $A\subset \mathfrak{B}^1_{R,c}$ and $D\subset \mathfrak{B}^1_{R,c}\cup \mathfrak{B}^{-1}_{R,c}$,  
   			\begin{equation}
   			\mathbb{P}\big( A \xleftrightarrow{(D  )} y\big) \gtrsim R^{d-2}\cdot \mathbb{P}\big( A \xleftrightarrow{(D )} x \big) \cdot \big[  \mathbb{P}\big( x\xleftrightarrow{(D )} y\big)-  \tfrac{C[\mathrm{diam}(A)]^{d-4}}{R^{d-4}|y|^{2-d}} \cdot \mathbbm{1}_{d\ge 7} \big].
   	\end{equation}
  	    \end{lemma}

%
%
%
%


  {\color{red}
  

  
 


  }


%
%
%
%
%

   The subsequent corollary provides a decomposition for two-point events.


   \begin{corollary}\label{coro28}
   	For any $3\le d\le 5$, there exist $C>0$ and $c>c'>c''>0$ such that for any $D\subset \widetilde{\mathbb{Z}}^d$, $R\ge C$, $R_1\le c''R$, $R_{-1}\ge \frac{R}{c''}$, and $w_{-1}\in \partial \mathcal{B}(R_{-1})$, 
   	   	\begin{equation}\label{ineq_coro_28}
   		\begin{split}
   		& 	 \sum\nolimits_{w_1\in \partial \mathcal{B}(R_1)}   \mathbb{P}\big(w_1 \xleftrightarrow{(D)} w_{-1}    \big)\\
   			 \lesssim &R^{-d} \sum_{w_1\in \partial \mathcal{B}(R_1),v_1\in \widetilde{\partial}\mathfrak{B}^{1}_{R,c'},v_{-1}\in \widetilde{\partial}\mathfrak{B}^{-1}_{R,c'}}\mathbb{P}\big(  v_1\xleftrightarrow{(D)}w_1, v_{-1}\xleftrightarrow{(D)}w_{-1} , \\
 		&  \ \ \ \ \ \ \ \ \ \ \ \ \ \ \ \   \big\{ v_{1} \xleftrightarrow{(D)} \widetilde{\partial }\mathfrak{B}^{1}_{R,c}\cup  \widetilde{\partial }\mathfrak{B}^{1}_{R_{1},c}  \big\}^c  , \big\{ v_{-1} \xleftrightarrow{(D)} \widetilde{\partial }\mathfrak{B}^{-1}_{R,c} \big\}^c \big). 
   		\end{split}
   	\end{equation} 
   \end{corollary}
   \begin{proof}
   For any $w_1\in \partial \mathcal{B}(R_1)$ and $w_{-1}\in \partial \mathcal{B}(R_{-1})$, it follows from (\ref{211}) that 
   \begin{equation}\label{do253}
   	  \mathbb{P}\big(w_1 \xleftrightarrow{(D)} w_{-1}    \big) \asymp \widetilde{\mathbb{P}}_{w_1}\big(\tau_{w_{-1}}<\tau_D \big)\cdot \sqrt{\tfrac{\widetilde{G}_{D}(w_{-1},w_{-1})}{\widetilde{G}_{D}(w_{1},w_{1})}}. 
   \end{equation}
  By the last-exit decomposition and the strong Markov property of Brownian motion,  the probability $\widetilde{\mathbb{P}}_{w_1}\big(\tau_{w_{-1}}<\tau_D \big)$ is upper-bounded by 
     \begin{equation}\label{final267}
   	\begin{split}
   	  &C\sum_{v_1\in \widetilde{\partial}\mathfrak{B}^{1}_{R,c'}} \widetilde{G}_{D}(w_1,v_1) \sum_{v_1'\in [\mathfrak{B}^{1}_{R,c'}]^c: \{v_1,v_1'\}\in \mathbb{L}^d} \widetilde{\mathbb{P}}_{v_1'}\big( \tau_{\partial B(R) }<\tau_{\mathfrak{B}^{1}_{R,c'}} \big)  \\
   	    &\cdot \max_{z\in  \partial B(R)} \sum\nolimits_{v_{-1}\in \widetilde{\partial}\mathfrak{B}^{-1}_{R,c'}} \widetilde{\mathbb{P}}_z\big( \tau_{\widetilde{\partial}\mathfrak{B}^{-1}_{R,c'}}= \tau_{v_{-1}}<\infty  \big)  \cdot \widetilde{\mathbb{P}}_{v_{-1}}\big(\tau_{w_{-1}}<\tau_{D} \big).
   	\end{split}
   \end{equation}
   Moreover, using the potential theory, we have   
   \begin{equation}\label{final268}
   	\widetilde{\mathbb{P}}_{v_1'}\big(  \tau_{\partial B(R) }<\tau_{\mathfrak{B}^{1}_{R,c'}}  \big) \lesssim R^{-1}.
   \end{equation}
   In addition, Lemma \ref{lemma_BM_uniform} implies that for any $z\in \partial B(R)$ and $v_{-1}\in \widetilde{\partial}\mathfrak{B}^{-1}_{R,c'}$, 
   \begin{equation}\label{do256}
   	\widetilde{\mathbb{P}}_z\big( \tau_{\widetilde{\partial}\mathfrak{B}^{-1}_{R,c'}}= \tau_{v_{-1}}<\infty  \big)  \lesssim R^{1-d}.
   \end{equation}
   Plugging (\ref{final268}) and (\ref{do256}) into (\ref{final267}), and then combining with (\ref{do253}), we have 
   \begin{equation*}
   	\begin{split}
   		 \mathbb{P}\big(w_1 \xleftrightarrow{(D)} w_{-1}    \big) \lesssim  & R^{-d} \sum_{v_1\in \widetilde{\partial}\mathfrak{B}^{1}_{R,c'},v_{-1}\in \widetilde{\partial}\mathfrak{B}^{-1}_{R,c'}}\widetilde{G}_{D}(w_1,v_1)      \widetilde{\mathbb{P}}_{v_{-1}}\big(\tau_{w_{-1}}<\tau_{D} \big)  \sqrt{\tfrac{\widetilde{G}_{D}(w_{-1},w_{-1})}{\widetilde{G}_{D}(w_{1},w_{1})}}\\
   		\overset{(\ref{211})}{ \lesssim} & R^{-d} \prod\nolimits_{i\in \{1,-1\}} \sum\nolimits_{v_i \in \widetilde{\partial}\mathfrak{B}^{i}_{R,c'} } \mathbb{P}\big(w_i\xleftrightarrow{(D)}v_i  \big).  
   	\end{split}
   \end{equation*}
   Summing over $w_1\in \partial \mathcal{B}(R_1)$, and using Lemmas \ref{lemma24} and \ref{lemma_new26}, we obtain 
   \begin{equation*}
   	\begin{split}
   			&  \sum\nolimits_{w_1\in \partial \mathcal{B}(R_1)}   \mathbb{P}\big(w_1 \xleftrightarrow{(D)} w_{-1}    \big) \\
   			 \lesssim & R^{-d} \sum\nolimits_{ w_1\in \partial \mathcal{B}(R_1),v_1\in \widetilde{\partial}\mathfrak{B}^{1}_{R,c'} } \mathbb{P}\big(w_1\xleftrightarrow{(D)}v_1   \big)\cdot \sum\nolimits_{v_{-1}\in \widetilde{\partial}\mathfrak{B}^{-1}_{R,c'}}\mathbb{P}\big(w_{-1}\xleftrightarrow{(D)}v_{-1}  \big)\\
   			\overset{(\text{Lemma}\ \ref{lemma_new26})}{\lesssim} & R^{-d} \sum\nolimits_{ w_1\in \partial \mathcal{B}(R_1),v_1\in \widetilde{\partial}\mathfrak{B}^{1}_{R,c'} } \mathbb{P}\big(w_1\xleftrightarrow{(D \cup \widetilde{\partial} \mathfrak{B}^1_{R,\sqrt{c\cdot c'}}\cup \widetilde{\partial} \mathfrak{B}^{1}_{R_1,\sqrt{c\cdot c'}} )}v_1   \big)\\
   			 & \cdot \sum\nolimits_{v_{-1}\in \widetilde{\partial}\mathfrak{B}^{-1}_{R,c'}}\mathbb{P}\big(w_{-1}\xleftrightarrow{(D  \cup \widetilde{\partial} \mathfrak{B}^{-1}_{R,\sqrt{c\cdot c'}})}v_{-1}  \big)\\
   			 \overset{(\text{Lemma}\ \ref{lemma24})}{\lesssim } &R^{-d}  \sum\nolimits_{ w_1\in \partial \mathcal{B}(R_1),v_1\in \widetilde{\partial}\mathfrak{B}^{1}_{R,c'} } \mathbb{P}\big(  v_{1} \xleftrightarrow{(D)} w_{1},  \big\{ v_{1} \xleftrightarrow{(D)} \widetilde{\partial }\mathfrak{B}^{1}_{R,c}\cup  \widetilde{\partial }\mathfrak{B}^{1}_{R_{1},c}  \big\}^c \big) \\
   			 &\cdot  \sum\nolimits_{ v_{-1}\in \widetilde{\partial}\mathfrak{B}^{-1}_{R,c'} }  \mathbb{P}\big(  v_{-1} \xleftrightarrow{(D)} w_{-1},  \big\{ v_{-1} \xleftrightarrow{(D)} \widetilde{\partial }\mathfrak{B}^{-1}_{R,c}  \big\}^c \big). 
   	\end{split}
   \end{equation*}
   Combined with Corollary \ref{coro27}, it implies the desired bound (\ref{ineq_coro_28}).
   \end{proof}

     \subsection{Brownian excursion measure}\label{subsection_BEM}

    We arbitrarily take $D\subset \widetilde{\mathbb{Z}}^d$. As pointed out in \cite{werner2025switching}, the loops of $\widetilde{\mathcal{L}}_{1/2}^D$ intersecting a fixed point can be equivalently described by a Poisson point process of excursions. Specifically, for any $v\in \widetilde{\mathbb{Z}}^d$, rooted loop $\widetilde{\varrho}$ with $\widetilde{\varrho}(0)=\widetilde{\varrho}(T)=v$ (where $T$ is the duration of $\widetilde{\ell}$), and $0<t<T$, define 
    \begin{equation}
    	\delta_t^{-}=\delta_t^{-}(\widetilde{\varrho},v):=\sup\{s\le t:\widetilde{\varrho}(s)=v\}, 
    \end{equation}     
    \begin{equation}
    	\delta_t^{+}=\delta_t^{+}(\widetilde{\varrho},v):=\inf\{s\ge t:\widetilde{\varrho}(s)=v\}. 
    \end{equation}
    Clearly, $\widetilde{\varrho}$ consists of the excursions $\{\widetilde{\varrho}[\delta_t^{-},\delta_t^{+}]:0<t<T,\delta_t^{-}<\delta_t^{+}\}$. Let $\mathscr{E}_v(\widetilde{\varrho})$ denote the point measure supported on these excursions. Note that $\mathscr{E}_v(\widetilde{\varrho})=\mathscr{E}_v(\widetilde{\varrho}')$ when $\widetilde{\varrho}$ and $\widetilde{\varrho}'$ belong to the same loop. Hence, for any loop $\widetilde{\ell}$ intersecting $v$, we define $\mathscr{E}_v(\widetilde{\ell})$ as $\mathscr{E}_v(\widetilde{\varrho})$ for any $\widetilde{\varrho}\in \widetilde{\ell}$ with $\widetilde{\varrho}(0)=\widetilde{\varrho}(T)=v$. For completeness, we set $\mathscr{E}_v(\widetilde{\ell}):=0$ when $\widetilde{\ell}$ is disjoint from $v$.

       For any $D\subset \widetilde{\mathbb{Z}}^d$, $v\in \widetilde{\mathbb{Z}}^d\setminus D$ and $a>0$, conditioned on the event $\big\{\widehat{\mathcal{L}}_{1/2}^{D,v}=a\big\}$, the point measure $\widetilde{\mathcal{E}}_{v}^D:= \sum\nolimits_{\widetilde{\ell}\in \widetilde{\mathcal{L}}^{D}_{1/2}} \mathscr{E}_v(\widetilde{\ell})$ is a Poisson point process, whose intensity measure $\mathbf{e}^{D}_{v,v}$ (commonly referred to as the ``Brownian excursion measure'') is a $\sigma$-finite measure on the space of excursions from $v$. (\textbf{P.S.} As in the case of ``It\^o's excursion measure'', $\mathbf{e}^{D}_{v,v}$ lacks a direct and simple definition.) In addition, for any $v\neq w\in \widetilde{\mathbb{Z}}^d$, we may restrict $\mathbf{e}_{v,v}^{D}$ to excursions that hit $w$, and for each excursion we only keep the part up to the first time it reaches $w$. This yields the measure $\mathbf{e}^{D}_{v,w}$ on the space of excursions from $v$ to $w$. According to \cite[Lemma 3.14]{inpreparation_twoarm}, the total mass of $\mathbf{e}^{D}_{v,w}$ exactly equals the boundary excursion kernel $\mathbb{K}_{D\cup \{v,w\}}(v,w)$.

       The Brownian excursion measure $\mathbf{e}_{v,w}^{D}$ (for both $v=w$ and $v\neq w$) possesses the strong Markov property, which greatly facilitates related computations. Precisely, for any $A\subset \widetilde{\mathbb{Z}}^d\setminus (D\cup \{v,w\})$, under the measure $\mathbf{e}_{v,w}^{D}$, given the configuration of an excursion $\widetilde{\eta}$ up to the first time it hits $A$ (suppose that the first hitting point is $z\in \widetilde{\partial}A$), the remaining part of $\widetilde{\eta}$ is distributed by $	\widetilde{\mathbb{P}}_z\big( \big\{\widetilde{S}_t \big\}_{0 \le t\le \tau_{w}} \in \cdot  \mid \tau_{w}<\tau_{D\cup \{v\}} \big)$. 

 The following lemma, as a strengthened version of \cite[Lemma 3.17]{inpreparation_twoarm}, provides an upper bound on the total mass of excursions that reach a distant location.

\begin{lemma}\label{lemma_for_new62}
	For any $3\le d\le 5$, there exist $C>0$, $C'>C^2$ and $c>0$ such that the following holds. Suppose that $r\ge 1$, $R\ge C'r$, $\{x_i,y_i\}\in \mathbb{L}^d$ satisfying $x_i\in B(\frac{r}{10})$ for $i\in \{1,2\}$, $v_i\in  I_{\{x_i,y_i\}}$ satisfying $|v_i-x_i|\land |v_i-y_i|\ge \frac{1}{10}$ for $i\in \{1,2\}$, $D\subset \widetilde{\mathbb{Z}}^d$ satisfying $x_i\in D$ and $D\cap I_{[v_i,y_i]}=\emptyset$ for $i\in \{1,2\}$, $v_i'\in I_{[v_i,y_i]}$ satisfying $|v_i'-v_i|\land |v_i'-y_i|\ge \frac{1}{100}$ for $i\in \{1,2\}$, and $A\subset [\widetilde{B}(R)]^c$. Then we have  
	\begin{equation}\label{addto242}
 	\begin{split}
 		& \mathbf{e}^D_{v_1,v_2}\big( \big\{ \widetilde{\eta}: \mathrm{ran}(\widetilde{\eta})\cap A \neq \emptyset \big\} \big)\\
 		\lesssim &r^{1-d}R^{1-d}\cdot    \sum\nolimits_{z\in \partial \mathcal{B}(Cr),z'\in \partial \mathcal{B}(\frac{R}{C})} \mathbb{P}\big( z\xleftrightarrow{(D)} z', \big\{z \xleftrightarrow{(D)} \partial B(10r )\cup \partial B(\tfrac{R}{10})\big\}^c\big)  \\
 		&\cdot   \mathbf{Q}(v_1',v_2')\cdot \max_{z''\in \partial \mathcal{B}(\frac{10R}{C})}   \widetilde{\mathbb{P}}_{z''}\big(\tau_A< \tau_{D} \big).
 	\end{split}
 	\end{equation} 
 	Here the quantity $\mathbf{Q}(v_1',v_2')$ is defined as follows:
 	\begin{itemize}
 		
 		\item when $\chi:=|v_1'-v_2'|\le 10c^{-1}$, $\mathbf{Q}(v_1',v_2'):=\widetilde{\mathbb{P}}_{v_1'}\big(\tau_{\partial B(r)}<\tau_{D} \big)$;

 		\item when $\chi> 10c^{-1}$, we define 
 		\begin{equation}
 			\mathbf{Q}(v_1',v_2'):= \chi^{-2}  \prod\nolimits_{i\in \{1,2\}} \sum\nolimits_{z_i\in \partial \mathcal{B}_{v_i'}(c\chi )}  \mathbb{P}\big( v_i'\xleftrightarrow{(D)} z_i, \big\{ v_i'\xleftrightarrow{(D)} \partial B_{v_i'}(\tfrac{\chi}{100}) \big\}^c \big). 
 		\end{equation}

 	\end{itemize}
\end{lemma}
\begin{proof}
	For $i\in \{1,2\}$, since $x_i\in D$, all excursions $\widetilde{\eta}$ starting from $v_i$ and reaching $A$ before $D$ must intersect $v_i'$. Thus, by the strong Markov property of $\mathbf{e}^D_{v_1,v_2}$, 
	\begin{equation}
		\mathbf{e}^D_{v_1,v_2}\big( \big\{ \widetilde{\eta}: \mathrm{ran}(\widetilde{\eta})\cap A \neq \emptyset \big\} \big) \asymp \widetilde{\mathbb{P}}_{v_1'}\big( \tau_{A}<\tau_{v_2'}<\tau_{D} \big). 
	\end{equation}
	Define $\mathbf{R}_1:=\widetilde{\mathbb{P}}_{v_1'}\big(\tau_{\partial B(r)}<\tau_{D} \big)$ and $\mathbf{R}_2:=1$ when $\chi\le 10c^{-1}$, and define $\mathbf{R}_i:=\widetilde{\mathbb{P}}_{v_i'}\big( \tau_{\partial \mathcal{B}_{v_i'}(10c\chi)} <\tau_D \big)$ for $i\in \{1,2\}$ when $\chi > 10c^{-1}$. Corollary \ref{coronew27} implies that 
	 \begin{equation}\label{addto246}
	 	\mathbf{R}_1\mathbf{R}_2
	 	\lesssim \mathbf{Q}(v_1',v_2'). 
	 \end{equation}	 
	Using the strong Markov property of Brownian motion repeatedly, we have 
	 \begin{equation}\label{newadd245}
	 	\begin{split}
	 		& \widetilde{\mathbb{P}}_{v_1'}\big( \tau_{A}<\tau_{v_2}<\tau_{D} \big)\\
	 		 \le & \mathbf{R}_1 \cdot  \max_{w\in  B(r)} \sum\nolimits_{z\in \partial B(Cr)} \widetilde{\mathbb{P}}_{w} \big( \tau_{\partial B(Cr)} = \tau_{z} < \tau_{D}\big)\cdot  \widetilde{\mathbb{P}}_{z}\big( \tau_{\partial B(\frac{10R}{C})} <\tau_D \big)\\
	 		 &\ \ \ \ \ \ \ \ \ \ \ \ \ \ \ \ \ \ \ \ \ \ \ \  \cdot \max_{w'\in \partial B(\frac{10R}{C})}  \widetilde{\mathbb{P}}_{w'}\big( \tau_{A} <\tau_D \big) 
	 		    \cdot  \max_{w''\in A} \widetilde{\mathbb{P}}_{w''}\big( \tau_{v_2}<\tau_{D} \big).
	 	\end{split}
	 \end{equation}

	 Next, we estimate the probabilities on the right-hand side separately. Firstly, by Lemma \ref{lemma_BM_uniform} one has: for any $w\in  B(r)$ and $z\in \partial B(Cr)$, 
	\begin{equation}\label{newadd246}
 	\widetilde{\mathbb{P}}_{w} \big( \tau_{\partial B(Cr)} = \tau_{z} < \tau_{D}\big)  \le \widetilde{\mathbb{P}}_{w} \big( \tau_{\partial B(Cr)} = \tau_{z} < \infty \big)  \lesssim  r^{1-d}. 
	\end{equation}
  Secondly, applying Lemma \ref{lemma_new26} and Corollary \ref{coronew27}, we have  
	 \begin{equation}\label{newadd247}
	 	\begin{split}
	 		& \sum\nolimits_{z\in \partial B(Cr)} \widetilde{\mathbb{P}}_{z}\big( \tau_{\partial B(\frac{10R}{C})} <\tau_D \big)  \\ \overset{(\text{Corollary}\ \ref{coronew27})}{\lesssim}
	 		&\frac{1}{R}\sum_{z\in \partial B(Cr), z'\in \partial \mathcal{B}(\frac{R}{C})} \mathbb{P}\big(   z\xleftrightarrow{(D)} z' , \big\{z \xleftrightarrow{(D)} \partial B(\tfrac{R}{10})\big\}^c  \big)\\ 
	 	 	\overset{(\text{Lemma}\ \ref{lemma_new26})}{	\lesssim } &\frac{1}{R}\sum_{z\in \partial B(Cr), z'\in \partial \mathcal{B}(\frac{R}{C})} \mathbb{P}\big( z\xleftrightarrow{(D)} z', \big\{z \xleftrightarrow{(D)} \partial B(10r )\cup \partial B(\tfrac{R}{10}) \big\}^c \big).
	 	\end{split}
	 \end{equation}
Thirdly, for any $w''\in A$, the same argument as in the proof of (\ref{revisenew_222}) yields 
	 \begin{equation}\label{newadd252}
	 	\widetilde{\mathbb{P}}_{w''}\big( \tau_{v_2'}<\tau_{D} \big)\lesssim R^{2-d}\cdot  \mathbf{R}_2. 	
	 	 \end{equation} 
 Plugging (\ref{newadd246}), (\ref{newadd247}) and (\ref{newadd252}) into (\ref{newadd245}), and then using (\ref{addto246}), we obtain the desired bound (\ref{addto242}).
	 \end{proof}

\textbf{Switching identity.} In what follows, we review the powerful tool ``switching identity'', which was introduced in \cite{werner2025switching}. Roughly speaking, it provides an explicit description of the loop cluster connecting two given points. We present the following version tailored to our needs in this paper.

  \begin{lemma}[switching identity]\label{lemma_switching}
 	 For any $d\ge 3$, $D\subset  \widetilde{\mathbb{Z}}^d$, $v\neq w\in \widetilde{\mathbb{Z}}^d\setminus D$, and $a,b>0$, conditioned on the event $\big\{ v\xleftrightarrow{(D)} w,  \widehat{\mathcal{L}}_{1/2}^{D,v}=a, \widehat{\mathcal{L}}_{1/2}^{D,w}=b  \big\}$, the occupation field $\big\{ \widehat{\mathcal{L}}_{1/2}^{D,z} \big\}_{z\in \widetilde{\mathbb{Z}}^d}$ has the same distribution as the total local time of the following four independent components:
 	 \begin{itemize}
 	 	\item[(1)] loops in the loop soup $\widetilde{\mathcal{L}}_{1/2}^{D\cup \{v,w\}}$;

 	 	\item[(2)] a Poisson point process with intensity measure $a\cdot \mathbf{e}_{v,v}^{ D\cup \{w\}}$;

 	 	\item[(3)] a Poisson point process with intensity measure $b \cdot \mathbf{e}_{w,w}^{D\cup \{v\}}$;

 	 	\item[(4)] a Poisson point process with intensity measure $\sqrt{ab} \cdot \mathbf{e}_{v,w}^D$, where the number of excursions is conditioned to be odd.

 	 \end{itemize}

 \end{lemma}

For ease of applying this lemma, we fix the following notation.
 \begin{itemize}
 
   \item  For each $i\in \{1,2,3,4\}$, we denote by $\mathcal{P}^{(i)}$ the Poisson point process in Component ($i$).

    \item We denote by $\widecheck{\mathbb{P}}_{v\leftrightarrow w,a,b}^D$ the probability measure generated by Components (1)-(4). The expectation under $\widecheck{\mathbb{P}}_{v\leftrightarrow w,a,b}^D$ is denoted by $\widecheck{\mathbb{E}}^D_{v\leftrightarrow w,a,b}$.

   \item We define $\widecheck{\mathcal{C}}:=\{z\in \widetilde{\mathbb{Z}}^d:
   z\xleftrightarrow{\cup (\sum_{1\le i\le 4}\mathcal{P}^{(i)})} v  \}$ as the cluster containing $v$ and consisting of all loops and excursions in $\sum_{1\le i\le 4}\mathcal{P}^{(i)}$.

  \item Let $\mathfrak{p}_{v\leftrightarrow w,a,b}^D$ denote the density of the event $\big\{v\xleftrightarrow{(D)} w,  \widehat{\mathcal{L}}_{1/2}^{D,v}=a, \widehat{\mathcal{L}}_{1/2}^{D,w}=b\big\}$.

 \end{itemize}

To facilitate the analysis of $\mathcal{P}^{(4)}$, we record the following fact. Let $ \mathbf{X}$ be a Poisson random variable with parameter $0<\lambda\lesssim 1$. In addition, assume that $\{\mathbf{q}_j\}_{j\in \mathbb{N}}$ is an independent  family of i.i.d. Bernoulli random variables with expectation $q$. Let $\mathbf{Y}:=\sum_{1 \le j\le \mathbf{X}}\mathbf{q}_j $. Then one has 
\begin{equation} \label{odd_formula}
\begin{split}
\mathbb{P}\big( \mathbf{Y}>0 \mid \mathbf{X}\ \text{is odd} \big) = & 1-\frac{\sum_{m\ge 1:m\ \text{odd}}e^{-\lambda}\cdot \frac{\lambda^m}{m!}\cdot (1-q)^{m}}{\sum_{m\ge 1:m\ \text{odd}}e^{-\lambda}\cdot \frac{\lambda^m}{m!}} \\
=  & \frac{\sinh(\lambda)-\sinh(\lambda(1-q))}{\sinh(\lambda)} \asymp q.
\end{split}
\end{equation}

\section{Proof of Theorem \ref{thm_minimal_distance}}\label{section_proof_thm1.1}

In this section, we present the proof of Theorem \ref{thm_minimal_distance} assuming Proposition \ref{lemma_four_arms}, which, as introduced later, serves as a crucial estimate for controlling the second moment of touching points between two large clusters.

We begin with the upper bound in (\ref{ineq_minimal_distance}). When $\chi^{\frac{1}{2}}N^{(\frac{d}{2}-1)\boxdot (d-4)}\ge 1$ (i.e., $\chi\ge N^{-[(d-2)\boxdot (2d-8)]}$), this upper bound holds trivially. Therefore, it suffices to consider the case when $\chi \le N^{-[(d-2)\boxdot (2d-8)]}$. Since $\widetilde{\mathbb{Z}}^d$ is locally one-dimensional, it follows from a standard covering argument that we may select subsets $A_1,...,A_K\subset \widetilde{B}(N)$ of diameter $2\chi$ (with $K\lesssim \chi^{-1} N^d$) such that any pair $v,w\in \widetilde{B}(N)$ with $|v-w|\le \chi$ is contained in some $A_i$. Thus, on $\{\mathcal{D}_{N,c}\le \chi \}$, some $A_i$ intersects two distinct sign clusters of diameter at least $cN$ and hence, there exist two distinct points $v,v'\in \widetilde{\partial }A_i$ (note that $|\widetilde{\partial }A_i|\le 2d$) that are connected to $\partial B_{v}(c'N)$ by two different sign clusters. By the isomorphism theorem and (\ref{thm1_small_n}), the probability of this event is $O(\chi ^{\frac{3}{2}}N^{-[(\frac{d}{2}+1)\boxdot 4]})$. In conclusion, we obtain the upper bound in (\ref{ineq_minimal_distance}): 
  \begin{equation}
  	\begin{split}
  		\mathbb{P}\big(\mathcal{D}_{N,c}\le \chi  \big)  \lesssim \sum\nolimits_{1\le i\le K}|\widetilde{\partial }A_i|^2 \cdot  \chi^{\frac{3}{2}}N^{-[(\frac{d}{2}+1)\boxdot 4]} \lesssim \chi^{\frac{1}{2}}N^{(\frac{d}{2}-1)\boxdot (d-4)}. 
  	\end{split}
  \end{equation}


%
%
%
%

It remains to prove the lower bound in (\ref{ineq_minimal_distance}). Since the right-hand side of (\ref{ineq_minimal_distance}) equals $1$ for all $\chi \ge N^{-[(d-2)\boxdot (2d-8)]}$, and $\mathbb{P}(\mathcal{D}_{N,c}\le \chi   )$ is increasing in $\chi$, it suffices to consider the case $\chi \le N^{-[(d-2)\boxdot (2d-8)]}$. We take a small constant $\cl\label{const_dagger_box}>0$, and denote the boxes 
 \begin{equation}\label{box_frak}
 	\mathfrak{B}_j:= B_{(2j\cdot \cref{const_dagger_box} N,0,...,0)}(\cref{const_dagger_box}^2 N),\ \ \forall j\in \{-1,0,1\}. 
 \end{equation}
  We define $\mathfrak{I}$ as the collection of edges $e= \{x,y\}\in \mathbb{L}^d$ with $x,y\in \mathfrak{B}_0$. Note that $|\mathfrak{I}|\asymp N^d$. For any edge $e\in \mathbb{L}^d$, we denote its two trisection points by $v_e^{-}$ and $v_e^{+}$, ordered such that the sum of the coordinates of $v_e^{-}$ is smaller than that of $v_e^{+}$. For each $e\in \mathfrak{I}$, we define the event 
     \begin{equation}\label{newadd321}
    \begin{split}
    		\mathsf{A}_{e}:=& \big\{ v_e^{+}\xleftrightarrow{\ge 0} \partial B(\cref{const_dagger_box} N), v_e^{-}\xleftrightarrow{\le 0} \partial B(\cref{const_dagger_box} N)  \big\}\cap  \big\{    |\widetilde{\phi}_{v_e^{+}}|, |\widetilde{\phi}_{v_e^{-}}|\in [\cref{const_lowd_four_point2},\Cref{const_lowd_four_point1} ]   \big\} \\
    		&\cap \big\{v_e^{+}  \xleftrightarrow{\ge 0 } \partial B(N) \big\}^c  \cap  \big\{v_e^{-}\xleftrightarrow{\le 0 } \partial B(N) \big\}^c. 
    \end{split}
    \end{equation}
    Let $\mathfrak{I}_{\mathsf{A}}$ denote the collection of edges $e\in \mathfrak{I}$ such that $\mathsf{A}_{e}$ occurs.

We claim that there exists $\cl\label{const_star}>0$ such that 
\begin{equation}\label{claim3.4}
		\mathbb{P}\big( |\mathfrak{I}_{\mathsf{A}}| \ge  \cref{const_star} N^{(\frac{d}{2}-1)\boxdot (d-4)}  \big)\gtrsim 1. 
\end{equation}
Before verifying this claim, we first prove Theorem \ref{thm_minimal_distance} using it. Precisely, on the event $\{\mathfrak{I}_{\mathsf{A}}=\omega\}$ (for any realization $\omega$ of $\mathfrak{I}_{\mathsf{A}}$), we explore the sign cluster $\mathcal{C}^{\pm}_{\partial B(\cref{const_dagger_box} N)}$ from $\partial B(\cref{const_dagger_box} N)$, locally pausing upon encountering any $v_e^{-}$ or $v_{e}^{+}$ for $e\in \omega$, and resuming the exploration elsewhere. After the exploration is finished, by the strong Markov property of $\widetilde{\phi}_\cdot$, the restrictions of the GFF to the subintervals $\{I_{[v_e^{-},v_e^{+}]}: e\in  \omega\}$ are independent for different $e\in \omega$. In addition, for each $e\in \omega$, the conditional probability of the event 
 \begin{equation}
 	\mathsf{G}_e:= \big\{   \mathrm{dist}(\mathcal{C}_{v_e^+}^{+}\cap I_{[v_e^{-},v_e^{+}]} , \mathcal{C}_{v_e^-}^{-}\cap I_{[v_e^{-},v_e^{+}]} )\le \chi\big\}
 \end{equation}
 is at least $c\chi^{\frac{1}{2}}$. To see this, we take $c'\chi^{-1}$ disjoint subintervals $I_{[z_1^-,z_1^+]},...,I_{[z_K^-,z_K^+]}\subset I_{[v_e^{-},v_e^{+}]}$ such that for all $1\le j\le K$, $\mathrm{dist}(I_{[z_j^-,z_j^+]}, \{v_e^{-},v_e^{+}\} )\ge \frac{d}{9}$, $|z_j^--v_e^-|<|z_j^--v_e^+|$ and $|z_j^--z_j^+|=\chi$. Note that the events 
 \begin{equation}
 	\mathsf{G}_e^j:= \big\{z_j^+ \xleftrightarrow{\ge 0} v_e^{+} , z_j^- \xleftrightarrow{\le 0} v_e^{-}\big\}  
 \end{equation}
 for $1\le j\le K$ are disjoint, and that $\mathsf{G}_e\supset \cup_{1\le j\le K} \mathsf{G}_e^j$. Moreover, by \cite[(6.27)]{inpreparation_twoarm}, the conditional probability (given the exploration) of each $\mathsf{G}_e^j$ is at least of order $\chi^{\frac{3}{2}}$. Thus, we obtain that $\mathsf{G}_e$ occurs with probability at least of order $\chi^{-1}\cdot \chi^{\frac{3}{2}}=\chi^{\frac{1}{2}}$, as claimed above. To sum up, we have 
\begin{equation}\label{613}
	\begin{split}
		 \mathbb{P}\big( \mathcal{D}_{N,\frac{1}{2}\cref{const_dagger_box}}\le \chi  \big)  
		\ge &\mathbb{P}\big(|\mathfrak{I}_{\mathsf{A}}|\ge \cref{const_star}  N^{(\frac{d}{2}-1)\boxdot (d-4)} , \cup_{e\in  \mathfrak{I}_{\mathsf{A}}} \mathsf{G}_e \big)\\
		\ge &\sum\nolimits_{\omega\subset \mathfrak{I}: |\omega|\ge \cref{const_star}N^{(\frac{d}{2}-1)\boxdot (d-4)}  }  \mathbb{P}\big(\mathfrak{I}_{\mathsf{A}} =\omega  \big) \cdot \mathbb{P}\big( \mathbf{Z}>0 \big),
	\end{split}
\end{equation}
  where $\mathbf{Z}\sim \mathrm{Binomial}(\lfloor \cref{const_star}N^{(\frac{d}{2}-1)\boxdot (d-4)}\rfloor, c\chi^{\frac{1}{2}})$, i.e., the sum of $\lfloor \cref{const_star}N^{(\frac{d}{2}-1)\boxdot (d-4)}\rfloor$ i.i.d. Bernoulli random variables with expectation $c\chi^{\frac{1}{2}}$. Note that a direct application of the second moment method yields $\mathbb{P}\big( \mathbf{Z}>0 \big)\gtrsim \chi^{\frac{1}{2}}N^{(\frac{d}{2}-1)\boxdot (d-4)}$. Plugging this into (\ref{613}), and using the claim (\ref{claim3.4}), we obtain the lower bound in (\ref{ineq_minimal_distance}): 
 \begin{equation}
 	\begin{split}
 		\mathbb{P}\big( \mathcal{D}_{N,\frac{1}{2}\cref{const_dagger_box}}\le \chi  \big)    \gtrsim \chi^{\frac{1}{2}}N^{(\frac{d}{2}-1)\boxdot (d-4)}\cdot \mathbb{P}\big(|\mathfrak{I}_{\mathsf{A}}|\ge \cref{const_star}  N^{\frac{d}{2}-1}\big) \overset{}{\gtrsim } \chi^{\frac{1}{2}}N^{(\frac{d}{2}-1)\boxdot (d-4)}. 
 	\end{split}
 \end{equation}

Next, we prove (\ref{claim3.4}) for the low- and high- dimensional cases separately.

\textbf{When $3\le d\le 5$.} According to \cite[Lemma 6.2]{inpreparation_twoarm}, there exist $\Cl\label{const_lowd_four_point1}>\cl\label{const_lowd_four_point2}>0$ such that for any $w_+\in \mathfrak{B}_1$, $w_-\in \mathfrak{B}_{-1}$ and $e\in \mathfrak{I}$, 
 \begin{equation}\label{revise_new_39}
 	\mathbb{P}\big( \mathsf{A}_e^{w_+,w_-} \big)\gtrsim  N^{-\frac{3d}{2}+1}, 
 \end{equation}
    where $\mathsf{A}_e^{w_+,w_-}$ is defined as $\mathsf{H}^{v_e^{+}, w_+ }_{v_e^{-}, w_-}\cap \{ |\widetilde{\phi}_{v_e^{+}}| , |\widetilde{\phi}_{v_e^{-}}| \in [\cref{const_lowd_four_point2},\Cref{const_lowd_four_point1}]\}\cap \big\{v_e^{+}  \xleftrightarrow{\ge 0 } \partial B(N) \big\}^c\cap  \big\{v_e^{-} \xleftrightarrow{\le 0 } \partial B(N) \big\}^c$. We denote by $\mathfrak{I}_{\mathsf{A}}^{w_+,w_-}$ the collection of edges $e\in \mathfrak{I}$ such that $\mathsf{A}_e^{w_+,w_-}$ occurs. As a result, we obtain that for any $w_+\in \mathfrak{B}_1$ and $w_-\in \mathfrak{B}_{-1}$, 
 \begin{equation}\label{add6.17}
 	\mathbb{E}\big[|\mathfrak{I}_{\mathsf{A}}^{w_+,w_-}| \big] = \sum\nolimits_{e\in \mathfrak{I}}  \mathbb{P}\big( \mathsf{A}_e^{w_+,w_-} \big) \overset{|\mathfrak{I}|\asymp N^{d},(\ref{revise_new_39})}{\gtrsim} N^{-\frac{d}{2}+1}.  
 \end{equation}

To bound the second moment of $|\mathfrak{I}_{\mathsf{A}}^{w_+,w_-}|$, we need the following proposition. For any events $\mathsf{A}_1$ and $\mathsf{A}_2$ measurable with respect to $\widetilde{\mathcal{L}}_{1/2}$, let $\mathsf{A}_1\Vert \mathsf{A}_2$ denote the event that $\mathsf{A}_1$ and $\mathsf{A}_2$ both occur and are certified by two disjoint collections of clusters in $\cup \widetilde{\mathcal{L}}_{1/2}$. For any $z_1,z_2,z_3\in \widetilde{\mathbb{Z}}^d$ and $D\subset \widetilde{\mathbb{Z}}^d$, we denote $\big\{z_1\xleftrightarrow{(D)} z_2,z_3\big\}:=\big\{z_1\xleftrightarrow{(D)} z_2\big\}\cap \big\{z_1\xleftrightarrow{(D)} z_3\big\}$. We then define the event (letting $\Cl\label{const_bar_lowd_four_point1}:=\frac{1}{2}(\Cref{const_lowd_four_point1})^2$)
\begin{equation}
\begin{split}
		\mathsf{F}(e,e';w_+,w_-):= & \big\{ w_+ \xleftrightarrow{} v_e^+,v_{e'}^+     \Vert  w_- \xleftrightarrow{} v_e^-,v_{e'}^- \big\} \\
		&\cap \big\{ \widehat{\mathcal{L}}_{1/2}^{v_e^{+}},\widehat{\mathcal{L}}_{1/2}^{v_{e'
  		 }^{+}} ,  \widehat{\mathcal{L}}_{1/2}^{v_e^{-}},\widehat{\mathcal{L}}_{1/2}^{v_{e'
  		 }^{-}} \in [0,\Cref{const_bar_lowd_four_point1}]   \big\}. 
\end{split}
\end{equation}

   \begin{proposition} \label{lemma_four_arms}
 	When $3\le d\le 5$, for any $w_+\in \mathfrak{B}_1$, $w_-\in \mathfrak{B}_{-1}$ and $e,e'\in \mathfrak{I}$, 
 	\begin{equation}\label{addto6.19}
 		\mathbb{P}\big(\mathsf{F}(e,e';w_+,w_-)   \big) \lesssim (|v_e^+-v_{e'}^+|+1)^{-\frac{d}{2}-1}N^{-\frac{3d}{2}+1}. 
 	\end{equation}
 \end{proposition}

   The proof of Proposition \ref{lemma_four_arms} will be provided in Section \ref{section_proof_two_branch}. Under the coupling given by the isomorphism theorem \cite[Proposition 2.1]{lupu2016loop}, Property (a) (recall it below (\ref{def_mu})) implies that on the event $\mathsf{A}_e^{w_+,w_-}\cap \mathsf{A}_{e'}^{w_+,w_-} $, for each $\diamond\in \{+,-\}$, the loop cluster containing $w_\diamond$ intersects both $v_e^\diamond$ and $ v_{e'}^{\diamond}$, and satisfies $\widehat{\mathcal{L}}_{1/2}^{v_e^{+}},\widehat{\mathcal{L}}_{1/2}^{v_{e'}^{+}}\in [0,\Cref{const_bar_lowd_four_point1}]$. I.e., $\mathsf{F}(e,e';w_+,w_-)$ occurs. Thus, applying Proposition \ref{lemma_four_arms}, we have
   \begin{equation}
	\begin{split}
		 	\mathbb{P}\big(\mathsf{A}_e^{w_+,w_-}\cap \mathsf{A}_{e'}^{w_+,w_-} \big)  
		 \le  \mathbb{P}\big( \mathsf{F}(e,e';w_+,w_-)\big) \lesssim (|v_e^+-v_{e'}^+|+1)^{-\frac{d}{2}-1}N^{-\frac{3d}{2}+1}. 
	\end{split}
\end{equation}
This further yields that  
\begin{equation}\label{newaddto3.6}
 \begin{split}
 	\mathbb{E}\big[|\mathfrak{I}_{\mathsf{A}}^{w_+,w_-}|^2\big]=&  \sum\nolimits_{e,e'\in \mathfrak{I}} \mathbb{P}\big(\mathsf{A}_e^{w_+,w_-}\cap \mathsf{A}_{e'}^{w_+,w_-} \big)\\
 	\lesssim  &N^{-\frac{3d}{2}+1} \sum\nolimits_{w,w'\in \mathfrak{B}_0}  (|w-w'|+1)^{-\frac{d}{2}-1} \\
 	\overset{}{\lesssim }& N^{-\frac{3d}{2}+1} \cdot |\mathfrak{B}_0|\cdot N^{\frac{d}{2}-1}\asymp 1,   
 \end{split}
\end{equation}
where in the second inequality we used the following basic fact (see e.g., \cite[(4.5)]{cai2024high}): for any $d\ge 3$, $a\neq d$, $x\in \mathbb{Z}^d$ and $M\ge 1$,  
\begin{equation}\label{computation_d-a}
 \sum\nolimits_{y\in B(M)}(|x-y|+1)^{-a}\lesssim M^{(d-a)\vee 0}. 
\end{equation} 
By the Paley-Zygmund inequality, (\ref{add6.17}) and (\ref{newaddto3.6}), we have
\begin{equation}\label{add6.20}
	\mathbb{P}\big( \mathfrak{I}_{\mathsf{A}}^{w_+,w_-}\neq \emptyset \big)\gtrsim \frac{\big(	\mathbb{E}\big[|\mathfrak{I}_{\mathsf{A}}^{w_+,w_-} |\big] \big)^2 }{\mathbb{E}\big[ | \mathfrak{I}_{\mathsf{A}}^{w_+,w_-} |^2 \big]} \gtrsim  N^{2-d}	
	\end{equation}
 for all $w_+\in \mathfrak{B}_1$ and $w_-\in \mathfrak{B}_{-1}$. In fact, this estimate is sharp: since $\{ \mathfrak{I}_{\mathsf{A}}^{w_+,w_-}  \neq \emptyset\}$ implies both $\{w_+\xleftrightarrow{\ge 0} \partial B_{w_+}(\cref{const_dagger_box} N)\}$ and $\{w_{-}\xleftrightarrow{\le 0} \partial B_{w_{-}}(\cref{const_dagger_box} N)\}$, it follows from the FKG inequality and (\ref{one_arm_low}) that 
 \begin{equation}
 	\mathbb{P}\big(	 \mathfrak{I}_{\mathsf{A}}^{w_+,w_-} \neq \emptyset    \big) \le  \big[ \theta_d(\cref{const_dagger_box} N )\big]^2 \lesssim N^{2-d}. 
 \end{equation}
 Combining $\mathbb{P}\big(\mathfrak{I}_{\mathsf{A}}^{w_+,w_-}\neq \emptyset    \big)\asymp N^{2-d}$ with (\ref{add6.17}) and (\ref{newaddto3.6}) respectively, one has 
 \begin{equation}
 	\mathbb{E}\big[ |\mathfrak{I}_{\mathsf{A}}^{w_+,w_-}| \mid \mathfrak{I}_{\mathsf{A}}^{w_+,w_-} \neq \emptyset \big]\gtrsim N^{\frac{d}{2}-1}, 
 	 \end{equation}
\begin{equation}
	\mathbb{E}\big[ | \mathfrak{I}_{\mathsf{A}}^{w_+,w_-} |^2 \mid \mathfrak{I}_{\mathsf{A}}^{w_+,w_-} \neq \emptyset \big]\lesssim N^{d-2}.
\end{equation}
These two estimates together with the Paley-Zygmund inequality imply that for some constant $\cref{const_star}>0$, 
 \begin{equation}
 	\mathbb{P}\big( |\mathfrak{I}_{\mathsf{A}}^{w_+,w_-}| \ge  \cref{const_star} N^{\frac{d}{2}-1} \mid  \mathfrak{I}_{\mathsf{A}}^{w_+,w_-}\neq \emptyset \big) \gtrsim  \frac{\big(	\mathbb{E}\big[ |\mathfrak{I}_{\mathsf{A}}^{w_+,w_-}|\mid  \mathfrak{I}_{\mathsf{A}}^{w_+,w_-}\neq \emptyset\big] \big)^2 }{\mathbb{E}\big[ |\mathfrak{I}_{\mathsf{A}}^{w_+,w_-}|^2\mid  \mathfrak{I}_{\mathsf{A}}^{w_+,w_-}\neq \emptyset\big]} \gtrsim 1. 
 \end{equation}
Combined with (\ref{add6.20}), it yields that for any $w_+\in \mathfrak{B}_1$ and $w_-\in \mathfrak{B}_{-1}$,
\begin{equation}\label{add6.25}
	\mathbb{P}\big( |\mathfrak{I}_{\mathsf{A}}^{w_+,w_-}| \ge  \cref{const_star} N^{\frac{d}{2}-1}  \big)\gtrsim   N^{2-d}. 
\end{equation}

Next, we consider the quantity 
 \begin{equation}
 	\mathbf{X}:= \sum\nolimits_{w_+\in \mathfrak{B}_1, w_{-}\in \mathfrak{B}_{-1}} \mathbbm{1}_{ |\mathfrak{I}_{\mathsf{A}}^{w_+,w_-}|\ge \cref{const_star}  N^{\frac{d}{2}-1} }. 
 \end{equation}
 For the first moment of $\mathbf{X}$, it follows from (\ref{add6.25}) that
 \begin{equation}\label{618}
 	\mathbb{E}[\mathbf{X} ]\gtrsim |\mathfrak{B}_1|\cdot |\mathfrak{B}_{-1}| \cdot  N^{2-d} \asymp N^{d+2}. 
 \end{equation}
For the second moment of $\mathbf{X}$, since $|\mathfrak{I}_{\mathsf{A}}^{w_+,w_-}| \ge  \cref{const_star}N^{\frac{d}{2}-1}$ implies $w_+ \xleftrightarrow{\ge 0} \partial B_{w_{+}}(\cref{const_dagger_box} N)$ and $w_{-}\xleftrightarrow{\le 0} \partial B_{w_{-}}(\cref{const_dagger_box} N)$, we have 
 \begin{equation}\label{619}
 	\begin{split}
 		\mathbb{E}\big[\mathbf{X}^2\big] \le & \sum_{w_+,w_+'\in \mathfrak{B}_1,w_{-},w_{-}'\in \mathfrak{B}_{-1} }\mathbb{P}\big(w_+\xleftrightarrow{\ge 0} \partial B_{w_+}(\cref{const_dagger_box} N),w_+'\xleftrightarrow{\ge 0} \partial B_{w_+'}(\cref{const_dagger_box} N),\\
   &\ \ \ \ \ \ \ \ \ \ \ \ \ \ \ \ \ \ \ \ \ \ \ \ \ \ \ \ \	w_{-}\xleftrightarrow{\le 0} \partial B_{w_{-}}(\cref{const_dagger_box} N),w_{-}'\xleftrightarrow{\le 0} \partial B_{w_{-}'}(\cref{const_dagger_box} N) \big) \\
 \overset{}{\le } & \Big[  \sum\nolimits_{w,w'\in \mathfrak{B}_0} \mathbb{P}\big( w\xleftrightarrow{\ge 0} \partial B_{w}(\cref{const_dagger_box} N),w'\xleftrightarrow{\ge 0} \partial B_{w'}(\cref{const_dagger_box} N) \big)\Big]^2 ,
 	\end{split}
 \end{equation}
 where in the last line we used the FKG inequality, the symmetry of the GFF, and the translation invariance of $\mathbb{Z}^d$. In addition, when $ w\xleftrightarrow{\ge 0} \partial B_{w}(\cref{const_dagger_box} N)$ and $w'\xleftrightarrow{\ge 0} \partial B_{w'}(\cref{const_dagger_box} N)$ both occur, their occurrences are certified either by a single cluster or by two distinct clusters. We denote the corresponding events by $\mathsf{C}_{w,w'}^{(1),+}$ and $\mathsf{C}_{w,w'}^{(2),+}$ respectively. According to \cite[Theorem 1.4]{cai2024incipient}, we have 
 \begin{equation}\label{newadd3.17}
 	\mathbb{P}\big(\mathsf{C}_{w,w'}^{(1),+}  \big) \lesssim  (|w-w'|+1)^{-\frac{d}{2}+1}N^{-\frac{d}{2}+1}.  
 \end{equation}
For $\mathsf{C}_{w,w'}^{(2),+}$, by the symmetry of GFF and the FKG inequality, we have 
\begin{equation}\label{newadd3.18}
	\mathbb{P}\big(\mathsf{C}_{w,w'}^{(2),+}  \big) = \mathbb{P}\Big( \mathsf{H}_{w',\partial B_{w'}(\cref{const_dagger_box} N)}^{w,\partial B_{w}(\cref{const_dagger_box} N)} \Big) \overset{(\text{FKG})}{\le} \big[\theta_d(\cref{const_dagger_box} N)\big]^2 \overset{(\ref{one_arm_low})}{\lesssim} N^{2-d}. 
\end{equation}
 Combining (\ref{newadd3.17}) and (\ref{newadd3.18}), one has 
\begin{equation}\label{newadd3.25}
\begin{split}
		&\mathbb{P}\big( w\xleftrightarrow{\ge 0} \partial B_{w}(\cref{const_dagger_box} N),w'\xleftrightarrow{\ge 0} \partial B_{w'}(\cref{const_dagger_box} N) \big) \\
		\lesssim &  (|w-w'|+1)^{-\frac{d}{2}+1}N^{-\frac{d}{2}+1}+ N^{2-d}. 
\end{split}
\end{equation}
Plugging (\ref{newadd3.25}) into (\ref{619}), we obtain 
\begin{equation}\label{add3.19}
	\begin{split}
			\mathbb{E}\big[\mathbf{X}^2\big] \lesssim &  \Big(  \sum\nolimits_{w,w'\in \mathfrak{B}_0} \big[  (|w-w'|+1)^{-\frac{d}{2}+1}\cdot N^{-\frac{d}{2}+1}+ N^{2-d} \big]  \Big)^2 \\
  			\overset{(\ref{computation_d-a})}{	\lesssim } &    \big(   |\mathfrak{B}_0 |\cdot   N^{\frac{d}{2}+1}\cdot N^{-\frac{d}{2}+1} + |\mathfrak{B}_0|^2\cdot N^{2-d}\big)^2  \lesssim N^{2d+4}. 
	\end{split}
\end{equation}
By the Paley-Zygmund inequality, (\ref{618}) and (\ref{add3.19}), we get 
 \begin{equation}\label{revisenew_to_329}
  	\mathbb{P}\big(\mathbf{X} >0 \big)\gtrsim \big(	\mathbb{E}[\mathbf{X} ] \big)^2/  \mathbb{E}\big[\mathbf{X}^2\big] \gtrsim 1.  
 \end{equation}
    On the event $\{\mathbf{X} >0\}$, there exist $w_+\in \mathfrak{B}_1$ and $ w_{-}\in \mathfrak{B}_{-1}$ such that the event $\mathsf{A}_e^{w_+,w_-}$ occurs for at least $\cref{const_star}  N^{\frac{d}{2}-1}$ edges in $\mathfrak{I}$. Combined with the inclusion that $\mathsf{A}_e^{w_+,w_-}\subset \mathsf{A}_{e}$ (recall $\mathsf{A}_{e}$ in (\ref{newadd321})) for all $e\in \mathfrak{I}$, $w_+\in \mathfrak{B}_1$ and $ w_{-}\in \mathfrak{B}_{-1}$, it yields $\{\mathbf{X} >0\}\subset \{ |\mathfrak{I}_{\mathsf{A}}|\ge \cref{const_star}  N^{\frac{d}{2}-1}\}$. This together with (\ref{revisenew_to_329}) implies the claim (\ref{claim3.4}): 
        \begin{equation}\label{fixadd322}
    	 \mathbb{P}\big( |\mathfrak{I}_{\mathsf{A}}|\ge \cref{const_star}  N^{\frac{d}{2}-1} \big)\ge \mathbb{P}\big(\mathbf{X} >0 \big) \gtrsim 1. 
    \end{equation}

 \textbf{When $d\ge 7$.} The proof in this case is relatively straightforward. According to \cite[Lemma 6.1]{inpreparation_twoarm}, we have 
 \begin{equation}
 	\mathbb{P}\big( \mathsf{A}_{e} \big)\gtrsim N^{-4}, \ \ \forall e\in \mathfrak{I}, 
 \end{equation}
which implies that the first moment of $\mathfrak{I}_{\mathsf{A}}$ satisfies  
 \begin{equation}\label{fixnew330}
 	\mathbb{E}[|\mathfrak{I}_{\mathsf{A}}| ]\gtrsim |\mathfrak{I}|\cdot N^{-4} \asymp  N^{d-4}. 
 \end{equation}
For the second moment, as in the proof of (\ref{619}), the FKG inequality yields
 \begin{equation}
 	\begin{split}
 		\mathbb{E}\big[|\mathfrak{I}_{\mathsf{A}}|^2 \big]=&  \sum\nolimits_{e,e'\in \mathfrak{I} } \mathbb{P}\big( \mathsf{A}_{e}\cap \mathsf{A}_{e'} \big)\\
 		\overset{}{\le} & \sum\nolimits_{e,e'\in \mathfrak{I} } \mathbb{P}\big(v_e^{+}\xleftrightarrow{\ge 0} \partial B(\cref{const_dagger_box} N), v_{e'}^{+}\xleftrightarrow{\ge 0} \partial B(\cref{const_dagger_box} N) \big)\\
 		&\ \ \ \ \ \ \ \ \ \ \  \cdot  \mathbb{P}\big(v_e^{-}\xleftrightarrow{\le 0} \partial B(\cref{const_dagger_box} N), v_{e'}^{-}\xleftrightarrow{\le 0} \partial B(\cref{const_dagger_box} N) \big).
 	\end{split}
 \end{equation}  
 For the same reason as in (\ref{newadd3.25}), it follows from \cite[Theorem 1.4]{cai2024incipient} and (\ref{one_arm_high}) that 
  \begin{equation}
	\mathbb{P}\big(v_e^{+}\xleftrightarrow{\ge 0} \partial B(\cref{const_dagger_box} N), v_{e'}^{+}\xleftrightarrow{\ge 0} \partial B(\cref{const_dagger_box} N) \big)
	\lesssim (|v_e^{+}-v_{e'}^{+}|+1)^{4-d}N^{-2} +N^{-4}.   
\end{equation} 
The same bound also holds for $\{v_e^{-}\xleftrightarrow{\le 0} \partial B(\cref{const_dagger_box} N), v_{e'}^{-}\xleftrightarrow{\le 0} \partial B(\cref{const_dagger_box} N)\}$. Thus, 
 \begin{equation}\label{6.11}
 	\begin{split}
 		\mathbb{E}\big[|\mathfrak{I}_{\mathsf{A}}|^2 \big] \lesssim & \Big( \sum\nolimits_{e,e'\in \mathfrak{I} }  \big[(|v_e^+-v_{e'}^+|+1)^{4-d}N^{-2} +N^{-4} \big] \Big)^2 \\
 			\lesssim & N^{-4} \sum\nolimits_{x,x' \in \mathfrak{B}_0}  (|x-x'|+1)^{8-2d}
 			 + N^{-8}|\mathfrak{B}_0|^2\\
 			 &+ N^{-6}\cdot |\mathfrak{B}_0|\cdot \max_{x\in \mathfrak{B}_0}\sum\nolimits_{x'\in \mathfrak{B}_0}(|x-x'|+1)^{4-d} \\
 			 \overset{|x-x'|\lesssim N,|\mathfrak{B}_0|\asymp N^d}{\lesssim } & N^{d-2}\cdot \max_{x\in \mathfrak{B}_0}\sum\nolimits_{x'\in \mathfrak{B}_0}(|x-x'|+1)^{6-2d}+ N^{2d-8}\\
 			 &+ N^{d-6}\cdot \max_{x\in \mathfrak{B}_0}\sum\nolimits_{x'\in \mathfrak{B}_0}(|x-x'|+1)^{4-d}\\
 			 \overset{(\ref{computation_d-a})}{\lesssim } & N^{d-2}+ N^{2d-8} \overset{(d\ge 7)}{\lesssim}  N^{2d-8}. 
 	\end{split}
 \end{equation}
By the Paley-Zygmund inequality, (\ref{fixnew330}) and (\ref{6.11}), we obtain the claim (\ref{claim3.4}):  
\begin{equation}\label{612}
	\mathbb{P}\big( |\mathfrak{I}_{\mathsf{A}}|\ge \cref{const_star} N^{d-4} \big)\gtrsim \big(\mathbb{E}[|\mathfrak{I}_{\mathsf{A}}|] \big)^2 / \mathbb{E}\big[|\mathfrak{I}_{\mathsf{A}}|^2 \big]  \gtrsim  1. 
	\end{equation}


 In conclusion, we have established Theorem \ref{thm_minimal_distance} assuming Proposition \ref{lemma_four_arms}.  \qed

 \section{Proof of Proposition \ref{lemma_four_arms}}\label{section_proof_two_branch}

 The proof of Proposition \ref{lemma_four_arms} mainly relies on the following lemma. To state this lemma precisely, we first introduce some notations as follows: 
 \begin{itemize}

   \item  Let $\Cl\label{const_boundary1},\Cl\label{const_boundary2}>10^4$ be constants that will be determined later. Here $\Cref{const_boundary2}$ is much larger than $\Cref{const_boundary1}$.

  \item  For any $D\subset \widetilde{\mathbb{Z}}^d$, $r\ge 1$, $v,v'\in \widetilde{\mathbb{Z}}^d$ with $r\le |v-v'|\le 10r$, 
   \begin{equation}\label{notation41}
   	\mathbf{Q}^{D}_{\mathrm{I}}(v,v';r):=  \sum\nolimits_{z\in \partial \mathcal{B}_{v}(\frac{r}{\Cref{const_boundary1}}),z' \in \partial \mathcal{B}_{v'}(\frac{r}{\Cref{const_boundary1}})}  \mathbb{P}\big(\mathsf{A}_{\mathrm{I}}^D(v,z;r) \big)\cdot \mathbb{P}\big(\mathsf{A}_{\mathrm{I}}^D(v',z';r) \big),
   \end{equation}   
   where $\mathsf{A}_{\mathrm{I}}^D(v,z;r):=\big\{v\xleftrightarrow{(D)} z \big\}\cap  \big\{v\xleftrightarrow{(D)} \partial B_{v}(\tfrac{r}{100}) \big\}^c$.

   \item For any $D\subset \widetilde{\mathbb{Z}}^d$, $v\in \widetilde{\mathbb{Z}}^d$, and $r,r'\ge 1$, 
   \begin{equation}\label{notation42}
   	\mathbf{Q}_{\mathrm{II}}^D(v;r,r'):= \sum_{z\in \partial \mathcal{B}_v(r),z'\in \partial \mathcal{B}_v(\frac{r'}{\Cref{const_boundary1}})}  \mathbb{P}\big( \mathsf{A}_{\mathrm{II}}^D(v,z,z';r,r') \big),
   \end{equation}
   where $\mathsf{A}_{\mathrm{II}}^D(v,z,z';r,r'):=\big\{ z\xleftrightarrow{(D)} z'\big\}\cap  \big\{z \xleftrightarrow{(D)} \partial B_v(\frac{r}{\Cref{const_boundary1}} )\cup \partial B_v(10r') \big\}^c$.


            \item  For any $r\ge 1$ and $k\in \mathbb{N}$, we denote $r_k:=(\Cref{const_boundary2})^k\cdot r$. For any $R\ge r_5$, let 
      \begin{equation}
  	k_\star(r,R):= \min\{k\ge 1: r_{k+3}\ge R\}. 
  \end{equation}
   For any $D\subset \widetilde{\mathbb{Z}}^d$, $v\in \widetilde{\mathbb{Z}}^d$, $r'\in [r_1,2r_{k_{\star}}]$ and $w\in [\widetilde{B}_v(2R)]^c$,  
     \begin{equation}\label{notation43}
   	\mathbf{Q}_{\mathrm{III}}^D(v,w;r'):= \sum\nolimits_{z\in \partial \mathcal{B}_v(dr')} \mathbb{P}\big(\mathsf{A}_{\mathrm{III}}^D(v,w,z;r') \big),
   \end{equation} 
   where $\mathsf{A}_{\mathrm{III}}^D(v,w,z;r'):=\big\{w\xleftrightarrow{(D)} z\big\}\cap \big\{ w \xleftrightarrow{(D)} \partial B_v(\tfrac{r'}{\Cref{const_boundary1}}) \big\}^c$. In addition, we define $\mathsf{A}_{\mathrm{III}}^D(v,w,z;r'):=\{\mathbf{A}_\star^z=1\}$ for all $r'>2r_{k_\star}$, where $\{\mathbf{A}_\star^z\}_{z\in \mathbb{Z}^d}$ is a family of i.i.d. Bernoulli random variables independent of $\widetilde{\mathcal{L}}_{1/2}$ with expectation $R^{2-d}$.

 \end{itemize}

 \begin{lemma}\label{lemma41}
 	For any $3\le d\le 5$, there exist $\Cref{const_boundary1},\Cref{const_boundary2}>0$ such that the following holds. Suppose that $r\ge 1$, $R\ge r_5$, $\{\{x_i,y_i\}\}_{i\in \{1,2\}}\subset\mathbb{L}^d$ satisfying $r\le |x_1-x_2|\le 10r$, $v_i\in  I_{e_i}$ satisfying $|v_i-x_i|\land |v_i-y_i|\ge \frac{1}{10}$ for $i\in \{1,2\}$, $v_i'\in I_{[v_i,y_i]}$ satisfying $|v_i'-v_i|\land |v_i'-y_i|\ge \frac{1}{100}$ for $i\in \{1,2\}$, $D\subset \widetilde{\mathbb{Z}}^d$ satisfying $x_i\in D$ and $D\cap I_{[v_i,y_i]}=\emptyset$ for all $i\in \{1,2\}$, and $w\in  [\widetilde{B}_{v_1}(2R)]^c$. Then we have 
 	\begin{equation}\label{ineq_lemma41}
 	\begin{split}
 		&\mathbb{P}\big( \big\{ w\xleftrightarrow{(D)} v_1,v_2  \big\} \cap \big\{  \widehat{\mathcal{L}}_{1/2}^{D,v_1},\widehat{\mathcal{L}}_{1/2}^{D,v_2}  \in [0,\Cref{const_bar_lowd_four_point1}] \big\} \big) \\
 		 \lesssim &r^{-d-1} \mathbf{Q}^{D}_{\mathrm{I}}(v_1',v_2';r)\sum\nolimits_{2 \le k\le k_\star(r,R)}  r_k^{-\frac{3d}{2}+1}  	\mathbf{Q}_{\mathrm{II}}^D(v_1;r_1,r_k) 	\mathbf{Q}_{\mathrm{III}}^D(v_1,w;r_{k+1}). 
 	\end{split}
 	\end{equation}
 \end{lemma}

 The rest of this section is organized as follows: in Section \ref{subsection_proof_prop3.1}, we prove Proposition \ref{lemma_four_arms} assuming Lemma \ref{lemma41}; in Section \ref{subsection_proof_lemma4.1}, we prove Lemma \ref{lemma41}.


 \subsection{Proof of Proposition \ref{lemma_four_arms}}\label{subsection_proof_prop3.1}

 We arbitrarily take $w_+\in \mathfrak{B}_1$, $w_-\in \mathfrak{B}_{-1}$ and $e,e'\in \mathfrak{I}$. Since $\mathbb{P}(v_e^{+}\xleftrightarrow{\ge 0} w_+, v_e^{-}\xleftrightarrow{\le 0} w_-)\lesssim N^{-\frac{3d}{2}+1}$ for all $e\in \mathfrak{I}$ (by \cite[Theorem 1.4]{inpreparation_twoarm}), it suffices to consider the case when $r:= |v_e^{+}-v_{e'}^+|$ is sufficiently large. For each $\diamond \in \{+,-\}$, we define  
 \begin{equation}
 	\begin{split}
 		\mathsf{F}_\diamond^D:= \big\{  w_\diamond \xleftrightarrow{(D)} v_e^{\diamond}, v_{e'}^{\diamond}  \big\}\cap  \big\{  \widehat{\mathcal{L}}_{1/2}^{D,v_e^{\diamond}},\widehat{\mathcal{L}}_{1/2}^{D,v_{e'
  		 }^{\diamond}}   \in [0,\Cref{const_bar_lowd_four_point1}] \big\}. 
 	\end{split}
 \end{equation}
 When $D=\emptyset$, we abbreviate $\mathsf{F}_\diamond:=\mathsf{F}_\diamond^\emptyset$. Note that $\mathsf{F}(e,e';w_+,w_-)=\{\mathsf{F}_+\Vert \mathsf{F}_-\}$. Therefore, by the restriction property and Lemma \ref{lemma41}, one has  
 \begin{equation}\label{48}
 \begin{split}
 		 		& \mathbb{P}\big(\mathsf{F}(e,e';w_+,w_-)   \big)=   \mathbb{E}\big[ \mathbbm{1}_{\mathsf{F}_-} \cdot \mathbb{P}\big(\mathsf{F}_+^{\mathcal{C}_{w_-}}  \big) \big] \\
 	 		\overset{}{\lesssim} &r^{-d-1}  \mathbb{E}\Big[ \mathbbm{1}_{\mathsf{F}_-} \cdot \mathbb{E}\big[   \mathbf{Q}^{\mathcal{C}_{w_-}}_{\mathrm{I}}(\hat{v}_e^+,\hat{v}_{e'}^+;r)\sum\nolimits_{2\le k \le k_\star(r,N)}  r_k^{-\frac{d}{2}-1} 	\mathbf{Q}_{\mathrm{II}}^{\mathcal{C}_{w_-}}(v_e^+;r_1,r_k) \\
 	 		&\ \ \ \ \ \ \ \ \ \ \ \ \ \ \ \ \ \ \ \ \ \ \ \ \ \ \ \ \ \ \ \ \ \ \ \ \ \ \ \ \ \ \ \ \ \ \ \ \ \ \ \cdot	\mathbf{Q}_{\mathrm{III}}^{\mathcal{C}_{w_-}}(v_e^+,w_+;r_{k+1}) \mid \mathcal{F}_{\mathcal{C}_{w_-}}\big] \Big] \\ 
 	 		=&r^{-d-1} \sum\nolimits_{2\le k\le k_\star(r,N)} r_k^{-\frac{3d}{2}+1} \sum\nolimits_{\bar{z}\in \Upsilon_+(r,r_{k})} \mathbb{E}\big[ \mathbbm{1}_{\mathsf{F}_-} \cdot \mathfrak{U}_+^{\mathcal{C}_{w_-} }(\bar{z};r,r_k) \big],   
 \end{split}
 \end{equation}
 where $\hat{v}_{e}^{\diamond}$ denotes the midpoint between $v_{e}^{\diamond}$ and the endpoint of $e$ that is closer to $v_{e}^{\diamond}$ (for $\diamond\in \{+,-\}$), $\bar{z}:=(z_1,z_2,z_3,z_4,z_5)$, $\Upsilon_{\diamond}(R,R'):=\partial \mathcal{B}_{\hat{v}_e^\diamond}(\frac{R}{\Cref{const_boundary1}})\times \partial \mathcal{B}_{\hat{v}_{e'}^\diamond}(\frac{R}{\Cref{const_boundary1}})\times \partial \mathcal{B}_{v_e^\diamond}(\Cref{const_boundary2} R) \times \partial \mathcal{B}_{v_e^\diamond}(\frac{R'}{\Cref{const_boundary1}})\times \partial \mathcal{B}_{v_e^\diamond}(d\Cref{const_boundary2}R' )$, and $\mathfrak{U}_{\diamond}^D(\bar{z};R,R')$ represents the product of the probabilities of the events $\mathsf{A}_{\mathrm{I}}^D(\hat{v}_e^{\diamond},z_1;R)$, $\mathsf{A}_{\mathrm{I}}^D(\hat{v}_{e'}^{\diamond},z_2;R)$, $\mathsf{A}_{\mathrm{II}}^D(v_e^{\diamond},z_3,z_4;\Cref{const_boundary2}R,R')$ and $\mathsf{A}_{\mathrm{III}}^D(v_e^{\diamond},w_{\diamond},z_5;\Cref{const_boundary2}R')$. In addition, we denote the intersection of these four events by $\overline{\mathsf{A}}^D_{\diamond}(\bar{z};R,R')$, and we abbreviate $\overline{\mathsf{A}}_{\diamond}:=\overline{\mathsf{A}}^{\emptyset}_{\diamond}$. (\textbf{P.S.} Here, $\hat{v}_{e}^{\diamond}$ is not necessarily the midpoint; it can be any point satisfying the conditions for $v_i'$ in Lemma \ref{lemma41}, where $v_i=v_{e}^{\diamond}$ and $y_i$ is the endpoint of $e$ closer to $v_{e}^{\diamond}$.)

 By using Corollary \ref{coro27} repeatedly, we have 
 \begin{equation}\label{new48}
 	\begin{split}
 		\mathfrak{U}_+^{\mathcal{C}_{w_-} }(\bar{z};r,r_k) \lesssim  \mathbb{P}\big(\overline{\mathsf{A}}^{\mathcal{C}_{w_-}}_+(\bar{z};r,r_{k} ) \mid \mathcal{F}_{\mathcal{C}_{w_-}}  \big). 
 	\end{split}
 \end{equation}
Let $\overline{\mathcal{C}}_{\bar{z}}:=\cup_{j\in \{1,2,3,5\}}\mathcal{C}_{z_j}$ (note that $z_4\in \mathcal{C}_{z_3}$). See Figure \ref{fig1} for an illustration of the event $\overline{\mathsf{A}}^{\mathcal{C}_{w_-}}_+(\bar{z};r,r_{k} )$. Plugging (\ref{new48}) into (\ref{48}), and then applying Lemma \ref{lemma41} and Corollary \ref{coro27} as in (\ref{48}) and (\ref{new48}), we get 
\begin{equation}\label{newadd49}
	\begin{split}
		& \mathbb{P}\big(\mathsf{F}(e,e';w_+,w_-)   \big) \\
	\overset{}{ \lesssim} & r^{-d-1} \sum\nolimits_{2\le k\le k_\star(r,N)} r_k^{-\frac{3d}{2}+1} \sum\nolimits_{\bar{z}\in \Upsilon_+(r,r_{k})} \mathbb{E}\big[ \mathbbm{1}_{\overline{\mathsf{A}}_+(\bar{z};r,r_{k} )} \cdot \mathbb{P}\big( \mathsf{F}^{\overline{\mathcal{C}}_{\bar{z}}}_- \mid \mathcal{F}_{\overline{\mathcal{C}}_{\bar{z}}}\big) \big] \\
	\overset{}{\lesssim} &r^{-2d-2}  \sum\nolimits_{2\le k,k'\le k_\star(r,N)} r_k^{-\frac{3d}{2}+1}r_{k'}^{-\frac{3d}{2}+1}\sum\nolimits_{\bar{z}\in \Upsilon_+(r,r_{k'})}\\
 		  &\ \ \ \ \ \ \ \  \sum\nolimits_{\bar{z}':=(z_1',z_2',z_3',z_4',z_5')\in \Upsilon_-(2r,2r_{k'})}\mathbb{P}\big( \overline{\mathsf{A}}_+(\bar{z};r,r_{k} ) \Vert  \overline{\mathsf{A}}_-(\bar{z}';2r,2r_{k'} )  \big). 	\end{split}
\end{equation}
  See Figure \ref{fig2} for an illustration for the event $\overline{\mathsf{A}}_+(\bar{z};r,r_{k} ) \Vert  \overline{\mathsf{A}}_-(\bar{z}';2r,2r_{k'} )$.

\begin{figure}[h]
	\centering
	\includegraphics[width=0.8\textwidth]{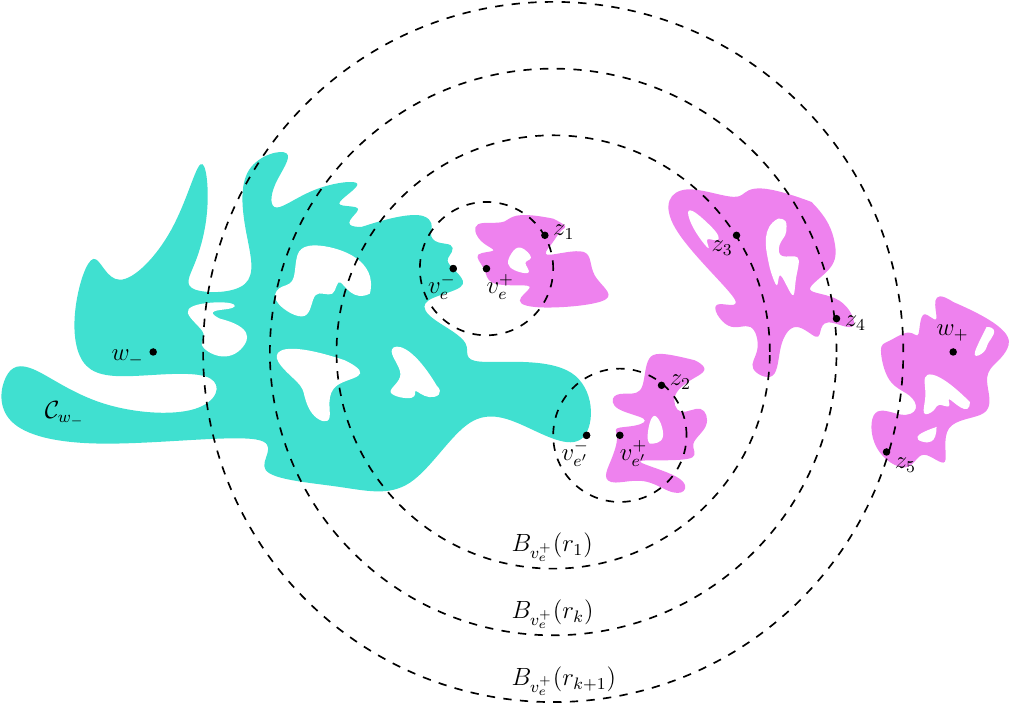}
	\caption{In this illustration, the cyan region is the cluster $\mathcal{C}_{w_-}$, and the pink regions represent the clusters $\{\mathcal{C}_{z_j}\}_{j\in \{1,2,3,5\}}$. Note that on $\overline{\mathsf{A}}^{\mathcal{C}_{w_-}}_+(\bar{z};r,r_{k} )$, the latter four clusters are restricted in different areas, and thus are constructed by disjoint collection of loops. This implies that given $\mathcal{C}_{w_1}$, the four sub-events involved in $\overline{\mathsf{A}}^{\mathcal{C}_{w_-}}_+(\bar{z};r,r_{k} )$ are conditionally independent. In the next step (\ref{newadd49}), we fix $\{\mathcal{C}_{z_j}\}_{j\in \{1,2,3,5\}}$ and then decompose $\mathcal{C}_{w_-}$ via Lemma \ref{lemma41}.  		 }\label{fig1}
\end{figure}

 \begin{figure}[h]
	\centering
	\includegraphics[width=0.75\textwidth]{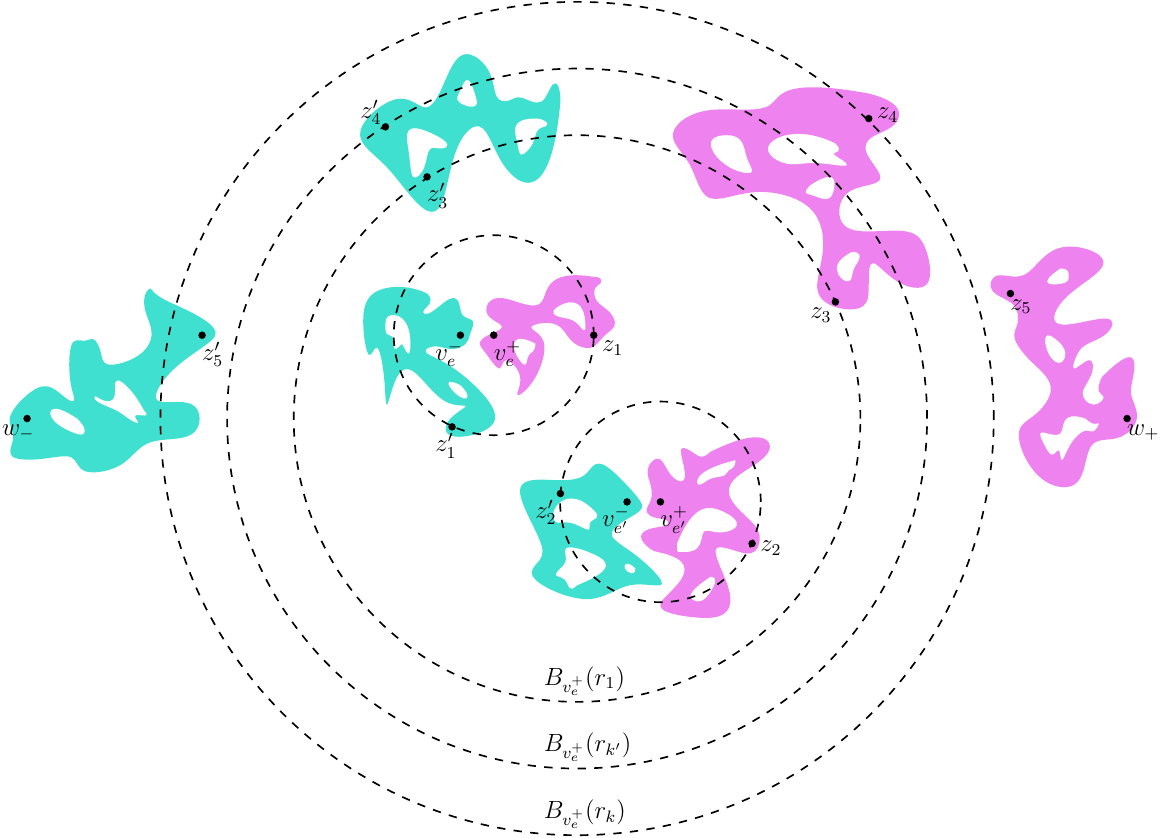}
	\caption{In this illustration, each colored region represents a loop cluster. The clusters $\mathcal{C}_{v_e^+}$ and $\mathcal{C}_{v_e^-}$ certify the four-point event $\{v_e^+ \xleftrightarrow{} z_1  \Vert v_e^-\xleftrightarrow{} z_1'\}$, whose probability has been estimated in the companion paper \cite{inpreparation_twoarm} (see (\ref{four_point_use})). The same applies to $\mathcal{C}_{v_{e'}^+}$ and $\mathcal{C}_{v_{e'}^-}$. To ensure that the remaining clusters certify four-point events restricted to disjoint annuli, we use Corollary \ref{coro28} to decompose $\mathcal{C}_{w_-}$ and $\mathcal{C}_{z_3}$ around $\partial B(r_k)$ and $\partial B(r_{k'})$ respectively (see (\ref{413}) and (\ref{414})). When $k=k'$, such a decomposition is not needed since the clusters $\mathcal{C}_{z_3}$, $\mathcal{C}_{z_3'}$, $\mathcal{C}_{w_+}$ and $\mathcal{C}_{w_-}$ already meet the requirements. }\label{fig2}
\end{figure}

 For any $\diamond \in \{+,-\}$ and $R,R'\ge 1$, we denote $\widehat{\Upsilon}_\diamond(R,R'):=\partial \mathcal{B}_{v_e^\diamond}(\Cref{const_boundary2} R) \times \partial \mathcal{B}_{v_e^\diamond}(\frac{R'}{\Cref{const_boundary1}})\times \partial \mathcal{B}_{v_e^\diamond}(d\Cref{const_boundary2}R')$. We claim that for any $2\le k,k'\le k_\star(r,N)$, $z_1\in \partial \mathcal{B}_{\hat{v}_e^+}(\frac{r}{\Cref{const_boundary1}})$, $z_2\in \partial \mathcal{B}_{\hat{v}_{e'}^+}(\frac{r}{\Cref{const_boundary1}}) $, $z_1'\in \partial \mathcal{B}_{\hat{v}_e^-}(\frac{2r}{\Cref{const_boundary1}})$ and $z_2'\in \partial \mathcal{B}_{\hat{v}_{e'}^-}(\frac{2r}{\Cref{const_boundary1}})$, 
 \begin{equation}\label{newadd410}
 	\begin{split}
 		&\sum\nolimits_{\hat{z}\in \widehat{\Upsilon}_+(r,r
 		_k),\hat{z}'\in \widehat{\Upsilon}_-(2r,2r
 		_{k'})}  \mathbb{P}\big(\overline{\mathsf{A}}_+(z_1,z_2,\hat{z};r,r_{k} ) \Vert  \overline{\mathsf{A}}_-(z_1',z_2',\hat{z}';2r,2r_{k'} ) \big)\\
 		\lesssim & r^{-\frac{3d}{2}+3}r_{k}^{d}r_{k'}^{d}N^{-\frac{3d}{2}+1}.
 	\end{split}
 \end{equation}
 Note that substituting (\ref{newadd410}) into (\ref{newadd49}) immediately yields Proposition \ref{lemma_four_arms}: 
 \begin{equation*}
 	\begin{split}
 	& 	\mathbb{P}\big(\mathsf{F}(e,e';w_+,w_-)   \big)\\
 		\overset{}{\lesssim} &  r^{-2d-2}\sum\nolimits_{2 \le k,k'\le k_\star(r,N)} r_k^{-\frac{3d}{2}+1}r_{k'}^{-\frac{3d}{2}+1} \cdot (r^{d-1})^4\cdot r^{-\frac{3d}{2}+3}r_{k}^{d}r_{k'}^{d}N^{-\frac{3d}{2}+1}\\
 		=& r^{\frac{d}{2}-3}N^{-\frac{3d}{2}+1} \sum\nolimits_{2 \le k,k'\le k_\star(r,N)} r_k^{-\frac{d}{2}+1}r_{k'}^{-\frac{d}{2}+1}\asymp r^{-\frac{d}{2}-1}N^{-\frac{3d}{2}+1}. 
 	\end{split}
 \end{equation*}
 Next, we verify the claim (\ref{newadd410}). Without loss of generality, we only consider the case $k\ge k'$. The proof is divided into the two subcases $k=k'$ and $k>k'$.

 \textbf{When $k=k'$.} In this case, it follows from (\ref{2.14}) that whenever the event on the left-hand side of (\ref{newadd410}) occurs (for $\hat{z}=(z_3,z_4,z_5)$ and $\hat{z}'=(z_3',z_4',z_5')$), the following events $\{	\mathsf{F}_{j}\}_{1\le j\le 3}$ occur as well. 
 \begin{itemize}

 	\item  $	\mathsf{F}_{1}:= \big\{ \hat{v}_e^+\xleftrightarrow{(\partial B_{v_e^+}(\frac{r}{100}))} z_1 \Vert \hat{v}_e^-\xleftrightarrow{(\partial B_{v_e^+}(\frac{r}{100}))} z_1' \big\}$;

     \item  $ \mathsf{F}_{2}:= \big\{ \hat{v}_{e'}^+\xleftrightarrow{(\partial B_{v_e^+}(\frac{r}{100}))} z_2 \Vert \hat{v}_e^-\xleftrightarrow{(\partial B_{v_{e'}^+}(\frac{r}{100}))} z_2' \big\}$;

 		\item  $ \mathsf{F}_{3}:=\big\{ z_3\xleftrightarrow{(\partial B_{v_e^+}(\frac{r_1}{\Cref{const_boundary1}})\cup \partial B_{v_e^+}(10 r_k))} z_4 \Vert z_3'\xleftrightarrow{(\partial B_{v_e^+}(\frac{r_1}{\Cref{const_boundary1}})\cup \partial B_{v_e^+}(10 r_k))} z_4' \big\}$;

 			\item  When $k<k_\star$, we define $	\mathsf{F}_{4}:=\big\{ z_5 \xleftrightarrow{(\partial B_{v_e^+}(\frac{r_{k+1}}{\Cref{const_boundary1}}))} w_+ \Vert z_5'\xleftrightarrow{(\partial B_{v_e^+}(\frac{r_{k+1}}{\Cref{const_boundary1}}))} w_- \big\}$, whereas for $k=k_\star$, we define $\mathsf{F}_{4}:= \big\{ \mathbf{A}_\star^{z_5}=\mathbf{A}_\star^{z_5'}= 1  \big\}$.

 \end{itemize}
By construction, $\mathsf{F}_{1}$, $\mathsf{F}_{2}$, $\mathsf{F}_{3}$ and $\mathsf{F}_{4}$ (for $k<k_\star$) depend on disjoint collections of loops in $\widetilde{\mathcal{L}}_{1/2}$. Together with the independence of $\{\mathbf{A}_\star^z\}_{z\in \mathbb{Z}^d}$ from $\widetilde{\mathcal{L}}_{1/2}$, this implies that $\{\mathsf{F}_{j}\}_{1\le j\le 4}$ are independent. Meanwhile, by the isomorphism theorem and the estimate on the heterochromatic four-point function in \cite[Lemma 6.4]{inpreparation_twoarm}, one has: for any $d\ge 3$ with $d\neq 6$, $m\ge C$, $M\ge Cm$, $v,v'\in  \widetilde{B}(10m)\setminus \widetilde{B}(m)$ with $|v-v'|\ge 1$, $w,w'\in \widetilde{B}(10M)\setminus \widetilde{B}(M)$ with $|w-w'|\ge M$ and $D\subset \widetilde{B}(\frac{m}{C'})\cup [\widetilde{B}(C'M)]^c$, 
 	\begin{equation}\label{four_point_use}
 		\mathbb{P}\big(v\xleftrightarrow{(D)}v' \Vert  w\xleftrightarrow{(D)}w'  \big) \asymp |v-v'|^{(3-\frac{d}{2})\boxdot 0} M^{ -[(\frac{3d}{2}-1)\boxdot (2d-4)]}. 
 	\end{equation}
 	This estimate together with the fact that $\mathbb{P}(\mathbf{A}_\star^{z})=N^{2-d}$ for all $z\in \mathbb{Z}^d$ yields 
  \begin{equation}
 	\mathbb{P}\big( \mathsf{F}_{1} \big) \vee \mathbb{P}\big( \mathsf{F}_{2} \big)  \lesssim r^{-\frac{3d}{2}+1}, \  	\mathbb{P}\big( \mathsf{F}_{3} \big) \lesssim r^{3-\frac{d}{2}}r_k^{-\frac{3d}{2}+1},
 \end{equation}
 \begin{equation}
 	\mathbb{P}\big( \mathsf{F}_{4} \big) \lesssim r_k^{3-\frac{d}{2}}N^{-\frac{3d}{2}+1}  \cdot \mathbbm{1}_{k<k_\star } + N^{4-2d} \cdot  \mathbbm{1}_{k=k_\star } \overset{(r_{k_\star}\asymp N)}{\asymp} r_k^{3-\frac{d}{2}}N^{-\frac{3d}{2}+1}. 
 \end{equation}
 Consequently, the left-hand side of (\ref{newadd410}) is bounded from above by 
 \begin{equation}
 	\begin{split}
 		   |\widehat{\Upsilon}_+(r,r_k)|\cdot | \widehat{\Upsilon}_-(2r,2r_{k'})|\cdot \prod\nolimits_{1\le j\le 4} \mathbb{P}\big( \mathsf{F}_{j} \big)  
 		 \lesssim  r^{-\frac{3d}{2}+3}r_{k}^{2d}N^{-\frac{3d}{2}+1}.
 	\end{split}
 \end{equation}

  \textbf{When $k>k'$.} For simplicity, we denote the event 
 \begin{equation}
 	\begin{split}
\widehat{\mathsf{A}}_{-}(\bar{z}';R,R' ) :=		\mathsf{A}_{\mathrm{I}}(\hat{v}_e^{-},z_1';R)\cap \mathsf{A}_{\mathrm{I}}(\hat{v}_{e'}^{-},z_2';R) \cap \mathsf{A}_{\mathrm{II}}(v_e^{-},z_3',z_4';\Cref{const_boundary2}R,R'), 
 	\end{split}
 \end{equation} 
 and the cluster $\widehat{\mathcal{C}}_-=\widehat{\mathcal{C}}_-(\bar{z},\bar{z}'):=(\cup_{i\in \{1,2,3,5\}}\mathcal{C}_{z_i})\cup (\cup_{1\le i\le 3}\mathcal{C}_{z_{i}'})$. We also denote $r_k^{\mathrm{out}}:=\Cref{const_boundary1}^{3/2}\Cref{const_boundary2}^{1/2}r_{k}$ and $r_k^{\mathrm{in}}:=\Cref{const_boundary1}^{-5/2}\Cref{const_boundary2}^{1/2}r_{k}$. For any $v\in \widetilde{\mathbb{Z}}^d$, $w\in [\widetilde{B}_{v_e^{-}}(\cref{const_dagger_box}N)]^c$, $w'\in \partial \mathcal{B}_{v_e^{-}}(\frac{2r_{k'+1}}{\Cref{const_boundary1}})$, $z \in \partial \mathcal{B}_{v_e^-}(r_k^{\mathrm{out}})$ and $z' \in \partial \mathcal{B}_{v_e^-}(r_k^{\mathrm{in}})$, we define the event 
  \begin{equation*}
\begin{split}
	\widehat{\mathsf{A}}_{\mathrm{III}}^{D }=\widehat{\mathsf{A}}_{\mathrm{III}}^{D }(v,w,w',z,z';k,k'):= &\mathsf{A}_{\mathrm{III}}^D(v,w,z;r_k^{\mathrm{out}})\cap \big\{w'\xleftrightarrow{(D)}z'\big\}\\
	& \cap    \big\{ w'\xleftrightarrow{(D)} \partial B_{v_e^{-}}(\tfrac{2dr_{k'+1}}{\Cref{const_boundary1}}) \cup \partial B_{v_e^{-}}(2\Cref{const_boundary1}r_k^{\mathrm{in}})  \big\}^c.
\end{split}
\end{equation*}
 Similarly, we denote $\widehat{\mathsf{A}}_{\mathrm{III}}^{}:=\widehat{\mathsf{A}}_{\mathrm{III}}^{\emptyset}$ when $D=\emptyset$.





 



 By the restriction property and Corollary \ref{coro28}, one has  \begin{equation}\label{413}
 	\begin{split}
 		&\sum\nolimits_{z_5'\in \partial \mathcal{B}_{v_e^{-}}(2dr_{k'+1})}	\mathbb{P}\big(\overline{\mathsf{A}}_+(\bar{z};r,r_{k} ) \Vert  \overline{\mathsf{A}}_-(\bar{z}';2r,2r_{k'} ) \big) \\ 
 		\overset{}{\le} & \sum_{z_5'\in \partial \mathcal{B}_{v_e^{-}}(2dr_{k'+1})} \mathbb{E}\big[ \mathbbm{1}_{\{ \overline{\mathsf{A}}_+(\bar{z};r,r_{k} ) \Vert  \widehat{\mathsf{A}}_{-}(\bar{z}';2r,2r_{k'} )     \}} \cdot \mathbb{P}\big( w_-\xleftrightarrow{(\widehat{\mathcal{C}}_-)} z_5'   \mid \mathcal{F}_{\widehat{\mathcal{C}}_-}\big) \big] \\
 	\overset{(\ref{ineq_coro_28})}{\lesssim }& r_k^{-d}\sum_{z_5'\in \partial \mathcal{B}_{v_e^{-}}(2dr_{k'+1}),z_6' \in \partial \mathcal{B}_{v_e^-}(2r_k^{\mathrm{out}}),z_7' \in \partial \mathcal{B}_{v_e^-}(2r_k^{\mathrm{in}})}\mathbb{E}\Big[ \mathbbm{1}_{\{ \overline{\mathsf{A}}_+(\bar{z};r,r_{k} ) \Vert  \widehat{\mathsf{A}}_{-}(\bar{z}';2r,2r_{k'} )     \}} \\
  &\ \ \ \ \ \ \ \ \ \ \ \ \ \ \ \ \ \ \ \ \ \ \ \ \ \ \ \ \ \ \ \ \ \ \  \ \ \cdot \mathbb{P}\big(\widehat{\mathsf{A}}_{\mathrm{III}}^{\widehat{\mathcal{C}}_- }(v_e^-,w_-,z_5',z_6',z_7';k,k')  \mid \mathcal{F}_{\widehat{\mathcal{C}}_-}\big) \Big]\\
\overset{}{=} &r_k^{-d} \sum_{z_5'\in \partial \mathcal{B}_{v_e^{-}}(2dr_{k'+1}),z_6' \in \partial \mathcal{B}_{v_e^-}(2r_k^{\mathrm{out}}),z_7' \in \partial \mathcal{B}_{v_e^-}(2r_k^{\mathrm{in}})}\mathbb{P}\big( \overline{\mathsf{A}}_+(\bar{z};r,r_{k} ) \Vert  \widetilde{\mathsf{A}}_{-} \big),
 	\end{split}
 \end{equation}
 where $\widetilde{\mathsf{A}}_{-}:=\widehat{\mathsf{A}}_{-}(\bar{z}';2r,2r_{k'} )\cap \widehat{\mathsf{A}}_{\mathrm{III}}(v_e^-,w_-,z_5',z_6',z_7';k,k')$. Following the same arguments as in (\ref{413}), we have
 \begin{equation}\label{414}
 	\begin{split}
 		&\sum\nolimits_{z_3\in \partial \mathcal{B}_{v_e^{+}}(r_1),z_4\in \partial \mathcal{B}_{v_e^{+}}(\frac{r_k}{\Cref{const_boundary1}}) }\mathbb{P}\big( \overline{\mathsf{A}}_+(\bar{z};r,r_{k} ) \Vert  \widetilde{\mathsf{A}}_{-} \big) \\
 		\lesssim & r_{k'}^{-d} \sum_{z_3\in \partial \mathcal{B}_{v_e^{+}}(r_1),z_4\in \partial \mathcal{B}_{v_e^{+}}(\frac{r_k}{\Cref{const_boundary1}}),z_6\in \partial \mathcal{B}_{v_e^+}(r_{k'}^{\mathrm{out}}),z_7 \in \partial \mathcal{B}_{v_e^+}(r_{k'}^{\mathrm{in}})}\mathbb{P}\big( \widetilde{\mathsf{A}}_{+} \Vert  \widetilde{\mathsf{A}}_{-} \big).
 	\end{split}
 \end{equation}
 Here the event $\widetilde{\mathsf{A}}_{+}$ is defined by 
 \begin{equation}
 	\begin{split}
 		\widetilde{\mathsf{A}}_{+}:= &\mathsf{A}_{\mathrm{I}}(\hat{v}_e^{+},z_1;r)\cap \mathsf{A}_{\mathrm{I}}(\hat{v}_{e'}^{+},z_2;r) \cap \mathsf{A}_{\mathrm{III}}(v_e^{+},w_+,z_5;r_{k+1}) 
 		   \cap \widehat{\mathsf{A}}_{\mathrm{II}},
 		\end{split}
 \end{equation}
 where $\widehat{\mathsf{A}}_{\mathrm{II}}:=\big\{z_3\xleftrightarrow{(D)} z_7 , z_4\xleftrightarrow{(D)}z_6\big\}\cap \big\{ z_3 \xleftrightarrow{(D)} \partial B_{v_e^{+}}( \frac{r_1}{\Cref{const_boundary1} } ) \cup \partial B_{v_e^{+}}(\Cref{const_boundary1}r_{k'}^{\mathrm{in}}  )  \big\}^c  
	 \cap \big\{ z_4 \xleftrightarrow{(D)} \partial B_{v_e^{+}}(r_k ) \cup \partial B_{v_e^{+}}(\tfrac{r_{k'}^{\mathrm{out}}}{\Cref{const_boundary1}})  \big\}^c$. Combining (\ref{413}) and (\ref{414}), we obtain  
 \begin{equation}\label{new4.17}
 	\begin{split}
 	&	\sum_{\hat{z}\in \widehat{\Upsilon}_+(r,r
 		_k),\hat{z}'\in \widehat{\Upsilon}_-(2r,2r
 		_{k'})}  \mathbb{P}\big(\overline{\mathsf{A}}_+(z_1,z_2,\hat{z};r,r_{k} ) \Vert  \overline{\mathsf{A}}_-(z_1',z_2',\hat{z}';2r,2r_{k'} ) \big) \\
 		\lesssim &r_k^{-d} r_{k'}^{-d} \sum_{\widetilde{z}:=(z_3,z_4,z_5,z_6,z_7)\in \widetilde{\Upsilon}_+(k,k'),\widetilde{z}':=(z_3',z_4',z_5',z_6',z_7')\in \widetilde{\Upsilon}_-(k,k')} \mathbb{P}\big( \widetilde{\mathsf{A}}_{+} \Vert  \widetilde{\mathsf{A}}_{-} \big),
 	\end{split}
 \end{equation}
 where we denote $\widetilde{\Upsilon}_+(k,k'):=\widehat{\Upsilon}_+(r,r_{k})\times \partial \mathcal{B}_{v_e^+}(r_{k'}^{\mathrm{out}})\times \partial \mathcal{B}_{v_e^+}(r_{k'}^{\mathrm{in}})$ and $\widetilde{\Upsilon}_-(2r,2r
 		_{k'}):=\widehat{\Upsilon}_-(2r,2r_{k'})\times \partial \mathcal{B}_{v_e^-}(2r_k^{\mathrm{out}}) \times \partial \mathcal{B}_{v_e^-}(2r_k^{\mathrm{in}})$. See Figure \ref{fig3} for an illustration of $\big\{\widetilde{\mathsf{A}}_{+} \Vert  \widetilde{\mathsf{A}}_{-}\big\}$.


 \begin{figure}[h]
	\centering
	\includegraphics[width=0.8\textwidth]{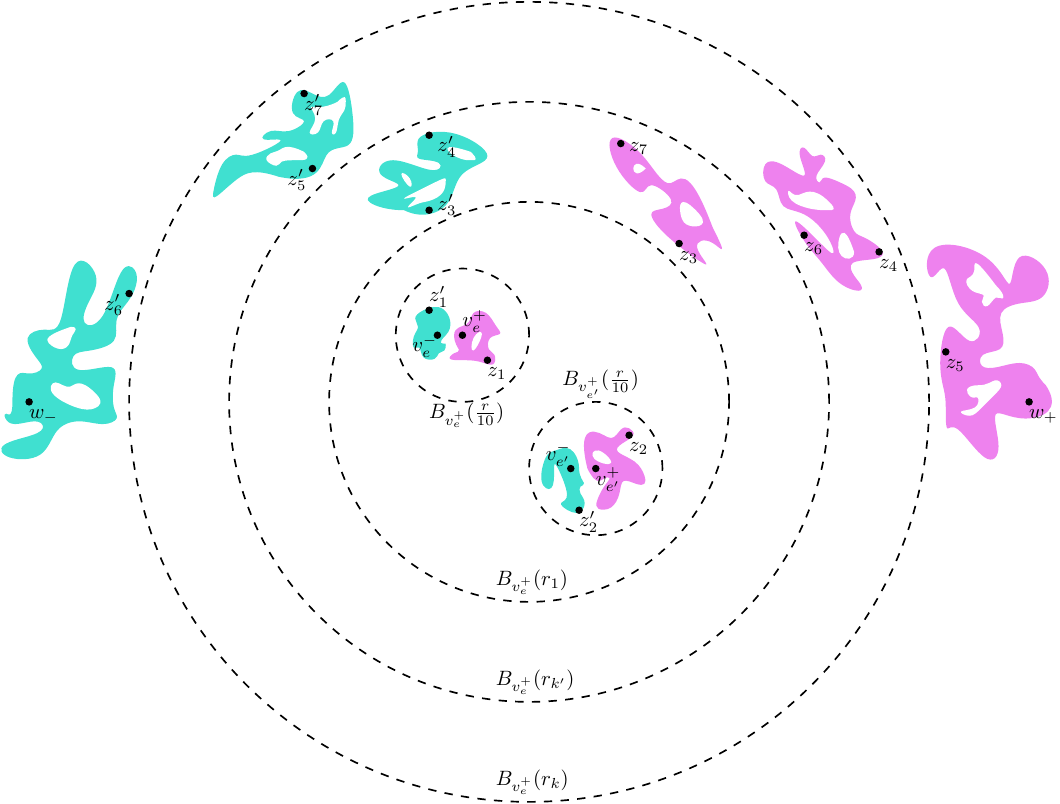}
	\caption{In this illustration, each colored region is a loop cluster. These clusters are arranged in disjoint balls and annuli in pairs such that each pair of clusters certifies a four-point event in $\{\mathsf{C}_j\}_{1\le j\le 5}$. In particular, when $k=k_\star$, as verified in the companion paper \cite{inpreparation_twoarm}, imposing $\mathcal{C}_{w_-}$ (resp. $\mathcal{C}_{w_+}$) as an absorbing boundary typically does not alter the order of the probability of $z_5 \xleftrightarrow{} w_{+}$ (resp. $z_6'\xleftrightarrow{} w_{-}$), i.e., $N^{2-d}$. In light of this, we substitute the four-point event $\{z_5\xleftrightarrow{} w_{+}  \Vert z_6'\xleftrightarrow{} w_{-}\}$ with the event $\{\mathbf{A}_\star^{z_5}=\mathbf{A}_\star^{z_6'}= 1\}$ to simplify the analysis, while changing the probability by only a constant factor. }\label{fig3}
\end{figure}

 		By (\ref{2.14}), on $\big\{\widetilde{\mathsf{A}}_{+} \Vert  \widetilde{\mathsf{A}}_{-}\big\}$, the events $\{	\mathsf{C}_{j}\}_{1\le j\le 5}$ (as defined below) occur.
 \begin{itemize}

 	\item  $	\mathsf{C}_{1}:= \big\{ \hat{v}_e^+\xleftrightarrow{(\partial B_{v_e^+}(\frac{r}{100}))} z_1 \Vert \hat{v}_e^-\xleftrightarrow{(\partial B_{v_e^+}(\frac{r}{100}))} z_1' \big\}$;

 	\item $\mathsf{C}_{2}:= \big\{ \hat{v}_{e'}^-\xleftrightarrow{(\partial B_{v_{e'}^+}(\frac{r}{100}))} z_2 \Vert \hat{v}_{e'}^-\xleftrightarrow{(\partial B_{v_{e'}^+}(\frac{r}{100}))} z_2' \big\}$;

 	\item $	\mathsf{C}_{3}:= \big\{ z_3\xleftrightarrow{(\partial B_{v_e^+}(\frac{r_1}{\Cref{const_boundary1}})\cup \partial B_{v_e^+}(10\Cref{const_boundary1}r_{k'}^{\mathrm{in}}))} z_7 \Vert z_3'\xleftrightarrow{(\partial B_{v_e^+}(\frac{r_1}{\Cref{const_boundary1}})\cup \partial B_{v_e^+}(\Cref{const_boundary1}r_{k'}^{\mathrm{in}}))} z_4' \big\}$;

 	\item $\mathsf{C}_{4}:= \big\{ z_6\xleftrightarrow{(\partial B_{v_e^+}(\frac{r_{k'}^{\mathrm{out}}}{\Cref{const_boundary1}})\cup \partial B_{v_e^+}(2\Cref{const_boundary1} r_k^{\mathrm{in}}))} z_4 \Vert z_5'\xleftrightarrow{(\partial B_{v_e^+}(\frac{r_{k'}^{\mathrm{out}}}{\Cref{const_boundary1}})\cup \partial B_{v_e^+}(2\Cref{const_boundary1} r_k^{\mathrm{in}}))} z_7' \big\}$;

 	\item When $k<k_{\star}$, we define $\mathsf{C}_{5}:= \big\{ z_5 \xleftrightarrow{(\partial B_{v_e^+}(\frac{r_{k}^{\mathrm{out}}}{\Cref{const_boundary1}}))} w_+ \Vert z_6'\xleftrightarrow{(\partial B_{v_e^+}(\frac{r_{k}^{\mathrm{out}}}{\Cref{const_boundary1}}))} w_- \big\}$, whereas for $k=k_{\star}$, we define $\mathsf{C}_{5}:= \big\{ \mathbf{A}_\star^{z_5}=\mathbf{A}_\star^{z_6'}= 1  \big\}$.

 \end{itemize} 


 	
 	Similar to $\{\mathsf{F}_{j}\}_{1\le j\le 4}$, the sub-events $\{\mathsf{C}_j\}_{1\le j\le 5}$ are independent. Moreover, the estimate (\ref{four_point_use}) and the fact that $\mathbb{P}(\mathbf{A}_\star^{z})=N^{2-d}$ for all $z\in \mathbb{Z}^d$ together imply  
 	\begin{equation}
 		\mathbb{P}(\mathsf{C}_1) \vee \mathbb{P}(\mathsf{C}_2) \lesssim r^{-\frac{3d}{2}+1}, 
 	\end{equation} 		
 	 \begin{equation}
 			\mathbb{P}(\mathsf{C}_3) \cdot  \mathbb{P}(\mathsf{C}_4) \lesssim r^{3-\frac{d}{2}} r_{k'}^{-\frac{3d}{2}+1}\cdot r_{k'}^{3-\frac{d}{2}} r_k^{-\frac{3d}{2}+1}= r^{3-\frac{d}{2}} r_{k'}^{-2d+4}  r_k^{-\frac{3d}{2}+1}, 
 		\end{equation}
 		\begin{equation}\label{addto427}
 			\mathbb{P}(\mathsf{C}_5)\lesssim r_k^{3-\frac{d}{2}}N^{-\frac{3d}{2}+1}  \cdot \mathbbm{1}_{k<k_\star } + N^{4-2d} \cdot  \mathbbm{1}_{k=k_\star } \overset{(r_{k_\star}\asymp N)}{\asymp} =r_k^{3-\frac{d}{2}} N^{-\frac{3d}{2}+1}. 
 		\end{equation}
 		 	Consequently, we obtain 
\begin{equation}\label{426}
	\begin{split}
		\mathbb{P}\big( \widetilde{\mathsf{A}}_{+} \Vert  \widetilde{\mathsf{A}}_{-} \big) \le \prod\nolimits_{1\le j\le 5} \mathbb{P}(\mathsf{C}_j)\lesssim r^{-\frac{7d}{2}+5} r_{k'}^{-2d+4} r_{k}^{-2d+4}N^{-\frac{3d}{2}+1}.
	\end{split}
\end{equation}
 	 Inserting (\ref{426}) into (\ref{new4.17}), we also derive the claim (\ref{newadd410}) for $k>k'$: 
 		\begin{equation}
 			\begin{split}
 			&	\sum_{\hat{z}\in \widehat{\Upsilon}_+(r,r
 		_k),\hat{z}'\in \widehat{\Upsilon}_-(2r,2r
 		_{k'})}  \mathbb{P}\big(\overline{\mathsf{A}}_+(z_1,z_2,\hat{z};r,r_{k} ) \Vert  \overline{\mathsf{A}}_-(z_1',z_2',\hat{z}';2r,2r_{k'} ) \big)\\
 		\lesssim &r_k^{-d} r_{k'}^{-d} \cdot  |\widehat{\Upsilon}_+(r,r
 		_k)| \cdot |\widehat{\Upsilon}_-(2r,2r
 		_{k'})|  \cdot r^{-\frac{7d}{2}+5} r_{k'}^{-2d+4} r_{k}^{-2d+4}N^{-\frac{3d}{2}+1}\\
 		\lesssim & r^{-\frac{3d}{2}+3}r_{k}^{d}r_{k'}^{d}N^{-\frac{3d}{2}+1}.
 			\end{split}
 		\end{equation}


 		In conclusion, we have confirmed Proposition \ref{lemma_four_arms} assuming Lemma \ref{lemma41}.  \qed

\subsection{Proof of Lemma \ref{lemma41}}\label{subsection_proof_lemma4.1}
 Using the switching identity, we have (recalling the notations $\widecheck{\mathbb{P}}^D_{\cdot}$, $\mathfrak{p}^D_{\cdot}$, $\mathcal{P}^{(i)}$ and $\widecheck{\mathcal{C}}$ below Lemma \ref{lemma_switching})
  \begin{equation}\label{428}
  \begin{split}
  	 &   	\mathbb{P}\big( \big\{ w\xleftrightarrow{(D)} v_1,v_2  \big\} \cap \big\{  \widehat{\mathcal{L}}_{1/2}^{D,v_1},\widehat{\mathcal{L}}_{1/2}^{D,v_2}  \in [0,\Cref{const_bar_lowd_four_point1}] \big\} \big)\\
  	  	=&  \int_{0 \le a\le b\le \Cref{const_bar_lowd_four_point1}}  \widecheck{\mathbb{P}}^D_{v_1 \leftrightarrow v_2 ,a,b}\big(   w  \in \widecheck{\mathcal{C}} \big)  \mathfrak{p}^D_{v_1 \leftrightarrow v_2 ,a,b} \mathrm{d}a  \mathrm{d}b. 
  \end{split}
  \end{equation}
  For each integer $k\in [1,k_\star-1]$, we define $\mathsf{G}_k$ as the event that $\cup (\sum_{i\in \{2,3,4\}}\mathcal{P}^{(i)})$ intersects $\partial B_{v_1}(r_{k})$ and is contained in $\widetilde{B}_{v_1}(r_{k+1})$. In addition, we denote the events 
  \begin{equation}
  	\mathsf{G}_0:= \big \{\cup \big(\sum\nolimits_{i\in \{2,3,4\}}\mathcal{P}^{(i)} \big)\subset \partial \widetilde{B}_{v_1}(r_1) \big\} \ \text{and}\ \mathsf{G}_{k_\star}:=\big( \cup_{0\le k\le k_\star-1} \mathsf{G}_k\big)^c.  
  \end{equation} 
  Therefore, it follows from (\ref{428}) that 
   \begin{equation}
   	\begin{split}
   		&	\mathbb{P}\big( \big\{ w\xleftrightarrow{(D)} v_1,v_2  \big\} \cap \big\{  \widehat{\mathcal{L}}_{1/2}^{D,v_1},\widehat{\mathcal{L}}_{1/2}^{D,v_2}  \in [0,\Cref{const_bar_lowd_four_point1}] \big\} \big) \\
   			 \lesssim & \mathbb{P}(v_1\xleftrightarrow{(D)} v_2 ) \cdot  \sum\nolimits_{0\le k\le k_\star} \sup_{0 \le a\le b\le \Cref{const_bar_lowd_four_point1}}  \widecheck{\mathbb{P}}^D_{v_1 \leftrightarrow v_2 ,a,b}( w  \in \widecheck{\mathcal{C}} , \mathsf{G}_k ). 
   	\end{split}
   \end{equation}
   For convenience, we denote the $k$-th term on the right-hand side of (\ref{ineq_lemma41}) by 
   \begin{equation}
   	\begin{split}
   		\mathbf{T}_k :=r^{-d-1} r_k^{-\frac{3d}{2}+1}\mathbf{Q}^{D}_{\mathrm{I}}(v_1',v_2';r)   	\mathbf{Q}_{\mathrm{II}}^D(v_1;r_1,r_k) 	\mathbf{Q}_{\mathrm{III}}^D(v_1,w;r_{k+1}). 
   	\end{split}
   \end{equation}
   Thus, it suffices to prove that for any $0\le k\le k_\star$ and $0 \le a\le b\le \Cref{const_bar_lowd_four_point1}$, 
 \begin{equation}\label{4.32}
 \begin{split}
 	 	\mathbb{P}(v_1\xleftrightarrow{(D)} v_2 ) \cdot    \widecheck{\mathbb{P}}^D_{v_1 \leftrightarrow v_2 ,a,b}( w  \in \widecheck{\mathcal{C}} , \mathsf{G}_k )\lesssim  \mathbf{T}_{k\vee 2}. 
 \end{split}
 \end{equation}


 \textbf{When $k\in \{0,1\}$.} Note that Corollary \ref{lemma_cut_two_points} directly implies    
  \begin{equation}\label{4.34}
  	\mathbb{P}(v_1\xleftrightarrow{(D)} v_2 )  \lesssim r^{-d}\mathbf{Q}_{\mathrm{I}}^{D}(v_1,v_2;r). 
  \end{equation} 
  Recall from the conditions that the endpoint $x_i$ of $e_i$ is contained in $D$, and that the points $x_i$ and $v_i'$ lie on opposite sides of $v_i$ along the interval $I_{e_i}$. As a result, any subset of $\widetilde{\mathbb{Z}}^d\setminus D$ connecting $v_i$ and $z\in \partial \mathcal{B}_{v_i}(\frac{r}{\Cref{const_boundary1}})$ must hit $v_i'$, which yields $\mathbf{Q}_{\mathrm{I}}^{D}(v_1,v_2;r)\le \mathbf{Q}_{\mathrm{I}}^{D}(v_1',v_2';r)$. Combined with (\ref{4.34}), it implies that  
   \begin{equation}\label{newv_4.34}
   	\mathbb{P}(v_1\xleftrightarrow{(D)} v_2 )  \lesssim r^{-d}\mathbf{Q}_{\mathrm{I}}^{D}(v_1',v_2';r).
   \end{equation}
 For $k\in \{0,1\}$, when $\mathsf{G}_k$ occurs, the union $\cup (\sum_{i\in \{2,3,4\}}\mathcal{P}^{(i)})$ is contained in $\widetilde{B}_{v_1}(r_2)$. Thus, if $w\in \widecheck{\mathcal{C}}$ also occurs, the cluster of $\cup \mathcal{P}^{(1)}$ containing $w$ must intersect $\widetilde{B}_{v_1}(r_3)$. Combined with $\mathcal{P}^{(1)}\le \widetilde{\mathcal{L}}_{1/2}^{D}$, it yields that      \begin{equation}
  	\begin{split}
  		 \widecheck{\mathbb{P}}^D_{v_1 \leftrightarrow v_2 ,a,b}( w  \in \widecheck{\mathcal{C}} , \mathsf{G}_k ) \le   \mathbb{P}\big(w\xleftrightarrow{(D)}  \widetilde{B}_{v_1}(r_{3}) \big).  
  	\end{split}
  \end{equation}
 In addition, Corollary \ref{lemma_point_to_boundary} implies that for any $0\le k\le k_\star$, 
  \begin{equation}\label{new4.36}
  	\begin{split}
  		 \mathbb{P}\big(w\xleftrightarrow{(D)}  \widetilde{B}_{v_1}(r_{k+1}) \big)   \lesssim r_k^{-\frac{d}{2}} \mathbf{Q}_{\mathrm{III}}^{D}(v_1,w;r_{k+1}).
  	\end{split}
  \end{equation}
  Meanwhile, it follows from (\ref{211}) and the definition of $\mathbf{Q}_{\mathrm{II}}^D$ that 
  \begin{equation}\label{revise_new_433}
  	\mathbf{Q}_{\mathrm{II}}^D(v_1;r_1,r_2)\asymp   r^{d}.  
  \end{equation}
  Combining (\ref{newv_4.34})-(\ref{revise_new_433}), we derive (\ref{4.32}) for $k\in \{0,1\}$.

   \textbf{When $2\le k\le k_\star-1$.}  Similarly, when $\{w \in \widecheck{\mathcal{C}} , \mathsf{G}_k\}$ occurs, one has $w\xleftrightarrow{\cup \mathcal{P}^{(1)}}  \widetilde{B}_{v_1}(r_{k+1})$, which implies that 
   \begin{equation}\label{new413}
    	\begin{split}
    		&\widecheck{\mathbb{P}}^D_{v_1 \leftrightarrow v_2 ,a,b}\big(   w \in \widecheck{\mathcal{C}},\mathsf{G}_k \big)\\
    		\le  &	\widecheck{\mathbb{P}}^D_{v_1 \leftrightarrow v_2 ,a,b}\Big( \bigcup_{i\in \{2,3,4\}} \big\{ (\cup \mathcal{P}^{(i)})\cap \partial B_{v_1}(r_k)\neq \emptyset \big\}  \Big) \cdot \mathbb{P}\big(w\xleftrightarrow{(D)}  \widetilde{B}_{v_1}(r_{k+1}) \big)\\
    		\le & \sum\nolimits_{i\in \{2,3,4\}}  \widecheck{\mathbb{P}}^D_{v_1 \leftrightarrow v_2 ,a,b}\big(  (\cup \mathcal{P}^{(i)})\cap \partial B_{v_1}(r_k)\neq \emptyset  \big) \cdot \mathbb{P}\big(w \xleftrightarrow{(D)}  \widetilde{B}_{v_1}(r_{k+1}) \big).
    	\end{split}
    \end{equation}
    For each $i\in \{2,3\}$, note that 
    \begin{equation}\label{pool442}
    	\begin{split}
    		   	&	\widecheck{\mathbb{P}}^D_{v_1 \leftrightarrow v_2 ,a,b}\big(  (\cup \mathcal{P}^{(i)})\cap \partial B_{v_1}(r_k)\neq \emptyset  \big) \\
   		\lesssim &	 \mathbf{e}^D_{v_{i-1} , v_{i-1}}\big( \big\{ \widetilde{\eta}: \mathrm{ran}(\widetilde{\eta} )\cap \partial B_{v_1}(r_k)\neq \emptyset \big\} \big). 
    	\end{split}
    \end{equation}
    Moreover, applying Lemma \ref{lemma_for_new62} (where we drop both two terms in the third line of (\ref{addto242})), we derive that  
    \begin{equation}
    	\begin{split}
    		 & \mathbf{e}^D_{v_{i-1} , v_{i-1}}\big( \big\{ \widetilde{\eta}: \mathrm{ran}(\widetilde{\eta} )\cap \partial B_{v_1}(r_k)\neq \emptyset \big\} \big) 
    		\overset{}{\lesssim}  
    		    r^{1-d}r_k^{1-d}  \mathbf{Q}_{\mathrm{II}}^D(v_1;r_1,r_k). 
    	\end{split}
    \end{equation}
    Combined with (\ref{pool442}), it implies that for $i\in \{2,3\}$, 
     \begin{equation}\label{new439}
   	\begin{split}
 \widecheck{\mathbb{P}}^D_{v_1 \leftrightarrow v_2 ,a,b}\big(  (\cup \mathcal{P}^{(i)})\cap \partial B_{v_1}(r_k)\neq \emptyset  \big) \lesssim  r^{1-d}r_k^{1-d}  \mathbf{Q}_{\mathrm{II}}^D(v_1;r_1,r_k). 
   		   	\end{split}
   \end{equation} 
   

   For $i=4$, according to (\ref{odd_formula}), $\widecheck{\mathbb{P}}^D_{v_1 \leftrightarrow v_2 ,a,b}\big(  (\cup \mathcal{P}^{(4)})\cap \partial B_{v_1}(r_k)\neq \emptyset  \big)$ is at most of the same order as the probability of an excursion hitting $\partial B_{v_1}(r_k)$ under the measure $\mathbf{e}^D_{v_1,v_2}$. Combined with the fact that the total mass of $\mathbf{e}^D_{v_1,v_2}$ equals $\mathbb{K}_{D\cup \{v_1,v_2\}}(v_1,v_2)$, it yields that 
   \begin{equation}\label{newfor439}
   	\begin{split}
   		&\widecheck{\mathbb{P}}^D_{v_1 \leftrightarrow v_2 ,a,b}\big(  (\cup \mathcal{P}^{(4)})\cap \partial B_{v_1}(r_k)\neq \emptyset  \big) \\
   		\lesssim& \big[ \mathbb{K}_{D\cup \{v_1,v_2\}}(v_1,v_2)\big]^{-1}\cdot   \mathbf{e}^D_{v_{1} , v_{2}}\big( \big\{ \widetilde{\eta}: \mathrm{ran}(\widetilde{\eta} )\cap \partial B_{v_1}(r_k)\neq \emptyset \big\} \big) \\
   		\overset{(\text{Lemma}\ \ref{lemma_for_new62})}{\lesssim } &\big[ \mathbb{K}_{D\cup \{v_1,v_2\}}(v_1,v_2)\big]^{-1}\cdot   r^{-1-d}r_k^{1-d}\mathbf{Q}_{\mathrm{I}}^{D}(v_1',v_2';r)  \mathbf{Q}_{\mathrm{II}}^D(v_1;r_1,r_k). 
   	\end{split}
   \end{equation}
Here in the application of Lemma \ref{lemma_for_new62}, we used the case $\chi>10c^{-1}$ and dropped the second term in the third line of (\ref{addto242}). Meanwhile, by (\ref{for2.16}) and (\ref{newv_4.34}), 
   \begin{equation}\label{newfor440}
   	\begin{split}
   		\mathbb{P}(v_1\xleftrightarrow{(D)} v_2 ) \lesssim  \big[\mathbb{K}_{D\cup \{v_1,v_2\}}(v_1,v_2)\big] \land  \big[r^{-d} \mathbf{Q}_{\mathrm{I}}^{D}(v_1',v_2';r) \big]. 
   	\end{split}
   \end{equation}
   Combining (\ref{new439}), (\ref{newfor439}) and (\ref{newfor440}), we obtain that 
   \begin{equation}\label{newfor441}
   	\begin{split}
   		&\sum\nolimits_{i\in \{2,3,4\}}\mathbb{P}(v_1\xleftrightarrow{(D)} v_2 ) \cdot \widecheck{\mathbb{P}}^D_{v_1 \leftrightarrow v_2 ,a,b}\big(  (\cup \mathcal{P}^{(i)})\cap \partial B_{v_1}(r_k)\neq \emptyset  \big) \\
   		\lesssim  &r^{-1-d}r_k^{1-d}\mathbf{Q}_{\mathrm{I}}^{D}(v_1',v_2';r)  \mathbf{Q}_{\mathrm{II}}^D(v_1;r_1,r_k). 
   	\end{split}
   \end{equation}
   Putting (\ref{new4.36}), (\ref{new413}) and (\ref{newfor441}) together, we derive (\ref{4.32}) for $2 \le k\le k_\star-1$.

  \textbf{When $k=k_\star$.} Since $\mathbf{Q}_{\mathrm{III}}^D(v_1,w;r_{k_\star+1})\asymp R$, it suffices to prove 
  \begin{equation}\label{good444}
  \begin{split}
  	 	&\mathbb{P}(v_1\xleftrightarrow{(D)} v_2 ) \cdot   \widecheck{\mathbb{P}}^D_{v_1 \leftrightarrow v_2 ,a,b}( w  \in \widecheck{\mathcal{C}} , \mathsf{G}_{k_\star} )\\
  	 	  	\lesssim & r^{-d-1} R^{-\frac{3d}{2}+2}\mathbf{Q}^{D}_{\mathrm{I}}(v_1',v_2';r)   	\mathbf{Q}_{\mathrm{II}}^D(v_1;r_1,r_{k_{\star}}).
  \end{split}
  	  \end{equation}
  We denote $l_\dagger:= \min\big\{ l\in \mathbb{N}:  2^{l+3}\ge R\big\}$. For each $0\le l\le l_\dagger-1$, we define the box $\mathfrak{B}_l:=\widetilde{B}_w(2^l)$. We also denote $\mathfrak{B}_{-1}:=\emptyset$ and $\mathfrak{B}_{l_\dagger}:=[\widetilde{B}_{v_1}(r_{k_\star})]^c$. For each $i\in \{2,3,4\}$ and $0\le l\le l_\dagger$, we define the event 
  \begin{equation}
  	\begin{split}
  		\mathsf{G}_\star^{i,l}:= \big\{ \big( \cup \mathcal{P}^{(i)} \big)\cap \mathfrak{B}_{l}\neq \emptyset, \big( \cup \mathcal{P}^{(i)} \big)\cap \mathfrak{B}_{l-1}= \emptyset \big\}.  
  	\end{split}
  \end{equation}
  Since $\mathsf{G}_{k_\star}\subset \cup_{i\in \{2,3,4\},0\le l\le l_\dagger}\mathsf{G}_\star^{i,l}$, it follows from the union bound that 
  \begin{equation}\label{443}
  	\begin{split}
    \widecheck{\mathbb{P}}^D_{v_1 \leftrightarrow v_2 ,a,b}\big(   w \in \widecheck{\mathcal{C}},\mathsf{G}_{k_\star}  \big)  
  		\lesssim \sum\nolimits_{i\in \{2,3,4\}} \sum\nolimits_{0\le l\le l_\dagger} \widecheck{\mathbb{P}}^D_{v_1 \leftrightarrow v_2 ,a,b}\big(   w \in \widecheck{\mathcal{C}},\mathsf{G}_\star^{i,l} \big). 
  	\end{split}
  \end{equation}
  Moreover, for $0\le l\le l_\dagger$, on the event $\{ w \in \widecheck{\mathcal{C}}\}\cap \mathsf{G}_{k_\star}$, the cluster of $\cup \mathcal{P}^{(1)}$ containing $w$ must intersect $\widetilde{\partial }\mathfrak{B}_{l-1}$ (where we set $\widetilde{\partial }\emptyset:= \widetilde{\mathbb{Z}}^d$ for completeness). Thus, 
  \begin{equation}\label{444}
  	\begin{split}
  		&\widecheck{\mathbb{P}}^D_{v_1 \leftrightarrow v_2 ,a,b}\big(   w \in \widecheck{\mathcal{C}},\mathsf{G}_\star^{i,l}   \big) \\
  		 \le&  \widecheck{\mathbb{P}}^D_{v_1 \leftrightarrow v_2 ,a,b}\big(  ( \cup \mathcal{P}^{(i)}  )\cap \mathfrak{B}_{l}\neq \emptyset \big)\cdot  \mathbb{P}\big( w\xleftrightarrow{(D)} \widetilde{\partial }\mathfrak{B}_{l-1}  \big). 
  	\end{split}
  \end{equation}
 Since $\widetilde{\mathcal{L}}_{1/2}^D\le \widetilde{\mathcal{L}}_{1/2}$, it follows from (\ref{one_arm_low}) that 
  \begin{equation}\label{addnew445}
  	\mathbb{P}\big( w\xleftrightarrow{(D)} \widetilde{\partial }\mathfrak{B}_{l-1}  \big) \lesssim \theta_d(2^l)\lesssim 2^{l(-\frac{d}{2}+1)}. 
  \end{equation}

  Next, we estimate the probability $\widecheck{\mathbb{P}}^D_{v_1 \leftrightarrow v_2 ,a,b}\big( ( \cup \mathcal{P}^{(i)} )\cap \mathfrak{B}_{l}\neq \emptyset \big)$ for $i\in \{2,3,4\}$. When $i\in \{2,3\}$, similar to (\ref{pool442}), one has 
  \begin{equation}\label{pool453}
  	\begin{split}
  		   \widecheck{\mathbb{P}}^D_{v_1 \leftrightarrow v_2 ,a,b}\big( ( \cup \mathcal{P}^{(i)} )\cap \mathfrak{B}_{l}\neq \emptyset \big)     
  			\lesssim   	 \mathbf{e}^D_{v_{i-1} , v_{i-1}}\big( \big\{ \widetilde{\eta}: \mathrm{ran}(\widetilde{\eta} )\cap \mathfrak{B}_{l} \neq \emptyset \big\} \big). 
  	\end{split}
  \end{equation}
  Moreover, by applying Lemma \ref{lemma_for_new62}, we have  
  \begin{equation}\label{pool454}
  	\begin{split}
  		\mathbf{e}^D_{v_{i-1} , v_{i-1}}\big( \big\{ \widetilde{\eta}: \mathrm{ran}(\widetilde{\eta} )\cap \mathfrak{B}_{l} \neq \emptyset \big\} \big) \lesssim  r^{1-d}R^{3-2d}   \mathbf{Q}_{\mathrm{II}}^D(v_1;r_1,r_{k_\star})\cdot 2^{l(d-2)}.  	\end{split}
  \end{equation}
  Here we dropped the first term in the third line of (\ref{addto242}), and we bound the second term using the fact that 
  \begin{equation}\label{newpool455}
  	\widetilde{\mathbb{P}}_z\big(\tau_{\mathfrak{B}_l} <\infty \big)  \lesssim \big(2^l/R  \big)^{d-2}, \ \ \forall z\in B(\tfrac{R}{2}).
  \end{equation}
Combining (\ref{pool453}) and (\ref{pool454}), we obtain that 
  \begin{equation}\label{addnew446}
  	\begin{split}
  		 \widecheck{\mathbb{P}}^D_{v_1 \leftrightarrow v_2 ,a,b}\big( ( \cup \mathcal{P}^{(i)} )\cap \mathfrak{B}_{l}\neq \emptyset \big)
  		\lesssim   r^{1-d}R^{3-2d}   \mathbf{Q}_{\mathrm{II}}^D(v_1;r_1,r_{k_\star})\cdot 2^{l(d-2)}.
  	\end{split}
  \end{equation}



   When $i=4$, as in (\ref{newfor439}), we have   
   \begin{equation}\label{pool457}
   	\begin{split}
   		&\widecheck{\mathbb{P}}^D_{v_1 \leftrightarrow v_2 ,a,b}\big( ( \cup \mathcal{P}^{(4)} )\cap \mathfrak{B}_{l}\neq \emptyset \big)\\
   	\lesssim &\big[ \mathbb{K}_{D\cup \{v_1,v_2\}}(v_1,v_2)\big]^{-1}\cdot   \mathbf{e}^D_{v_{1} , v_{2}}\big( \big\{ \widetilde{\eta}: \mathrm{ran}(\widetilde{\eta} )\cap \mathfrak{B}_{l}\neq \emptyset \big\} \big). 
   	\end{split}
   \end{equation}
   Moreover, Lemma \ref{lemma_for_new62} implies that 
   \begin{equation}\label{pool458}
   	\begin{split}
   		& \mathbf{e}^D_{v_{1} , v_{2}}\big( \big\{ \widetilde{\eta}: \mathrm{ran}(\widetilde{\eta} )\cap \mathfrak{B}_{l}\neq \emptyset \big\} \big) \\
   		 \lesssim & r^{-d-1}R^{3-2d}  \mathbf{Q}_{\mathrm{I}}^{D}(v_1',v_2';r)  \mathbf{Q}_{\mathrm{II}}^D(v_1;r_1,r_{k_\star})  \cdot 2^{l(d-2)}. 
   	\end{split}
   \end{equation}
   Here we used the case $\chi=r>10c^{-1}$, and bounded the second term in the third line of (\ref{addto242}) through (\ref{newpool455}). Plugging (\ref{pool458}) into (\ref{pool457}), one has  
   \begin{equation}\label{addnew448}
   	\begin{split}
   	&\widecheck{\mathbb{P}}^D_{v_1 \leftrightarrow v_2 ,a,b}\big( ( \cup \mathcal{P}^{(4)} )\cap \mathfrak{B}_{l}\neq \emptyset \big)\\ 
   		\lesssim &\big[ \mathbb{K}_{D\cup \{v_1,v_2\}}(v_1,v_2)\big]^{-1}r^{-d-1}R^{3-2d}  \mathbf{Q}_{\mathrm{I}}^{D}(v_1',v_2';r)  \mathbf{Q}_{\mathrm{II}}^D(v_1;r_1,r_{k_\star})  \cdot 2^{l(d-2)}.
   	\end{split}
   \end{equation}
   Putting (\ref{newfor440}), (\ref{addnew445}), (\ref{addnew446}) and (\ref{addnew448}) together, we obtain that 
   \begin{equation}
   \begin{split}
  &  	\mathbb{P}(v_1\xleftrightarrow{(D)} v_2 ) \cdot \sum_{i\in \{2,3,4\}} \sum_{0\le l\le l_\dagger} \widecheck{\mathbb{P}}^D_{v_1 \leftrightarrow v_2 ,a,b}\big( ( \cup \mathcal{P}^{(i)} )\cap \mathfrak{B}_{l}\neq \emptyset \big)\mathbb{P}\big( w\xleftrightarrow{(D)} \widetilde{\partial }\mathfrak{B}_{l-1}  \big) \\
   	\lesssim & r^{-d-1}R^{3-2d}  \mathbf{Q}_{\mathrm{I}}^{D}(v_1',v_2';r)  \mathbf{Q}_{\mathrm{II}}^D(v_1;r_1,r_{k_\star})\cdot \sum\nolimits_{0\le l\le l_\dagger}  2^{l(\frac{d}{2}-1)}\\
   	\lesssim & r^{-d-1}R^{-\frac{3d}{2}+2}  \mathbf{Q}_{\mathrm{I}}^{D}(v_1',v_2';r)  \mathbf{Q}_{\mathrm{II}}^D(v_1;r_1,r_{k_\star}). 
   	   \end{split}
   \end{equation}
 Combined with (\ref{443}) and (\ref{444}), it yields (\ref{good444}).

To sum up, we have verified Lemma \ref{lemma41}, thereby completing the proof of  Proposition \ref{lemma_four_arms}.  \qed

 \section{Typical number of pivotal loops at scale $1$}\label{section_pivotal}

 In this section, we aim to prove Theorem \ref{thm_pivotal_loop}.

 \subsection{Proof of Theorem \ref{thm_pivotal_loop}: upper bound}\label{subsection5.1}

 To achieve this, we first estimate the first moment of $\widetilde{\mathbf{Piv}}_N$ as follows. Recall that for each $e\in \mathbb{L}^d$, we denote the two trisection points of $I_e$ by $v_e^+$ and $v_e^-$. In addition, we denote the point measure 
 \begin{equation}\label{51}
 	\mathfrak{L}_e:=  \widetilde{\mathcal{L}}_{1/2}\cdot \mathbbm{1}_{\mathrm{ran}(\widetilde{\ell})\subset I_e,\{v_e^+,v_e^- \}\subset \mathrm{ran}(\widetilde{\ell}) }. 
 \end{equation}
 For any edge $e\in \mathbb{L}^d$ and event $\mathsf{A}$, we denote by $\mathsf{L}_e(\mathsf{A})$ the event that there exists a loop $\widetilde{\ell}\in 	\mathfrak{L}_e$ which is pivotal for $\mathsf{A}$. In addition, we define the set 
 \begin{equation}
 	\widetilde{\mathbf{Piv}}(\mathsf{A}):= \big\{ e\in \mathbb{L}^d: \mathsf{L}_e(\mathsf{A})\ \text{occurs} \big\}. 
 \end{equation}
 Recall that we denote $\widetilde{\mathbf{Piv}}_N=\widetilde{\mathbf{Piv}}(\bm{0}\xleftrightarrow{} \partial B(N))$.

 \begin{lemma}\label{lemma_first_moment}
 	For any $d\ge 3$ with $d\neq 6$, and any $N\ge 1$, 
 	 \begin{equation}\label{52}
 	\mathbb{E}\big[ \big|\widetilde{\mathbf{Piv}}_N \big| \mid \bm{0}\xleftrightarrow{} \partial B(N) \big]  \lesssim   N^{(\frac{d}{2}-1)\boxdot 2}. 
 \end{equation}
 \end{lemma}
 \begin{proof}
 In this proof, we abbreviate $\mathsf{L}_e\big(\bm{0}\xleftrightarrow{} \partial B(N)\big)$ as $\mathsf{L}_e$. For each $e=\{x_e^1,x_e^2\}\in  \mathbb{L}^d$, on the event $\mathsf{L}_e$, we know that $\mathsf{F}_e:=  \big\{ \mathfrak{L}_e \neq 0 \big\}$ and 
 \begin{equation}
 		 \mathsf{C}_e:= \cup_{i\in \{1,2\}} \big\{ x_e^i \xleftrightarrow{\cup (\widetilde{\mathcal{L}}_{1/2}-\mathfrak{L}_e)} \bm{0} \Vert  x_e^{3-i}\xleftrightarrow{\cup (\widetilde{\mathcal{L}}_{1/2}-\mathfrak{L}_e)} \partial B(N) \big\} 
 	\end{equation} 
 	both occur. Note that $\mathsf{F}_e$ and $\mathsf{C}_e$ are independent (since they rely on disjoint collections of loops in $\widetilde{\mathcal{L}}_{1/2}$), and that $\mathbb{P}(\mathsf{F}_e)$ is bounded away from $0$ and $1$. Therefore, 
 \begin{equation}\label{5.4}
 	\begin{split}
 		\mathbb{P}\big( \mathsf{L}_e \big) \le \mathbb{P}\big( \mathsf{C}_e\cap \mathsf{F}_e \big) = \mathbb{P}\big( \mathsf{C}_e\cap \mathsf{F}_e^c \big) \cdot \frac{\mathbb{P}(\mathsf{F}_e)}{\mathbb{P}(\mathsf{F}^c_e)} \asymp  \mathbb{P}\big( \mathsf{C}_e\cap \mathsf{F}_e^c \big). 
 	\end{split}
 \end{equation}
 	Moreover, on $\mathsf{F}_e^c=\{ \mathfrak{L}_e=0\}$, the event $\mathsf{C}_e$ is equivalent to 
 	 \begin{equation}\label{5.5}
 		\overline{\mathsf{C}}_e :=\cup_{i\in \{1,2\}} \overline{\mathsf{C}}_e^i:= \cup_{i\in \{1,2\}} \big\{ x_e^i \xleftrightarrow{} \bm{0} \Vert  x_e^{3-i}\xleftrightarrow{} \partial B(N) \big\}.  
 	\end{equation} 
 	Combining (\ref{5.4}) and (\ref{5.5}), and using the union bound, we have 
 	\begin{equation}\label{5.6}
 		\begin{split}
 		 	 \mathbb{E}\big[ \big|\widetilde{\mathbf{Piv}}_N \big| \cdot \mathbbm{1}_{ \bm{0}\xleftrightarrow{} \partial B(N)} \big]   
 			=&\sum\nolimits_{e\in \mathbb{L}^d: I_e\subset \widetilde{B}(N)} \mathbb{P}\big( \mathsf{L}_e \big) \\
 			\lesssim  & \sum\nolimits_{e\in \mathbb{L}^d: I_e\subset \widetilde{B}(N)} \sum\nolimits_{i\in \{1,2\}}\mathbb{P}\big(\overline{\mathsf{C}}_e^i  \big).  
 		\end{split}
 	\end{equation}

 	We claim that for any $e=\{x,y\}\in \mathbb{L}^d$ with $x,y\in B(N)$, 
 	\begin{equation}\label{5.7}
 		  \mathbb{P}\big(x \xleftrightarrow{} \bm{0} \Vert  y\xleftrightarrow{} \partial B(N)  \big) \lesssim \left\{
\begin{aligned}
&(|x|+1)^{-[(\frac{d}{2}+1)\boxdot (d-2)]}N^{-[(\frac{d}{2}-1)\boxdot 2]}      &  x\in B(\tfrac{N}{2}); \\ 
&\big(\mathrm{dist}(x,\partial B(N) )+1\big)^{-2}N^{2-d}    &   x\notin B(\tfrac{N}{2}).
\end{aligned}
\right.
 	\end{equation}
 	In fact, this claim together with (\ref{5.6}) implies the desired bound (\ref{52}). Specifically, by (\ref{5.6}) and (\ref{5.7}), we have 
 	\begin{equation}\label{new58}
 		\begin{split}
 		 \mathbb{E}\big[ \big|\widetilde{\mathbf{Piv}}_N \big| \cdot \mathbbm{1}_{ \bm{0}\xleftrightarrow{} \partial B(N)} \big]  \lesssim  \mathbb{I}^{\mathrm{in}}+ \mathbb{I}^{\mathrm{out}}, 
 		\end{split}
 	\end{equation}
 	where the quantities $\mathbb{I}^{\mathrm{in}}$ and $\mathbb{I}^{\mathrm{out}}$ are defined by 
 	\begin{equation}
 		\mathbb{I}^{\mathrm{in}}:=N^{-[(\frac{d}{2}-1)\boxdot 2]}\sum\nolimits_{x\in B(\frac{N}{2})} (|x|+1)^{-[(\frac{d}{2}+1)\boxdot (d-2)]}, 
 	\end{equation}
 	\begin{equation}
 		\mathbb{I}^{\mathrm{out}}:=N^{2-d}\sum\nolimits_{x\in B(N)\setminus B(\frac{N}{2})}\big(\mathrm{dist}(x,\partial B(N) )+1\big)^{-2}. 
 	\end{equation}
 	Using the fact that $|\partial B(k)|\asymp k^{d-1}$ for $k\ge 1$, we have 
 	\begin{equation}\label{511}
 			\mathbb{I}^{\mathrm{in}} \lesssim N^{-[(\frac{d}{2}-1)\boxdot 2]} \sum\nolimits_{0\le k\le \frac{N}{2}} (k+1)^{(\frac{d}{2}-2)\boxdot 1}\lesssim  1 , 
 	\end{equation}
 \begin{equation}\label{512}
 	\mathbb{I}^{\mathrm{out}} \lesssim N^{2-d}  \sum\nolimits_{ \frac{N}{2}\le k\le N} (N-k+1)^{d-3}\lesssim 1.  
 \end{equation}
 	Plugging (\ref{511}) and (\ref{512}) into (\ref{new58}), and using (\ref{one_arm_low}) and (\ref{one_arm_high}), we obtain (\ref{52}).

 	Next, we prove the claim (\ref{5.7}) for the cases $x\in B(\tfrac{N}{2})$ and $x\notin B(\tfrac{N}{2})$ separately.

  \textbf{Case 1: when $ x\in B(\tfrac{N}{2})$.} Applying Lemma \ref{lemma_relation}, one has 
  \begin{equation}\label{late514}
  	\mathbb{P}\big(y\xleftrightarrow{(D)} \partial B(N)  \big)  \overset{ }{\lesssim}    N^{-(\frac{d}{2}\boxdot 3)} \sum\nolimits_{z\in \partial \mathcal{B}(d^{-1}N)} \mathbb{P}\big(y\xleftrightarrow{(D)} z  \big).
  \end{equation}
  Combined with the restriction property, it yields that 
\begin{equation}\label{late515}
	\begin{split}
		  \mathbb{P}\big(x \xleftrightarrow{} \bm{0} \Vert  y\xleftrightarrow{} \partial B(N)  \big) 
	  \overset{ }{=} &  \mathbb{E}\big[ \mathbbm{1}_{x \xleftrightarrow{} \bm{0} } \cdot \mathbb{P}\big( y\xleftrightarrow{(\mathcal{C}_{\bm{0}})} \partial B(N) \mid \mathcal{F}_{\mathcal{C}_{\bm{0}}}\big) \big] \\
	  \overset{ }{\lesssim } & N^{-(\frac{d}{2}\boxdot 3)} \sum\nolimits_{z\in \partial \mathcal{B}(d^{-1}N)} \mathbb{P}\big(x \xleftrightarrow{} \bm{0} \Vert  y\xleftrightarrow{} z \big). 
	\end{split}
\end{equation}
  Moreover, for any $z\in \partial \mathcal{B}(d^{-1}N)$, by the restriction property and Corollary \ref{newcoro_29}, 
  \begin{equation}\label{late516}
  	\begin{split}
  		& \mathbb{P}\big(x \xleftrightarrow{} \bm{0} \Vert  y\xleftrightarrow{} z \big) =  \mathbb{E}\big[ \mathbbm{1}_{x \xleftrightarrow{} \bm{0} } \cdot \mathbb{P}\big( y\xleftrightarrow{(\mathcal{C}_{\bm{0}})} z \mid \mathcal{F}_{\mathcal{C}_{\bm{0}}}\big) \big]\\
  		\lesssim &N^{2-d}  (|x|+1)^{-1} \sum\nolimits_{z' \in \partial \mathcal{B}(d^{-1}|x|)}  \mathbb{E}\big[ \mathbbm{1}_{x \xleftrightarrow{} \bm{0} } \cdot \mathbb{P}\big( y\xleftrightarrow{(\mathcal{C}_{\bm{0}})} z' \mid \mathcal{F}_{\mathcal{C}_{\bm{0}}} \big) \big]\\
  	= &	N^{2-d}  (|x|+1)^{-1} \sum\nolimits_{z' \in \partial \mathcal{B}(d^{-1}|x|)}  \mathbb{P}\big(x \xleftrightarrow{} \bm{0} \Vert  y\xleftrightarrow{} z'  \big) \\
  		 \overset{(\ref{four_point_use})}{\lesssim } & (|x|+1)^{-[(\frac{d}{2}+1)\boxdot (2-d)]}N^{2-d}. 
  	\end{split}
  \end{equation}
  This together with (\ref{late515}) yields that 
  \begin{equation}
  	\mathbb{P}\big(x \xleftrightarrow{} \bm{0} \Vert  y\xleftrightarrow{} \partial B(N)  \big)  \lesssim (|x|+1)^{-[(\frac{d}{2}+1)\boxdot (2-d)]}N^{-[(\frac{d}{2}-1)\boxdot 2]}.
  \end{equation}

 \textbf{Case 2: when $ x\notin B(\tfrac{N}{2})$.} Assume that $x\in \partial B(N-k)$ with $0\le k\le \frac{N}{2}$. Thus, 
 \begin{equation}
 	  \mathbb{P}\big(x \xleftrightarrow{} \bm{0} \Vert  y\xleftrightarrow{} \partial B(N)  \big) \le  \mathbb{P}\big(x \xleftrightarrow{} \bm{0} \Vert  y\xleftrightarrow{} \partial B_y(k)  \big). 
 \end{equation}
 Moreover, similar to (\ref{late515}), by the restriction property and Lemma \ref{lemma_relation}, one has 
 \begin{equation}\label{late519}
 	\begin{split}
 			  \mathbb{P}\big(x \xleftrightarrow{} \bm{0} \Vert  y\xleftrightarrow{} \partial B_y(k)  \big) 
 			\lesssim  (k+1)^{-(\frac{d}{2}\boxdot 3)} \sum\nolimits_{z\in \partial \mathcal{B}_y(d^{-1}k)}  \mathbb{P}\big( x \xleftrightarrow{} \bm{0} \Vert  y\xleftrightarrow{} z \big). 
 	\end{split}
 \end{equation}
 In addition, as in (\ref{late516}), it follows from Corollary \ref{newcoro_29} that for any $z\in \partial \mathcal{B}_y(d^{-1}k)$, 
 \begin{equation}\label{late520}
 \begin{split}
 	 	\mathbb{P}\big( x \xleftrightarrow{} \bm{0} \Vert  y\xleftrightarrow{} z \big)  \lesssim & N^{2-d}(k+1)^{-1} \sum\nolimits_{z'\in \partial \mathcal{B}_y(k)}   \mathbb{P}\big( x \xleftrightarrow{} z' \Vert  y\xleftrightarrow{} z \big)\\
 	 	\overset{(\ref{four_point_use})}{\lesssim } & N^{2-d}(k+1)^{-[(\frac{d}{2}+1)\boxdot (d-2)]}.
 \end{split}
 \end{equation}
  Substituting (\ref{late520}) into (\ref{late519}), we get
  \begin{equation}
  	  \mathbb{P}\big(x \xleftrightarrow{} \bm{0} \Vert  y\xleftrightarrow{} \partial B_y(k)  \big) 
 			\lesssim N^{2-d}(k+1)^{-2}.  
  \end{equation}
  To sum up, we have verified the claim (\ref{5.7}), thereby completing the proof.
 \end{proof}

 By Lemma \ref{lemma_first_moment} and Markov's inequality, we have: for any $\epsilon>0$, there exists $\Cref{const_pivotal_3}(d,\epsilon)>0$ such that 
   \begin{equation}\label{ineq_before_54}
 	\begin{split}
 		\mathbb{P}\big( \big|\widetilde{\mathbf{Piv}}_N \big| \ge \Cref{const_pivotal_3} N^{(\frac{d}{2}-1)\boxdot 2} \mid \bm{0}\xleftrightarrow{} \partial B(N) \big)\le \tfrac{1}{2}\epsilon. 
 	\end{split}
 \end{equation}

 \subsection{Proof of Theorem \ref{thm_pivotal_loop}: lower bound}


 We now show that for any $\epsilon>0$, there exists $\cref{const_pivotal_1}(d,\epsilon)>0$ such that  
   \begin{equation}\label{54}
 	\begin{split}
 		\mathbb{P}\big( \big|\widetilde{\mathbf{Piv}}_N \big| \le \cref{const_pivotal_1} N^{(\frac{d}{2}-1)\boxdot 2} \mid \bm{0}\xleftrightarrow{} \partial B(N) \big)\le \tfrac{1}{2}\epsilon. 
 	\end{split}
 \end{equation}

 To prove (\ref{54}), we first introduce some notations as follows.
 
  \begin{itemize}

 	\item   Let $n_0,\lambda>1$ be two large numbers that will be determined later. For each $i\in \mathbb{N}^+$, we denote 
 	 \begin{equation}
 	 	n_i=n_i(n_0,\lambda):= \lambda^2i^4n_{i-1}. 
 	 \end{equation}

 	\item  For sufficiently large $N\ge 1$, we define 
 	\begin{equation}\label{cite525}
 		i_\star= 	i_\star(n_0,\lambda,N ):= \min\big\{ i\ge 1: n_{2i+3}\ge N \big\}. 
 	\end{equation}

 	\item For each $i\in \mathbb{N}$, we employ the partial cluster $\widehat{\mathfrak{C}}_i$ introduced in \cite[Section 3.2]{cai2024incipient}. Roughly speaking, $\widehat{\mathfrak{C}}_i$ is the loop cluster where only the intersections inside $\widetilde{B}(n_{2i})$ are taken into account. To be precise, one may define $\widehat{\mathfrak{C}}_i$ via the following inductive exploration process:  
 	\begin{itemize}

 	\item   Step $0$: We set $\mathcal{C}_0:=\{\bm{0}\}$;

 	\item   Step $j$ ($j\ge 1$): Given $\mathcal{C}_{j-1}$, we define $\mathcal{C}_{j}$ as the union of $\mathcal{C}_{j-1}$ and the ranges of all loops in $\widetilde{\mathcal{L}}_{1/2}$ that intersect $\mathcal{C}_{j-1}\cap \widetilde{B}(n_{2i})$. If $\mathcal{C}_{j}=\mathcal{C}_{j-1}$, we stop the construction and define $\widehat{\mathfrak{C}}_i:=\mathcal{C}_{j}$; otherwise (i.e., $\mathcal{C}_{j-1}\subsetneq \mathcal{C}_{j}$), we proceed to Step $(j+1)$.  	
 	
 	\end{itemize}

 \end{itemize}

According to the construction of $\widehat{\mathfrak{C}}_i$, we have the following observations: 
\begin{enumerate}

	\item[(i)]  Conditioned on $\{\widehat{\mathfrak{C}}_i=D\}$ (for any proper realization $D$ of $\widehat{\mathfrak{C}}_i$), the point measure composed of all remaining loops (i.e., loops that are not used in the construction of $\widehat{\mathfrak{C}}_i$) is distributed by $\widetilde{\mathcal{L}}_{1/2}^{D\cap \widetilde{B}(n_{2i})}$. Therefore, for any $A\subset [\widetilde{B}(n_{2i})]^c$, one has $\big\{ \bm{0}\xleftrightarrow{} A  \big\}=\big\{ \widehat{\mathfrak{C}}_i\xleftrightarrow{(\widehat{\mathfrak{C}}_i\cap \widetilde{B}(n_{2i}))} A  \big\}$.

	

	\item[(ii)]   For each $0\le i\le i_\star$, on the event $\{\bm{0}\xleftrightarrow{} \partial B(N)\}$, given $\widehat{\mathfrak{C}}_i$, for any $e\in \mathbb{L}^d$, if there exists $\widetilde{\ell}\in \mathfrak{L}_e$ with $\mathrm{ran}(\widetilde{\ell})\not\subset \widehat{\mathfrak{C}}_i$ and pivotal for the event $\{\widehat{\mathfrak{C}}_i \xleftrightarrow{} \partial B(n_{2i+2}) \}$ (we denote by $\widetilde{\mathbf{Piv}}_{N,i}$ the collection of all these edges), then $\widetilde{\ell}$ is also pivotal for $\{\bm{0}\xleftrightarrow{} \partial B(N)\}$. (\textbf{P.S.} To see this, if we remove the loop $\widetilde{\ell}$ from $\widetilde{\mathcal{L}}_{1/2}$, then $\widehat{\mathfrak{C}}_i$ and $\partial B(n_{2i+2})$ are not connected by $\cup \widetilde{\mathcal{L}}_{1/2}^{\widehat{\mathfrak{C}}_i\cap \widetilde{B}(n_{2i})}$. Combined with Observation (i), it yields that $\{\bm{0}\xleftrightarrow{} \partial B(N)\}$ does not occur.) Furthermore, since $\widetilde{\ell}$ is pivotal for $\{\widehat{\mathfrak{C}}_i \xleftrightarrow{} \partial B(n_{2i+2}) \}$, it must be connected to $\widehat{\mathfrak{C}}_i$ by $\cup \widetilde{\mathcal{L}}_{1/2}^{\partial B(n_{2i+2})}$, and thus is included in $\widehat{\mathfrak{C}}_{i+1}$. As a result, $\widetilde{\mathbf{Piv}}_{N,i}$ is measurable with respect to $\mathcal{F}_{\widehat{\mathfrak{C}}_{i+1}}$.


\end{enumerate}

 Next, we present two auxiliary lemmas, and establish (\ref{54}) using them. The first lemma states that conditioned on the event $\{\bm{0}\xleftrightarrow{} \partial B(N)\}$, it is unlikely to have a loop that crosses an annulus with a large ratio between the outer and inner radii. For any $M>m\ge 1$, we denote the point measure
 \begin{equation}\label{def_crossing}
 	\mathfrak{L}[m,M]:= \widetilde{\mathcal{L}}_{1/2}\cdot \mathbbm{l}_{\mathrm{ran}(\widetilde{\ell})\cap \partial B(m)\neq \emptyset,\mathrm{ran}(\widetilde{\ell})\cap \partial B(M)\neq \emptyset }. 
 \end{equation}
  \begin{lemma} \label{lemma_large_loop}
 	For any $d\ge 3$ with $d\neq 6$, there exist $C,c>0$ such that for any $n\ge 1$, $A,D\subset \widetilde{B}(n)$, $m\ge Cn$, $M\ge Cm$ and $N\ge (CM)\vee n^{1 \boxdot \frac{d-4}{2}}$,
 \begin{equation}\label{ineq_large_crossing_loop1}
 	\mathbb{P}\big(  	\mathfrak{L}[m,M]\neq 0  \mid A \xleftrightarrow{(D)} \partial B(N) \big) \lesssim  \big( \tfrac{m}{M} \big)^{(\frac{d}{2}-1)\boxdot (d-4)} . 
 \end{equation}
 \end{lemma}


 Since the proof of Lemma \ref{lemma_large_loop} is relatively standard, we postpone it to Appendix \ref{app_large_loop}. By the construction of $\widehat{\mathfrak{C}}_i$, we know that every loop included in $\widehat{\mathfrak{C}}_i$ must intersect $\widetilde{B}(n_{2i})$. As a result, on $\{ \widehat{\mathfrak{C}}_i \cap  \partial B(n_{2i+1}) \ne \emptyset\}$, there exists a loop $\widetilde{\ell}\in \widetilde{\mathcal{L}}_{1/2}$ intersecting both $\partial B(n_{2i})$ and $\partial B(n_{2i+1})$, i.e., $\mathfrak{L}[n_{2i},n_{2i+1}]\neq 0 $. Thus, for $0\le i\le i_\star$, 
 \begin{equation}\label{lemma52}
 	\begin{split}
 	&	\mathbb{P}\big( \widehat{\mathfrak{C}}_i \cap  \partial B(n_{2i+1}) \ne \emptyset \mid \bm{0}\xleftrightarrow{} \partial B(N) \big)  \\
 		\le & 	\mathbb{P}\big( \mathfrak{L}[n_{2i},n_{2i+1}]\neq 0  \mid \bm{0}\xleftrightarrow{} \partial B(N) \big)  \overset{(\text{Lemma}\ \ref{lemma_large_loop})}{\lesssim }\lambda^{-1}(i+1)^{-2}.   
 	\end{split}
 \end{equation}
By taking $\lambda=\epsilon^{-2}$, it follows from (\ref{lemma52}) that 
 \begin{equation}\label{510}
 	\mathbb{P}\big( \mathsf{F}_i^c \mid \bm{0}\xleftrightarrow{} \partial B(N)  \big)\le \tfrac{1}{4}\epsilon, \ \ \forall 0\le i\le i_\star.
 \end{equation}
 Here the event $\mathsf{F}_i$ is defined by $\cap_{0\le j \le i} \big\{  \widehat{\mathfrak{C}}_j \subset  \widetilde{B}(n_{2j+1})  \big\}$.

 (\textbf{P.S.} Although in (\ref{lemma52}) we only used Lemma \ref{lemma_large_loop} for the special case $A=\{\bm{0}\}$ and $D=\emptyset$, the lemma in its full generality will play a crucial role in the companion paper \cite{inpreparation_gap}.)

 {\color{orange}

}



The second lemma shows that conditioned on a general set being connected to a distant box boundary of size $N$, with a uniformly positive probability the number of pivotal loops at scale $1$ for this event is at least of order $N^{(\frac{d}{2}-1)\boxdot 2}$. 

  \begin{lemma}\label{lemma53}
 	For any $d\ge 3$ with $d\neq 6$, there exist $C>0$ and $c,c',c''>0$ such that for any $M\ge 1$, $N\ge CM$, and $A,D\subset \widetilde{B}(cM)$,
 	\begin{equation}
 			\mathbb{P}\big( \big|\widetilde{\mathbf{Piv}} \big( A  \xleftrightarrow{(D)} \partial B(M)\big) \big|\ge  c' M^{(\frac{d}{2}-1)\boxdot 2}    \mid A  \xleftrightarrow{(D)} \partial B(N) \big)\ge  c''.
 	\end{equation}
 \end{lemma}

{\color{red}  
 
 
}

 The proof of Lemma \ref{lemma53} will be provided in Section \ref{subsection_lemma53}. 
 

 \begin{proof}[Proof of (\ref{54})]
 Without loss of generality, we assume that $\epsilon>0$ is sufficiently small. We set $\lambda=\epsilon^{-2}$, $n_0=\mathrm{exp}(-\epsilon^{-10})N$ and $\cref{const_pivotal_1}=\mathrm{exp}(-\epsilon^{-100})$. Note that $i_\star\ge \epsilon^{-1}$. For each $0\le i\le i_\star$, we define the event 
 	  	 \begin{equation}
 	  	 	\mathsf{G}_i:= \cap_{0\le j \le i  }  \big\{ 	 \big|\widetilde{\mathbf{Piv}}_{N,j} \big| \le \cref{const_pivotal_1} N^{(\frac{d}{2}-1)\boxdot 2}  \big\}. 
 	  	 \end{equation}
 	  	  Using Observation (ii), we have $\cup_{0\le i\le i_\star}	\widetilde{\mathbf{Piv}}_{N,i}  \subset  	\widetilde{\mathbf{Piv}}_{N}$, which implies that 
 	\begin{equation}\label{late512}
 	 \big\{ 	 \big|\widetilde{\mathbf{Piv}}_N \big| \le \cref{const_pivotal_1} N^{(\frac{d}{2}-1)\boxdot 2}  \big\} \subset \mathsf{G}_{i_{\star}}. 
 	\end{equation}

For each $1\le i\le i_\star$, recall from Observation (i) below (\ref{cite525}) that conditioned on $\mathcal{F}_{\widehat{\mathfrak{C}}_i}$ (note that $\widehat{\mathfrak{C}}_i$ is now determined), the point measure consisting of the remaining loops has the same distribution as $\widetilde{\mathcal{L}}_{1/2}^{\widehat{\mathfrak{C}}_i\cap \widetilde{B}(n_{2i})}$. As a result, we have 
\begin{equation}\label{nice533}
	\begin{split}
		&\mathbb{P}\big( \bm{0 }\xleftrightarrow{} \partial B(N) , \mathsf{F}_{i}, \mathsf{G}_{i}  \big)\\
 	  	 	=&\mathbb{E}\big[  \mathbbm{1}_{ \mathsf{F}_i\cap  \mathsf{G}_{i-1}   } \cdot   \mathbb{P}\big( \widehat{\mathfrak{C}}_i \xleftrightarrow{(\widehat{\mathfrak{C}}_i\cap \widetilde{B}(n_{2i})) } \partial B(N), \big|\widetilde{\mathbf{Piv}}_{N,i}\big| \le \cref{const_pivotal_1} N^{(\frac{d}{2}-1)\boxdot 2}   \mid \mathcal{F}_{\widehat{\mathfrak{C}}_i} \big) \big]. 
	\end{split}
\end{equation}
In addition, since $\mathsf{F}_i\subset \big\{\widehat{\mathfrak{C}}_i\subset \widetilde{B}(n_{2i+1})\big\}$, it follows from Lemma \ref{lemma53} (with $A=\widehat{\mathfrak{C}}_i$ and $D=\widehat{\mathfrak{C}}_i\cap \widetilde{B}(n_{2i})$) that there exists $c_\dagger\in (0,1)$ such that on $\mathsf{F}_i$, 
\begin{equation}\label{nice534}
	\begin{split}
	&	\mathbb{P}\big( \widehat{\mathfrak{C}}_i \xleftrightarrow{(\widehat{\mathfrak{C}}_i\cap \widetilde{B}(n_{2i})) } \partial B(N), \big|\widetilde{\mathbf{Piv}}_{N,i}\big| \le \cref{const_pivotal_1} N^{(\frac{d}{2}-1)\boxdot 2}   \mid \mathcal{F}_{\widehat{\mathfrak{C}}_i} \big)  \\
		\le & (1-c_\dagger) \cdot \mathbb{P}\big( \widehat{\mathfrak{C}}_i\xleftrightarrow{(\widehat{\mathfrak{C}}_i\cap \widetilde{B}(n_{2i}))} \partial B(N)  \mid \mathcal{F}_{\widehat{\mathfrak{C}}_i} \big). 
	\end{split}
\end{equation}
Plugging (\ref{nice534}) into (\ref{nice533}), and using the inclusion $\mathsf{F}_i\subset \mathsf{F}_{i-1}$, we get 
\begin{equation}
	\begin{split}
			&\mathbb{P}\big( \bm{0 }\xleftrightarrow{} \partial B(N) , \mathsf{F}_{i}, \mathsf{G}_{i}  \big)\\
\le  & (1-c_\dagger) \cdot \mathbb{E}\big[  \mathbbm{1}_{ \mathsf{F}_{i-1}\cap  \mathsf{G}_{i-1}   } \cdot   \mathbb{P}\big( \widehat{\mathfrak{C}}_i \xleftrightarrow{(\widehat{\mathfrak{C}}_i\cap \widetilde{B}(n_{2i})) } \partial B(N)   \mid \mathcal{F}_{\widehat{\mathfrak{C}}_i} \big) \big] \\
=& (1-c_\dagger) \cdot \mathbb{P}\big( \bm{0 }\xleftrightarrow{} \partial B(N) , \mathsf{F}_{i-1}, \mathsf{G}_{i-1}  \big),
	\end{split}
\end{equation}
where the transition in the last line follows from Observation (i). 
 	 By repeating this step, we obtain that 
 	  	 \begin{equation}\label{513}
 	  	 	\begin{split}
 	  	 			&\mathbb{P}\big( \bm{0 }\xleftrightarrow{} \partial B(N) , \mathsf{F}_{i_{\star}}, \mathsf{G}_{i_{\star}}  \big) \\
 	  	 			\le & (1-c_\dagger)^{i_{\star}}\cdot \mathbb{P}\big( \bm{0 }\xleftrightarrow{} \partial B(N) , \mathsf{F}_{0}, \mathsf{G}_{0}  \big) \\
 	  	 			\le & (1-c_\dagger)^{i_{\star}}\cdot \mathbb{P}\big( \bm{0 }\xleftrightarrow{} \partial B(N)  \big)\le \tfrac{1}{4}\epsilon\cdot \mathbb{P}\big( \bm{0 }\xleftrightarrow{} \partial B(N)  \big). 
 	  	 	\end{split}
 	  	 \end{equation}
 	  	 Combined with (\ref{late512}), it implies that 
 	  	 \begin{equation}\label{514}
 	  	 	\begin{split}
 	  	 		& 	\mathbb{P}\big( \bm{0 }\xleftrightarrow{} \partial B(N) , \mathsf{F}_{i_\star}, 	  	 \big|\widetilde{\mathbf{Piv}}_N \big| \le \cref{const_pivotal_1} N^{(\frac{d}{2}-1)\boxdot 2}    \big)  \\
 	  	 			\overset{ (\ref{late512})}{\le} & \mathbb{P}\big( \bm{0 }\xleftrightarrow{} \partial B(N) , \mathsf{F}_{i_{\star}}, \mathsf{G}_{i_{\star}}  \big) \overset{ (\ref{513})}{\le}\tfrac{1}{4}\epsilon\cdot \mathbb{P}\big( \bm{0 }\xleftrightarrow{} \partial B(N)  \big). 
 	  	 	\end{split}
 	  	 \end{equation}
 	   Putting (\ref{510}) and (\ref{514}) together, we derive (\ref{54}). 	 \end{proof}

  Combining (\ref{ineq_before_54}) and (\ref{54}) completes the proof of Theorem \ref{thm_pivotal_loop}. \qed

 \subsection{Proof of Lemma \ref{lemma53}} \label{subsection_lemma53}

We arbitrarily take  a point $w_\dagger\in \partial B(\frac{M}{2})$, and denote by $\mathfrak{D}$ the collection of edges $e\in \mathbb{L}^d$ with $I_e\subset \widetilde{B}_{w_\dagger}(\frac{M}{\Cref{const_boundary1}^2})$. Recall the notations $\mathfrak{L}_e$ and $\mathsf{L}_e(\cdot)$ at the beginning of Section \ref{subsection5.1}. In this proof, we write $\mathsf{L}_e(A  \xleftrightarrow{(D)} \partial B(M))$ as $\mathsf{L}_e$ for brevity. On the event $\mathsf{L}_e$, after removing all loops in $\mathfrak{L}_e$, the endpoints of $e$ belong to two distinct loop clusters that intersect $A$ and $\partial B(M)$ respectively (we denote these two clusters by $\mathcal{C}_{\mathrm{in}}$ and $\mathcal{C}_{\mathrm{out}}$). In addition, if $\mathcal{C}_{\mathrm{out}}$ reaches $\partial B(N)$, then $A\xleftrightarrow{(D)} \partial B(N)$ also occurs. Let $c_\dagger,C_{\ddagger}>0$ be constants that will be determined later. We say a cluster $\mathcal{C}$ is good if it satisfies the following conditions: 
 \begin{itemize}

 	\item  When $3\le d\le 5$, $\mathcal{C}\cap B(c_\dagger M)=\emptyset$;
 

 	\item  When $d\ge 7$, $\mathcal{C} \cap B(c_\dagger M)=\emptyset$ and 
 	\begin{equation}\label{add_require_540} 
 		\sum\nolimits_{z\in \mathbb{Z}^d:\mathcal{C}\cap \widetilde{B}_z(1)\neq \emptyset}|z|^{2-d}\le C_{\ddagger}N^{6-d}.
 	\end{equation}

 \end{itemize}
 Recall that $A,D\subset \widetilde{B}(cM)$ and $N\ge CM$. After $c_\dagger$ and $C_{\ddagger}$ are fixed, we choose $c\ll c_\dagger$ and take $C>0$ sufficiently large so that $C_{\ddagger}N^{6-d}\ll M^{6-d}$ for $d\ge 7$.


   We define $\widehat{\mathsf{L}}_e$ as the subevent of $\mathsf{L}_e$ with the additional requirements that the cluster $\mathcal{C}_{\mathrm{out}}$ reaches $\partial B(N)$ and is good. 
  Let $\mathbf{X}:= \sum_{e\in \mathfrak{D}} \mathbbm{1}_{\widehat{\mathsf{L}}_e}$. Since $\mathbf{X}\le \big|\widetilde{\mathbf{Piv}} \big( A  \xleftrightarrow{(D)} \partial B(M)\big) \big|$, it is sufficient to prove that for some constants $c',c''>0$, 
\begin{equation}\label{534}
 		\mathbb{P}\big( \mathbf{X} \ge  c' M^{(\frac{d}{2}-1)\boxdot 2}    \mid A  \xleftrightarrow{(D)} \partial B(N) \big)\ge  c''.
 \end{equation}
 We derive (\ref{534}) through the following estimates. Arbitrarily take $x_\star\in \partial B(N)$.  
 \begin{enumerate}

 	\item For any $e\in \mathfrak{D}$, 
 	 \begin{equation}\label{535}
 		\mathbb{P}\big( \widehat{\mathsf{L}}_e \big)  \gtrsim  M^{-[(\frac{d}{2}+1)\boxdot (d-2)] }N^{(\frac{d}{2}-1)\boxdot (d-4)}\cdot \mathbb{P}\big(A\xleftrightarrow{(D)} x_\star \big);
 \end{equation}

 	\item  For any $e=\{x,y\},e'=\{x',y'\}\in \mathfrak{D}$, 
 	\begin{equation}\label{536}
		\mathbb{P}\big( \widehat{\mathsf{L}}_e\cap \widehat{\mathsf{L}}_{e'} \big)  
 		\lesssim     [M\cdot (|x-y|+1)]^{-[(\frac{d}{2}+1)\boxdot (d-2)] }  N^{(\frac{d}{2}-1)\boxdot (d-4)}\cdot  \mathbb{P}\big(A\xleftrightarrow{(D)} x_\star\big). 
 	\end{equation}

 \end{enumerate}
We first present the proof of (\ref{534}) assuming (\ref{535}) and (\ref{536}). By (\ref{535}) one has 
  \begin{equation}\label{537}
  	\mathbb{E}\big[\mathbf{X}  \big] \gtrsim   M^{(\frac{d}{2}-1)\boxdot 2 } N^{(\frac{d}{2}-1)\boxdot (d-4)}\cdot \mathbb{P}\big(A\xleftrightarrow{(D)} x_\star \big).
  \end{equation}
  Meanwhile, it follows from (\ref{536}) that 
  \begin{equation}\label{538}
  	\begin{split}
  			\mathbb{E}\big[\mathbf{X}^2  \big]\lesssim & M^{-[(\frac{d}{2}+1)\boxdot (d-2)] }N^{(\frac{d}{2}-1)\boxdot (d-4)}\cdot  \mathbb{P}\big(A\xleftrightarrow{(D)} x_\star\big)\\
  		& 	\cdot  \sum\nolimits_{x,y\in B_{w_\dagger}(\frac{M}{5})}   (|x-y|+1)^{-[(\frac{d}{2}+1)\boxdot (d-2)] }\\
  		\overset{(\ref{computation_d-a})}{\lesssim} & M^{(d-2)\boxdot 4 }N^{(\frac{d}{2}-1)\boxdot (d-4)}\cdot  \mathbb{P}\big(A\xleftrightarrow{(D)} x_\star\big).
  	\end{split}
  \end{equation}
  Applying the Paley-Zygmund inequality, (\ref{537}) and (\ref{538}), we have 
  \begin{equation}\label{539}
  	\mathbb{P}\big(  \mathbf{X} >0\big)\gtrsim \frac{	\big( \mathbb{E}\big[\mathbf{X}  \big]\big)^2}{\mathbb{E}\big[\mathbf{X}^2  \big]} \gtrsim N^{(\frac{d}{2}-1)\boxdot (d-4)}\cdot  \mathbb{P}\big(A\xleftrightarrow{(D)} x_\star\big). \end{equation}
  Combined with (\ref{538}), it implies that 
  \begin{equation}\label{540}
  	\mathbb{E}\big[\mathbf{X}^2  \mid \mathbf{X}>0 \big] \lesssim M^{(d-2)\boxdot 4 }. 
  \end{equation}
   Meanwhile, since $\widehat{\mathsf{L}}_e\subset \{ A\xleftrightarrow{(D)} \partial B(N)\}$ holds for all $e\in \mathfrak{D}$, one has 
  \begin{equation}
  	\mathbb{P}\big(  \mathbf{X} >0\big) \le \mathbb{P}\big(A\xleftrightarrow{(D)} \partial B(N) \big) \overset{(\text{Lemma}\ \ref{lemma_213})}{\lesssim }  N^{(\frac{d}{2}-1)\boxdot (d-4)}\cdot \mathbb{P}\big(A\xleftrightarrow{(D)} x_\star \big). 
  \end{equation}
  This together with (\ref{537}) yields that 
 \begin{equation}\label{542}
 	\mathbb{E}\big[\mathbf{X}   \mid \mathbf{X}>0 \big]\gtrsim M^{(\frac{d}{2}-1)\boxdot 2 }. 
 \end{equation}
     By the Paley-Zygmund inequality, (\ref{540}) and (\ref{542}), we have 
     \begin{equation}\label{543}
     	\mathbb{P}\big(\mathbf{X} \ge c' M^{(\frac{d}{2}-1)\boxdot 2 }  \mid  \mathbf{X}>0 \big) \gtrsim \frac{\big( \mathbb{E}\big[\mathbf{X}   \mid \mathbf{X}>0 \big]\big)^2}{	\mathbb{E}\big[\mathbf{X}^2  \mid \mathbf{X}>0 \big] } \gtrsim 1. 
     \end{equation}
     Combining (\ref{539}) and (\ref{543}), and using Lemma \ref{lemma_213}, we obtain (\ref{534}): 
     \begin{equation}
     	\mathbb{P}\big( \mathbf{X} \ge  c' M^{(\frac{d}{2}-1)\boxdot 2}  \big)  \gtrsim N^{(\frac{d}{2}-1)\boxdot (d-4)}\cdot  \mathbb{P}\big(A\xleftrightarrow{(D)} x_\star\big)\gtrsim  \mathbb{P}\big(  A  \xleftrightarrow{(D)} \partial B(N) \big).  
     \end{equation}


  In what follows, we establish (\ref{535}) and (\ref{536}).

 \textbf{Proof of (\ref{535}).} Recall that $v_e^+$ and $v_e^-$ denote the two trisection points of $I_e$. For any $e\in \mathfrak{D}$, we define the event 
 \begin{equation}
 	\widehat{\mathsf{C}}_{e}:= \big\{ v_e^+\xleftrightarrow{(D)} A \Vert v_e^- \xleftrightarrow{(D)} \partial B(N)   \big\}\cap \big\{ \mathcal{C}_{v_e^-}^{D}\ \text{is good} \big\} . 
 \end{equation}
 Note that on the event $\widehat{\mathsf{C}}_{e}$, one has $\mathfrak{L}_e=0$; in addition, if adding a single loop into $\mathfrak{L}_e$, then that loop becomes pivotal and hence $\widehat{\mathsf{L}}_e$ occurs. Consequently, 
 \begin{equation}\label{new545}
 	\begin{split}
 		\mathbb{P}\big(\widehat{\mathsf{L}}_e \big)  \ge  \mathbb{P}\big(\widehat{\mathsf{C}}_e \big)\cdot \frac{ \mathbb{P}\big( |\mathfrak{L}_e| =1\big)}{ \mathbb{P}\big(\mathfrak{L}_e =0  \big)}\gtrsim  \mathbb{P}\big(\widehat{\mathsf{C}}_e \big). 
 	\end{split}
 \end{equation}  
 By the restriction property, and Lemmas \ref{lemma_new_stable} and \ref{lemma_new_decompose_point_to_set}, we have   \begin{equation}\label{new546}
 	\begin{split}
 		\mathbb{P}\big(\widehat{\mathsf{C}}_e \big) \overset{}{=} & \mathbb{E}\Big[ \mathbbm{1}_{v_e^- \xleftrightarrow{(D)} \partial B(N), \mathcal{C}_{v_e^-}^D\ \text{is good} }\cdot \mathbb{P}\big( v_e^+\xleftrightarrow{(D\cup \mathcal{C}_{v_e^- })} A \big)  \Big]\\
 		\overset{(\text{Lemma}\ \ref{lemma_new_decompose_point_to_set})}{\gtrsim } &M^{d-2} \cdot \mathbb{E}\Big[ \mathbbm{1}_{v_e^- \xleftrightarrow{(D)} \partial B(N), \mathcal{C}_{v_e^-}^D\ \text{is good} } \cdot \mathbb{P}\big( A \xleftrightarrow{(D\cup \mathcal{C}_{v_e^- })}x_{ \spadesuit} \big)  \\
 	&\ \ \ \ \ \ \ \ \ \ \ \ \ 	\cdot \Big( \mathbb{P}\big( v_e^+\xleftrightarrow{(D\cup \mathcal{C}_{v_e^- })} x_{ \spadesuit} \big)- \tfrac{C'[\mathrm{diam}(A)]^{d-4}}{M^{d-4}|v_e^+|^{d-2}}\cdot \mathbbm{1}_{d\ge 7} \Big)\Big] \\
	\overset{(\text{Lemma}\ \ref{lemma_new_stable}),(\ref{one_arm_high})}{\gtrsim }  & M^{d-2}\cdot  \mathbb{P}\big( A \xleftrightarrow{(D)}x_{ \spadesuit} \big)\cdot \Big[  \mathbb{P}\big(  \widehat{\mathsf{C}}_e' \big)- \tfrac{C''[\mathrm{diam}(A)]^{d-4}}{N^2M^{2d-6}}\cdot \mathbbm{1}_{d\ge 7}\Big],  
 	\end{split}
 \end{equation}  
   where $c_{ \spadesuit}\in (c,c_\dagger)$ (recall that $c$ comes from the assumption $A,D\subset \widetilde{B}(cM)$), $x_{ \spadesuit} $ is an arbitrary point in $\partial B(c_{ \spadesuit}M)$, and $\widehat{\mathsf{C}}_e'$ is defined by 
 \begin{equation}
 	\widehat{\mathsf{C}}_e':= \big\{  v_e^+\xleftrightarrow{(D)}  x_{ \spadesuit}   \Vert v_e^- \xleftrightarrow{(D)} \partial B(N) \big\} \cap  \big\{ \mathcal{C}_{v_e^-}^D\ \text{is good} \big\}. 
 \end{equation}  
 Note that in (\ref{new546}), the occurrence of $\{ \mathcal{C}_{v_e^-}^D\ \text{is good}\}$ is necessary for the application of Lemmas \ref{lemma_new_stable} and \ref{lemma_new_decompose_point_to_set}.

 Next, we estimate the probability of $\widehat{\mathsf{C}}_e'$ through the second moment method. Consider the quantity 
 \begin{equation}
 	\mathbf{Y}:= \sum\nolimits_{y\in \partial B(N)} \mathbbm{1}_{ \widetilde{\mathsf{C}}''_{e,y}}, 
 \end{equation}
 where $\widetilde{\mathsf{C}}''_{e,y}:=\big\{  v_e^+\xleftrightarrow{(D)}  x_{ \spadesuit}   \Vert v_e^- \xleftrightarrow{(D)} y \big\} \cap  \big\{ \mathcal{C}_{v_e^-}^D\ \text{is good} \big\}$. In fact, conditioned on $\big\{v_e^- \xleftrightarrow{(D)} y\big\}$, the event $\big\{ \mathcal{C}_{v_e^-}^D\ \text{is good} \big\}$ occurs with high probability. Precisely, since $w_\dagger\in \partial B(\frac{M}{2})$ and $v_e^-\in \widetilde{B}_{w_\dagger}(\frac{M}{\Cref{const_boundary1}^2})$, one has $v_e^-\in [\widetilde{B}(\frac{M}{3})]^c$. Thus, by \cite[Lemma 4.6]{inpreparation_twoarm}, under this conditioning, the probability of $\big\{v_e^- \xleftrightarrow{(D)} B(c_\dagger M)\big\}$ uniformly converges to zero as $c_\dagger\to 0$. Meanwhile, the probability of the absence of the additional condition (\ref{add_require_540}) for $d\ge 7$ also uniformly converges to zero as $C_{\ddagger}\to \infty$, following from Markov's inequality and the lemma below:
  \begin{lemma}\label{lemma_highd_volume}
 	For any $d\ge 7$, for any $y\in \partial B(N)$, 
 	\begin{equation}
 		\mathbb{E}\Big[ \sum\nolimits_{z\in \mathbb{Z}^d} \mathbbm{1}_{z\xleftrightarrow{(D )} v_e^- } (|z|+1)^{2-d} \mid v_e^- \xleftrightarrow{(D)} y  \Big] \lesssim N^{6-d}. 
 	\end{equation}
 \end{lemma}
 The proof of Lemma \ref{lemma_highd_volume} will be provided in Appendix \ref{app_lemma_highd_volume}.

Combining the aforementioned observation (i.e., $\big\{ \mathcal{C}_{v_e^-}^D\ \text{is good} \big\}$ occurs with high probability given $\big\{v_e^- \xleftrightarrow{(D)} y\big\}$) with the rigidity lemma (see \cite[Lemma 4.1]{inpreparation_twoarm}), we obtain that conditioned on $\big\{  v_e^+\xleftrightarrow{(D)}  x_{ \spadesuit}   \Vert v_e^- \xleftrightarrow{(D)} y \big\} $, the probability of $\big\{ \mathcal{C}_{v_e^-}^D\ \text{is good} \big\}$ is bounded away from zero (note that this application of the rigidity lemma relies on the fact that $\big\{ \mathcal{C}_{v_e^-}^D\ \text{is good} \big\}$ is decreasing with respect to the range of $\mathcal{C}_{v_e^-}^D$). I.e., 
\begin{equation}\label{553}
\begin{split}
		\mathbb{P}\big(\widetilde{\mathsf{C}}''_{e,y}  \big) \asymp  \mathbb{P}\big( v_e^+\xleftrightarrow{(D)}  x_{ \spadesuit}   \Vert v_e^- \xleftrightarrow{(D)} y \big) . 
\end{split}
\end{equation}
We take a large constant $C_{\clubsuit}>0$.
 By the restriction property, the right-hand side of (\ref{553}) can be bounded from below by
 \begin{equation}\label{554}
 	\begin{split}
 	& \mathbb{E}\Big[ \mathbbm{1}_{  v_e^+\xleftrightarrow{(D)}  x_{ \spadesuit} ,  \{ v_e^+\xleftrightarrow{(D)} \partial B(C_{ \clubsuit}M)   \}^c  } \cdot \mathbb{P}\big(v_e^- \xleftrightarrow{(D\cup \mathcal{C}_{v_e^+})} y   \big)  \Big]\\
 	\overset{(\text{Lemma}\ \ref{lemma_point_to_set})}{\asymp} &\big( \tfrac{M}{N} \big)^{d-2}\cdot  \mathbb{E}\Big[ \mathbbm{1}_{  v_e^+\xleftrightarrow{(D)}  x_{ \spadesuit} ,  \{ v_e^+\xleftrightarrow{(D)} \partial B(C_{ \clubsuit}M)   \}^c  } \cdot \mathbb{P}\big(v_e^- \xleftrightarrow{(D\cup \mathcal{C}_{v_e^+})} z_{\clubsuit}   \big)  \Big]\\
 	\asymp & \big( \tfrac{M}{N} \big)^{d-2}\cdot  \mathbb{P}\big( \big\{v_e^+\xleftrightarrow{(D)}  x_{ \spadesuit}\Vert v_e^- \xleftrightarrow{(D )} z_{\clubsuit}   \big\} \cap \{ v_e^+\xleftrightarrow{(D)} \partial B(C_{ \clubsuit}M)   \}^c \big). 
 	\end{split}
 \end{equation}
 Here $z_{\clubsuit}$ is an arbitrarily point in $\partial B(C_{\clubsuit}^2M)$. Employing the same argument as in (\ref{553}), and using (\ref{four_point_use}), one has 
 \begin{equation}\label{555}
 	\begin{split}
 & 	\mathbb{P}\big( \big\{v_e^+\xleftrightarrow{(D)}  x_{ \spadesuit}\Vert v_e^- \xleftrightarrow{(D )} z_{\clubsuit}   \big\} \cap \{ v_e^+\xleftrightarrow{(D)} \partial B(C_{ \clubsuit}M)   \}^c \big) \\
 	\asymp &	\mathbb{P}\big( \big\{v_e^+\xleftrightarrow{(D)}  x_{ \spadesuit}\Vert v_e^- \xleftrightarrow{(D )} z_{\clubsuit}   \big\}   \big) \asymp    M^{-[ (\frac{3d}{2}-1) \boxdot (2d-4)] }. 
 	\end{split}
 \end{equation}
 Combining (\ref{553}),  (\ref{554}) and (\ref{555}), we obtain 
 \begin{equation}\label{556}
 	\mathbb{E}\big[\mathbf{Y} \big] = \sum\nolimits_{y\in \partial B(N) }\mathbb{P}\big( \widetilde{\mathsf{C}}''_{e,y} \big)  \asymp  \ M^{-[ (\frac{d}{2}+1) \boxdot (d-2)] }N. 
 \end{equation}

 Next, we estimate the second moment of $\mathbf{Y}$ for the cases $3\le d\le 5$ and $d\ge 7$ separately. Note that $\mathbb{E}[\mathbf{Y}^2]=\sum\nolimits_{y_1,y_2\in \partial B(N)}\mathbb{P}\big( 	\widetilde{\mathsf{C}}''_{e,y_1}\cap  	\widetilde{\mathsf{C}}''_{e,y_2}\big)$, and that 
 \begin{equation}\label{revise_new_560}
 	\mathbb{P}\big(  	\widetilde{\mathsf{C}}''_{e,y_1}\cap  	\widetilde{\mathsf{C}}''_{e,y_2} \big) \lesssim  \mathbb{I}_{y_1,y_2}:= \mathbb{P}\big( \big\{  v_e^+\xleftrightarrow{(D)}  x_{ \spadesuit} \big\}   \Vert \big\{ v_e^- \xleftrightarrow{(D)} y_1,v_e^- \xleftrightarrow{(D)} y_2 \big\} \big). 
 \end{equation}

%

 \textbf{When $3\le d\le 5$.} For any $y_1,y_2\in \partial B(N)$, for the same reason as in \cite[(4.104)]{inpreparation_twoarm}, 
 \begin{equation}\label{later558}
 	\begin{split}
 		\mathbb{I}_{y_1,y_2} \lesssim &  N^{d-2}\cdot \mathbb{P}\big(v_e^- \xleftrightarrow{} y_1,v_e^- \xleftrightarrow{ } y_2  \big) \cdot \mathbb{P}\big(   v_e^+\xleftrightarrow{(D)}  x_{ \spadesuit}    \Vert   v_e^- \xleftrightarrow{(D)} y_1  \big)\\
 		\lesssim &  (|y_1-y_2|+1)^{-\frac{d}{2}+1}\cdot \mathbb{P}\big(   v_e^+\xleftrightarrow{(D)}  x_{ \spadesuit}    \Vert   v_e^- \xleftrightarrow{(D)} y_1  \big),
 		 	\end{split}
 \end{equation}
 where in the last inequality we used \cite[Lemma 5.3]{cai2024quasi}. Meanwhile, by the restriction property, Corollary \ref{newcoro_29} and (\ref{four_point_use}), we have 
 \begin{equation}\label{later559}
 	\begin{split}
 	& 	 \mathbb{P}\big(   v_e^+\xleftrightarrow{(D)}  x_{ \spadesuit}    \Vert   v_e^- \xleftrightarrow{(D)} y_1  \big) 
 		\overset{}{ =}  \mathbb{E}\Big[\mathbbm{1}_{v_e^+\xleftrightarrow{(D)}  x_{ \spadesuit} } \cdot \mathbb{P}\big(v_e^- \xleftrightarrow{(D\cup \mathcal{C}_{v_e^+ })} y_1   \big)  \Big]\\
 		\overset{(\text{Corollary}\ \ref{newcoro_29})}{ \lesssim} & N^{2-d}M^{-1} \sum_{z\in \partial \mathcal{B}(M)} \mathbb{P}\big(   v_e^+\xleftrightarrow{(D)}  x_{ \spadesuit}    \Vert   v_e^- \xleftrightarrow{(D)} z  \big)  
 	\overset{(\ref{four_point_use})}{  \lesssim}   M^{-\frac{d}{2}-1}N^{2-d}. 
 		 \end{split}
 \end{equation}
Combining (\ref{revise_new_560}), (\ref{later558}) and (\ref{later559}), we obtain 
\begin{equation}\label{newrevise_563}
	\begin{split}
			\mathbb{E}[\mathbf{Y}^2] \lesssim  M^{-\frac{d}{2}-1}N^{2-d} \sum\nolimits_{y_1,y_2\in \partial B(N) }  (|y_1-y_2|+1)^{-\frac{d}{2}+1} \overset{(\ref{computation_d-a})}{\lesssim } M^{-\frac{d}{2}-1}N^{\frac{d}{2}+1}, 
	\end{split}
\end{equation}
 where the last inequality follows from a direct computation (see e.g., \cite[(4.4)]{cai2024high}).

 \textbf{When $d\ge 7$.} By the BKR inequality (see e.g., \cite[Lemma 3.3]{cai2024high}) and (\ref{two-point1}),  
 \begin{equation}
 	\begin{split}
 		\mathbb{I}_{y_1,y_2} \lesssim  M^{2-d}\cdot  \mathbb{P}\big(  v_e^- \xleftrightarrow{ } y_1,v_e^- \xleftrightarrow{ } y_2  \big).
 	\end{split}
 \end{equation}
 Summing over all $y_1,y_2\in \partial B(N)$, and using \cite[Lemma 4.1]{cai2024high}, we get 
 \begin{equation}\label{newrevise_565}
 	\mathbb{E}[\mathbf{Y}^2] \lesssim M^{2-d} N^4.
 \end{equation}

 We are now ready to apply the Paley-Zygmund inequality in both regimes $3\le d\le 5$ and $d\ge 7$. Since $\{\mathbf{Y}>0\}\subset \widehat{\mathsf{C}}_e'$, we have 
\begin{equation}\label{new548}
	\begin{split}
		\mathbb{P}\big(\widehat{\mathsf{C}}_e' \big) \ge \mathbb{P}\big(\mathbf{Y}>0 \big) \gtrsim  \frac{\big( \mathbb{E}[ \mathbf{Y} ]\big)^2 }{\mathbb{E}\big[ \mathbf{Y}^2\big]}  \overset{(\ref{556}),(\ref{newrevise_563}),(\ref{newrevise_565})}{\gtrsim} M^{-[(\frac{d}{2}+1)\boxdot (d-2)]}N^{-[(\frac{d}{2}-1)\boxdot 2]}. 
	\end{split}
\end{equation}
Consequently, when $d\ge 7$, since $A\subset \widetilde{B}(cN)$, one may choose a sufficiently small $c>0$ such that 
\begin{equation}\label{newnew548}
	\tfrac{C''[\mathrm{diam}(A)]^{d-4}}{N^2M^{2d-6}} \le \tfrac{1}{2} \cdot  \mathbb{P}\big(\widehat{\mathsf{C}}_e' \big). 
\end{equation}
By plugging (\ref{new548}) and (\ref{newnew548}) into (\ref{new546}), we have 
 \begin{equation}
 \begin{split}
 	 	\mathbb{P}\big(\widehat{\mathsf{C}}_e \big) \gtrsim &  M^{-[(3-\frac{d}{2})\boxdot 0] }N^{-[(\frac{d}{2}-1)\boxdot 2]}\cdot \mathbb{P}\big(A\xleftrightarrow{(D)} x_{ \spadesuit} \big)\\
 	 	\overset{(\text{Lemma}\ \ref{lemma_point_to_set})}{\asymp } & M^{-[(\frac{d}{2}+1)\boxdot (d-2)] }N^{(\frac{d}{2}-1)\boxdot (d-4)}\cdot \mathbb{P}\big(A\xleftrightarrow{(D)} x_\star \big).
 \end{split}	
 \end{equation}
   This together with (\ref{new545}) implies the estimate (\ref{535}).


%
%
%
%
%
%
%
%

  \textbf{Proof of (\ref{536}).} We denote by $\Psi_{e,e'}$ the collection of all quadruples $(x_1,y_1,x_2,y_2)$ such that $x_1\neq y_1$ are the endpoints of $e$ and $x_2\neq y_2$ are the endpoints of $e'$. On $\widehat{\mathsf{L}}_e\cap \widehat{\mathsf{L}}_{e'}$, both $\mathfrak{L}_e$ and $\mathfrak{L}_{e'}$ are non-zero. In addition, after removing all loops in $\mathfrak{L}_e+\mathfrak{L}_{e'}$, there exists $\psi=(x_1,y_1,x_2,y_2)\in \Psi_{e,e'}$ such that the event  
    \begin{equation}
  	\begin{split}
  		\overline{\mathsf{C}}
  		_{\psi}:= \big\{ x_1 \xleftrightarrow{(D)} A \Vert x_2\xleftrightarrow{(D)} \partial B(N)  \Vert y_1\xleftrightarrow{(D)} y_2  \big\}\cap \big\{ \{x_2,y_2\} \xleftrightarrow{(D)} B(c_\dagger M) \big\}^c
  	\end{split}
  \end{equation}
   occurs. Therefore, applying the union bound and following the same reasoning as in (\ref{5.4}), one obtains
 \begin{equation}\label{use566}
 	\begin{split}
 		\mathbb{P}\big( \widehat{\mathsf{L}}_e\cap \widehat{\mathsf{L}}_{e'}\big) \lesssim \sum\nolimits_{ \psi\in \Psi_{e,e'}} \mathbb{P}\big( \overline{\mathsf{C}}
  		_{\psi} \big)\cdot \frac{\mathbb{P}(\mathfrak{L}_e\neq 0, \mathfrak{L}_{e'}\neq 0)}{\mathbb{P}(\mathfrak{L}_e=  \mathfrak{L}_{e'}= 0)} \asymp  \sum\nolimits_{ \psi\in \Psi_{e,e'}} \mathbb{P}\big( \overline{\mathsf{C}}
  		_{\psi} \big). 
 	\end{split}
 \end{equation} 
 Subsequently, we estimate the probability of $\overline{\mathsf{C}}
  		_{\psi}$ for $3\le d\le 5$ and $d\ge 7$ separately.


 \textbf{When $3\le d \le 5$.} For any subset $D'\subset \widetilde{\mathbb{Z}}^d$ satisfying $D\subset D'$ and $(D'\setminus D)\cap B(c_\dagger M)=\emptyset$, by Lemmas \ref{lemma_stability_bd}, \ref{lemma_213} and \ref{lemma_quasi}, we have (recalling that $x_\star\in \partial B(N)$)
 \begin{equation}\label{567}
 	\begin{split}
 			\mathbb{P}\big(  x_1 \xleftrightarrow{(D')} A \big) \overset{(\text{Lemma}\ \ref{lemma_quasi})}{\asymp} &\mathbb{P}\big( A \xleftrightarrow{(D)} \partial B(c_{ \spadesuit}M)  \big) \cdot  \mathbb{P}\big(  x_1\xleftrightarrow{(D'\setminus D)} \partial B(c_{ \spadesuit}M) \big) \\
 			\overset{(\text{Lemmas}\ \ref{lemma_stability_bd}\ \text{and}\  \ref{lemma_213})}{\lesssim }  & N^{d-2} \cdot  \mathbb{P}\big( A \xleftrightarrow{(D)} x_\star  \big) \cdot  \mathbb{P}\big(  x_1 \xleftrightarrow{(D')} x_{\spadesuit}   \big). 
 	 	\end{split}
 \end{equation}
By the restriction property and (\ref{567}), one has  
 \begin{equation}\label{good568}
 	\begin{split}
 		\mathbb{P}\big(\overline{\mathsf{C}}_{\psi} \big) =&  \mathbb{E}\Big[ \mathbbm{1}_{ \{ x_2\xleftrightarrow{(D)} \partial B(N)  \Vert y_1\xleftrightarrow{(D)} y_2 \}\cap   \{ \{x_2,y_2\} \xleftrightarrow{(D)} B(c_\dagger M)  \}^c}   \mathbb{P}\big(  x_1 \xleftrightarrow{(D\cup \mathcal{C}_{\{x_2,y_2\}})} A \big)  \Big]  \\
 \overset{(\ref{567})}{	\lesssim  } & N^{d-2} \cdot  \mathbb{P}\big( A \xleftrightarrow{(D)} x_\star  \big) \cdot    \mathbb{P}\big( \overline{\mathsf{C}}
  		_{\psi}'\big),
 	\end{split}
 \end{equation}   
 where the event $\overline{\mathsf{C}}_{\psi}'$ is defined by $\{ x_2\xleftrightarrow{(D)} \partial B(N)\Vert    x_1 \xleftrightarrow{(D)} x_{\spadesuit} \Vert y_1\xleftrightarrow{(D)} y_2 \}$. Moreover, by applying Corollary \ref{newcoro_29} and Lemma \ref{lemma_relation}, we have 
 \begin{equation}\label{good569}
 	\begin{split}
 		\mathbb{P}\big( \overline{\mathsf{C}}_{\psi}'\big)  =&  \mathbb{E}\Big[  \mathbbm{1}_{x_1 \xleftrightarrow{(D)} x_{\spadesuit} \Vert y_1\xleftrightarrow{(D)} y_2} \cdot \mathbb{P}\big(x_2\xleftrightarrow{(D\cup \mathcal{C}_{\{x_1,y_1\}})} \partial B(N)  \big) \Big]\\
 		\overset{ (\text{Lemma}\ \ref{lemma_relation})}{\lesssim } & N^{-\frac{d}{2}} \sum\nolimits_{v\in \partial \mathcal{B}(d^{-1}N)}\mathbb{E}\Big[  \mathbbm{1}_{x_1 \xleftrightarrow{(D)} x_{\spadesuit} \Vert y_1\xleftrightarrow{(D)} y_2} \cdot \mathbb{P}\big(x_2\xleftrightarrow{(D\cup \mathcal{C}_{\{x_1,y_1\}})} v \big) \Big] \\
 			\overset{(\text{Corollary}\ \ref{newcoro_29} )}{\lesssim } &M^{-1}  N^{-\frac{d}{2}+1 } \sum\nolimits_{w\in \partial \mathcal{B}(M)} \mathbb{P}\big(  \overline{\mathsf{C}}_{\psi,w}''  \big), 
 	\end{split}
 \end{equation}
 where we denote $\overline{\mathsf{C}}_{\psi,w}'':=\big\{x_1 \xleftrightarrow{(D)} x_{\spadesuit} \Vert x_2\xleftrightarrow{(D)} w   \Vert y_1\xleftrightarrow{(D)} y_2\big\}$.

 When $r:=|y_1-y_2|\lesssim 1$, it follows from (\ref{four_point_use}) that 
 \begin{equation}
 	\mathbb{P}\big(  \overline{\mathsf{C}}_{\psi,w}''  \big) \le \mathbb{P}\big(x_1 \xleftrightarrow{(D)} x_{\spadesuit} \Vert x_2\xleftrightarrow{(D)} w     \big)  \lesssim M^{-\frac{3d}{2}+1}. 
 \end{equation}
 Next, we estimate the probability of $\overline{\mathsf{C}}_{\psi,w}''$ under the assumption that $r$ is sufficiently large. Recall the events $\mathsf{A}_\diamond^{\cdot}(\cdot)$ for $\diamond \in \{\mathrm{I},\mathrm{II},\mathrm{III} \}$ at the beginning of Section \ref{section_proof_two_branch}. For any $D'\subset \widetilde{\mathbb{Z}}^d$, Corollary \ref{lemma_cut_two_points} implies that  
\begin{equation}\label{addto570}
\begin{split}
		& \mathbb{P}\big( y_1\xleftrightarrow{(D')} y_2\big) \\
		 \lesssim &r^{-d}   \sum_{z_1\in \partial \mathcal{B}_{x_1}(\frac{r}{\Cref{const_boundary1}}),z_2 \in \partial \mathcal{B}_{x_2}(\frac{r}{\Cref{const_boundary1}})}   \mathbb{P}\big(\mathsf{A}_{\mathrm{I}}^{D'}(x_1,z_1; r), \mathsf{A}_{\mathrm{I}}^{D'}(x_2,z_2; r) \big).  
\end{split}
\end{equation}
 Meanwhile, it follows from Corollary \ref{coro28} (with $R_1=0$) that 
 \begin{equation}\label{addto571}
 	\begin{split}
 			\mathbb{P}\big( &		x_1\xleftrightarrow{(D')}x_{\spadesuit} \big) \lesssim  r^{-d}\sum\nolimits_{z_3\in  \partial \mathcal{B}_{x_1}(\frac{2r}{\Cref{const_boundary1}}) ,z_4 \in \partial \mathcal{B}_{x_1}(\Cref{const_boundary1}r) } \\
 			&\ \ \ \ \ \ \ \ \ \ \ \ \ \ \ \ \ \ \ \ \ \ \ \ \  \mathbb{P}\big( \mathsf{A}_{\mathrm{I}}^{D'}(x_1,z_3; r), \mathsf{A}_{\mathrm{III}}^{D'}(x_1,x_{\spadesuit}, z_4 ;  100\Cref{const_boundary1}r ) \big), 
 	\end{split}
 \end{equation}
  \begin{equation}\label{addto572}
 	\begin{split}
 			\mathbb{P}\big( 	&	x_2\xleftrightarrow{(D')} w \big)  \lesssim r^{-d}  \sum\nolimits_{z_5\in  \partial \mathcal{B}_{x_2}(\frac{2r}{\Cref{const_boundary1}})  ,z_6 \in \partial \mathcal{B}_{x_2}(2\Cref{const_boundary1}r)  } \\
 			&\ \ \ \ \ \ \ \ \ \ \ \ \ \ \ \ \ \ \ \ \ \ \ \ \mathbb{P}\big( \mathsf{A}_{\mathrm{I}}^{D'}(x_2,z_5; r) , \mathsf{A}_{\mathrm{III}}^{D'}(x_1,w, z_6 ;  100\Cref{const_boundary1}r ) \big).  
 	\end{split}
 \end{equation}
 We write $\bar{z}:= (z_1,z_2,z_3,z_4,z_5,z_6)$, and let $\overline{\mathsf{A}}^{D'}_{\bar{z}}$ denote the intersection of all six events on the right-hand sides of (\ref{addto570}), (\ref{addto571}) and (\ref{addto572}). Combining these three inequalities with the restriction property, we have  
 \begin{equation}
 	\begin{split}
 		\mathbb{P}\big( \overline{\mathsf{C}}_{\psi,w}'' \big) \lesssim r^{-3d}\sum\nolimits_{\bar{z}:=(z_1,z_2,z_3,z_4,z_5,z_6)\in \Upsilon  } \mathbb{P}\big( \overline{\mathsf{A}}_{\bar{z}}^D \big), 
 	\end{split}
 \end{equation}
where the set $\Upsilon$ is defined by 
\begin{equation}
	\begin{split}
		\Upsilon:= & \partial \mathcal{B}_{y_1}(\tfrac{r}{\Cref{const_boundary1}})\times \partial \mathcal{B}_{y_2}(\tfrac{r}{\Cref{const_boundary1}}) \times\partial \mathcal{B}_{x_1}(\tfrac{2r}{\Cref{const_boundary1}})\\
		&\times \partial \mathcal{B}_{x_1}(\Cref{const_boundary1}r) \times  \partial \mathcal{B}_{x_2}(\tfrac{2r}{\Cref{const_boundary1}}) \times \partial \mathcal{B}_{x_2}(2\Cref{const_boundary1}r).  
	\end{split}
\end{equation}

 For each $\bar{z}\in \Upsilon$, on $\overline{\mathsf{A}}_{\bar{z}}^D$, the following three independent events occur: 
 \begin{equation}
 	\begin{split}
 		 \mathsf{C}_1:=   \big\{ y_1\xleftrightarrow{(  \partial B_{x_1}(\frac{r}{100}))} z_1 \Vert x_1  \xleftrightarrow{(  \partial B_{x_1}(\frac{r}{100}))} z_3  \big\}, 
 	\end{split}
 \end{equation}
  \begin{equation}
 	\begin{split}
 		 \mathsf{C}_2:=   \big\{   y_2\xleftrightarrow{(  \partial B_{x_2}(\frac{r}{100}))} z_2  \Vert x_2 \xleftrightarrow{(  \partial B_{x_2}(\frac{r}{100}))} z_5   \big\}, 
 	\end{split}
 \end{equation}
  \begin{equation}
 	\begin{split}
 		 \mathsf{C}_3:=   \big\{  x_{\spadesuit}  \xleftrightarrow{(D\cup \partial B_{x_1}(100r ) )} z_4   \Vert w  \xleftrightarrow{(D\cup \partial B_{x_1}(100r ))} z_6 \big\}.  
 	\end{split}
 \end{equation}
 Combined with (\ref{four_point_use}) and $|\Upsilon|\asymp r^{6d-6}$, it yields that 
 \begin{equation}\label{good578}
 	\begin{split}
 		\mathbb{P}\big( \overline{\mathsf{C}}_{\psi,z}'' \big) \lesssim r^{-\frac{d}{2}-1}M^{-\frac{3d}{2}+1}.  
 	\end{split}
 \end{equation}
Putting (\ref{good568}), (\ref{good569}) and (\ref{good578}) together, we obtain 
\begin{equation}\label{use579}
	\mathbb{P}\big(\overline{\mathsf{C}}_{\psi} \big)   \lesssim (rM)^{-\frac{d}{2}-1}N^{\frac{d}{2}-1}\cdot \mathbb{P}\big(A\xleftrightarrow{(D)} x_\star \big). 
\end{equation}

  \textbf{When $d\ge 7$.} Using the BKR inequality, we have 
  \begin{equation}\label{late554}
  	\begin{split}
  		\mathbb{P}\big(\overline{\mathsf{C}}_{\psi} \big) \le & \mathbb{P}\big(x_1 \xleftrightarrow{(D)} A \big) \cdot \mathbb{P}\big(x_2\xleftrightarrow{(D)} \partial B(N)  \big) \cdot \mathbb{P}\big(y_1\xleftrightarrow{(D)} y_2   \big).
  	\end{split}
  \end{equation}
 Note that Lemma \ref{lemma_213} implies  
  \begin{equation}\label{late555}
  		\mathbb{P}\big(x_1 \xleftrightarrow{(D)} A \big) \lesssim M^{2-d}N^{d-2} \cdot \mathbb{P}\big(A\xleftrightarrow{(D)} x_\star \big).
  \end{equation}  
 Meanwhile, using (\ref{one_arm_high}) and (\ref{two-point1}) respectively, one has 
 \begin{equation}\label{late556}
 	\mathbb{P}\big(x_2\xleftrightarrow{(D)} \partial B(N)  \big) \lesssim N^{-2}\ \ \text{and}\ \ \mathbb{P}\big(y_1 \xleftrightarrow{(D)} y_2    \big) \lesssim r^{2-d}. 
 \end{equation}
 Plugging (\ref{late554}) and (\ref{late555}) into (\ref{late556}), we get 
 \begin{equation}\label{use583}
 		\mathbb{P}\big(\overline{\mathsf{C}}_{\psi} \big)   \lesssim (rM)^{2-d} N^{d-4}\cdot \mathbb{P}\big(A\xleftrightarrow{(D)} x_\star \big). 
 \end{equation}
  Combining (\ref{use566}), $|\Psi_{e,e'}|\asymp 1$, (\ref{use579}) and (\ref{use583}), we obtain (\ref{536}).

To sum up, we have completed the proof of Lemma \ref{lemma53}.   \qed

  \section{Typical number of pivotal edges}\label{section_pivotal_dege}

 This section is devoted to proving Theorem \ref{thm_1.5}. We denote by $\mathbf{Piv}_N$ the analogue of $\mathbf{Piv}_N^+$ obtained by replacing $\widetilde{E}^{\ge 0}$ with $\cup \widetilde{\mathcal{L}}_{1/2}$ (i.e., replacing connectivity via GFF level sets to connectivity via loop clusters). Applying the isomorphism theorem, we have 
 \begin{equation}
 	\mathbb{P}\big( \bm{0}\xleftrightarrow{} \partial B(N) \big)=  2  \mathbb{P}\big( \bm{0}\xleftrightarrow{\ge 0} \partial B(N) \big), 
 \end{equation}
 \begin{equation}
 	\mathbb{P}\big(  | \mathbf{Piv}_N   |  \in F \big) = 2\mathbb{P}\big( | \mathbf{Piv}_N^+   |  \in F \big) 
 \end{equation}
 for any non-empty subset $F\subset \mathbb{N}$. Thus, it suffices to show that 
 \begin{equation}
 		\mathbb{P}\big(  | \mathbf{Piv}_N   | \in \big[\cref{const_GFF_pivotal1} N^{(\frac{d}{2}-1)\boxdot 2} ,\Cref{const_GFF_pivotal4} N^{(\frac{d}{2}-1)\boxdot 2} \big] \mid \bm{0}\xleftrightarrow{ } \partial B(N) \big)\ge 1-\epsilon.
 \end{equation}
 To see this, for any $e\in \mathbb{L}^d$, when $e\in \widetilde{\mathbf{Piv}}_N$ occurs, removing all loops in $\mathfrak{L}_e$ from $\widetilde{\mathcal{L}}_{1/2}$ violates the event $\{\bm{0}\xleftrightarrow{} \partial B(N)\}$ (where $\widetilde{\mathbf{Piv}}_N$ and $\mathfrak{L}_e$ are defined in (\ref{close_def_Piv_N}) and (\ref{51}) respectively). As a result, removing the edge $e$ also has the same effect. Thus, one has $\widetilde{\mathbf{Piv}}_N \subset \mathbf{Piv}_N$. Combined with (\ref{54}), it implies that for $\cref{const_GFF_pivotal1}:= \cref{const_pivotal_1}$, 
  \begin{equation}\label{64}
  \begin{split}
  		&  \mathbb{P}\big(  | \mathbf{Piv}_N   | \le \cref{const_GFF_pivotal1} N^{(\frac{d}{2}-1)\boxdot 2}   \mid \bm{0}\xleftrightarrow{ } \partial B(N) \big) \\
  		 \le & \mathbb{P}\big(  |\widetilde{\mathbf{Piv}}_N    | \le \cref{const_GFF_pivotal1} N^{(\frac{d}{2}-1)\boxdot 2}   \mid \bm{0}\xleftrightarrow{ } \partial B(N) \big) \le \tfrac{1}{2}\epsilon. 
  \end{split}
 \end{equation}
 As in the proof of Theorem \ref{thm_pivotal_loop}, we need to bound the first moment of $|\mathbf{Piv}_N |$:

 \begin{lemma}\label{lemma61}
 	 	For any $d\ge 3$ with $d\neq 6$, and any $N\ge 1$, 
 	 \begin{equation}
 	\mathbb{E}\big[ \big|\mathbf{Piv}_N \big|  \mid \bm{0}\xleftrightarrow{ } \partial B(N) \big]  \lesssim   N^{(\frac{d}{2}-1)\boxdot 2}. 
 \end{equation}
 \end{lemma}


  By Markov's inequality and Lemma \ref{lemma61}, there exists $\Cref{const_GFF_pivotal4}(d,\epsilon)>0$ such that 
  \begin{equation}\label{67}
 	\mathbb{P}\big(  | \mathbf{Piv}_N   | \ge \Cref{const_GFF_pivotal4} N^{(\frac{d}{2}-1)\boxdot 2}   \mid \bm{0}\xleftrightarrow{ } \partial B(N) \big) \le \tfrac{1}{2}\epsilon.
 \end{equation}
 Combining (\ref{64}) and (\ref{67}), we obtain Theorem \ref{thm_pivotal_loop} (assuming Lemma \ref{lemma61}).           \qed



 \vspace{0.2cm}

 The rest of this section is devoted to the proof of Lemma 6.1. We denote by $\mathsf{L}_e$ the event that the edge $e$ is pivotal for $\{\bm{0}\xleftrightarrow{} \partial B(N)\}$. Like in the proof of Lemma \ref{lemma_first_moment}, it is sufficient to prove that for any $e=\{x_1,x_2\} \in \mathbb{L}^d$ with $I_e\subset \widetilde{B}(N)$, 
	\begin{equation} \label{68}
 		  \mathbb{P}\big(\mathsf{L}_e  \big) \lesssim \mathbb{I}_e:= \left\{
\begin{aligned}
&(|x_1|+1)^{-[(\frac{d}{2}+1)\boxdot (d-2)]}N^{-[(\frac{d}{2}-1)\boxdot 2]}      &  x_1\in B(\tfrac{N}{2}); \\ 
&\big(\mathrm{dist}(x_1,\partial B(N) )+1\big)^{-2}N^{2-d}    &   x_1\notin B(\tfrac{N}{2}).
\end{aligned}
\right.
 	\end{equation}
 	In fact, when $|x_1| \land  \mathrm{dist}(x_1,\partial B(N) )\lesssim 1$, the bound (\ref{68}) is straightforward. To be precise, on the event $\mathsf{L}_e$, the set $\{x_1,x_2\}$ are connected to both $\bm{0}$ and $\partial B(N)$ by $\cup \widetilde{\mathcal{L}}_{1/2}$. Therefore, when $|x_1|\lesssim 1$, it follows from (\ref{one_arm_low}) and (\ref{one_arm_high}) that 
 	\begin{equation}
 		  \mathbb{P}\big(\mathsf{L}_e  \big) \le \sum\nolimits_{j\in \{1,2\}} \mathbb{P}\big(x_j \xleftrightarrow{} \partial B(N)  \big) \asymp  N^{-[(\frac{d}{2}-1)\boxdot 2]} \asymp \mathbb{I}_e. 
 	\end{equation}
   Meanwhile, when $\mathrm{dist}(x_1,\partial B(N) )\lesssim 1$, by (\ref{two-point1}) one has 
   \begin{equation}
   	  \mathbb{P}\big(\mathsf{L}_e  \big) \le \sum\nolimits_{j\in \{1,2\}} \mathbb{P}\big(x_j \xleftrightarrow{} \bm{0}  \big) \asymp N^{2-d}  \asymp \mathbb{I}_e. 
   \end{equation}
 	Thus, it remains to establish (\ref{68}) under the assumption that $R_1:=|x_1|$ and $R_2:=\mathrm{dist}(x_1,\partial B(N) )$ are both sufficiently large.

We arbitrarily fix an edge $e$. For each loop $\widetilde{\ell}\in \widetilde{\mathcal{L}}_{1/2}$, it can be decomposed into the following parts: for each $j\in \{1,2\}$, 
\begin{enumerate}

	\item[(i)] the excursions that start from $x_j$, end at $x_j$, and do not touch $I_e$;

	\item[(ii)] the excursions that start from $x_j$, end at $x_{3-j}$, and do not touch $I_e$;

	\item[(iii)] the excursions that start from $x_j$, end at $\{x_1,x_2\}$, and are contained in $I_e$.

\end{enumerate} 
For $\diamond\in \{\mathrm{i},\mathrm{ii},\mathrm{iii} \}$, we define $\mathfrak{E}^{\diamond}_j$ as the union of ranges of excursions in Part ($\diamond$) for all loops $\widetilde{\ell}\in \widetilde{\mathcal{L}}_{1/2}$. In addition, let $\mathcal{C}_1$ (resp. $\mathcal{C}_2$) denote the cluster of $\cup \widetilde{\mathcal{L}}_{1/2}^{\{x_1,x_2\}}$ intersecting $\bm{0}$ (resp. $\partial B(N)$). Observe that on the event $\mathsf{L}_e$, one has $\mathfrak{E}^{\mathrm{ii}}_1=\mathfrak{E}^{\mathrm{ii}}_2=\emptyset$; otherwise, removing the edge $e$ will not change the connectivity between $x_1$ and $x_2$ in the cluster $\mathcal{C}_{\bm{0}}$, and thus $e$ is not pivotal. Moreover, the event $\mathsf{L}_e$ is independent of the excursions in $\mathfrak{E}^{\mathrm{iii}}_1$ and $\mathfrak{E}^{\mathrm{iii}}_2$. Therefore, when $\mathsf{L}_e$ occurs, there exists $j\in \{1,2\}$ such that $\mathfrak{E}^{\mathrm{i}}_j$ intersects $\mathcal{C}_1$, $\mathfrak{E}^{\mathrm{i}}_{3-j}$ intersects $\mathcal{C}_2$, and that $\mathfrak{E}^{\mathrm{i}}_j\cup \mathcal{C}_1$ is disjoint from $\mathfrak{E}^{\mathrm{i}}_{3-j}\cup \mathcal{C}_2$ (we denote this event by $\mathsf{F}_j$). Consequently, in order to verify (\ref{68}), it suffices to show that 
 \begin{equation}\label{69}
 	\mathbb{P}\big(\mathsf{F}_j \big)\lesssim \mathbb{I}_e, \ \ \forall  j\in \{1,2\}. 
 \end{equation}
 Without loss of generality, we only consider the case $j=1$, and write $\mathsf{F}_1$ as $\mathsf{F}$.





 
 We first present the proof of (\ref{69}) for $d\ge 7$, since it is relatively simpler than that of the low-dimensional case $3\le d\le 5$. On the event $\mathsf{F}$, there exist $z_1,z_2\in B(N)$ such that $\widetilde{B}_{z_j}(1) \cap  \mathfrak{E}^{\mathrm{i}}_j\cap \mathcal{C}_j \neq \emptyset$ holds for $j\in \{1,2\}$; moreover, $\mathcal{C}_1$ is disjoint from $\mathcal{C}_2$. In addition, note that $\mathfrak{E}^{\mathrm{i}}_1$ and $\mathfrak{E}^{\mathrm{i}}_2$ are independent of $\mathcal{C}_1$ and $\mathcal{C}_2$. Thus, by the union bound, we have 
 \begin{equation}\label{add610}
 	\begin{split}
 		\mathbb{P}\big( \mathsf{F}\big) \lesssim & \sum\nolimits_{z_1,z_2\in B(N)} \mathbb{P}\big(\widetilde{B}_{z_j}(1)\cap \mathfrak{E}^{\mathrm{i}}_j \neq \emptyset, \forall j\in \{1,2\}   \big)  \\
 		&\cdot \mathbb{P}\big(\widetilde{B}_{z_1}(1)  \xleftrightarrow{(\{x,y\})}   \bm{0}  \Vert  \widetilde{B}_{z_2}(1)  \xleftrightarrow{(\{x,y\})}   \partial B(N)  \big). 
 	\end{split}
 \end{equation}
 By the strong Markov property of Brownian excursion measure, we have  \begin{equation}\label{add611}
 \begin{split}
 	&\mathbb{P}\big(\widetilde{B}_{z_j}(1)\cap \mathfrak{E}^{\mathrm{i}}_j \neq \emptyset, \forall j\in \{1,2\}   \big) \\
 	\lesssim &   \prod\nolimits_{j\in \{1,2\}} \max_{w_j\in \widetilde{\mathbb{Z}}^d: |w_j-x_j|=1} \widetilde{\mathbb{P}}_{w_j}\big( \tau_{\widetilde{B}_{z_j}(1)} < \tau_{\{x_1,x_2\}}= \tau_{x_j} \big)    \\
 	 \lesssim &\prod\nolimits_{j\in \{1,2\}}  (|x_j-z_j|+1)^{4-2d}. 
 \end{split} 	 	 
 \end{equation}
Meanwhile, by the BKR inequality, (\ref{two-point1}) and (\ref{one_arm_high}), we have 
\begin{equation}\label{add612}
	\begin{split}
	&	\mathbb{P}\big(\widetilde{B}_{z_1}(1)  \xleftrightarrow{(\{x,y\})}   \bm{0}  \Vert  \widetilde{B}_{z_2}(1)  \xleftrightarrow{(\{x,y\})}   \partial B(N)  \big) \\
		\overset{(\text{BKR})}{\lesssim } & \mathbb{P}\big(\widetilde{B}_{z_1}(1)  \xleftrightarrow{(\{x,y\})}   \bm{0}    \big)\cdot \mathbb{P}\big(   \widetilde{B}_{z_2}(1)  \xleftrightarrow{(\{x,y\})}   \partial B(N)  \big) \\
		\overset{(\ref{two-point1}),(\ref{one_arm_high})}{\lesssim} &( |z_1|+1)^{2-d} \big(\mathrm{dist}(z_2,\partial B(N))+1\big)^{-2}. 
	\end{split}
\end{equation}
Plugging (\ref{add611}) and (\ref{add612}) into (\ref{add610}), we get 
\begin{equation}\label{to613}
			\mathbb{P}\big( \mathsf{F}\big) \lesssim \mathbb{U}_1\cdot  \mathbb{U}_2,
\end{equation}
where $\mathbb{U}_1$ and $\mathbb{U}_2$ are defined by 
\begin{equation}
	\mathbb{U}_1:= \sum\nolimits_{z_1\in B(N)} (|x_1-z_1|+1)^{4-2d} ( |z_1|+1)^{2-d} ,
\end{equation}
\begin{equation}
	\mathbb{U}_2:= \sum\nolimits_{z_2\in B(N)}(|x_2-z_2|+1)^{4-2d} \big(\mathrm{dist}(z_2,\partial B(N))+1\big)^{-2}. 
\end{equation}
For $\mathbb{U}_1$, recall the basic fact that (see e.g., \cite[(4.9)]{cai2024high}) 
\begin{equation}\label{6-2d_2-d}
	\sum\nolimits_{z\in \mathbb{Z}^d} (|z-v_1|+1)^{6-2d} (|z-v_2|+1)^{2-d}\lesssim  (|v_1-v_2|+1)^{2-d}, \ \forall v_1,v_2\in \widetilde{\mathbb{Z}}^d. 
\end{equation}
This immediately implies that  \begin{equation}\label{616}
	\mathbb{U}_1\le   \sum\nolimits_{z\in B(N)} (|x_1-z|+1)^{6-2d} ( |z|+1)^{2-d} \lesssim  R_1^{2-d}. 
\end{equation} 
For $\mathbb{U}_2$, since $\mathrm{dist}(z_2,\partial B(N))\gtrsim R_2$ for $z\in B_{x_1}(\frac{R_2}{10})$, one has 
 \begin{equation}\label{to617}
 	\begin{split}
 		\mathbb{U}_2 \lesssim &   R_2^{-2}\sum\nolimits_{z_2\in B_{x_1}(\frac{R_2}{10})}(|x_2-z_2|+1)^{4-2d}  \\
 		&  +  \sum_{z_2\in B(N)\setminus B_{x_1}(\frac{R_{2}}{10})}(|x_2-z_2|+1)^{4-2d}\big(\mathrm{dist}(z_2,\partial B(N))+1\big)^{-2}\\
 	\overset{(\ref{computation_d-a})  }{ \lesssim } & R_2^{-2} +  \sum\nolimits_{\frac{R_{2}}{10}  \le k\le N} k^{d-1}\cdot k^{4-2d}\cdot (N-k+1)^{-2} \lesssim R_2^{-2}+R_2^{3-d}\asymp R_2^{-2}.  
 	\end{split}
 \end{equation}
 Inserting (\ref{616}) and (\ref{to617}) into (\ref{to613}), we derive the bound (\ref{69}) for $d\ge 7$.

  We now turn to the case when $3\le d \le 5$. Take a large constant $\Cl\label{const_last_section}>0$, and then for each $j \in \{1,2\}$, we define   
 \begin{equation}
 	k_{j}:= \min\{k\ge 1: r_{k+3}^{(j)} \ge R_j\},
 \end{equation}
 where $r_{i}^{(j)}:=(\Cref{const_last_section})^{i+\frac{j}{2}}$ for $i\in \mathbb{N}$. We may assume that $k_1\land k_2\ge 10$, as $R_1$ and $R_2$ are both sufficiently large. Moreover, for any $1\le k  \le k_{j}-1$, we define   
 \begin{equation}\label{verynew621}
 	\mathsf{G}_{k}^{(j)}:= \big\{  \mathcal{C}_j \cap B_{x_1}(r_{k+1}^{(j)}) \neq \emptyset, \mathcal{C}_j \cap B_{x_1}(r_{k}^{(j)}) = \emptyset \big\}.  
 \end{equation}
For completeness, we also define 
 \begin{equation}\label{verynew622}
 	\mathsf{G}_{0}^{(j)}:= \big\{\mathcal{C}_j \cap B_{x_1}(r_1^{(j)})  \neq \emptyset \big\} \ \ \text{and}\ \  \mathsf{G}_{k_j}^{(j)}:= \big\{ \mathcal{C}_j  \cap B_{x_1}(r^{(j)}_{ k_{j}})  =\emptyset   \big\}. 
 \end{equation}
 For any $0 \le m_1\le k_1$ and $0 \le m_2\le k_2$, we define the event  
 \begin{equation}
 	\mathsf{F}_{m_1,m_2}: = \mathsf{F}\cap \mathsf{G}_{m_1}^{(1)}\cap \mathsf{G}_{m_2}^{(2)}. 
 \end{equation}
 Since the union of $\mathsf{G}_{m}^{(j)}$ for $0\le m\le k_j$ is the sure event, one has 
 \begin{equation}
 	\mathsf{F}= \cup_{ 0 \le m_1\le k_1,0 \le m_2\le k_2 } \mathsf{F}_{m_1,m_2}. 
 \end{equation}
  Thus, to obtain (\ref{68}), it is sufficient to establish that for any $\delta\in (0,\frac{1}{2})$, 
 	\begin{equation}\label{615}
 		  \mathbb{P}\big(\mathsf{F}_{m_1,m_2}  \big) \lesssim \mathbb{J}_{m_1,m_2}:= \left\{
\begin{aligned}
&(r_{m_1\vee m_2}^{(1)})^{-\frac{d}{2}+1+\delta}R_1^{-\frac{d}{2}-1}N^{-\frac{d}{2}+1 }      &  x_1\in B(\tfrac{N}{2}); \\ 
&(r_{m_1\vee m_2}^{(1)})^{-\frac{d}{2}+1+\delta} R_2^{-2}N^{2-d}    &   x_1\notin B(\tfrac{N}{2}).
\end{aligned}
\right.
 	\end{equation}
 We only provide its proof for the case $m_1\le m_2$, since that of $m_1\ge m_2$ is  similar.

  \textbf{Case 1: When $m_1=m_2=0$.} By definition of $\mathsf{F}_{0,0}$, we have 
  \begin{equation}\label{verynew626}
  \begin{split}
  		\mathsf{F}_{0,0} \subset & \big\{ B_{x_1}(r_1^{(1)}) \xleftrightarrow{\{x_1,x_2\}} \bm{0} \Vert B_{x_1}(r_1^{(2)}) \xleftrightarrow{\{x_1,x_2\}} \partial B(N)   \big\}\\
   \subset & 	\bigcup_{y_1\in \partial B_{x_1}(r_1^{(1)}), y_2\in B_{x_1}(r_1^{(2)})}   \big\{y_1 \xleftrightarrow{\{x_1,x_2\}} \bm{0} \Vert y_2\xleftrightarrow{\{x_1,x_2\}} \partial B_{y_2}(R_2)   \big\}. 
  \end{split}
  \end{equation}
  Moreover, by the restriction property and Lemma \ref{lemma_relation}, we have 
  \begin{equation}\label{verynew627}
  \begin{split}
  	  	& \mathbb{P}\big(y_1 \xleftrightarrow{\{x_1,x_2\}} \bm{0} \Vert y_2\xleftrightarrow{\{x_1,x_2\}} \partial B_{y_2}(R_2)  \big) \\
  	  	 \lesssim  & R_2^{-\frac{d}{2}}  \sum\nolimits_{z \in \partial \mathcal{B}_{y_2}(d^{-1}R_2)} 	\mathbb{P}\big(y_1 \xleftrightarrow{\{x_1,x_2\}} \bm{0} \Vert y_2\xleftrightarrow{\{x_1,x_2\}}  z \big).
  \end{split}
  \end{equation}
In addition, it has been verified in (\ref{late516}) and (\ref{late520}) that
  \begin{equation}\label{verynew628}
  	\mathbb{P}\big(y_1 \xleftrightarrow{\{x_1,x_2\}} \bm{0} \Vert y_2\xleftrightarrow{\{x_1,x_2\}}  z \big) \lesssim (R_1R_2)^{2-d}\cdot \big(R_1\land R_2 \big)^{\frac{d}{2}-3}.
  \end{equation}
  Combining (\ref{verynew626}), (\ref{verynew627}) and (\ref{verynew628}), we derive (\ref{615}) for $m_1=m_2=0$: 
  \begin{equation}\label{verynew_629}
  	\begin{split}
  		  \mathbb{P}\big(\mathsf{F}_{0,0}  \big) \lesssim R_1^{2-d}R_2^{-\frac{d}{2}+1}(R_1\land R_2)^{\frac{d}{2}-3} \asymp  \mathbb{J}_{0,0}. 
  	\end{split}
  \end{equation}

 \textbf{Case 2: When $0\le m_1\le k_1-1$ and $m_2\ge m_1+1$.} For  $m_1,m_2\ge 1$, when $\mathsf{F}_{m_1,m_2}$ occurs, we know that $\mathfrak{E}^{\mathrm{i}}_1$ must intersect $\partial B_{x_{1}}(r_{m_1}^{(1)})$, and that $\mathfrak{E}^{\mathrm{i}}_2$ must intersect $\partial B_{x_{1}}(r_{m_2}^{(2)})$ and is disjoint from both $\mathfrak{E}^{\mathrm{i}}_1$ and $\mathcal{C}^{\mathrm{in}}$. Similar to \cite[Lemma 3.17]{inpreparation_twoarm}, the following lemma provides the order of the probability of $\{ \mathfrak{E}^{\mathrm{i}}_1 \cap \partial B_{x_{1}}(r_{m_1}^{(1)})\neq \emptyset\}$:

\begin{lemma}\label{lemma_pool_62}
	For any $d\ge 3$ and $R\ge 1$, 
	\begin{equation}\label{pool627}
		\mathbb{P}\big( \mathfrak{E}^{\mathrm{i}}_1\cap \partial B_{x_{1}}(R)\neq \emptyset \big)  \asymp R^{2-d}. 
	\end{equation}
\end{lemma}
 \begin{proof}
 	By the strong Markov property of Brownian excursion measure, one has 
 	\begin{equation}\label{pool628}
 		\begin{split}
 				 \mathbb{P}\big( \mathfrak{E}^{\mathrm{i}}_1\cap \partial B_{x_{1}}(R)\neq \emptyset \big)   
 				\asymp    \sum\nolimits_{z\in F} \widetilde{\mathbb{P}}_{z}\big( \tau_{ \partial B_{x_{1}}(R)}  < \tau_{\{x_1,x_2\}}= \tau_{x_1}<\infty  \big),
 		\end{split}
 	\end{equation}
 	where $F:=\{z\in \widetilde{\mathbb{Z}}^d: |z-x_1|=1,z\notin I_{\{x_1,x_2\}}\}$. Note that $|F|=2d-1$. Moreover, by the strong Markov property of Brownian motion, we have 
 	\begin{equation}\label{pool629}
 		\begin{split}
 		& \frac{	\widetilde{\mathbb{P}}_{z}\big( \tau_{ \partial B_{x_{1}}(R)}  < \tau_{\{x_1,x_2\}}= \tau_{x_1}<\infty  \big)}{ 	\widetilde{\mathbb{P}}_{z}\big( \tau_{ \partial B_{x_{1}}(R)}  < \tau_{\{x_1,x_2\}}   \big)}\\
 			\in &  \Big[ \min_{w\in \partial B_{x_{1}}(R) } 	\widetilde{\mathbb{P}}_{w}\big( \tau_{\{x_1,x_2\}}= \tau_{x_1}<\infty  \big) , \max_{w\in \partial B_{x_{1}}(R) } 	\widetilde{\mathbb{P}}_{w}\big( \tau_{\{x_1,x_2\}}= \tau_{x_1}<\infty  \big) \Big].
 		\end{split}
 	\end{equation}

 	For each  $z\in F$, it follows from the transience of $\widetilde{\mathbb{Z}}^d$ that  	\begin{equation}\label{pool630}
 			\widetilde{\mathbb{P}}_{z}\big( \tau_{ \partial B_{x_{1}}(R)}  < \tau_{\{x_1,x_2\}}   \big) \asymp 1. 
 	\end{equation}
 	Meanwhile, for any $w\in \partial B_{x_{1}}(R)$, by (\ref{cap1}) one has 
 	\begin{equation}\label{pool631}
 		\widetilde{\mathbb{P}}_{w}\big( \tau_{ \{x_1,x_2\}}  < \infty  \big)  \asymp R^{2-d}. 
 	\end{equation} 	
 	Moreover, given that a Brownian motion starting from $w\in \partial B_{x_{1}}(R)$ hits $\{x_1,x_2\}$, \cite[Proposition 6.5.4]{lawler2010random} implies that the conditional probability of reaching $x_1$ before $x_2$ is proportional to the harmonic measure of $\{x_1,x_2\}$ at $x_1$, which exactly equals $\frac{1}{2}$ by the symmetry. Combined with (\ref{pool631}), it yields that 
 	 	\begin{equation}\label{pool632}
 			\widetilde{\mathbb{P}}_{w}\big( \tau_{ \{x_1,x_2\}}  = \tau_{  x_1 }   < \infty \big)  \asymp R^{2-d}, \ \ \forall w\in \partial B_{x_{1}}(R). 
 	\end{equation}
 	Putting (\ref{pool629}), (\ref{pool630}) and (\ref{pool632}) together, we get 
 	\begin{equation}
 		\widetilde{\mathbb{P}}_{z}\big( \tau_{ \partial B_{x_{1}}(R)}  < \tau_{\{x_1,x_2\}}= \tau_{x_1}<\infty  \big) \asymp R^{2-d}, \ \ \forall z\in F. 
 	\end{equation}
 	This together with (\ref{pool628}) and $|F|=2d-1$ yields (\ref{pool627}). 
 \end{proof}


 Additionally, given that $\{ \mathfrak{E}^{\mathrm{i}}_1\cap \partial B_{x_{1}}(r_{m_1}^{(1)})\neq \emptyset\}$ occurs, we have the following stochastic domination for $\mathfrak{E}^{\mathrm{i}}_1$. For any path $\widetilde{\eta}$ on $\widetilde{\mathbb{Z}}^d$, let $\mathrm{proj}(\widetilde{\eta})$ denote its projection onto the lattice $\mathbb{Z}^d$. For any path $\overline{\eta}$ on $\mathbb{Z}^d$, we denote its length (i.e., the number of jumps in $\overline{\eta}$) by $\mathrm{len}(\overline{\eta})$. Recall that $\widetilde{\partial}\widetilde{B}_{x_1}(1)=\{z\in \mathbb{Z}^d: \{z,x_1\}\in \mathbb{L}^d \}$.




   \begin{lemma}\label{lemma_pure62}
For any $m_1\ge 1$, conditioned on the event $\big\{ \mathfrak{E}^{\mathrm{i}}_1\cap \partial B_{x_{1}}(r_{m_1}^{(1)})\neq \emptyset\big\}$, there exist two (random) paths $\widehat{\eta}_1$ and $\widehat{\eta}_2$ on $\widetilde{\mathbb{Z}}^d$ satisfying the following properties: 
  \begin{enumerate}

  	\item[(i)]  $\cup_{j \in \{1,2\}}\mathrm{ran}(\widehat{\eta}_i) \subset \mathfrak{E}^{\mathrm{i}}_1$;

  	\item[(ii)]   For each $j\in \{1,2\}$, $\widehat{\eta}_j$ starts from $\widetilde{\partial} \widetilde{B}_{x_1}(1)\setminus \{x_2\}$, stops when reaching $\partial B_{x_1}(r_{m_1-1}^{(1)})$, and does not hit $\{x_1,x_2\}$. We denote by $\Gamma$ the collection of all possible projections of such paths onto $\mathbb{Z}^d$.

  	\item[(iii)] For any $\overline{\eta}_1,\overline{\eta}_2\in \Gamma$, 
  	     \begin{equation}
  	\mathbb{P}\big( \mathrm{proj}(\widehat{\eta}_j) =\overline{\eta}_j, \forall j\in \{1,2\}  \big) \lesssim  (2d)^{-\mathrm{len}(\overline{\eta}_1)-\mathrm{len}(\overline{\eta}_2)}. 
  \end{equation}

  \end{enumerate}
  
%
%
%
%
%
%
%
%
%
%

 \end{lemma}
  \begin{proof}
  	Let $\eta$ be an excursion in $\mathfrak{E}^{\mathrm{i}}_1$ that reaches $\partial B_{x_1}(r_{m_1}^{(1)})$. We define $\widehat{\eta}_1$ as its subpath from the first time it hits $\widetilde{\partial}\widetilde{B}_{x_1}(1)$ until it reaches $\partial B_{x_1}(r_{m_1}^{(1)})$, and define $\widehat{\eta}_2$ as the reversed path of its subpath from the last time it stays outside $\widetilde{B}_{x_1}(r_{m_1-1}^{(1)})$ until the last time it visits $\widetilde{\partial}\widetilde{B}_{x_1}(1)$. By the construction, $\widehat{\eta}_1$ and $\widehat{\eta}_2$ satisfy both Properties (i) and (ii).

  	For Property (iii), we denote the starting and ending point of the path $\overline{\eta}_j$ by $y_j$ and $z_j$ respectively. Let $\overline{\eta}_2'$ be the path on $\mathbb{Z}^d$ that starts from $x_1$, jumps to $y_1$ in the first step, and then moves along $\overline{\eta}_2$. Note that $\mathrm{len}(\overline{\eta}_2')=\mathrm{len}(\overline{\eta}_2)+1$. By the strong Markov property of Brownian excursion measure and the last-exit decomposition, we have 
  	 \begin{equation}\label{pool635}
  	 	\begin{split}
  	 		& \mathbb{P}\big( \mathfrak{E}^{\mathrm{i}}_1\cap \partial B_{x_{1}}(r_{m_1}^{(1)})\neq \emptyset,  \mathrm{proj}(\widehat{\eta}_j) =\overline{\eta}_j, \forall j\in \{1,2\}   \big) \\
  	 		\lesssim &  \widetilde{\mathbb{P}}_{y_1}\big( \mathrm{proj}\big(\big\{\widetilde{S}_t :0\le t\le \tau_{\partial B_{x_1}(r_{m_1-1}^{(1)})} \big\}\big) = \overline{\eta}_1\big) \cdot  \max_{w\in \partial B_{x_{1}}(r_{m_1}^{(1)})}  G(w,z_2)  \\
  	 		& \cdot \widetilde{\mathbb{P}}_{z_2}\big( \mathrm{proj}\big(\big\{\widetilde{S}_{\tau_{\{x_1,x_2\}}-s}:0\le s\le \tau_{\{x_1,x_2\}} \big\}\big) = \overline{\eta}_2'\big).
  	 	\end{split}
  	 \end{equation}
  Recalling from Section \ref{section_brownian_motion} that the projection of $\widetilde{S}_\cdot\sim \widetilde{\mathbb{P}}_{y_1}$ is a simple random walk starting from $y_1$, we have 
  \begin{equation}\label{pool636}
  	 \widetilde{\mathbb{P}}_{y_1}\big( \mathrm{proj}\big(\big\{\widetilde{S}_t :0\le t\le \tau_{\partial B_{x_1}(r_{m_1-1}^{(1)})} \big\}\big) = \overline{\eta}_1\big) = (2d)^{-\mathrm{len}(\overline{\eta}_1)}, 
  \end{equation}  
   \begin{equation}\label{pool637}
  	\widetilde{\mathbb{P}}_{z_2}\big( \mathrm{proj}\big(\big\{\widetilde{S}_{\tau_{\{x_1,x_2\}}-s}:0\le s\le \tau_{\{x_1,x_2\}} \big\}\big) = \overline{\eta}_2'\big) = (2d)^{-\mathrm{len}(\overline{\eta}_2)+1}.  
  \end{equation} 
  	Plugging (\ref{pool636}), (\ref{pool637}) and $\max\limits_{w\in \partial B_{x_{1}}(r_{m_1}^{(1)})}G(w,z_2)\asymp (r_{m_1}^{(1)})^{2-d}$ into (\ref{pool635}), we get 
  	  	  \begin{equation}\label{newpool635}
  	  \begin{split}
  	  	  	 	 & \mathbb{P}\big( \mathfrak{E}^{\mathrm{i}}_1\cap \partial B_{x_{1}}(r_{m_1}^{(1)})\neq \emptyset,  \mathrm{proj}(\widehat{\eta}_j) =\overline{\eta}_j, \forall j\in \{1,2\}   \big) \\
  	  	  	 	 \lesssim &  (r_{m_1}^{(1)})^{2-d} \cdot (2d)^{-\mathrm{len}(\overline{\eta}_1)-\mathrm{len}(\overline{\eta}_2)}. 
  	  \end{split}
  	 \end{equation}
Combined with $\mathbb{P}(\mathfrak{E}^{\mathrm{i}}_1\cap \partial B_{x_{1}}(r_{m_1}^{(1)})\neq \emptyset)\asymp (r_{m_1}^{(1)})^{2-d}$ (by Lemma \ref{lemma_pool_62}), it completes the proof of this lemma. 
  	 \end{proof}

 Recall that $\mathsf{F}_{m_1,m_2}$ implies the following sub-events: 
 \begin{equation}
 	\mathsf{D}_{m_1}':= \left\{
\begin{aligned}
 &\big\{ \mathcal{C}_1 \cap B_{x_1}(r_{m_1+1}^{(1)}) \neq \emptyset , \mathfrak{E}^{\mathrm{i}}_1\cap \partial B_{x_{1}}(r_{m_1}^{(1)})\neq \emptyset \big\} &  \ \  &   m_1\ge 1; \\ 
&\big\{\mathcal{C}_1 \cap B_{x_1}(r_{1}^{(1)}) \neq \emptyset   \big\} &   &  m_1=0,
\end{aligned}
\right.
 \end{equation} 
  \begin{equation}
 		\mathsf{D}_{m_2}'' := \big\{  \mathfrak{E}^{\mathrm{i}}_2\cap \partial B_{x_{1}}(r_{m_2}^{(2)})\neq\emptyset, \mathfrak{E}^{\mathrm{i}}_2\cap (\mathfrak{E}^{\mathrm{i}}_1\cup \mathcal{C}_1) =\emptyset   \big\}, 
 \end{equation}
 \begin{equation}\label{verynew_D'''}
 	\mathsf{D}_{m_2}''':= \left\{
\begin{aligned}
 &\big\{ B_{x_{1}}(r_{m_2+1}^{(2)})\xleftrightarrow{(\{x_1,x_2\})} B_{x_{1}}(R_2) \big\} & \ \ \   &   m_2\le k_2-1; \\ 
&\text{the sure event} &   &  m_2=k_2.
\end{aligned}
\right.
 \end{equation}
 Given $\mathfrak{E}^{\mathrm{i}}_1$ and $\mathcal{C}^{\mathrm{in}}$, by applying Lemma \ref{lemma_for_new62} (where we use the case $\chi\le 10c^{-1}$ and drop the second term in the third line of (\ref{addto242})), we derive that the conditional probability of $\mathsf{D}_{m_2}''$ is upper-bounded by
  \begin{equation}\label{pure628}
 	\begin{split}
 		& C(r_{m_1}^{(1)}  r_{m_2}^{(2)})^{1-d} \mathbf{K}_{m_1}       \\
 		&\cdot \sum\nolimits_{z_1\in \partial \mathcal{B}_{x_1}(  r_{m_1+1}^{(1)}),z_2\in \partial \mathcal{B}_{x_1}(\frac{r_{m_2}^{(2)}}{\Cref{const_final_section}})}\mathbb{P}\big(\mathsf{A}^{\mathcal{C}_1\cup \{x_1,x_2\}}(z_1,z_2; \tfrac{ r_{m_1+1}^{(1)} }{ \Cref{const_final_section} }  ,r_{m_2}^{(2)}) \big), 
 	\end{split}
 \end{equation}
where $\mathbf{K}_{m_1}:=\mathbbm{1}_{m_1\ge 1}\cdot \widetilde{\mathbb{P}}_{x_2}\big(\tau_{\partial B_{x_1}(r_{m_1-1}^{(1)}) }<\tau_{\mathfrak{E}^{\mathrm{i}}_1} \big)+ \mathbbm{1}_{m_1=0} $, $\Cl\label{const_final_section}:=\Cref{const_last_section}^{\frac{1}{100}} $, and the event $\mathsf{A}^D(z,z';r,r')$ is defined by $\big\{z\xleftrightarrow{(D)} z'\big\}\cap  \big\{z \xleftrightarrow{(D)} \partial B_{x_1}(r)\cup \partial B_{x_1}(r')\big\}^c$. Taking expectation over the event $\mathsf{D}_{m_1}'\cap \mathsf{D}_{m_2}'''$, we obtain that 
  \begin{equation}\label{newpool640}
 	\begin{split}
 				\mathbb{P}\big(\mathsf{F}_{m_1,m_2}  \big) \lesssim & (r_{m_1}^{(1)} r_{m_2}^{(2)})^{1-d} \mathbf{U}_{m_1}  \cdot  \sum_{z_1\in \partial \mathcal{B}_{x_1}(  r_{m_1+1}^{(1)}),z_2\in \partial \mathcal{B}_{x_1}(\frac{r_{m_2}^{(2)}}{\Cref{const_final_section}})}  \mathbb{P}\big( 	\overline{\mathsf{C}}_{z_1,z_2}^{m_1,m_2}\big). 
 	\end{split}
 \end{equation}
 Here $\mathbf{U}_{m_1}:=\mathbbm{1}_{m_1\ge 1}\cdot \mathbb{E}_{\mathfrak{E}^{\mathrm{i}}_1}\big[ \mathbbm{1}_{\mathfrak{E}^{\mathrm{i}}_1\cap \partial B_{x_{1}}(r_{m_1}^{(1)})\neq \emptyset} \cdot  \widetilde{\mathbb{P}}_{x_2} \big(  \tau_{\partial B_{x_1}(r_{m_1-1}^{(1)}) }<\tau_{\mathfrak{E}^{\mathrm{i}}_1} \big) \big]+\mathbbm{1}_{m_1=0}$, and the event $\overline{\mathsf{C}}_{z_1,z_2}^{m_1,m_2}$ is defined by 
 \begin{equation}\label{forD}
\overline{\mathsf{C}}_{z_1,z_2}^{m_1,m_2}:=   \big\{ \bm{0}\xleftrightarrow{(\{x_1,x_2\})}  B_{x_1}(r_{m_1+1}^{(1)}) \Vert \mathsf{A}^{ \{x_1,x_2\}}(z_1,z_2;  \tfrac{ r_{m_1+1}^{(1)} }{ \Cref{const_final_section} }  ,r_{m_2}^{(2)})\cap \mathsf{D}_{m_2}''' \big\}.  
 \end{equation}

 The following lemma provides an upper bound for the expectation on the right-hand side of (\ref{newpool640}). With a slight abuse of notation, for any path $\overline{\eta}$ on $\mathbb{Z}^d$, we denote its range (i.e., the set of lattice points visited by $\overline{\eta}$) by $\mathrm{ran}(\overline{\eta})$. For $x\in \mathbb{Z}^d$, we denote by $\mathbb{P}_x$ the law of the simple random walk on $\mathbb{Z}^d$ starting from $x$.

  \begin{lemma}\label{lemma_escape_3d}
 	 For any $d\ge 3$, $\delta\in (0,\frac{1}{2})$ and $m_1\ge 1$, 
 	\begin{equation}\label{pool641}
 \mathbb{E}_{\mathfrak{E}^{\mathrm{i}}_1}\Big[ \mathbbm{1}_{ \mathfrak{E}^{\mathrm{i}}_1\cap \partial B_{x_{1}}(r_{m_1}^{(1)})\neq \emptyset}  \cdot \widetilde{\mathbb{P}}_{x_2} \big( \tau_{\partial B_{x_1}(r_{m_1-1}^{(1)}) }<\tau_{\mathfrak{E}^{\mathrm{i}}_1} \big) \Big]  \lesssim (r_{m_1}^{(1)})^{2-d-(1-\delta)\cdot \mathbbm{1}_{d=3}}. 
 	\end{equation}
 \end{lemma}
 \begin{proof}
 	  Note that Lemma \ref{lemma_pool_62} implies 
 	\begin{equation}\label{pool642}
 		 \mathbb{P}\big( \mathfrak{E}^{\mathrm{i}}_1\cap \partial B_{x_{1}}(r_{m_1}^{(1)})\neq \emptyset  \big) \asymp   (r_{m_1}^{(1)})^{2-d}.  
 	\end{equation}
 	This already gives the bound (\ref{pool641}) for $d\ge 4$. Hence, we focus on the case $d=3$ for the remainder of the proof.

 	Moreover, by Lemma \ref{lemma_pure62}, conditioned on $\big\{ \mathfrak{E}^{\mathrm{i}}_1\cap \partial B_{x_{1}}(r_{m_1}^{(1)})\neq \emptyset\big\} $, the expectation of the probability of $\widetilde{\mathbb{P}}_{x_2}\big(\tau_{\partial B_{x_1}(r_{m_1-1}^{(1)}) }<\tau_{\mathfrak{E}^{\mathrm{i}}_1} \big)$ is upper-bounded by 
 	\begin{equation}\label{pool643}
 		\begin{split}
 	& 	\mathbb{E}_{\widehat{\eta}_1,\widehat{\eta}_2}\Big[ \widetilde{\mathbb{P}}_{x_2}\big(\tau_{\partial B_{x_1}(r_{m_1-1}^{(1)}) }<\tau_{\cup_{j\in \{1,2\}}\mathrm{ran}(\widehat{\eta}_j)} \big)  \Big] \\
 		=&  \sum\nolimits_{ \overline{\eta}_1,\overline{\eta}_2\in \Gamma}  	\mathbb{P}\big( \mathrm{proj}(\widehat{\eta}_j) =\overline{\eta}_j, \forall j\in \{1,2\}  \big)    \widetilde{\mathbb{P}}_{x_2}\big(\tau_{\partial B_{x_1}(r_{m_1-1}^{(1)}) }<\tau_{\cup_{j\in \{1,2\}}\mathrm{ran}(\overline{\eta}_j)} \big) \\
 	\lesssim  & \sum\nolimits_{ \overline{\eta}_1,\overline{\eta}_2\in \Gamma}  (2d)^{-\mathrm{len}(\overline{\eta}_1)-\mathrm{len}(\overline{\eta}_2)} \cdot \widetilde{\mathbb{P}}_{x_2}\big(\tau_{\partial B_{x_1}(r_{m_1-1}^{(1)}) }<\tau_{\cup_{j\in \{1,2\}}\mathrm{ran}(\overline{\eta}_j)} \big). 
 		\end{split}
 	\end{equation}
 	For any $y \in \widetilde{B}_{x_1}(1)\setminus \{x_2\}$, we denote by $\Gamma_y$ the subset of $\Gamma$ restricted to the paths that start from $y$. We arbitrarily take $y_1,y_2\in \widetilde{B}_{x_1}(1)\setminus \{x_2\}$, and take three independent random walks $S^1_\cdot \sim  \mathbb{P}_{y_1}$, $S^2_\cdot \sim \mathbb{P}_{y_2}$ and $S^3_\cdot \sim  \mathbb{P}_{x_2}$. For each $1\le i\le 3$, we denote by $\mathring{\tau}_{i}$ the first time when $S^i_\cdot$ hits $\partial B_{x_1}(r_{m_1-1}^{(1)})$, and we define the path $\mathring{\eta}_i$ (resp. $\mathring{\eta}_i^*$) as $\big\{S^i_n: 0\le n\le \mathring{\tau}_{i}\big\}$ (resp. $\big\{S^i_n: 0\le n\le (r_{m_1}^{(1)})^{2-2\delta}\big\}$). Note that 
 	\begin{equation}\label{pool644}
 		\begin{split}
 			& \sum\nolimits_{ \overline{\eta}_1\in \Gamma_{y_1},\overline{\eta}_2\in \Gamma_{y_2}}  (2d)^{-\mathrm{len}(\overline{\eta}_1)-\mathrm{len}(\overline{\eta}_2)} \cdot \widetilde{\mathbb{P}}_{x_2}\big(\tau_{\partial B_{x_1}(r_{m_1-1}^{(1)}) }<\tau_{\cup_{j\in \{1,2\}}\mathrm{ran}(\overline{\eta}_j)} \big) \\
 			\le &  \mathbb{P}\big( \big[\cup_{i=1,2}\mathrm{ran}(\mathring{\eta}_i) \big]\cap \mathrm{ran}(\mathring{\eta}_3) =\emptyset   \big). 
 		\end{split}
 	\end{equation}
 	Moreover, since $\{\mathring{\tau}_{i}\ge(r_{m_1}^{(1)})^{2-2\delta}\}$ implies $\mathrm{ran}(\mathring{\eta}_i^*)\subset \mathrm{ran}(\mathring{\eta}_i)$, we have 
 	\begin{equation}\label{pool645}
 		 	\begin{split}
 		& \mathbb{P}\big(\big[\cup_{i=1,2}\mathrm{ran}(\mathring{\eta}_i) \big]\cap \mathrm{ran}(\mathring{\eta}_3) =\emptyset   \big)\\
 		\le &  \mathbb{P}\big(\big[\cup_{i=1,2}\mathrm{ran}(\mathring{\eta}_i^*) \big]\cap \mathrm{ran}(\mathring{\eta}_3^*) =\emptyset   \big) + \sum\nolimits_{1\le i\le 3} \mathbb{P}\big( \mathring{\tau}_{i}\le(r_{m_1}^{(1)})^{2-2\delta}  \big)\\
 		\lesssim & (r_{m_1}^{(1)})^{-1+\delta} + \mathrm{exp}(-c(r_{m_1}^{(1)})^{2\delta}), 
 	\end{split} 
 	\end{equation} 	
 	 	where in the last inequality we used \cite[Proposition 2.1.2 and Corollary 10.2.2]{lawler2010random}.

 	 	Inserting (\ref{pool645}) into (\ref{pool644}), and then summing over all $y_1,y_2\in \widetilde{B}_{x_1}(1)\setminus \{x_2\}$, we obtain that the right-hand side of (\ref{pool643}) is at most 
 	 	\begin{equation}\label{pool647}
 	 		C\big[ (r_{m_1}^{(1)})^{-1+\delta} + \mathrm{exp}(-c(r_{m_1}^{(1)})^{2\delta})\big] \lesssim (r_{m_1}^{(1)})^{-1+\delta}. 
 	 	\end{equation}
 	Combined with (\ref{pool642}) and (\ref{pool643}), it yields the desired bound (\ref{pool641}) for $d=3$, and thus completes the proof.
 	\end{proof}

 \begin{remark}
 	 Following the same argument as in the proof of Lemma \ref{lemma_escape_3d}, one may derive the following result for the discrete loop soup on $\mathbb{Z}^3$ (denoted by $\mathcal{L}_{\alpha}$, where $\alpha$ is the intensity). For any $N> n\ge 1$, let $\widehat{\mathfrak{L}}[n,N]$ denote the collection of loops in $\mathcal{L}_{1/2}$ that cross the annulus $B(N)\setminus B(n)$. Suppose that $\ell$ is a loop in $\widehat{\mathfrak{L}}[n,N]$, sampled from the loop measure. Then for any $\delta\in (0,\frac{1}{2})$ and $x\in \partial B(n)$,  
 	 \begin{equation}\label{nicenice648}
 	 	\mathbb{E}_{\ell}\big[ \mathbbm{1}_{ \widehat{\mathfrak{L}}[n,N] \neq \emptyset}\cdot  \mathbb{P}_{x} \big(  \tau_{ \partial B(cN) } < \tau_{\mathrm{ran}(\ell) } \big) \big]  \lesssim  \big(n/N \big)^{2-\delta}. 
 	 \end{equation}
 	First of all, the probability of $\{ \widehat{\mathfrak{L}}[n,N] \neq \emptyset\}$ is $O(n/N)$. In addition, as shown in (\ref{pool644}), for the simple random walk starting from $x$, the probability of escaping $\ell$ can be upper-bounded, up to a constant factor, by that of escaping two independent random walks starting from $B(n)$ and stopped upon hitting $\partial B(cN)$. Parallel to (\ref{pool645}), the latter probability is $O\big( (n/N)^{1-\delta} \big)$. These observations together yield the bound (\ref{nicenice648}).


 \end{remark}

 We now estimate the probability of $\overline{\mathsf{C}}_{z_1,z_2}^{m_1,m_2}$. For any $D'\subset \widetilde{\mathbb{Z}}^d$, by applying Lemma \ref{lemma_relation} and Corollary \ref{coro27} successively, we have  
\begin{equation}\label{newlate629}
	\begin{split}	
	&\mathbb{P}\big(B_{x_{1}}(r_{m_2+1}^{(2)})\xleftrightarrow{(D')} B_{x_{1}}(R_2) \big) \\
		\lesssim  & (r_{m_2}^{(2)} R_2)^{-\frac{d}{2}}  \sum_{z_3\in \partial \mathcal{B}_{x_1}(dr_{m_2+1}^{(2)}),z_4\in \partial \mathcal{B}_{x_1}(d^{-1}R_2)}  	\mathbb{P}\big( \mathsf{A}^{D'}(z_3,z_4; \tfrac{r_{m_2+1}^{(2)} }{\Cref{const_final_section}},\Cref{const_final_section}R_2) \big).  
	\end{split}
\end{equation}  
Meanwhile, it follows from Corollary \ref{lemma_cut_two_points} and Lemma \ref{lemma_relation} that  
\begin{equation}\label{late629}
	\begin{split}
		&  \mathbb{P}\big( \bm{0}\xleftrightarrow{(D')}  B_{x_1}(r_{m_1+1}^{(1)})\big)  \\\lesssim & (r_{m_1}^{(1)})^{-\frac{d}{2}}(r_{m_2}^{(2)})^{-\frac{d}{2}-1}   \sum\nolimits_{z_5\in \partial \mathcal{B}_{x_1}(  d r_{m_1+1}^{(1)}),z_6\in \partial \mathcal{B}_{x_1}(\frac{dr_{m_2}^{(2)}}{\Cref{const_final_section}}),z_7\in \partial \mathcal{B}_{x_1}( r_{m_2+1}^{(2)})}  \\
		 &\ \ \ \ \ \ \ \ \ \ \ \ \ \ \ \ \ \ \ \ \ \ \ \ \  \mathbb{P}\big(  \mathsf{A}^{D'}(z_5,z_6;   \tfrac{ r_{m_1+1}^{(1)} }{ \Cref{const_final_section} }  ,r_{m_2}^{(2)}) \cap  \mathsf{A}^{D'}(z_7,\bm{0}; \tfrac{r_{m_2
		 +1}^{(2)}}{\Cref{const_final_section} } ,\Cref{const_final_section} R_1)  \big). 
	\end{split}
\end{equation} 
 Using the restriction property, (\ref{newlate629}) and (\ref{late629}), we get 
\begin{equation}\label{need631}
	\begin{split}
		\mathbb{P}\big( \overline{\mathsf{C}}_{\bar{z}}^{m_1,m_2} \big) \lesssim   (r_{m_1}^{(1)})^{-\frac{d}{2}}(r_{m_2}^{(2)})^{-d-1} R_2^{-\frac{d}{2}}     \sum_{\bar{z}:=(z_3,z_4,z_5,z_6,z_7)\in \Psi_{m_1,m_2}}   \mathbb{P}\big( 	\widetilde{\mathsf{C}}_{\bar{z}}^{m_1,m_2}\big), 
	\end{split} 
\end{equation} 
 where we denote $\Psi_{m_1,m_2}:=\partial \mathcal{B}_{x_1}(dr_{m_2+1}^{(2)})\times \partial \mathcal{B}_{x_1}(d^{-1}R_2) \times \partial \mathcal{B}_{x_1}(  d r_{m_1+1}^{(1)}) \times \partial \mathcal{B}_{x_1}(\frac{dr_{m_2}^{(2)}}{\Cref{const_final_section}}) \times \partial \mathcal{B}_{x_1}( r_{m_2+1}^{(2)})$, and define $\widetilde{\mathsf{C}}_{\bar{z}}^{m_1,m_2}$ as the intersection of $\mathsf{A}^{ \{x_1,x_2\}}(z_1,z_2;  \tfrac{ r_{m_1+1}^{(1)} }{ \Cref{const_final_section} }  ,r_{m_2}^{(2)})$, $\mathsf{A}^{ \{x_1,x_2\}}(z_3,z_4; \tfrac{r_{m_2+1}^{(2)} }{\Cref{const_final_section}},\Cref{const_final_section}R_2)$, $ \mathsf{A}^{ \{x_1,x_2\}}(z_5,z_6;   \tfrac{ r_{m_1+1}^{(1)} }{ \Cref{const_final_section} }  ,r_{m_2}^{(2)})$ and $ \mathsf{A}^{ \{x_1,x_2\}}(z_7,\bm{0}; \tfrac{r_{m_2
		 +1}^{(2)}}{\Cref{const_final_section} } ,\Cref{const_final_section} R_1)$. Moreover, on the event $\widetilde{\mathsf{A}}_{\bar{z}}^{m_1,m_2}$, the following independent events occur: 
 \begin{equation*}
 	\mathsf{C}_1:= \big\{ z_1\xleftrightarrow{(\{x,y\}\cup \partial B_{x_1}(\frac{ r_{m_1+1}^{(1)} }{ \Cref{const_final_section} }  )\cup \partial B_{x_1}(r_{m_2}^{(2)}))} z_2  \Vert z_5\xleftrightarrow{(\{x,y\}\cup \partial B_{x_1}(\frac{ r_{m_1+1}^{(1)} }{ \Cref{const_final_section} }  )\cup \partial B_{x_1}(r_{m_2}^{(2)}))} z_6  \big\}, 
 \end{equation*} 
 \begin{equation*}
 	\mathsf{C}_2:=  \big\{ z_3\xleftrightarrow{(\{x,y\}\cup \partial B_{x_1}(\frac{r_{m_2
		 +1}^{(2)}}{\Cref{const_final_section} }) )} z_4  \Vert z_7 \xleftrightarrow{(\{x,y\}\cup \partial B_{x_1}(\frac{r_{m_2
		 +1}^{(2)}}{\Cref{const_final_section} }  ))} \bm{0}  \big\}.
 \end{equation*} 
By the estimate (\ref{four_point_use}), one has 
\begin{equation}\label{need632}
	\mathbb{P}\big(\mathsf{C}_1 \big) \lesssim  (r_{m_1}^{(1)})^{ 3-\frac{d}{2}}(r_{m_2}^{(2)})^{-\frac{3d}{2}+1}.
	\end{equation} 
	Moreover, referring to (\ref{late516}) and (\ref{late520}), we have  
	\begin{equation}\label{need633}
		\begin{split}
				\mathbb{P}\big(\mathsf{C}_2 \big) \lesssim  (R_1R_2)^{2-d}\cdot \big(\tfrac{R_1\land R_2}{r_{m_2}^{(2)}}\big)^{\frac{d}{2}-3}. 
		\end{split}
	\end{equation}
	 Putting (\ref{need631}), (\ref{need632}) and (\ref{need633}) together, we get 
	 \begin{equation}\label{pure637}
	 	\begin{split}
	 		\mathbb{P}\big( \overline{\mathsf{C}}_{\bar{z}}^{m_1,m_2} \big) \lesssim & |\Psi_{m_1,m_2}|\cdot (r_{m_1}^{(1)})^{3-d}(r_{m_2}^{(2)})^{-3d+3} R_1^{2-d}R_2^{-\frac{3d}{2}+2}(R_1\land R_2)^{\frac{d}{2}-3}\\
	 		\lesssim & (r_{m_1}^{(1)})^{2}R_1^{2-d} R_2^{-\frac{d}{2}+1}(R_1\land R_2)^{\frac{d}{2}-3}. 
	 	\end{split}
	 \end{equation} 
 	Plugging (\ref{pool641}) and (\ref{pure637}) into (\ref{newpool640}), we obtain (\ref{615}): 
	\begin{equation}\label{nice635}
		\begin{split}
			\mathbb{P}\big(\mathsf{F}_{m_1,m_2}  \big) \lesssim  (r_{m_1}^{(1)})^{4-d-(1-\delta)\cdot \mathbbm{1}_{d=3}} R_1^{2-d} R_2^{-\frac{d}{2}+1}(R_1\land R_2)^{\frac{d}{2}-3}\lesssim \mathbb{J}_{m_1,m_2}. 
		\end{split}
	\end{equation}

  \textbf{When $m_1=k_1$.} We first decompose the event $\mathsf{F}_{k_1,m_2}$ as follows. We define 
  \begin{equation}
  	l_\star:= \min\{ l\ge 1: 2^{l+3}\ge R_1 \}. 
  \end{equation}
  For $1\le l \le l_\star-1$, we denote the event 
  \begin{equation}
  \mathsf{U}_l:= \big\{ \mathfrak{E}^{\mathrm{i}}_1\cap \partial B(2^l)\neq \emptyset,\mathfrak{E}^{\mathrm{i}}_1\cap \partial B(2^{l-1})= \emptyset  \big\}. 
  \end{equation}
  In addition, we also denote $\mathsf{U}_0:=\big\{ \mathfrak{E}^{\mathrm{i}}_1\cap \partial B(1)\neq \emptyset  \big\}$ and 
  \begin{equation}
   \mathsf{U}_{l_\star}:=  \big\{ \mathfrak{E}^{\mathrm{i}}_1\cap \partial B(2^{l_\star-1})=  \emptyset,\mathfrak{E}^{\mathrm{i}}_1\cap \partial B_{x_{1}}(r_{k_1}^{(1)})\neq \emptyset \big\}. 
  \end{equation}
  Clearly, one has $\{\mathfrak{E}^{\mathrm{i}}_1\cap \partial B_{x_{1}}(r_{k_1}^{(1)})\neq \emptyset \}\subset \cup_{0\le l\le l_\star} \mathsf{U}_l$. Moreover, for each $1\le l\le l_\star$, on the event $\mathsf{U}_l$, if $\mathfrak{E}^{\mathrm{i}}_1$ intersects $\mathcal{C}_1$, then $\mathsf{V}_l:=\big\{\bm{0}\xleftrightarrow{\{x_1,x_2\}}  \partial B( 2^{l-1}) \big\}$ must occur. For convenience, we let $\mathsf{V}_0$ be the sure event.

   Recall that on $\mathsf{F}_{k_1,m_2}$, the events $\mathfrak{E}^{\mathrm{i}}_1\cap \partial B_{x_{1}}(r_{k_1}^{(1)})\neq \emptyset$, $\mathfrak{E}^{\mathrm{i}}_1\cap \mathcal{C}_1\neq \emptyset$ and $\mathsf{D}_{m_2}'''$ (defined in (\ref{verynew_D'''})) all occur; in addition, $\mathfrak{E}^{\mathrm{i}}_2$ intersects $\partial B_{x_{1}}(r_{m_2}^{(2)})$ and is disjoint from $\mathfrak{E}^{\mathrm{i}}_1$. By the observation in the last paragraph, there exists $0\le l\le l_\star$ such that $\mathsf{U}_{l}\cap \mathsf{V}_{l}$ occurs. Note that the estimate (\ref{one_arm_low}) implies  
   \begin{equation}
   	\mathbb{P}\big(\mathsf{V}_{l} \big)\lesssim  2^{(-\frac{d}{2}+1)l}, \ \ \forall 0\le l\le l_\star. 
   \end{equation}
Moreover, on the event $\mathsf{U}_{l}$, there exists an excursion in $\mathfrak{E}^{\mathrm{i}}_1$ that contains a trajectory of Brownian motion from $\partial B_{x_{1}}(r_{k_1}^{(1)})$ to $B(2^l)$, and one from $\partial B(2^l)$ to $x_1$. Thus, by the strong Markov property of Brownian excursion measure, we have 
   \begin{equation}
   	\mathbb{P}\big(\mathsf{U}_{l} \big)\lesssim \big(\tfrac{R_1}{2^{l}} \big)^{2-d} \cdot R_1^{2-d} =2^{(d-2)l} R_1^{4-2d}, \ \ \forall 0\le l\le l_\star. 
   \end{equation}
   Moreover, conditioned on $\mathsf{U}_{l}$, when $\{\mathfrak{E}^{\mathrm{i}}_2\cap B_{x_{1}}(r_{m_2}^{(2)})\neq \emptyset,\mathfrak{E}^{\mathrm{i}}_2\cap \mathfrak{E}^{\mathrm{i}}_1=\emptyset \}$ occurs, there exists an excursion in $\mathfrak{E}^{\mathrm{i}}_2$ that contains a trajectory from $x_2$ to $\partial B_{x_{1}}(r_{k_1}^{(1)})$ without hitting $\mathfrak{E}^{\mathrm{i}}_1$ (whose probability, for the same reason as in proving (\ref{pool647}), is at most $CR_1^{-(1-\delta)\cdot \mathbbm{1}_{d=3}}$), and one from $B_{x_{1}}(r_{m_2}^{(2)})$ to $x_2$ (which happens with probability $O((r_{m_2}^{(2)})^{2-d})$). Furthermore, it follows from (\ref{crossing_low}) that arbitrarily given $\mathfrak{E}^{\mathrm{i}}_1$ and $\mathcal{C}_1$, the conditional probability of $\mathsf{D}_{m_2}'''$ is at most $C(R_2/r_{m_2}^{(2)})^{-\frac{d}{2}+1}$. To sum up, 
   \begin{equation}\label{nice641}
   	\begin{split}
   			\mathsf{F}_{k_1,m_2}\lesssim  & \sum\nolimits_{0\le l\le l_\star } 2^{(-\frac{d}{2}+1)l}\cdot  2^{(d-2)l} R_1^{4-2d}  \cdot   R_1^{-(1-\delta)\cdot\mathbbm{1}_{d=3}}(r_{m_2}^{(2)})^{2-d} \cdot   \big( \tfrac{R_2}{r_{m_2}^{(2)}} \big)^{-\frac{d}{2}+1}\\
   			=&  (r_{m_2}^{(2)})^{-\frac{d}{2}+1}R_1^{4-2d-(1-\delta)\cdot \mathbbm{1}_{d=3}}R_2^{-\frac{d}{2}+1} \sum\nolimits_{0\le l\le l_\star } 2^{(\frac{d}{2}-1)l} \\
   			\lesssim &(r_{m_2}^{(2)})^{-\frac{d}{2}+1}R_1^{-\frac{3}{2}d+3-(1-\delta)\cdot\mathbbm{1}_{d=3}}R_2^{-\frac{d}{2}+1} \lesssim  \mathbb{J}_{k_1,m_2}. 
   	\end{split}
   \end{equation}


 By (\ref{verynew_629}), (\ref{nice635}) and (\ref{nice641}), we have confirmed the desired bound (\ref{615}), thereby completing the proof of Lemma \ref{lemma61}. \qed

  \section*{Acknowledgments}

 We warmly thank Wendelin Werner for his valuable suggestions, which greatly enhanced the clarity of the manuscript. J. Ding is supported by the National Natural Science Foundation of China (Grant No. 12231002, 12595284, 12595280), and by the New Cornerstone Science Foundation through the New Cornerstone Investigator Program and XPLORER PRIZE.


  \appendix

%
%
%
%
%
%

 \section{Proof of Lemma \ref{lemma_new_stable}} \label{app0_stable}

 When $3\le d\le 5$, the desired bound (\ref{ineq_new_stable}) can be derived as follows: 
 \begin{equation}
 \begin{split}
 	\mathbb{P}\big( A\xleftrightarrow{(D)} x \big) \asymp & N^{-\frac{d}{2}+1} 	\cdot \mathbb{P}\big( A\xleftrightarrow{(D)} \partial B(N) \big) \\
 \overset{(\text{Lemma}\ \ref{lemma_stability_bd})}{\asymp }& N^{-\frac{d}{2}+1} 	\cdot  	\mathbb{P}\big( A\xleftrightarrow{(D\cup D')} \partial B(N) \big)   
 	\asymp   \mathbb{P}\big( A\xleftrightarrow{(D\cup D')} x  \big). 
 \end{split}
 \end{equation}
 Here the first and last inequalities both follow from \cite[Proposition 1.9]{cai2024quasi}.

 Now we focus on the case when $d\ge 7$. Note that 
 \begin{equation}\label{findA2}
 	\begin{split}
0\le  &	\mathbb{P}\big( A\xleftrightarrow{(D)} x  \big) -	\mathbb{P}\big( A\xleftrightarrow{(D\cup D')} x  \big) \\
=&\mathbb{P}(\mathsf{F}) :=	\mathbb{P}\big( A\xleftrightarrow{(D)} x  , \big\{ \mathbb{P}\big( A\xleftrightarrow{(D\cup D')} x  \big\}^c \big). 
 	\end{split}
 \end{equation}
 On $\mathsf{F}$, the event $\{A\xleftrightarrow{(D)} x \}$ occurs, and removing all loops that intersect $D'$ must violate it. As a result, there exists a loop $\widetilde{\ell}_\dagger\in \widetilde{\mathcal{L}}_{1/2}$ with $\mathrm{ran}(\widetilde{\ell}_\dagger)\cap D'\neq \emptyset$ such that $\mathrm{ran}(\widetilde{\ell}_\dagger) \xleftrightarrow{(D)} A$ and $\mathrm{ran}(\widetilde{\ell}_\dagger) \xleftrightarrow{(D)} x$ occur disjointly (to see this, one may remove the loops intersecting 
 $D'$ one by one, stopping upon the absence of $\{A\xleftrightarrow{(D)} x \}$; then the last removed loop meets the requirements). Here ``$\mathsf{A}$ and $\mathsf{A}'$ occur disjointly'' (denoted by $\mathsf{A}\circ \mathsf{A}'$) means that $\mathsf{A}$ and $\mathsf{A}'$ are certified by two disjoint collections of loops (see \cite[Section 3.3]{cai2024high} for further details).

 Recall that $D'\subset [\widetilde{B}(c^{-1}N)]^c$. In what follows, we estimate the probability of $\mathsf{F}$ separately in two cases, according to whether $\widetilde{\ell}_\dagger$ intersects $\widetilde{B}(c^{-1/3}N)$.

  \textbf{Case 1: $\mathrm{ran}(\widetilde{\ell}_\dagger)\subset [\widetilde{B}(c^{-\frac{1}{3}}N)]^c$.} By the construction of $\widetilde{\ell}_\dagger$, there exist $w_1,w_2,w_3\in \mathbb{Z}^d$ such that $\widetilde{B}_{w_3}(1)\cap D'\neq \emptyset$, $\widetilde{\ell}_\dagger$ intersects $\widetilde{B}_{w_j}(1)$ for all $j\in \{1,2,3\}$, and that $w_1\xleftrightarrow{(D)} A$ and $w_2\xleftrightarrow{(D)} x$ occur disjointly. By \cite[Lemma 2.3]{cai2024high}, the probability of having a loop in $\widetilde{\mathcal{L}}_{1/2}^D$ intersecting $\widetilde{B}_{w_j}(1)$ for all $j\in \{1,2,3\}$ is at most 
 	\begin{equation}\label{A3_loop}
 		C|w_1-w_2|^{2-d}|w_2-w_3|^{2-d}|w_3-w_1|^{2-d}. 
 	\end{equation}
 (\textbf{P.S.} In this proof and the next three sections, we denote $0^{-a}:=1$ for $a>0$.) In addition, it follows from Lemma \ref{lemma_point_to_set} and (\ref{two-point1}) that 
 \begin{equation}\label{verynewA4}
 	\mathbb{P}\big( w_1\xleftrightarrow{(D)} A \big) \lesssim 	c|w_1|^{2-d}N^{d-2}\cdot \mathbb{P}\big(A\xleftrightarrow{(D)} x\big),
 \end{equation}
 \begin{equation}\label{verynewA5}
 	\mathbb{P}\big(w_2\xleftrightarrow{(D)} x  \big)\lesssim |w_2-x|^{2-d}. 
 \end{equation}

 Note that the assumption $\mathrm{ran}(\widetilde{\ell}_\dagger)\subset [\widetilde{B}(c^{-\frac{1}{3}}N)]^c$ implies $w_1\in [B(N)]^c$. Thus, applying the BKR inequality, and using the estimates (\ref{A3_loop}), (\ref{verynewA4}) and (\ref{verynewA5}), the probability of this case is bounded from above by  
   \begin{equation}\label{A9}
   	\begin{split}
   		&CN^{d-2}  \mathbb{P}\big(A\xleftrightarrow{(D)} x\big) \sum\nolimits_{w_1,w_2,w_3\in \mathbb{Z}^d:w_1\in [B(N)]^c, \widetilde{B}_{w_3}(1)\cap D'\neq \emptyset }|w_1-w_2|^{2-d}\\
   		&\ \ \ \ \ \ \ \ \ \ \ \ \ \ \ \ \ \ \ \ \ \ \ \ \ \ \ \  \cdot |w_2-w_3|^{2-d}|w_3-w_1|^{2-d}|w_1|^{2-d} |w_2-x|^{2-d}.
   	\end{split}
   \end{equation}
   By $|w_1|\gtrsim N$ and \cite[(4.11)]{cai2024high}, the sum on the right-hand side is at most 
   \begin{equation}
   \begin{split}
   	   	CN^{2-d}\sum\nolimits_{w_3\in \mathbb{Z}^d:\widetilde{B}_{w_3}(1)\cap D'\neq \emptyset } |w_3-x|^{2-d} \overset{|w_3-x|\asymp |w_3|,(\ref{condition_new_stable})}{\lesssim } N^{8-2d}. 
   \end{split}
   \end{equation}
   Combined with (\ref{A9}), it implies that the probability of this case is at most of order $N^{6-d}\cdot \mathbb{P}\big(A\xleftrightarrow{(D)} x\big)$.


 Before proceeding to the next case, we first review the decomposition argument of loops introduced in \cite[Section 2.6.3]{cai2024high}. Precisely, for $M>m\ge 1$, each loop in $\mathfrak{L}[m,M]$ (i.e., the point measure consisting of loops in $\widetilde{\mathcal{L}}_{1/2}^D$ that crosses the annulus $B(M)\setminus B(m)$) can be divided into forward and backward crossing paths, where each forward (resp. backward) crossing path is a Brownian motion starting from $\partial B(m)$ (resp. $\partial B(M)$) and stopping upon hitting $\partial B(M)$ (resp. $\partial B(m)$). Note that the starting point of a forward crossing path is the ending point of a backward crossing path, and vice versa. We enumerate the forward (resp. backward) crossing paths of all loops in $\widetilde{\mathcal{L}}_{1/2}^D$ crossing $B(M)\setminus B(m)$ as $\{\widetilde{\eta}^{\mathrm{F}}_i\}_{1\le i\le \kappa}$ (resp. $\{\widetilde{\eta}^{\mathrm{B}}_i\}_{1\le i\le \kappa}$). By \cite[Lemma 2.11]{cai2024quasi}, the total number $\kappa$ of crossings decays exponentially. I.e.,  
   \begin{equation}\label{decay_A10}
   	\mathbb{P}\big(\kappa \ge l \big)\le \big( \tfrac{Cm}{M} \big)^{(d-2)l}, \ \ \forall l\ge 1.  
   \end{equation}
   We denote the starting and ending points of $\widetilde{\eta}^{\mathrm{F}}_i$ by $\mathbf{y}_i$ and $\mathbf{z}_i$ respectively. Let $\mathcal{F}_{\mathbf{yz}}$ denote the $\sigma$-field generated by $\{(\mathbf{y}_i,\mathbf{z}_i)\}_{1\le i\le \kappa}$. Referring to \cite[Lemma 2.4]{cai2024high}, conditioned on $\mathcal{F}_{\mathbf{yz}}$, the crossing paths $\{\widetilde{\eta}^{\mathrm{F}}_i\}_{1\le i\le \kappa}\cup \{\widetilde{\eta}^{\mathrm{B}}_i\}_{1\le i\le \kappa}$ are independent. In addition, for any $1\le i\le \kappa$, $y\in \partial B(m)$ and $z\in \partial B(M)$, if $\widetilde{\eta}^{\mathrm{F}}_i$ starts from $y$ and ends at $z$, then it is distributed as
      \begin{equation}
   	\widetilde{\eta}^{\mathrm{F}}_i \sim \widetilde{\mathbb{P}}_{y}\big( \big\{\widetilde{S}_t\big\}_{0\le t\le \tau_{\partial B(M)}} \in \cdot   \mid \tau_{\partial B(M) }=\tau_{z}<\tau_{D} \big). 
   \end{equation} 
   Similarly, if $\widetilde{\eta}^{\mathrm{B}}_i$ starts from $z$ and ends at $y$, then it has distribution
       \begin{equation}
   	\widetilde{\eta}^{\mathrm{B}}_i \sim \widetilde{\mathbb{P}}_{z}\big( \big\{\widetilde{S}_t\big\}_{0\le t\le \tau_{\partial B(m)}} \in \cdot   \mid \tau_{\partial B(m) }=\tau_{y}<\tau_{D} \big).  
   \end{equation}

  \textbf{Case 2: $\mathrm{ran}(\widetilde{\ell}_\dagger) \cap \widetilde{B}(c^{-\frac{1}{3}}N)\neq \emptyset$.} We consider the crossing paths for $m=c^{-\frac{1}{3}}N$ and $M=c^{-\frac{2}{3}}N$. Since $\mathrm{ran}(\widetilde{\ell}_\dagger) \cap \widetilde{B}(c^{-\frac{1}{3}}N)\neq \emptyset$ and $D'\subset [\widetilde{B}(c^{-1}N)]^c$, the loop must cross the annulus $B(c^{-\frac{2}{3}}N)\setminus B(c^{-\frac{1}{3}}N)$. Therefore, in this case, the following events all occur: 
 \begin{itemize}
 	
 	\item  $\mathsf{F}_1$: $\exists$ a crossing path $\eta_1$ that is connected to $A$ by a loop cluster $\mathcal{C}_1$;

 	\item $\mathsf{F}_2$: $\exists$ a crossing path $\eta_2$ that is connected to $x$ by a loop cluster $\mathcal{C}_2$;

 	\item $\mathsf{F}_3$: $\exists$ a backward crossing path $\eta_3$ that intersects $D'$.  	
 	
 \end{itemize}
 Here $\mathcal{C}_1$ and $\mathcal{C}_2$ are constructed by two disjoint collections of loops in $\widetilde{\mathcal{L}}_{1/2}^D-\mathfrak{L}[c^{-\frac{1}{3}}N,c^{-\frac{2}{3}}N]$. Conditioned on $\mathcal{F}_{\mathbf{yz}}$, we say that events $\mathsf{A}$ and $\mathsf{A}'$ occur disjointly if they are certified by two disjoint collections of loops and crossing paths. In this setting, the BKR inequality remains valid (see \cite[Remark 2.15]{cai2024quasi}). Next, we estimate the probability of $\cap_{i\in\{1,2,3\}} \mathsf{F}_i$ under different restrictions on $\{\eta_i\}_{i\in\{1,2,3\}}$.

  
 
  \textbf{Case 2.1: $\eta_1$, $\eta_2$ and $\eta_3$ are distinct.} In this case, the events $\mathsf{F}_1$, $\mathsf{F}_2$ and $\mathsf{F}_3$ occur disjointly. In addition, they imply $\{A\xleftrightarrow{(D)} \partial B(c^{-\frac{1}{3}}N)\}$, $\{x\xleftrightarrow{(D)} \partial B(c^{-\frac{1}{3}}N)\}$ and $\cup_{1\le i\le \kappa}\{\mathrm{ran}(\widetilde{\eta}^{\mathrm{B}}_i)\cap D'\neq \emptyset \}$ respectively. Therefore, by the BKR inequality, the conditional probability of this case given $\mathcal{F}_{\mathbf{yz}}$ is bounded from above by 
   \begin{equation}\label{veryniceA10}
 	\begin{split}
 		 & \mathbb{P}\big(A\xleftrightarrow{(D)} \partial B(c^{-\frac{1}{3}}N) \mid \mathcal{F}_{\mathbf{yz}}  \big) \cdot  \mathbb{P}\big(x\xleftrightarrow{(D)} \partial B(c^{-\frac{1}{3}}N) \mid \mathcal{F}_{\mathbf{yz}} \big) \\
 		& \cdot  \sum\nolimits_{1\le i\le \kappa }  \mathbb{P}\big(\mathrm{ran}(\widetilde{\eta}^{\mathrm{B}}_i)\cap D'\neq \emptyset \mid \mathcal{F}_{\mathbf{yz}}  \big).
 	\end{split}
 \end{equation}
 Moreover, it follows from \cite[Lemma 4.3]{cai2024quasi} that 
 \begin{equation}\label{veryniceA11}
 	\begin{split}
 		&\mathbb{P}\big(A\xleftrightarrow{(D)} \partial B(c^{-\frac{1}{3}}N) \mid \mathcal{F}_{\mathbf{yz}}  \big) \\
 		 \lesssim &(\kappa+1) \cdot  \mathbb{P}\big(A\xleftrightarrow{(D)} \partial B(10N)   \big) \overset{(\text{Lemma}\ \ref{lemma_213})}{ \lesssim}  (\kappa+1)N^{d-4}\cdot \mathbb{P}\big(A\xleftrightarrow{(D)} x  \big), 
 	\end{split}
 \end{equation}
  \begin{equation}\label{veryniceA12}
 	\begin{split}
 		&  \mathbb{P}\big(x\xleftrightarrow{(D)} \partial B(c^{-\frac{1}{3}}N) \mid \mathcal{F}_{\mathbf{yz}} \big)  \\
 		\lesssim & (\kappa+1) \cdot \mathbb{P}\big(x\xleftrightarrow{(D)} \partial B(10N)  \big) \overset{(\ref{one_arm_high})}{\lesssim }(\kappa+1)\cdot N^{-2}. 
 	\end{split}
 \end{equation}
 Meanwhile, for each $1\le i\le \kappa$, the union bound yields that 
 \begin{equation}\label{veryniceA13}
 \begin{split}
 	 & 	\mathbb{P}\big(\mathrm{ran}(\widetilde{\eta}^{\mathrm{B}}_i)\cap D'\neq \emptyset \mid \mathcal{F}_{\mathbf{yz}}  \big) \\
 	 	\le &\sum\nolimits_{w_3\in \mathbb{Z}^d:\widetilde{B}_{w_3}(1)\cap D'\neq \emptyset } \mathbb{P}\big( w_3 \in \mathrm{ran}(\widetilde{\eta}^{\mathrm{B}}_i) \mid \mathcal{F}_{\mathbf{yz}}  \big)\\
 	 	\lesssim & \sum\nolimits_{w_3\in \mathbb{Z}^d:\widetilde{B}_{w_3}(1)\cap D'\neq \emptyset } |w_3|^{4-2d}\cdot N^{d-2} \\
 	 \overset{(|w_3|\gtrsim N)}{\lesssim} & \sum\nolimits_{w_3\in \mathbb{Z}^d:\widetilde{B}_{w_3}(1)\cap D'\neq \emptyset }   |w_3|^{2-d} \overset{(\ref{condition_new_stable})}{\le} c'N^{6-d}, 
 \end{split}
 \end{equation}
 where the second inequality follows from \cite[(2.63)]{cai2024quasi}. Plugging (\ref{veryniceA11}), (\ref{veryniceA12}) and (\ref{veryniceA13}) into (\ref{veryniceA10}), and then combining with (\ref{decay_A10}), we obtain that the probability this case is at most of order $c' \cdot \mathbb{P}\big(A\xleftrightarrow{(D)} x  \big)$. 

 \textbf{Case 2.2: $\eta_1=\eta_2\neq \eta_3$.} In this case, the events $\{x\xleftrightarrow{(D)} A,x\xleftrightarrow{(D)}\partial B(c^{-\frac{1}{3}}N) \}$ and $\cup_{1\le i\le \kappa}\{\mathrm{ran}(\widetilde{\eta}^{\mathrm{B}}_i)\cap D'\neq \emptyset \}$ occur disjointly, as they are certified by $\{\eta_1,\mathcal{C}_1,\mathcal{C}_2\}$ and $\eta_3$ respectively. Therefore, the conditional probability of this case given $\mathcal{F}_{\mathbf{yz}}$ is bounded from above by  
  \begin{equation}\label{verynice_A14}
 	\begin{split}
 & 	 \mathbb{P}\big(  A\xleftrightarrow{(D)} x, A\xleftrightarrow{(D)}\partial B(c^{-\frac{1}{3}}N)  \mid \mathcal{F}_{\mathbf{yz}} \big)   \sum\nolimits_{1\le i\le \kappa }  \mathbb{P}\big(\mathrm{ran}(\widetilde{\eta}^{\mathrm{B}}_i)\cap D'\neq \emptyset \mid \mathcal{F}_{\mathbf{yz}}  \big) .
 	\end{split}
 \end{equation}
 Moreover, by \cite[Lemma 4.3]{cai2024quasi} one has 
 \begin{equation}\label{verynice_A15}
 	\begin{split}
 		 \mathbb{P}\big(  A\xleftrightarrow{(D)} x, A\xleftrightarrow{(D)}\partial B(c^{-\frac{1}{3}}N)  \mid \mathcal{F}_{\mathbf{yz}} \big)  \lesssim (\kappa+1)\cdot  \mathbb{P}\big(  A\xleftrightarrow{(D)} x \big). 
 	\end{split}
 \end{equation}
 Inserting (\ref{veryniceA13}) and (\ref{verynice_A15}) into (\ref{verynice_A14}), and then combining with (\ref{decay_A10}), we derive that the probability of this case is at most of order $N^{6-d} \cdot  \mathbb{P}\big(  A\xleftrightarrow{(D)} x \big) $.

  \textbf{Case 2.3: $\eta_1=\eta_3$.} Recall that $\eta_3$ is a backward crossing path, and hence is contained in $[\widetilde{B}(c^{-\frac{1}{3}}N)]^c$. Therefore, since $\mathcal{C}_1$ connects $A$ and $\eta_1$($=\eta_3$), it certifies the event $\{A\xleftrightarrow{(D)} \partial B(c^{-\frac{1}{3}}N)\}$. Meanwhile, all loops crossing $B(c^{-\frac{2}{3}}N)\setminus B(c^{-\frac{1}{3}}N)$ together with the cluster $\mathcal{C}_2$ certify $\{x\xleftrightarrow{} D'\}$. Thus, by the BKR inequality, the probability of this case is upper-bounded by 
  \begin{equation}
  	\begin{split}
  	 & \mathbb{P}\big(A\xleftrightarrow{(D)} \partial B(c^{-\frac{1}{3}}N) \big)\cdot \mathbb{P}\big( x\xleftrightarrow{} D' \big)  \\
  	\overset{(\text{Lemma}\ \ref{lemma_213})}{ \lesssim}  &\mathbb{P}\big(A\xleftrightarrow{(D)} x \big) \cdot \sum\nolimits_{w_3\in \mathbb{Z}^d:\widetilde{B}_{w_3}(1)\cap D'\neq \emptyset }  |x-w_3|^{2-d} \\
  	 \overset{|x-w_3|\asymp |w_3|,(\ref{condition_new_stable})}{\asymp} &  N^{6-d}\cdot \mathbb{P}\big(A\xleftrightarrow{(D)} x \big). 
  	\end{split}
  \end{equation}

 \textbf{Case 2.4: $\eta_2=\eta_3\neq \eta_1$.} Similarly, since $\eta_2$($=\eta_3$) is a backward crossing path, the cluster $\mathcal{C}_2$ certifies the event $\{x\xleftrightarrow{(D)} \partial   B(c^{-\frac{1}{3}}N) \} $. Meanwhile, the events $\{A\xleftrightarrow{(D)} \partial   B(c^{-\frac{1}{3}}N) \}$ and $\cup_{1\le i\le \kappa}\{\mathrm{ran}(\widetilde{\eta}^{\mathrm{B}}_i)\cap D'\neq \emptyset \}$ are certified by $\{\eta_1,\mathcal{C}_1\}$ and $\eta_3$ respectively. Hence, these three events occur disjointly. Thus, for the same reason as in \textbf{Case 2.1}, the probability of this case is at most of order $c' \cdot \mathbb{P}\big(A\xleftrightarrow{(D)} x  \big)$.

  \textbf{Case 2.5: $\eta_1=\eta_2= \eta_3$.} In this case, the clusters $\mathcal{C}_1$ and $\mathcal{C}_2$ certify the events $\{x\xleftrightarrow{(D)} \partial   B(c^{-\frac{1}{3}}N) \} $ and $\{A\xleftrightarrow{(D)} \partial   B(c^{-\frac{1}{3}}N) \}$ respectively. In addition, the loop containing $\eta_3$ must cross $B(c^{-\frac{2}{3}}N)\setminus B(c^{-\frac{1}{3}}N)$ and intersect $D'$. Thus, by the BKR inequality, the probability of this case is at most 
  \begin{equation*}
  	\begin{split}
  	&	\mathbb{P}\big( x\xleftrightarrow{(D)} \partial   B(c^{-\frac{1}{3}}N) \big) \cdot \mathbb{P}\big( A\xleftrightarrow{(D)} \partial   B(c^{-\frac{1}{3}}N) \big)  \\
  	\cdot & \mathbb{P}\big( \exists \ell \in \widetilde{\mathcal{L}}_{1/2}
  	^D\ \text{crossing}\  B(c^{-\frac{2}{3}}N)\setminus B(c^{-\frac{1}{3}}N)\ \text{and intersecting}\  D' \big)  \\
 \overset{(\ref{one_arm_high}),(\ref{ineqforlemma216})}{\lesssim}  & N^{d-6}\cdot \mathbb{P}\big( A\xleftrightarrow{(D)}x\big)  \sum_{w_3\in \mathbb{Z}^d:\widetilde{B}_{w_3}(1)\cap D'\neq \emptyset } \mathbb{P}\big( \exists \ell \in \widetilde{\mathcal{L}}_{1/2}
  	^D\ \text{intersecting}\  \partial B(c^{-\frac{2}{3}}N)\ \text{and}\  w_3 \big) \\  \lesssim &N^{d-6}\cdot \mathbb{P}\big( A\xleftrightarrow{(D)}x\big)   \sum_{w_3\in \mathbb{Z}^d:\widetilde{B}_{w_3}(1)\cap D'\neq \emptyset }  |w_3|^{2-d} \overset{(\ref{condition_new_stable})}{\lesssim }  c'\cdot \mathbb{P}\big( A\xleftrightarrow{(D)}x\big). 
  	\end{split}
  \end{equation*}

  Combining the estimates in \textbf{Case 1} and \textbf{Cases 2.1-2.5}, we obtain 
    \begin{equation}
    	\mathbb{P}(\mathsf{F})\lesssim (N^{6-d}+c')\cdot \mathbb{P}\big(A\xleftrightarrow{(D)}x \big). 
    \end{equation}
    By choosing $c'>0$ sufficiently small and assuming $N$ large enough, this  implies $\mathbb{P}(\mathsf{F})\le \tfrac{1}{2}\cdot \mathbb{P}\big(A\xleftrightarrow{(D)}x \big)$. Together with (\ref{findA2}), it yields the desired bound (\ref{ineq_new_stable}).  \qed

 \section{Proof of Lemma \ref{lemma_new_decompose_point_to_set}} \label{app_new_decompose_point_to_set}

As a preparation, we present the following extension of \cite[Lemma 2.1]{cai2024incipient}.

\begin{lemma}\label{lemmaB1}
	For any $d\ge 3$, $D\subset \widetilde{\mathbb{Z}}^d$ and $v \in \widetilde{\mathbb{Z}}^d\setminus D$, suppose that all GFF values on $D$ are given and are all non-negative. Then the probability of $\big\{v\xleftrightarrow{\ge 0}D\big\}$ is of the same order as $([\widetilde{G}_D(v,v)]^{-\frac{1}{2}}  \mathcal{H}_v)\land 1$, where 
	\begin{equation}
		\mathcal{H}_v:= \sum\nolimits_{w\in \widetilde{\partial}D} \widetilde{\mathbb{P}}_v(\tau_{D}=\tau_{w}<\infty)\widetilde{\phi}_w. 
	\end{equation} 
\end{lemma}
\begin{proof}
	Without loss of generality, we assume $\mathcal{H}_v>0$ (otherwise, for each $w\in \widetilde{\partial}D$, either $\widetilde{\phi}_w=0$ or there is no path starting from $v$ and first hitting $D$ at $w$; hence, the event $\big\{v\xleftrightarrow{\ge 0}D\big\}$ does not occur.) By \cite[(18)]{lupu2018random}, the probability of $\big\{v\xleftrightarrow{\ge 0}D\big\}$ can be equivalently written as  
\begin{equation}\label{use255}
	\begin{split}
	\mathbb{I}:=   \mathbb{E} \Big[ \mathbbm{1}_{\widetilde{\phi}_v\ge 0}\cdot \big(  1- e^{-2\widetilde{\phi}_v\sum_{w\in \widetilde{\partial}D} \mathbb{K}_{  D \cup \{v\} }(v,w)\widetilde{\phi}_w}  \big) \Big],
	\end{split}
\end{equation} 
	where the values of $\{\widetilde{\phi}_w\}_{w\in \widetilde{\partial}D}$ are given by the boundary condition on $D$, and $\widetilde{\phi}_v$ is distributed by $N(\mu,\sigma^2)$ with $\mu=\mathcal{H}_v$ and $\sigma^2=\widetilde{G}_D(v,v)$. Moreover, referring to \cite[Lemma 3.2]{inpreparation_twoarm}, there exist $C_\star>c_\star>0$ such that 
	\begin{equation}\label{use257}
	\begin{split}
c_\star \le 	\tfrac{\mathbb{K}_{  D \cup \{v\} }(v,w) }{  [\widetilde{G}_D(v,v) ]^{-1}\cdot \widetilde{\mathbb{P}}_v(\tau_{D}=\tau_{w}<\infty) }	 \le C_\star. 
	\end{split}
\end{equation}
 Plugging (\ref{use257}) into (\ref{use255}), we obtain that  
 \begin{equation}\label{greatB4}
\mathbb{J}(c_\star) \le  	\mathbb{I}  \le \mathbb{J}(C_\star),
 \end{equation}
 where the quantity $\mathbb{J}(a)$ for $a>0$ is defined by 
 \begin{equation}\label{niceB3}
 \begin{split}
  	\mathbb{J}(a):= & 	\int_{0}^{\infty} \frac{1}{ \sqrt{2\pi \sigma^2 }} e^{-\frac{(t-\mu)^2}{2\sigma^2}}\big(1-  e^{-a\sigma^{-2}  \mu   t   } \big)\mathrm{d}t \\
  	\overset{(t=\sigma s)}{=}&   \int_{0}^{\infty} \frac{1}{ \sqrt{2\pi   }} e^{-\frac{(s-\frac{\mu}{\sigma})^2}{2 }}\big(1-  e^{-a\sigma^{-1}  \mu  s   } \big)\mathrm{d}s. 
 \end{split}
 \end{equation}


	Next, we estimate $\mathbb{J}(c_\star)$ and $\mathbb{J}(C_\star)$ separately in two cases.

	 \textbf{When $0< \sigma^{-1}\mu \le \frac{1}{10}$.} On the one hand, one has 
\begin{equation}\label{niceB4}
	\begin{split}
		\mathbb{J}(c_\star) \gtrsim &    \int_{0}^{\frac{\sigma}{\mu}} \frac{1}{ \sqrt{2\pi   }} e^{-\frac{(s-\frac{\mu}{\sigma})^2}{2 }}\big(1-  e^{-c_{\star}\sigma^{-1}  \mu  s   } \big)\mathrm{d}s \\
		\gtrsim & \sigma^{-1}\mu   \int_{0}^{\frac{\sigma}{\mu}} \frac{s}{ \sqrt{2\pi   }} e^{-\frac{(s-\frac{\mu}{\sigma})^2}{2 }} \mathrm{d}s\\
		\overset{(s=r+\frac{\mu}{\sigma})}{\ge} &  \sigma^{-1}\mu   \int_{-\frac{\mu}{\sigma}}^{\frac{\sigma}{\mu}-\frac{\mu}{\sigma}} \frac{r}{ \sqrt{2\pi   }} e^{-\frac{r^2}{2 }} \mathrm{d}s \gtrsim  \sigma^{-1}\mu, 
	\end{split}
\end{equation}
where in the last inequality we used the fact that $\frac{\sigma}{\mu}\ge \frac{3\mu}{\sigma}\vee 10$ for all $0\le \mu\le \frac{\sigma}{10}$. On the other hand, since $1-e^{-b}\le b$ for $b\ge 0$, one has 
\begin{equation}\label{greatB7}
	\begin{split}
		\mathbb{J}(C_\star)\lesssim & \sigma^{-1}  \mu \int_{0}^{\infty} \frac{s}{ \sqrt{2\pi   }} e^{-\frac{(s-\frac{\mu}{\sigma})^2}{2 }} \mathrm{d}s \\
		\overset{(r=s-\frac{\mu}{\sigma})}{=}& \sigma^{-1}  \mu \Big( \int_{0}^{\infty} \frac{r}{ \sqrt{2\pi   }} e^{-\frac{r^2}{2 }} \mathrm{d}r + \sigma^{-1}  \mu  \int_{0}^{\infty} \frac{1}{ \sqrt{2\pi   }} e^{-\frac{r^2}{2 }} \mathrm{d}r \Big)\\
		\lesssim & \sigma^{-1}  \mu  \big(1+ \sigma^{-1}  \mu  \big) 	\lesssim   \sigma^{-1}  \mu . 
	\end{split}
\end{equation}
Combining (\ref{greatB4}), (\ref{niceB4}) and (\ref{greatB7}), we derive $\mathbb{I}\asymp \sigma^{-1}  \mu$ in this case.

 \textbf{When $ \sigma^{-1}\mu \ge  \frac{1}{10}$.} Note that $\mathbb{J}(C_\star)\le 1$ holds trivially. Moreover, since $1-  e^{-c_{\star}\sigma^{-1}  \mu  s }\gtrsim 1$ for all $s\ge 1$ in this case, we have 
 \begin{equation}\label{niceB5}
 	\begin{split}
 		\mathbb{J}(c_\star ) \gtrsim \int_{1}^{\infty} \frac{1}{ \sqrt{2\pi   }} e^{-\frac{(s-\frac{\mu}{\sigma})^2}{2 }} \mathrm{d}s \gtrsim 1. 
 	\end{split}
 \end{equation}
 Combining the bounds on $\mathbb{J}(C_\star)$ and $\mathbb{J}(c_\star )$ with (\ref{greatB4}), we conclude that $\mathbb{I}\asymp 1$, which completes the proof of this lemma.
  \end{proof}


Next, we establish Lemma \ref{lemma_new_decompose_point_to_set} for $3\le d\le 5$ and $d\ge 7$ separately.

\textbf{When $3\le d\le 5$.} By the isomorphism theorem, it suffices to prove that 
\begin{equation}\label{newgreatB9}
		\mathbb{P}^D\big( A \xleftrightarrow{\ge 0} y\big) \gtrsim R^{d-2}\cdot \mathbb{P}^D\big( A \xleftrightarrow{\ge 0 } x \big) \cdot  \mathbb{P}^D\big( x\xleftrightarrow{\ge 0} y\big). 
\end{equation}
Since $A\subset \mathfrak{B}^1_{R,c}$, the isomorphism theorem together with (\ref{crossing_low}) implies that 
\begin{equation}
\begin{split}
		  \mathbb{P}^D\big(A\xleftrightarrow{\le 0}  \widetilde{\partial} \mathfrak{B}^1_{R,(10d)^{-1}}\big)  
		 \le  & 	\mathbb{P}\big(\mathfrak{B}^1_{R,c} \xleftrightarrow{}  \widetilde{\partial} \mathfrak{B}^1_{R,(10d)^{-1}}\big) \\
		  \lesssim &  \rho_d(cR,(10d)^{-1}R) \lesssim c^{\frac{d}{2}-1}. 
\end{split} 
\end{equation}
Therefore, for all sufficiently small $c>0$, one has 
\begin{equation}\label{niceB8}
	\mathbb{P}^D\big(\mathsf{F}\big):= \mathbb{P}^D\big(\{A\xleftrightarrow{\le 0} \widetilde{\partial} \mathfrak{B}^1_{R,(10d)^{-1}}\}^c\big)\le \tfrac{1}{2}.  
\end{equation}
 (\textbf{P.S.} When $d\ge 7$, the estimate (\ref{niceB8}) fails for $A=\mathfrak{B}^1_{R,c}$ and $D= \emptyset$; see (\ref{crossing_high}). This is why we need to treat the high-dimensional case individually.)

By the FKG inequality and (\ref{niceB8}), we have 
  	 \begin{equation} 
 	\begin{split}
 		  \mathbb{P}^{D}\big( \mathsf{F}   \mid  A \xleftrightarrow{\ge 0} x  \big)  
 	\overset{(\text{FKG})}{\le}  \mathbb{P}^{D}\big( \mathsf{F}  \big) \overset{(\mathrm{\ref{niceB8}}) }{\le } \tfrac{1}{2}, 
 	\end{split}
 \end{equation}
which further implies that 
\begin{equation}\label{newA7}
	 \mathbb{P}^{D}(A  \xleftrightarrow{\ge 0} x, \mathsf{F}) \asymp 	\mathbb{P}^D \big(  A  \xleftrightarrow{\ge 0 } x \big).  
\end{equation}

 
 Recall that $\mathcal{C}_A^-:=\{v\in \widetilde{\mathbb{Z}}^d: v\xleftrightarrow{\le 0} A\}$. For each $v\in \widetilde{\mathbb{Z}}^d$, we define  
 \begin{equation}\label{greatB14}
 		\mathcal{H}_v:= \left\{
\begin{aligned}
&\sum\nolimits_{w\in \widetilde{\partial} \mathcal{C}_A^- } \widetilde{\mathbb{P}}_v\big( \tau_{\widetilde{\partial} \mathcal{C}_A^-} = \tau_w <\tau_D \big)  \widetilde{\phi}_w \ \  &      v\notin \mathcal{C}_A^-; \\ 
&0 & v\in \mathcal{C}_A^-.
\end{aligned}
\right.
 \end{equation} 
Conditioned on $\mathcal{F}_{\mathcal{C}_A^-}$, the boundary condition is non-negative, and only takes positive values on $A$. In addition, the event $\mathsf{F}$ implies that $\mathcal{C}_A^- \subset \widetilde{B}(\frac{N}{10d})$, and hence,  
  \begin{equation}\label{greatB16}
 	\widetilde{G}_{D\cup \mathcal{C}_A^-}(y,y)  \overset{(\ref{order_green_function})}{\asymp}     	\widetilde{G}_{D}(y,y),
 \end{equation} 
 \begin{equation}\label{newgreatB16}
 	\widetilde{G}_{D\cup \mathcal{C}_A^-}(x,x)\asymp 1. 
 \end{equation}
 Moreover, by the strong Markov property and Harnack's inequality, it follows that on the event $\mathsf{F}$, 
  \begin{equation}\label{greatB15}
 	\begin{split}
 		\mathcal{H}_y \gtrsim  \widetilde{\mathbb{P}}_y\big( \tau_{\partial B(10R)}<\tau_{D} \big) \cdot   \mathcal{H}_x \overset{}{\gtrsim} R^{d-2}  \widetilde{G}_D(y,x) \cdot   \mathcal{H}_x. 
 	\end{split}
 \end{equation} 
 Thus, by applying Lemma \ref{lemmaB1}, we have 
 \begin{equation}\label{greatB18}
 	\begin{split}
 		\mathbb{P}^D\big( A  \xleftrightarrow{\ge 0} y  \big)  \gtrsim &  \mathbb{E}^D\big[ \mathbbm{1}_{\mathsf{F}}\cdot  \big( [\widetilde{G}_{D\cup \mathcal{C}_A^-}(y,y)]^{-\frac{1}{2}} \mathcal{H}_y  \big) \land 1  \big]\\
 		\overset{(\mathrm{\ref{greatB16}}),(\mathrm{\ref{greatB15}}),\widetilde{G}_D(x,x)\asymp 1}{\gtrsim } &  \mathbb{E}^D\big[ \mathbbm{1}_{\mathsf{F}}\cdot  \big( R^{d-2}\cdot  \tfrac{\widetilde{G}_D(y,x)}{\sqrt{\widetilde{G}_D(x,x)\widetilde{G}_D(y,y)}} \cdot  \mathcal{H}_x  \big) \land 1  \big]  \\
 		\overset{  }{\gtrsim }  &  R^{d-2}\cdot \mathbb{P} (x\xleftrightarrow{(D)} y) \cdot \mathbb{E}\big[  \mathbbm{1}_{\mathsf{F}}\cdot  \big(   \mathcal{H}_x\land 1  \big)  \big],
 	\end{split}
 \end{equation}
 where in the last inequality we used the fact that 
 \begin{equation}
 	R^{d-2}\cdot \tfrac{\widetilde{G}_D(y,x)}{\sqrt{\widetilde{G}_D(x,x)\widetilde{G}_D(y,y)}} \overset{(\ref{211})}{\asymp } R^{d-2}\cdot  \mathbb{P} (x\xleftrightarrow{(D)} y) \overset{(\ref{two-point1})}{\lesssim }  1. 
 \end{equation}
 Meanwhile, using (\ref{newA7}), (\ref{newgreatB16}) and Lemma \ref{lemmaB1}, one has 
 \begin{equation}\label{niceB20}
 	\begin{split}
 		\mathbb{E}^D\big[  \mathbbm{1}_{\mathsf{F}}\cdot  \big(   \mathcal{H}_x\land 1  \big)  \big]  \overset{ ( \mathrm{\ref{newgreatB16}}) }{ \asymp } & \mathbb{E}^D\big[  \mathbbm{1}_{\mathsf{F}}\cdot  \big(  [\widetilde{G}_{D\cup \mathcal{C}_A^-}(x,x)]^{-\frac{1}{2}} \mathcal{H}_x  \big) \land 1  \big]   \\
 		\overset{ ( \text{Lemma}\ \mathrm{\ref{lemmaB1}}) }{ \asymp }   & \mathbb{P}^D\big(A \xleftrightarrow{\ge 0} x, \mathsf{F}  \big)    \overset{ ( \mathrm{\ref{newA7}}) }{ \asymp }    \mathbb{P}^D\big(A \xleftrightarrow{\ge 0} x   \big). 
 	\end{split}
 \end{equation}
 Combined with (\ref{greatB18}), it implies the desired bound (\ref{newgreatB9}) for $3\le d\le 5$.

\textbf{When $d\ge 7$.} In this case, by the isomorphism theorem, it suffices to show that 
\begin{equation}\label{weB21}
	\mathbb{P}^D\big( A \xleftrightarrow{\ge 0} y\big) \gtrsim R^{d-2}\cdot \mathbb{P}^D\big( A \xleftrightarrow{\ge 0} x \big) \cdot \big[  \mathbb{P}^D\big( x\xleftrightarrow{\ge 0} y\big)-  \tfrac{C[\mathrm{diam}(A)]^{d-4}}{R^{d-4}|y|^{d-2}}   \big].
\end{equation}
We denote $\chi:= \mathrm{dist}(y,D)\land 1$, and define $\mathsf{G}:= \big\{ \mathrm{dist}( \mathcal{C}_A^-,y)\ge  \tfrac{1}{2}\chi \big\}$. Note that on the event $\mathsf{G}$, it follows from (\ref{order_green_function}) that 
\begin{equation}\label{greatB21}
	\widetilde{G}_{D\cup \mathcal{C}_A^-}(y,y)\asymp \chi \asymp \widetilde{G}_{D}(y,y)
\end{equation}
Using Lemma \ref{lemmaB1}, we have  
\begin{equation}\label{coolB21}
\mathbb{P}^D\big( A \xleftrightarrow{\ge 0} y\big) \gtrsim  \mathbb{E}^D\big[\mathbbm{1}_{\mathsf{G}} \cdot  \big(  [\widetilde{G}_{D\cup \mathcal{C}_A^-}(y,y)]^{-\frac{1}{2}} \mathcal{H}_y  \big) \land 1  \big]. 
\end{equation}

Next, we aim to give a lower bound for $\mathcal{H}_y$. To this end, we need a lower bound on the probability $\widetilde{\mathbb{P}}_y(\tau_{\mathcal{C}_A^- }=\tau_{w}<\tau_{D} ) $ for $w\in \widetilde{\partial}\mathcal{C}_A^-$. Recall that $A\subset \mathfrak{B}^1_{R,c}$. For simplicity, we denote $\mathcal{C}:=\mathcal{C}_A^-$ and $\widehat{\mathcal{C}}:=\mathcal{C}  \cap  \widetilde{B}(\frac{R}{10})$. For any $w\in A$, one has 
\begin{equation}\label{greatB23}
	\begin{split}
		\widetilde{\mathbb{P}}_y\big( \tau_{\mathcal{C}} =\tau_{w}<\tau_{D} \big)= \widetilde{\mathbb{P}}_y\big( \tau_{\widehat{\mathcal{C}}} =\tau_{w}<\tau_{D} \big)  -  \widetilde{\mathbb{P}}_y\big( \tau_{\mathcal{C}\setminus \widetilde{B}(\frac{R}{10})} <\tau_{\widehat{\mathcal{C}}} =\tau_{w}<\tau_{D} \big). 
	\end{split}
\end{equation}
 The first probability on the right-hand side is bounded from below by 
 \begin{equation}\label{greatB24}
 	\begin{split}
   \widetilde{\mathbb{P}}_y\big( \tau_{\partial B(10dR)}  <\tau_{D} \big)  \cdot \min_{z\in \partial B(10dR)} \sum\nolimits_{z'\in \partial \mathcal{B}(R)} & \widetilde{\mathbb{P}}_z \big( \tau_{\partial \mathcal{B}(R)} =\tau_{z'}<\tau_{D} \big) \\
&   \cdot  \widetilde{\mathbb{P}}_{z'}\big( \tau_{\widehat{\mathcal{C}}} =\tau_{w}<\tau_{D} \big). 
 	\end{split}
 \end{equation}
  By the strong Markov property of Brownian motion, one has 
 \begin{equation}\label{greatB25}
 	\begin{split}
 \widetilde{\mathbb{P}}_y\big( \tau_{x}<\tau_{D}  \big) 	 \le & \widetilde{\mathbb{P}}_y\big( \tau_{\partial B(10dR)}  <\tau_{D} \big) \cdot \max_{z\in \partial B(10dR)}  \widetilde{\mathbb{P}}_z\big( \tau_{x}<\tau_{D} \big) \\
 		\lesssim &R^{2-d}\cdot \widetilde{\mathbb{P}}_y\big( \tau_{\partial B(10dR)}  <\tau_{D} \big). 
 	\end{split}
 \end{equation}
  Meanwhile, for any $z\in \partial B(10dR)$ and $z'\in \partial \mathcal{B}(R)$, we have 
 \begin{equation}\label{greatB26}
	\begin{split}
 	 \widetilde{\mathbb{P}}_z \big( \tau_{\partial \mathcal{B}(R)} =\tau_{z'}<\tau_{D} \big) \overset{(\text{Lemma}\ \ref{lemma_BM_stable})}{\asymp}  \widetilde{\mathbb{P}}_z \big( \tau_{\partial \mathcal{B}(R)} =\tau_{z'}<\infty \big) \overset{(\text{Lemma}\ \ref{lemma_BM_uniform})}{\asymp}  R^{1-d}. 
	\end{split}
\end{equation}
 In addition, using Harnack's inequality, we have 
\begin{equation}\label{greatB27}
	 \begin{split}
	 	 \widetilde{\mathbb{P}}_{z'}\big( \tau_{\widehat{\mathcal{C}}} =\tau_{w}<\tau_{D} \big)  \asymp \widetilde{\mathbb{P}}_{x}\big( \tau_{\widehat{\mathcal{C}}} =\tau_{w}<\tau_{D} \big) \ge \widetilde{\mathbb{P}}_{x}\big( \tau_{\mathcal{C}} =\tau_{w}<\tau_{D} \big). 
	 \end{split}
\end{equation}
 Substituting (\ref{greatB25}), (\ref{greatB26}) and (\ref{greatB27}) into (\ref{greatB24}), we get 
\begin{equation}\label{greatB28}
	\begin{split}
		\widetilde{\mathbb{P}}_y\big( \tau_{\widehat{\mathcal{C}}} =\tau_{w}<\tau_{D} \big)  \gtrsim R^{d-2}\cdot  \widetilde{\mathbb{P}}_y\big( \tau_{x}<\tau_{D} \big)  \cdot \widetilde{\mathbb{P}}_{x}\big( \tau_{\mathcal{C}} =\tau_{w}<\tau_{D} \big).
	\end{split}
\end{equation}

For the second probability on the right-hand side of (\ref{greatB23}), by the union bound, 
\begin{equation}\label{greatB29}
\begin{split}
		&	\widetilde{\mathbb{P}}_y\big( \tau_{\mathcal{C}\setminus \widetilde{B}(\frac{R}{10})} <\tau_{\widehat{\mathcal{C}}} =\tau_{w}<\tau_{D} \big) \\
		 \le  &  \sum\nolimits_{z\in  [B(\frac{R}{10})]^c} \mathbbm{1}_{A\xleftrightarrow{\le 0} \widetilde{B}_{z}(1)}   \cdot  \widetilde{\mathbb{P}}_y\big( \tau_{z} <\tau_{\widehat{\mathcal{C}}} =\tau_{w}<\tau_{D} \big). 
\end{split}
\end{equation}
Moreover, for each $z\in [B(\frac{R}{10})]^c$, one has  
\begin{equation}\label{greatB30}
	\begin{split}
		\widetilde{\mathbb{P}}_y\big( \tau_{z} <\tau_{\widehat{\mathcal{C}}} =\tau_{w}<\tau_{D} \big) \le &  \widetilde{\mathbb{P}}_y\big( \tau_{z} <\tau_D \big)\cdot \widetilde{\mathbb{P}}_z\big(\tau_{\widehat{\mathcal{C}}} =\tau_{w}<\tau_{D} \big).
	\end{split}
\end{equation}
In addition, by the strong Markov property of Brownian motion, we have 
\begin{equation}\label{newweB32}
	\begin{split}
		\widetilde{\mathbb{P}}_y\big( \tau_{z} <\tau_D   \big)\le & \widetilde{\mathbb{P}}_y\big( \tau_{\widetilde{\partial}\widetilde{B}_y(1)} <\tau_D   \big) \cdot \max_{y'\in \widetilde{\partial}\widetilde{B}_y(1)} \widetilde{\mathbb{P}}_{y'}\big( \tau_{z} <\tau_D   \big)\\
		\lesssim &\big( \mathrm{dist}(y,D)\land 1 \big)\cdot |y-z|^{2-d} \overset{(\ref{order_green_function})}{\asymp }\widetilde{G}_D(y,y)\cdot  |y-z|^{2-d},
	\end{split}
\end{equation}
\begin{equation}\label{weB32}
	\begin{split}
		\widetilde{\mathbb{P}}_z\big(\tau_{\widehat{\mathcal{C}}} =\tau_{w}<\tau_{D} \big)\lesssim &	\widetilde{\mathbb{P}}_z \big( \tau_{ \partial B(\frac{R}{100})  }<\infty \big)\cdot \max_{z'\in \partial B(\frac{R}{100})} \sum\nolimits_{z''\in \partial \mathcal{B}( \sqrt{c}R)} \\
		&\ \ \ \ \ \ \ \ \  \widetilde{\mathbb{P}}_{z'} \big( \tau_{ \partial \mathcal{B}(\sqrt{c}R) } = \tau_{z''}<\infty \big) \cdot  \widetilde{\mathbb{P}}_{z''} \big(\tau_{\widehat{\mathcal{C}}} =\tau_{w}<\tau_{D} \big), 
	\end{split}
\end{equation}
where the constant $c>0$ is the same as in the condition $A\subset \widetilde{B}(cR)$. Here one may derive from (\ref{cap1}) and (\ref{cap2}) that 
\begin{equation}\label{weB33}
		\widetilde{\mathbb{P}}_z \big( \tau_{ \partial B(\frac{R}{100})  }<\infty \big) \lesssim |z|^{2-d}R^{d-2}. 
\end{equation}
Meanwhile, Lemma \ref{lemma_BM_uniform} implies that 
\begin{equation}\label{weB34}
	\widetilde{\mathbb{P}}_{z'} \big( \tau_{ \partial \mathcal{B}(\sqrt{c}R) } = \tau_{z''}<\infty \big) \lesssim R^{1-d}. 
\end{equation}
 Applying Lemma \ref{lemma_new_ BM_stable}, one has 
 \begin{equation}\label{weB35}
 	  \sum\nolimits_{z''\in \partial \mathcal{B}(\sqrt{c}R)}  \widetilde{\mathbb{P}}_{z''} \big(\tau_{\widehat{\mathcal{C}}} =\tau_{w}<\tau_{D} \big) \lesssim 	  \sum\nolimits_{z''\in \partial \mathcal{B}(\sqrt{c}R)}  \widetilde{\mathbb{P}}_{z''} \big(\tau_{\mathcal{C}} =\tau_{w}<\tau_{D} \big). 
 \end{equation}
 Inserting (\ref{weB33}), (\ref{weB34}) and (\ref{weB35}) into (\ref{weB32}), we get 
 \begin{equation}
 	\widetilde{\mathbb{P}}_z\big(\tau_{\widehat{\mathcal{C}}} =\tau_{w}<\tau_{D} \big) \lesssim  |z|^{2-d}R^{-1} \sum\nolimits_{z''\in \partial \mathcal{B}(\sqrt{c}R)}  \widetilde{\mathbb{P}}_{z''} \big(\tau_{\mathcal{C}} =\tau_{w}<\tau_{D} \big).  
 \end{equation}
 This together with (\ref{greatB29}), (\ref{greatB30}) and (\ref{newweB32}) yields that 
 \begin{equation}\label{greatB32}
 	\begin{split}
 			  \widetilde{\mathbb{P}}_y\big( \tau_{\mathcal{C}\setminus \widetilde{B}(\frac{R}{10})} <\tau_{\widehat{\mathcal{C}}} =\tau_{w}<\tau_{D} \big)  
 			\lesssim &  R^{-1}\widetilde{G}_D(y,y)  \sum_{z\in  [B(\frac{R}{10})]^c,z''\in \partial \mathcal{B}(\sqrt{c}R)} \mathbbm{1}_{A\xleftrightarrow{\le 0} \widetilde{B}_{z}(1)}  \\
 		 &\ \ \ \ \ \ \ \	\cdot    |y-z|^{2-d}|z|^{2-d} \widetilde{\mathbb{P}}_{z''} \big(\tau_{\mathcal{C}} =\tau_{w}<\tau_{D} \big).
 	\end{split}
 \end{equation}

Recall $\mathcal{H}_\cdot$ in (\ref{greatB14}). Plugging (\ref{greatB28}) and (\ref{greatB32}) into (\ref{greatB23}), and then using the definition of $\mathcal{H}_\cdot$ in (\ref{greatB14}), we obtain that for some constants $C',c'>0$,
\begin{equation}
	\begin{split}
		\mathcal{H}_y \ge & c'R^{d-2}  \widetilde{\mathbb{P}}_y\big( \tau_{x}<\tau_{D} \big)   \mathcal{H}_x \\
		& - C' R^{-1} \widetilde{G}_D(y,y)  \sum_{z\in  [B(\frac{R}{10})]^c,z''\in \partial \mathcal{B}(\sqrt{c}R)}   \mathbbm{1}_{A\xleftrightarrow{\le 0} \widetilde{B}_{z}(1)}    |y-z|^{2-d}|z|^{2-d}  \mathcal{H}_{z''}.
	\end{split}
\end{equation}
Combined with (\ref{coolB21}), it implies that  
\begin{equation}\label{greatB33}
	\begin{split}
		\mathbb{P}^D\big( A \xleftrightarrow{\ge 0} y\big) \ge \mathbb{V}_1-\mathbb{V}_2, 
	\end{split}
\end{equation}
 where the quantities $\mathbb{V}_1$ and $\mathbb{V}_2$ satisfy that 
 \begin{equation*}
 	\mathbb{V}_1\gtrsim  \mathbb{E}^D\big[ \mathbbm{1}_{\mathsf{G}} \cdot  \big(  [\widetilde{G}_{D\cup \mathcal{C}}(y,y)]^{-\frac{1}{2}}R^{d-2}  \widetilde{\mathbb{P}}_y\big( \tau_{x}<\tau_{D} \big)   \mathcal{H}_x  \big) \land 1 \big], 
 \end{equation*}
 \begin{equation*}
 	\begin{split}
 		\mathbb{V}_2\lesssim & \mathbb{E}^D\big[  \mathbbm{1}_{\mathsf{G}} \cdot  \big(  [\widetilde{G}_{D\cup \mathcal{C}}(y,y)]^{-\frac{1}{2}}R^{-1}   \widetilde{G}_D(y,y)\sum\nolimits_{z\in  [B(\frac{R}{10})]^c,z''\in \partial \mathcal{B}(\frac{R}{100d})}   \mathbbm{1}_{A\xleftrightarrow{\le 0} \widetilde{B}_{z}(1)} \\
 	 &\ \ \ \ \ \ \ \ \ \ \ \ \ \ \ \ \ \ \ \ \ \ \ \ \ \ \ \ \ \ \ \ \ \ \ \ \ \ \ \ \ \ \ \ \ \ \ \ \ \ \ \ \ \ \ \ \ \ 	\cdot  |y-z|^{2-d}|z|^{2-d} \mathcal{H}_z \big) \land 1 \big]. 
 	\end{split}
 \end{equation*}

 For $\mathbb{V}_1$, by $\widetilde{G}_{D\cup \mathcal{C}}(y,y)\le \widetilde{G}_{D}(y,y)$, $\widetilde{G}_{D}(x,x)\asymp 1$ and (\ref{211}), we have 
 \begin{equation}\label{weB40}
 \begin{split}
 	 		\mathbb{V}_1 \gtrsim &  \mathbb{E}^D\big[ \mathbbm{1}_{\mathsf{G}} \cdot  \big(   R^{d-2} \mathbb{P}^{D}\big( x\xleftrightarrow{\ge 0} y \big) \cdot    \mathcal{H}_x  \big) \land 1 \big] \\
 		\gtrsim & R^{d-2} \mathbb{P}^{D}\big( x\xleftrightarrow{\ge 0} y \big)  \cdot  \mathbb{E}^D\big[ \mathbbm{1}_{\mathsf{G}} \cdot     \mathcal{H}_x  \land 1   \big], 
 \end{split}
 \end{equation}
 where the second inequality relies on the fact that $R^{d-2} \mathbb{P}^{D}\big( y\xleftrightarrow{\ge 0} x \big)\lesssim 1$ (by (\ref{two-point1})). Note that the event $\mathsf{G'}:=\big\{ \mathrm{dist}( \mathcal{C}, x)\ge  1 \big\}$ implies $\widetilde{G}_{D\cup \mathcal{C}}(x,x)\asymp 1$, and satisfies 
 \begin{equation}\label{verynewB42}
 	\mathbb{P}^D\big( \mathsf{G}'\big) \ge \mathbb{P}^D\big( \cap_{z\in \widetilde{\partial} \widetilde{B}_x(1)}\big\{x\xleftrightarrow{\ge 0}z \big\} \big) \gtrsim 1. 
 \end{equation}
It follows from the inclusion $\mathsf{G'}\subset \{\widetilde{G}_{D\cup \mathcal{C}}(x,x)\asymp 1\}$ that  
 \begin{equation}\label{verynew_B43}
 	\begin{split}
 		\mathbb{E}^D\big[ \mathbbm{1}_{\mathsf{G}} \cdot     \mathcal{H}_x  \land 1   \big]   \gtrsim   &		\mathbb{E}^D\big[ \mathbbm{1}_{\mathsf{G}\cap \mathsf{G}'} \cdot   \big( [\widetilde{G}_{D\cup \mathcal{C}}(x,x)]^{-\frac{1}{2}}   \mathcal{H}_x   \big) \land 1 \big]\\
 		\overset{(\text{FKG})}{\gtrsim  } &	\mathbb{E}^D\big[\big( [\widetilde{G}_{D\cup \mathcal{C}}(x,x)]^{-\frac{1}{2}}   \mathcal{H}_x   \big) \land 1 \big] \cdot \mathbb{P}^D\big( \mathsf{G}\big)\cdot 	\mathbb{P}^D\big( \mathsf{G}'\big)\\
 		\overset{(\text{Lemma}\ \mathrm{\ref{lemmaB1}}),(\mathrm{\ref{verynewB42}}) }{\gtrsim}  & \mathbb{P}^D\big( A \xleftrightarrow{\ge 0} x \big)      \cdot \mathbb{P}^D\big( \mathsf{G}\big), 
 	\end{split}
 \end{equation}
 where the use of the FKG inequality is based on the fact that $\mathsf{G}$, $\mathsf{G}'$, $[\widetilde{G}_{D\cup \mathcal{C}}(x,x)]^{-\frac{1}{2}}$ and $\mathcal{H}_x $ are all measurable with respect to $\mathcal{F}_{\mathcal{C}}$, and decreasing with respect to the GFF. Plugging (\ref{verynew_B43}) into (\ref{weB40}), we get 
 \begin{equation}\label{weB42}
 \begin{split}
 	\mathbb{V}_1 \gtrsim & R^{d-2} \mathbb{P}^{D}\big( x\xleftrightarrow{\ge 0} y \big)       \cdot \mathbb{P}^D\big( A \xleftrightarrow{\ge 0} x \big)      \cdot \mathbb{P}^D\big( \mathsf{G}\big). 
 \end{split}
 \end{equation}
 For $\mathbb{P}^D\big( \mathsf{G}\big)$, by the FKG inequality, (\ref{order_green_function}) and (\ref{211}), we have 
\begin{equation}\label{greatB37}
\begin{split}
		\mathbb{P}^D\big( \mathsf{H} \big):=& \mathbb{P}^D\big( \cap_{z\in \widetilde{\mathbb{Z}}^d:|z-y|=\frac{1}{2}\chi} \big\{z\xleftrightarrow{\le 0} y  \big\} \big) \\
		 \overset{(\text{FKG})}{\ge}  & \prod\nolimits_{z\in \widetilde{\mathbb{Z}}^d:|z-y|=\frac{1}{2}\chi}\mathbb{P}^D\big(  z\xleftrightarrow{\le 0} y  \big)\\
		 \overset{(\ref{211})}{\asymp }  & \prod\nolimits_{z\in \widetilde{\mathbb{Z}}^d:|z-y|=\frac{1}{2}\chi} \widetilde{\mathbb{P}}_{y}\big( \tau_{z}<\tau_{D} \big) \cdot \sqrt{\tfrac{\widetilde{G}_D(z,z)}{\widetilde{G}_D(y,y)}} \gtrsim 1,  
\end{split}
\end{equation}
where the last inequality follows from the facts that $\widetilde{\mathbb{P}}_{y}\big( \tau_{z}<\tau_{D} \big)\asymp 1$ (by $|y-z|\le \frac{1}{2}\mathrm{dist}(y,D)$), and that $\widetilde{G}_D(y,y)\asymp \widetilde{G}_D(z,z)$ (using (\ref{order_green_function})). In addition, since $\mathsf{H}\cap \mathsf{G}^c\subset \{A\xleftrightarrow{\le 0} y\}$, one has 
\begin{equation}\label{greatB38}
	\begin{split}
	\mathbb{P}^D\big(	A\xleftrightarrow{\ge 0} y \big) \overset{(\text{symmetry})}{ =	}\mathbb{P}^D\big(	A\xleftrightarrow{\le 0} y \big) \overset{(\text{FKG})}{\ge} \mathbb{P}^D\big( \mathsf{G}^c \big) \cdot \mathbb{P}^D\big( \mathsf{H} \big) \overset{(\mathrm{\ref{greatB37}})}{\gtrsim } \mathbb{P}^D\big( \mathsf{G}^c \big). 
	\end{split}
\end{equation}
 Combined with (\ref{weB42}), it implies that 
  \begin{equation}\label{greatB39}
 	\mathbb{V}_1 \gtrsim \big[1-C'' \mathbb{P}^D\big(	A\xleftrightarrow{\ge 0} y \big) \big] \cdot R^{d-2} \mathbb{P}^D\big( A \xleftrightarrow{\ge 0 } x \big)\mathbb{P}^D\big( x\xleftrightarrow{\ge 0} y\big). 
 \end{equation}

We now bound the quantity $\mathbb{V}_2$. By (\ref{greatB21}) and $\widetilde{G}_{D}(y,y)\lesssim 1$ one has 
\begin{equation}\label{weB46}
	\begin{split}
			\mathbb{V}_2 \lesssim & R^{-1}  \sum\nolimits_{z\in  [B(\frac{R}{10})]^c, z''\in \partial \mathcal{B}(\sqrt{c}R)}    |y-z|^{2-d}|z|^{2-d}  \cdot  \mathbb{E}^D\big[   \mathbbm{1}_{A\xleftrightarrow{\le 0} \widetilde{B}_{z}(1)} \cdot   \mathcal{H}_{z''}    \big].
	\end{split}
\end{equation}
 In addition, since $\big\{ A\xleftrightarrow{\le 0} \widetilde{B}_{z}(1)\big\}$ (resp. $\mathcal{H}_{z'} $) is decreasing (resp. increasing) with respect to the GFF, using the FKG inequality, we have 
  \begin{equation}\label{greatB41}
 	\begin{split}
 		\mathbb{E}^D\big[   \mathbbm{1}_{A\xleftrightarrow{\le 0} \widetilde{B}_{z}(1)} \cdot   \mathcal{H}_{z''}    \big] \le \mathbb{P}^D\big(A\xleftrightarrow{\le 0} \widetilde{B}_{z}(1) \big) \cdot  \mathbb{E}^D\big[ \mathcal{H}_{z''} \big]. 
 	\end{split}
 \end{equation}
  Recall that $A\subset \widetilde{B}(cR)$. By the isomorphism theorem, one has: for $z\in  [B(\frac{R}{10})]^c$, 
   \begin{equation}\label{greatB42}
 	\begin{split}
 		\mathbb{P}^D\big(A\xleftrightarrow{\le 0} \widetilde{B}_{z}(1) \big) \le  \sum\nolimits_{z'\in \widetilde{\partial} \widetilde{B}_z(1)} \mathbb{P}\big(A \xleftrightarrow{ }z' \big) \lesssim |z|^{2-d}[\mathrm{diam}(A)]^{d-4},
 	\end{split}
 \end{equation}
 where the last inequality is derived from \cite[Corollary 2.18]{cai2024quasi}. Meanwhile, it follows from \cite[Lemma 3.4]{cai2024one} that 
 \begin{equation}\label{greatB43}
 	\mathbb{E}^D\big[ \mathcal{H}_{z''} \big] \lesssim \mathbb{P}^D\big( A\xleftrightarrow{\ge 0} z'' \big) \overset{(\text{Lemma}\ \ref{lemma_point_to_set})}{\asymp}\mathbb{P}^D\big( A\xleftrightarrow{\ge 0} x \big). 
 \end{equation}
  Plugging (\ref{greatB42}) and (\ref{greatB43}) into (\ref{greatB41}), and combining it with (\ref{weB46}), we get 
 \begin{equation}\label{greatB44}
 		\mathbb{V}_2 \lesssim R^{d-2}[\mathrm{diam}(A)]^{d-4} \mathbb{P}^D\big( A \xleftrightarrow{\ge 0 } x \big) \sum\nolimits_{z\in  [B(\frac{R}{10})]^c } |y-z|^{2-d}|z|^{4-2d}. 
 	 \end{equation}
 	 For the sum on the right-hand side, since $|z|\gtrsim |y|$ for all $z\in B_y(\frac{|y|}{10})$, we have 
 	 \begin{equation}\label{weB52}
 	 	\sum\nolimits_{z\in B_y(\frac{|y|}{10})} |y-z|^{2-d}|z|^{4-2d} \lesssim |y|^{4-2d} 	\sum\nolimits_{z\in B_y(\frac{|y|}{10})} |y-z|^{2-d} \overset{(\ref{computation_d-a})}{\lesssim } |y|^{6-2d}. 
 	 \end{equation}
 	 In addition, since $|z-y|\gtrsim |y|$ for all $z\in [B(\frac{R}{10})]^c \setminus  B_y(\frac{|y|}{10}) $, one has 
 	 \begin{equation}\label{weB53}
 	 \begin{split}
 	 	 	 	&\sum\nolimits_{z\in [B(\frac{R}{10})]^c \setminus  B_y(\frac{|y|}{10}) }   |y-z|^{2-d}|z|^{4-2d} \\
 	 	 	 	 \lesssim & |y|^{2-d} \sum\nolimits_{z\in [B(\frac{R}{10})]^c } |z|^{4-2d} \lesssim |y|^{2-d}\sum\nolimits_{k\ge \frac{R}{10}} k^{d-1}\cdot k^{4-2d}\lesssim |y|^{2-d}R^{4-d}.  
 	 \end{split}
 	 \end{equation}
 	 Note that $|y|^{6-2d}\lesssim |y|^{2-d}R^{4-d}$ since $|y|\gtrsim R$. Therefore, combining (\ref{greatB44}), (\ref{weB52}) and (\ref{weB53}), we obtain 
 	 \begin{equation}\label{weB54}
 	 	\begin{split}
 	 		 \mathbb{V}_2 \lesssim R^{2}[\mathrm{diam}(A)]^{d-4}|y|^{2-d}\mathbb{P}^D\big( A \xleftrightarrow{\ge 0} x \big). 
 	 	\end{split}
 	 \end{equation}

 By (\ref{greatB33}), (\ref{greatB39}) and (\ref{weB54}), there exists $C'''>0$ such that 
 \begin{equation}\label{greatB45}
 \begin{split}
 		\mathbb{P}^D\big(	A\xleftrightarrow{\ge 0} y \big) \gtrsim   & R^{d-2}\cdot \mathbb{P}^D\big( A \xleftrightarrow{\ge 0 } x \big) \cdot  \\
 		&\cdot  \Big( \mathbb{P}^D\big( x\xleftrightarrow{\ge 0} y\big)\cdot \big[1- C''\mathbb{P}^D\big(	A\xleftrightarrow{\ge 0} y \big) \big]-  \tfrac{C'''[\mathrm{diam}(A)]^{d-4}}{R^{d-4}|y|^{d-2}}   \Big). 
 \end{split}
 \end{equation}
 If $\mathbb{P}^D\big(	A\xleftrightarrow{\ge 0} y \big)\ge (2C'')^{-1}$, (\ref{weB21}) holds immediately (since $\mathbb{P}^D\big( x\xleftrightarrow{\ge 0} y\big)\lesssim R^{2-d}$); otherwise, one has $1- C''\mathbb{P}^D\big(	A\xleftrightarrow{\ge 0} y \big) \ge \frac{1}{2}$, and hence (\ref{weB21}) follows from (\ref{greatB45}). We thus complete the proof of Lemma \ref{lemma_new_decompose_point_to_set}.     \qed

  \section{Proof of Lemma \ref{lemma_large_loop}} \label{app_large_loop}

  For $3\le d\le 5$, the desired bound (\ref{lemma_large_loop}) has been proved in \cite[(2.37)]{cai2024incipient}. For the high-dimensional case $d\ge 7$, we present the following lemma as a preliminary.

   
   \begin{lemma}\label{lemma_goodc1}
  	 For $d\ge 7$, there exists $C>0$ such that for any $r\ge 1$, $A,D\subset \widetilde{B}(r)$ and $R\ge Cr^{\frac{d-4}{2}}$:
  	 \begin{equation}\label{goodC2}
  	 	\begin{split}
  \mathbb{P}\big( A\xleftrightarrow{(D)} \partial B(r) \big) \lesssim 	  \big( \tfrac{R}{r} \big)^2 \cdot 	\mathbb{P}\big( A\xleftrightarrow{(D)} \partial B(R) \big).  	 	\end{split}
  	 \end{equation}
  \end{lemma}
 \begin{proof}
 	According to \cite[Proposition 1.9]{cai2024quasi}, one has: for any $z\in \partial B(R)$, 
 	\begin{equation}
 		\begin{split}
 			\mathbb{P}\big( A\xleftrightarrow{(D)} \partial B(R) \big) \asymp R^{d-4}\cdot \mathbb{P}\big( A\xleftrightarrow{(D)}z  \big). 
 		\end{split}
 	\end{equation} 
 Combined with Lemma \ref{lemma_213}, it yields the desired bound (\ref{goodC2}).  
 \end{proof}
  

  We now turn to the proof of Lemma \ref{lemma_large_loop} for $d\ge 7$. By \cite[Corollary 2.2]{cai2024high},  
 \begin{equation}\label{B1}
 	\mathbb{P}\big(\mathfrak{L}[m,M]\neq 0  \big) \lesssim \big(\tfrac{m}{M} \big)^{d-2}. 
 \end{equation}
  By the BKR inequality and (\ref{B1}), we have 
  \begin{equation}\label{concludeC5}
  \begin{split}
  	  	\mathbb{P}\big( \mathsf{F} \big):=  &	\mathbb{P}\big(  \big\{ \mathfrak{L}[m,M]\neq 0 \big\} \circ \big\{A\xleftrightarrow{(D)} \partial B(N) \big\} \big) \\
  	   \overset{}{ \lesssim  }  & \big(\tfrac{m}{M} \big)^{d-2} \cdot \mathbb{P}\big(   A\xleftrightarrow{(D)} \partial B(N) \big).
  \end{split}
  \end{equation}

 When $\{\mathfrak{L}[m,M]\neq 0, A \xleftrightarrow{(D)} \partial B(N)\}$ and $  \mathsf{F}^c$ both occur, there exists a loop $\widetilde{\ell}_\dagger\in \mathfrak{L}[m,M]$ that is pivotal for $\{A\xleftrightarrow{(D)} \partial B(N)\}$. Otherwise, the loop $\widetilde{\ell}_\dagger$ certifies ${\mathfrak{L}[m,M]\neq 0}$, while the collection of other loops certifies ${A \xleftrightarrow{(D)} \partial B(N)}$; hence $\mathsf{F}$ occurs, a contradiction. Recall that $m$ and $M$ satisfy $M\ge Cm$. Let $L^-:=C^{\frac{1}{20}}m$ and $L^+:=C^{\frac{1}{10}}m$. Since $\widetilde{\ell}_\dagger\in \mathfrak{L}[m,M]$, it must cross the annulus $B(L^+)\setminus B(L^-)$. We now employ the notations for crossing paths of this annulus. By the pivotality of $\widetilde{\ell}_\dagger$, the following event $\mathsf{F}'$ must occurs: there exist crossing paths $\eta_1$, $\eta_2$ and $\eta_3$, and clusters $\mathcal{C}_1$ and $\mathcal{C}_2$ constructed from two disjoint collections of loops in $\widetilde{\mathcal{L}}_{1/2}^{D}-\mathfrak{L}[L^-,L^+]$, such that $\mathcal{C}_1$ connects $\eta_1$ and $A$, $\mathcal{C}_2$ connects $\eta_2$ and $\partial B(N)$, and that $\eta_3$ intersects $\partial B(M)$. Thus, by (\ref{concludeC5}), it remains to show that for $d\ge 7$, 
 \begin{equation}\label{greatC5}
 		\mathbb{P}\big( \mathsf{F}' \big) \lesssim \big( \tfrac{m}{M}\big)^{d-4}\cdot \mathbb{P}\big(   A\xleftrightarrow{(D)} \partial B(N) \big).  
 \end{equation}

    \textbf{Case 1: $\mathcal{C}_1\cap \partial B(m) \neq \emptyset$ and $\eta_2\neq \eta_3$.} In this case, the cluster $\mathcal{C}_1$ certifies $A\xleftrightarrow{(D)} \partial B(m)$. In addition, $\eta_2$ and $\mathcal{C}_2$ certify the event that for some $1\le i\le \kappa$ and $\diamond \in \{\mathrm{F}, \mathrm{B}\}$, the path $\widetilde{\eta}_i^{\diamond}$ is connected to $\partial B(N)$ by $\widetilde{\mathcal{L}}_{1/2}^{D}-\mathfrak{L}[L^-,L^+]$ (we denote this event by $\mathsf{G}_i^{\diamond}$). On the event $\mathsf{G}_i^{\diamond}$, either $\widetilde{\eta}_i^{\diamond}$ intersects $\partial B(\frac{N}{2})$ (which has probability $O((\frac{m}{N})^{d-2})$), or it stays inside $\widetilde{B}(\frac{N}{2})$ and is connected to $\partial B(N)$ by $\widetilde{\mathcal{L}}_{1/2}$ (which occurs with probability $O(m^2 N^{-2})$, since $\widetilde{\eta}_i^{\diamond}$ has expected volume $O(m^2)$ and by (\ref{one_arm_high}), each point in $\widetilde{B}(\frac{N}{2})$ is connected to $\partial B(N)$ by $\widetilde{\mathcal{L}}_{1/2}$ with probability $O(N^{-2})$). Meanwhile, $\eta_3$ certifies that one of the backward crossing path intersects $\partial B(M)$, which occurs with probability $O(\kappa\cdot (\frac{m}{M})^{d-2})$. Thus, by the BKR inequality and the union bound, the probability of this case is at most proportional to 
   \begin{equation}\label{greatC6}
   	\begin{split}
  &\mathbb{P}\big( A\xleftrightarrow{(D)} \partial B(m) \big) \cdot \big[(\tfrac{m}{N})^{d-2} + m^2N^{-2} \big]  (\tfrac{m}{M})^{d-2} \cdot   \mathbb{E}\big[  \kappa^2\big] \\
  \overset{(\mathrm{\ref{decay_A10}}),\mathrm{Lemma\  \ref{lemma_goodc1}}}{\lesssim } & (\tfrac{m}{M})^{d-2}\cdot \mathbb{P}\big( A\xleftrightarrow{(D)} \partial B(N) \big). 
   	\end{split}
   \end{equation}

  \textbf{Case 2: $\mathcal{C}_1\cap \partial B(m)\neq \emptyset$ and $\eta_2=\eta_3$.} As in \textbf{Case 1}, $\mathcal{C}_1$ certifies $A\xleftrightarrow{(D)} \partial B(m)$. Moreover, $\eta_2$($=\eta_3$), which is a backward crossing path) together with $\mathcal{C}_2$ certifies that for some $1\le i\le \kappa$, the path $\widetilde{\eta}^{\mathrm{B}}_i$ intersects $\partial B(M)$ and is connected to $\partial B(N)$ by $\widetilde{\mathcal{L}}_{1/2}^{D}-\mathfrak{L}[L^-,L^+]$ (we denote this event by $\mathsf{H}_i$). Conditioned on $\widetilde{\eta}^{\mathrm{B}}_i$ intersecting $\partial B(M)$ (which occurs with probability $O((\frac{m}{M})^{d-2})$), its expected volume is $O(M^2)$, as the Brownian motion on $\widetilde{\mathbb{Z}}^d$ has two-dimensional scaling. Proceeding as in the analysis of $\mathsf{G}_i^{\diamond}$ (where we distinguish whether $\widetilde{\eta}^{\mathrm{B}}_i$ reaches $\partial B(\frac{N}{2})$), it follows that the probability of $\mathsf{H}_i$ is at most of order $(\frac{m}{N})^{d-2} +  (\frac{m}{M})^{d-2}\cdot  M^2N^{-2}$. Therefore, the probability of this case is at most proportional to 
     \begin{equation}
    	\begin{split}
    	&	\mathbb{P}\big(A\xleftrightarrow{(D)} \partial B(m)  \big)\cdot \big[ (\tfrac{m}{N})^{d-2} +  (\tfrac{m}{M})^{d-2}\cdot  M^2N^{-2} \big]\cdot  \mathbb{E}[\kappa]   \\
    		\overset{(\mathrm{\ref{decay_A10}}),\mathrm{Lemma\  \ref{lemma_goodc1}}}{\lesssim } & (\tfrac{m}{M})^{d-4}\cdot \mathbb{P}\big( A\xleftrightarrow{(D)} \partial B(N) \big). 
    	\end{split}
    \end{equation}

 \textbf{Case 3: $\eta_1\neq \eta_2$ and $\eta_2\neq \eta_3$.} In this case, $\eta_1$ together with the cluster $\mathcal{C}_1$ certifies the event $A\xleftrightarrow{(D)} \partial B(L^-)$. By (\ref{veryniceA11}), the conditional probability of this event given $\mathcal{F}_{\mathbf{yz}}$ is at most of order $(\kappa+1)\cdot \mathbb{P}\big( A\xleftrightarrow{(D)} \partial B(m) \big)$. Meanwhile, $\eta_2$ and $\mathcal{C}_2$ certify $\cup_{1\le i\le \kappa,\diamond\in \{ \mathrm{F}, \mathrm{B}\}}\mathsf{G}_i^{\diamond}$, and $\eta_3$ certifies $\cup_{1\le i\le \kappa}\{ \mathrm{ran}(\widetilde{\eta}^{\mathrm{B}}_i)\cap \partial B(M)\neq \emptyset \}$. Thus, similar to \textbf{Case 1}, the probability of this case is at most proportional to
  \begin{equation}
  \begin{split}
  	  	&\mathbb{P}\big( A\xleftrightarrow{(D)} \partial B(m) \big) \cdot \big[(\tfrac{m}{N})^{d-2} + m^2N^{-2} \big]  (\tfrac{m}{M})^{d-2} \cdot   \mathbb{E}\big[(\kappa+1) \kappa^2\big] \\
  \overset{(\mathrm{\ref{decay_A10}}),\mathrm{Lemma\  \ref{lemma_goodc1}}}{\lesssim } & (\tfrac{m}{M})^{d-2}\cdot \mathbb{P}\big( A\xleftrightarrow{(D)} \partial B(N) \big). 
  \end{split}
  \end{equation}

 \textbf{Case 4: $\eta_1\neq \eta_2$ and $\eta_2=\eta_3$.} Like in \textbf{Case 3}, $\eta_1$ and $\mathcal{C}_1$ certify $A\xleftrightarrow{(D)} \partial B(L^-)$. In addition, $\eta_2$ and $\mathcal{C}_2$ together certify the event $\mathsf{H}_i$. Thus, applying the same argument as in \textbf{Case 2}, the probability of this case is at most of order 
 \begin{equation}
 	\begin{split}
 		 	&	\mathbb{P}\big(A\xleftrightarrow{(D)} \partial B(m)  \big)\cdot \big[ (\tfrac{m}{N})^{d-2} +  (\tfrac{m}{M})^{d-2}\cdot  M^2N^{-2} \big]\cdot  \mathbb{E}[(\kappa+1)\kappa]   \\
    		\overset{(\mathrm{\ref{decay_A10}}),\mathrm{Lemma\  \ref{lemma_goodc1}}}{\lesssim } & (\tfrac{m}{M})^{d-4}\cdot \mathbb{P}\big( A\xleftrightarrow{(D)} \partial B(N) \big). 
 	\end{split}
 \end{equation}

 It remains to consider the following case.

 \textbf{Case 5: $\mathcal{C}_1\subset \widetilde{B}(m)$ and $\eta_1=\eta_2$.} If $\eta_1$($=\eta_2$) and $\eta_3$ belong to two distinct loops, then the events $\mathfrak{L}[L^-,L^+]\neq \emptyset$ and $\mathsf{F}'$ occur disjointly, since the former is certified by the loop containing $\eta_3$, whereas the latter is certified by the remaining loops. By the BKR inequality, the probability of this event is upper-bounded by 
 \begin{equation}
 		\mathbb{P}\big(\mathsf{F}' \big)  \cdot \mathbb{P}\big(\mathfrak{L}[L^-,L^+]\neq \emptyset  \big) \overset{(\mathrm{\ref{B1}})}{\lesssim } C^{\frac{2-d}{20}}\cdot \mathbb{P}\big( \mathsf{F}' \big). 
 \end{equation} 
 We now consider the subcase when $\eta_1$ and $\eta_3$ belong to the same loop. In fact, since $\eta_3$ reaches $\partial B(M)$, the path $\eta_1$ must be contained in a forward crossing path for the annulus $B(L^\star)\setminus B(L^-)$ with $L^\star:=C^{\frac{1}{5}}m$ (for this annulus, we denote by $\kappa_\star$ the number of crossings, and enumerate the crossing paths as $\{\widetilde{\eta}_{i}^{\diamond,\star}\}_{1\le i\le \kappa_\star,\diamond\in \{\mathrm{F},\mathrm{B}\}}$). In other word, there exists $1\le i\le \kappa_\star$ such that $A$ and $\partial B(N)$ are connected by $\cup(\widetilde{\mathcal{L}}_{1/2}^D-\mathfrak{L}[L^-,L^+])\cup [\mathrm{ran}(\widetilde{\eta}_{i}^{\mathrm{F},\star})\cap \widetilde{B}(L^-)]$. Moreover, the probability of this event is at most $\mathbb{P}\big(A\xleftrightarrow{(D)} \partial B(N)\big)$, since $\mathrm{ran}(\widetilde{\eta}_{i}^{\mathrm{F},\star})\cap \widetilde{B}(L^-)$ is stochastically dominated by $\cup \mathfrak{L}[L^-,L^+]$ (see \cite[Lemma 4.2]{cai2024quasi}). Meanwhile, $\mathrm{ran}(\eta_3)\cap \partial B(M)\neq \emptyset$ also implies that for some $1\le i\le \kappa_\star$, the backward crossing path $\widetilde{\eta}_{i}^{\mathrm{B},\star}$ intersects $\partial B(M)$, which occurs with probability $O((\frac{m}{M})^{d-2})$. Thus, by the BKR inequality, the probability of this subcase is at most proportional to 
 \begin{equation}\label{greatC11}
 	\begin{split}
 		\mathbb{P}\big(A\xleftrightarrow{(D)} \partial B(N)\big) \cdot (\tfrac{m}{M})^{d-2}  \cdot \mathbb{E}\big[ \kappa_\star^2 \big]  \overset{(\mathrm{\ref{decay_A10}}) }{\lesssim } (\tfrac{m}{M})^{d-2}  \cdot \mathbb{P}\big(A\xleftrightarrow{(D)} \partial B(N)\big). 
 	\end{split}
 \end{equation}

 Combining the estimates (\ref{greatC6})-(\ref{greatC11}), we obtain 
 \begin{equation}
 	\mathbb{P}\big( \mathsf{F}' \big) \lesssim C^{\frac{2-d}{20}}\cdot \mathbb{P}\big( \mathsf{F}' \big) + \big[(\tfrac{m}{M})^{d-4} + (\tfrac{m}{M})^{d-2} \big]  \cdot \mathbb{P}\big(A\xleftrightarrow{(D)} \partial B(N)\big). 
 \end{equation}
By choosing $C>0$ sufficiently large, this implies (\ref{greatC5}), which completes the proof of Lemma \ref{lemma_large_loop}.  \qed

 {\color{blue}

  }

\section{Proof of Lemma \ref{lemma_highd_volume}}\label{app_lemma_highd_volume}

 Recall that $N\ge CM$, $D\subset \widetilde{B}(cM)$, $y\in \partial B(N)$, $w_\dagger\in \partial B(\frac{N}{2})$ and $v_e^-\in \widetilde{B}_{w_\dagger}(\frac{M}{\Cref{const_boundary1}^2})$. Therefore, by (\ref{211}) one has 
\begin{equation}
	\mathbb{P}\big( v_e^- \xleftrightarrow{(D)} y  \big)  \asymp \widetilde{G}_D(v_e^- ,y) \asymp N^{2-d}. 
\end{equation}
Thus, by $\widetilde{\mathcal{L}}_{1/2}^D\le \widetilde{\mathcal{L}}_{1/2}$, it is sufficient to show that 
\begin{equation}\label{C2}
 \mathbb{I}:=\sum\nolimits_{z\in \mathbb{Z}^d} \mathbb{P}\big(y  \xleftrightarrow{} z ,y \xleftrightarrow{} v_e^- \big) \cdot |z|^{2-d} \lesssim N^{8-2d}. 
\end{equation}

On the event $\{y  \xleftrightarrow{} z , y \xleftrightarrow{} v_e^-\}$, by the tree expansion argument (see \cite[Section 3.4]{cai2024high}), there exists a loop $\widetilde{\ell}_\dagger\in \widetilde{\mathcal{L}}_{1/2}$ such that $\mathrm{ran}(\widetilde{\ell}_\dagger)\xleftrightarrow{} y$, $\mathrm{ran}(\widetilde{\ell}_\dagger)\xleftrightarrow{} z $ and $\mathrm{ran}(\widetilde{\ell}_\dagger)\xleftrightarrow{} v_e^-$ occur disjointly. In other words, there exist $w_1,w_2,w_3\in \mathbb{Z}^d$ such that $\widetilde{\ell}_\dagger$ intersects $\widetilde{B}_{j}(1)$ for $j\in \{1,2,3\}$, and that $w_1\xleftrightarrow{} y$, $w_2\xleftrightarrow{} z $ and $w_3\xleftrightarrow{} v_e^-$ occur disjointly. Thus, by the BKR inequality, (\ref{two-point1}) and (\ref{A3_loop}), we have 
\begin{equation}\label{C3}
 \begin{split}
  \mathbb{I} \lesssim & \sum\nolimits_{z,w_1,w_2,w_3\in \mathbb{Z}^d}  |w_1-w_2|^{2-d}|w_2-w_3|^{2-d}|w_3-w_1|^{2-d} \\
  &\ \ \ \ \ \ \ \ \ \ \ \ \ \ \ \ \ \  \cdot |z|^{2-d} | w_1-y|^{2-d} |w_2-z|^{2-d} |w_3-v_e^-|^{2-d}\\
   \overset{}{\lesssim } & \sum\nolimits_{w_1,w_2,w_3\in \mathbb{Z}^d} |w_1-w_2|^{2-d}|w_2-w_3|^{2-d}|w_3-w_1|^{2-d} \\
 &\ \ \ \ \ \ \ \ \ \ \ \ \ \ \ \  \cdot  | w_1-y|^{2-d} |w_2|^{4-d} |w_3-v_e^-|^{2-d}. 
   \end{split}
\end{equation}
where the second inequality follows from the following fact (see \cite[Lemma 4.3]{cai2024high}): 
 	\begin{equation}\label{2-d_2-d}
 		\sum\nolimits_{z\in \mathbb{Z}^d} |v_1-z|^{2-d}|z-v_2|^{2-d}\lesssim |v_1-v_2|^{4-d}.
 	\end{equation}

We decompose the sum on the right-hand side of (\ref{C3}) into three parts, which we denote by $\widehat{\mathbb{I}}_1$, $\widehat{\mathbb{I}}_2$ and $\widehat{\mathbb{I}}_3$, according to the following regions:
\begin{itemize}
	\item $\widehat{\mathbb{I}}_1$: restricted to $w_1 \in [B_{y}(\frac{N}{10})]^c$;

	\item $\widehat{\mathbb{I}}_2$: restricted to $w_3 \in [B_{v_e^-}(\frac{N}{10})]^c$;

	\item $\widehat{\mathbb{I}}_3$: restricted to $w_1 \in B_{y}(\frac{N}{10})$ and $w_3 \in B_{v_e^-}(\frac{N}{10})$.

\end{itemize}
In what follows, we estimate $\widehat{\mathbb{I}}_1$, $\widehat{\mathbb{I}}_2$ and $\widehat{\mathbb{I}}_3$ separately.

 \textbf{For $\widehat{\mathbb{I}}_1$.} Since $|w_1-y|\gtrsim N$, we have 
 \begin{equation}\label{C4}
 	\begin{split}
 		\widehat{\mathbb{I}}_1 \lesssim  & N^{2-d} \sum\nolimits_{w_1,w_2,w_3\in \mathbb{Z}^d} |w_1-w_2|^{2-d}|w_2-w_3|^{2-d}|w_3-w_1|^{2-d} \\
 &\ \ \ \ \ \ \ \ \ \ \ \ \ \ \ \ \ \ \ \ \ \ \ \cdot   |w_2|^{4-d}  |w_3-v_e^-|^{2-d}\\
  \overset{(\mathrm{\ref{2-d_2-d}}),(\mathrm{\ref{6-2d_2-d}})}{ \lesssim} & N^{2-d} \sum\nolimits_{w_2\in \mathbb{Z}^d}  |w_2|^{4-d} |w_2-v_e^-|^{2-d} \overset{(\mathrm{\ref{4-d_2-d}})}{ \lesssim} N^{8-2d},
 	\end{split}
 \end{equation}  
where in the last inequality we used the following bound (see \cite[(4.10)]{cai2024high}): 
 	\begin{equation}\label{4-d_2-d}
 		\sum\nolimits_{z\in \mathbb{Z}^d} |v_1-z|^{4-d}|z-v_2|^{2-d}\lesssim |v_1-v_2|^{6-d}.  
 	\end{equation}

 \textbf{For $\widehat{\mathbb{I}}_2$.} Similarly, since $|w_3-v_e^-|\gtrsim N$, we have 
  \begin{equation}\label{C5}
 	\begin{split}
 		\widehat{\mathbb{I}}_2 \lesssim  & N^{2-d} \sum\nolimits_{w_1,w_2,w_3\in \mathbb{Z}^d} |w_1-w_2|^{2-d}|w_2-w_3|^{2-d}|w_3-w_1|^{2-d} \\
 &\ \ \ \ \ \ \ \ \ \ \ \ \ \ \ \ \ \ \ \ \ \ \ \cdot | w_1-y|^{2-d}  |w_2|^{4-d}  \\
\overset{(\mathrm{\ref{2-d_2-d}}),(\mathrm{\ref{6-2d_2-d}})}{ \lesssim} & N^{2-d} \sum\nolimits_{w_2\in \mathbb{Z}^d}  |w_2|^{4-d} |w_2-y|^{2-d}  \overset{(\mathrm{\ref{4-d_2-d}})}{ \lesssim} N^{8-2d}. 
 	\end{split}
 \end{equation}

 \textbf{For $\widehat{\mathbb{I}}_3$.} In this case, one has $|w_1-w_3|\gtrsim N$. Thus, we have 
 \begin{equation}\label{C6}
 	\begin{split}
 		\widehat{\mathbb{I}}_3 \lesssim &  N^{2-d}    \sum\nolimits_{w_1,w_2,w_3\in \mathbb{Z}^d} |w_1-w_2|^{2-d}|w_2-w_3|^{2-d}| w_1-y|^{2-d} \\
 &\ \ \ \ \ \ \ \ \ \ \ \ \ \ \ \ \ \ \ \ \ \ \  \cdot   |w_2|^{4-d} |w_3-v_e^-|^{2-d}\\
   \overset{(\mathrm{\ref{2-d_2-d}})}{ \lesssim} & N^{2-d}  \cdot \big[ \mathbb{J}(B_y(\tfrac{N}{10}))+ \mathbb{J}(B(2N)\setminus B_y(\tfrac{N}{10}))+ \mathbb{J}([B(2N)]^c)   \big],  
 	\end{split}
 \end{equation}
 where we denote $\mathbb{J}(A):= \sum\nolimits_{w\in A } |w|^{4-d}|w-y|^{4-d}|w-v_e^-|^{4-d}$ for $A\subset \mathbb{Z}^d$.

Firstly, since $|w-v_e^-|\gtrsim N$ and $|w|\lesssim N$ for all $w\in B_y(\tfrac{N}{10})$, one has 
\begin{equation}\label{C7}
	\begin{split}
 		\mathbb{J}(B_y(\tfrac{N}{10})) \lesssim N^{6-d}  \sum\nolimits_{w\in \mathbb{Z}^d} |w|^{2-d}|w-y|^{4-d} \overset{(\mathrm{\ref{4-d_2-d}})}{ \lesssim} N^{12-2d}.  
	\end{split}
\end{equation}
 Secondly, since $|w-y|\gtrsim N$ and $|w|\lesssim N$ for all $w\in B(2N)\setminus B_y(\tfrac{N}{10})$, we have 
 \begin{equation}\label{C8}
	\begin{split}
 	\mathbb{J}(B(2N)\setminus B_y(\tfrac{N}{10})) \lesssim N^{6-d}  \sum\nolimits_{w\in \mathbb{Z}^d} |w|^{2-d}|w-v_e^-|^{4-d} \overset{(\mathrm{\ref{4-d_2-d}})}{ \lesssim} N^{12-2d}.  
	\end{split}
\end{equation}
Thirdly, since $|w|\asymp |w-y|\asymp |w-v_e^-|$ for all $w\in [B(2N)]^c$, one has 
\begin{equation}\label{C9}
	\mathbb{J}([B(2N)]^c) \lesssim \sum\nolimits_{w\in \mathbb{Z}^d}|w|^{12-3d}\overset{|\partial B(k)|\asymp k^{d-1}}{\lesssim} \sum\nolimits_{k\ge N}k^{11-2d} \lesssim N^{12-2d}. 
\end{equation}
Plugging (\ref{C7}), (\ref{C8}) and (\ref{C9}) into (\ref{C6}), we get 
\begin{equation}\label{C10}
	\widehat{\mathbb{I}}_3 \lesssim N^{14-3d} \lesssim N^{8-2d}. 
\end{equation}

  Combining (\ref{C3}), (\ref{C4}), (\ref{C5}) and (\ref{C10}), we obtain (\ref{C2}), and thus complete the proof of Lemma \ref{lemma_highd_volume}. \qed


\begin{thebibliography}{10}

\bibitem{benjamini1999noise}
I.~Benjamini, G.~Kalai, and O.~Schramm.
\newblock Noise sensitivity of {Boolean} functions and applications to percolation.
\newblock {\em Publications Math{\'e}matiques de l'Institut des Hautes {\'E}tudes Scientifiques}, 90(1):5--43, 1999.

\bibitem{cai2024incipient}
Z.~Cai and J.~Ding.
\newblock {Incipient infinite clusters and self-similarity for metric graph Gaussian free fields and loop soups}.
\newblock {\em arXiv preprint arXiv:2412.05709}, 2024.

\bibitem{cai2024high}
Z.~Cai and J.~Ding.
\newblock {One-arm exponent of critical level-set for metric graph Gaussian free field in high dimensions}.
\newblock {\em Probability Theory and Related Fields}, pages 1--86, 2024.

\bibitem{cai2024one}
Z.~Cai and J.~Ding.
\newblock {One-arm probabilities for metric graph Gaussian free fields below and at the critical dimension}.
\newblock {\em arXiv preprint arXiv:2406.02397}, 2024.

\bibitem{cai2024quasi}
Z.~Cai and J.~Ding.
\newblock {Quasi-multiplicativity and regularity for critical metric graph Gaussian free fields}.
\newblock {\em arXiv preprint arXiv:2412.05706}, 2024.

\bibitem{inpreparation_twoarm}
Z.~Cai and J.~Ding.
\newblock {Heterochromatic two-arm probabilities for metric graph Gaussian free fields}.
\newblock {\em arXiv preprint arXiv:2510.20492}, 2025.

\bibitem{inpreparation_gap}
Z.~Cai and J.~Ding.
\newblock {On the gap between cluster dimensions of loop soups on $\mathbb{R}^3$ and the metric graph of $\mathbb{Z}^3$}.
\newblock {\em arXiv preprint arXiv:2510.20526}, 2025.

\bibitem{cai2025minimal}
Z.~Cai, E.~B. Procaccia, and Y.~Zhang.
\newblock Minimal harmonic measure on {2D} lattices.
\newblock {\em Probability Theory and Related Fields}, pages 1--55, 2025.

\bibitem{SLECLE}
F.~Camia and C.~M. Newman.
\newblock {SLE(6) and CLE(6) from critical percolation}.
\newblock {\em arXiv:math/0611116}, 2006.

\bibitem{camia2006two}
F.~Camia and C.~M. Newman.
\newblock Two-dimensional critical percolation: the full scaling limit.
\newblock {\em Communications in Mathematical Physics}, 268(1):1--38, 2006.

\bibitem{camia2007critical}
F.~Camia and C.~M. Newman.
\newblock Critical percolation exploration path and $\mathrm{SLE}_6$: a proof of convergence.
\newblock {\em Probability theory and related fields}, 139(3):473--519, 2007.

\bibitem{chang2024percolation}
Y.~Chang, H.~Du, and X.~Li.
\newblock Percolation threshold for metric graph loop soup.
\newblock {\em Bernoulli}, 30(4):3324--3333, 2024.

\bibitem{ding2020percolation}
J.~Ding and M.~Wirth.
\newblock Percolation for level-sets of {Gaussian} free fields on metric graphs.
\newblock {\em The Annals of Probability}, 48(3):1411--1435, 2020.

\bibitem{drewitz2023arm}
A.~Drewitz, A.~Pr{\'e}vost, and P.-F. Rodriguez.
\newblock {Arm exponent for the Gaussian free field on metric graphs in intermediate dimensions}.
\newblock {\em arXiv preprint arXiv:2312.10030}, 2023.

\bibitem{drewitz2023critical}
A.~Drewitz, A.~Pr{\'e}vost, and P.-F. Rodriguez.
\newblock Critical exponents for a percolation model on transient graphs.
\newblock {\em Inventiones mathematicae}, 232(1):229--299, 2023.

\bibitem{drewitz2024cluster}
A.~Drewitz, A.~Pr{\'e}vost, and P.-F. Rodriguez.
\newblock {Cluster volumes for the Gaussian free field on metric graphs}.
\newblock {\em arXiv preprint arXiv:2412.06772}, 2024.

\bibitem{drewitz2024critical}
A.~Drewitz, A.~Pr{\'e}vost, and P.-F. Rodriguez.
\newblock {Critical one-arm probability for the metric Gaussian free field in low dimensions}.
\newblock {\em Probability Theory and Related Fields}, pages 1--24, 2025.

\bibitem{ganguly2024critical}
S.~Ganguly and K.~Jing.
\newblock Critical level set percolation for the {GFF} in $d> 6$: comparison principles and some consequences.
\newblock {\em arXiv preprint arXiv:2412.17768}, 2024.

\bibitem{ganguly2024ant}
S.~Ganguly and K.~Nam.
\newblock The ant on loops: {Alexander-Orbach} conjecture for the critical level set of the {Gaussian} free field.
\newblock {\em arXiv preprint arXiv:2403.02318}, 2024.

\bibitem{garban2010fourier}
C.~Garban, G.~Pete, and O.~Schramm.
\newblock The fourier spectrum of critical percolation.
\newblock {\em Acta mathematica}, 205(1):19--104, 2010.

\bibitem{garban2013pivotal}
C.~Garban, G.~Pete, and O.~Schramm.
\newblock Pivotal, cluster, and interface measures for critical planar percolation.
\newblock {\em Journal of the American Mathematical Society}, 26(4):939--1024, 2013.

\bibitem{kesten1986incipient}
H.~Kesten.
\newblock The incipient infinite cluster in two-dimensional percolation.
\newblock {\em Probability theory and related fields}, 73:369--394, 1986.

\bibitem{KMT_sharp}
J.~Koml\'{o}s, P.~Major, and G.~Tusn\'{a}dy.
\newblock {An approximation of partial sums of independent RV'-s, and the sample DF. I}.
\newblock {\em Probability theory and related fields}, 32(1):111--131, 1975.

\bibitem{lawler2007random}
G.~Lawler and J.~Trujillo~Ferreras.
\newblock Random walk loop soup.
\newblock {\em Transactions of the American Mathematical Society}, 359(2):767--787, 2007.

\bibitem{lawler2010random}
G.~F. Lawler and V.~Limic.
\newblock {\em Random walk: a modern introduction}, volume 123.
\newblock Cambridge University Press, 2010.

\bibitem{lupu2016loop}
T.~Lupu.
\newblock From loop clusters and random interlacements to the free field.
\newblock {\em Annals of Probability}, 44(3):2117--2146, 2016.

\bibitem{lupu2019convergence}
T.~Lupu.
\newblock Convergence of the two-dimensional random walk loop soup clusters to {CLE}.
\newblock {\em J. Eur. Math. Soc}, 21(4):1201--1227, 2019.

\bibitem{lupu2018random}
T.~Lupu and W.~Werner.
\newblock {The random pseudo-metric on a graph defined via the zero-set of the Gaussian free field on its metric graph}.
\newblock {\em Probability Theory and Related Fields}, 171:775--818, 2018.

\bibitem{morters2010brownian}
P.~M{\"o}rters and Y.~Peres.
\newblock {\em Brownian motion}, volume~30.
\newblock Cambridge University Press, 2010.

\bibitem{schramm2010quantitative}
O.~Schramm and J.~E. Steif.
\newblock Quantitative noise sensitivity and exceptional times for percolation.
\newblock {\em Annals of Mathematics}, pages 619--672, 2010.

\bibitem{sheffield2012conformal}
S.~Sheffield and W.~Werner.
\newblock {Conformal loop ensembles: the Markovian characterization and the loop-soup construction}.
\newblock {\em Annals of Mathematics}, pages 1827--1917, 2012.

\bibitem{SMIRNOV2001239}
S.~Smirnov.
\newblock {Critical percolation in the plane: conformal invariance, Cardy's formula, scaling limits}.
\newblock {\em Comptes Rendus de l'Académie des Sciences - Series I - Mathematics}, 333(3):239--244, 2001.

\bibitem{werner2021clusters}
W.~Werner.
\newblock On clusters of {B}rownian loops in $d$ dimensions.
\newblock {\em In and Out of Equilibrium 3: Celebrating Vladas Sidoravicius}, pages 797--817, 2021.

\bibitem{werner2025switching}
W.~Werner.
\newblock A switching identity for cable-graph loop soups and {Gaussian} free fields.
\newblock {\em arXiv preprint arXiv:2502.06754}, 2025.

\end{thebibliography}
\end{document}